\documentclass[10pt]{amsart} 
\usepackage{latexsym} 
\usepackage{amssymb} 
\usepackage{amscd}
\usepackage{epsf} 
\usepackage{graphicx}
\usepackage{verbatim}
\usepackage{longtable}
\usepackage{setspace}

\begin{document}

%Style Declaration \theoremstyle{plain} 
\newtheorem{Thm}{Theorem}[section] \newtheorem{TitleThm}[Thm]{} 
\newtheorem*{TitleThm1}{Theorem 1} 
\newtheorem*{TitleThm2}{Corollary 2} 
\newtheorem*{TitleThm3}{Theorem 3} 
\newtheorem*{TitleThm4}{Theorem 4} 
\newtheorem*{TitleThm5}{Theorem 5} 
\newtheorem*{TitleThm6}{Theorem 6} 
\newtheorem*{TitleThm7}{Corollary 7} 
\newtheorem*{TitleThm8}{Theorem 8} 
\newtheorem*{TitleThm9}{Theorem 9} 
\newtheorem*{TitleThm10}{Theorem 10}

\newtheorem{Corollary}[Thm]{Corollary} 
\newtheorem{Proposition}[Thm]{Proposition} 
\newtheorem{Lemma}[Thm]{Lemma} 
\newtheorem{Conjecture}[Thm]{Conjecture} 
\theoremstyle{definition} 
\newtheorem{Definition}[Thm]{Definition} 
\theoremstyle{definition} 
\newtheorem{Example}[Thm]{Example} 
\newtheorem{TitleExample}[Thm]{} 
\newtheorem{Remark}[Thm]{Remark} 
\newtheorem{SimpRemark}{Remark} 
\renewcommand{\theSimpRemark}{} 
\numberwithin{equation}{section}
\newcommand{\C}{{\mathbb C}}
\newcommand{\R}{{\mathbb R}} 
\newcommand{\Z}{{\mathbb Z}}
\newcommand{\bbP}{{\mathbb P}}

\newcommand{\flushpar}{\par \noindent}

\newcommand{\codim}{{\rm codim}\,} 
\newcommand{\proj}{{\rm proj}} 
\newcommand{\coker}{{\rm coker}\,}
\newcommand{\vol}{{\rm vol}\,}
\newcommand{\emb}{{\rm Emb}\,} 
\newcommand{\half}{\frac12} 
\newcommand{\wt}{{\rm wt}\,} 

\newcommand{\supp}{{\rm supp}\,}
\newcommand{\grad}{{\rm grad}\,}
\newcommand{\graph}{{\rm graph}\,}
\newcommand{\sing}{{\rm sing}\,}
\newcommand{\bdyM}{\partial M}
\newcommand{\Diff}{{\rm Diff}\,}
\newcommand{\intr}{{\rm int}\,}
\newcommand{\ptgG}{\partial \Gamma}
\newcommand{\Cinf}{\rm C^{\infty}} 
\newcommand{\ol}{\overline}
\newcommand{\ul}{\underline}
\newcommand{\barbdyM}{\overline{\partial M}}
\newcommand{\cirsp}[1]{\stackrel{\circ}{#1}}
\newcommand{\cirX}{\stackrel{\circ}{X}}
\newcommand{\ind}{{\rm ind}} 
\newcommand{\pd}[2]{\dfrac{\partial#1}{\partial#2}} 
\newcommand{\vnG}[1]{\mbox{vol}_{n+1}(\Gamma_t( #1 ))} 
\newcommand{\vnpG}[1]{\mbox{vol}_n(\partial \Gamma_t( #1 ))}

\def \ba {\mathbf {a}} 
\def \bb {\mathbf {b}} 
\def \bc {\mathbf {c}} 
\def \bD {\mathbf {D}}
\def \bA {\mathbf {A}}
\def \be {\mathbf {e}}
\def \bh {\mathbf {h}} 
\def \bk {\mathbf {k}}
\def \bl {\mathbf {\ell}} 
\def \bm {\mathbf {m}} 
\def \bn {\mathbf {n}} 
\def \bs {\mathbf {s}}
\def \bt {\mathbf {t}} 
\def \bu {\mathbf {u}} 
\def \bv {\mathbf {v}} 
\def \bx {\mathbf {x}} 
\def \by {\mathbf {y}}
\def \bw {\mathbf {w}} 
\def \b1 {\mathbf {1}}

\def \bzero {\boldsymbol {0}}
\def \bga {\boldsymbol \alpha} 
\def \bgb {\boldsymbol \beta} 
\def \bgg {\boldsymbol \gamma}
\def \bgW {\boldsymbol \Omega}
\def \bgth {\boldsymbol \theta}
\def \bgD {\boldsymbol \Delta}

\def \itc {\text{\it c}} 
\def \iti {\text{\it i}} 
\def \itj {\text{\it j}} 
\def \itm {\text{\it m}} 
\def \itM {\text{\it M}}  
\def \ithn {\text{\it hn}} 
\def \itt {\text{\it t}}

\def \cA {\mathcal{A}} 
\def \cB {\mathcal{B}} 
\def \cC {\mathcal{C}} 
\def \cD {\mathcal{D}} 
\def \cE {\mathcal{E}} 
\def \cF {\mathcal{F}} 
\def \cG {\mathcal{G}} 
\def \cH {\mathcal{H}} 
\def \cI {\mathcal{I}}
\def \cJ {\mathcal{J}}
\def \cK {\mathcal{K}} 
\def \cL {\mathcal{L}} 
\def \cM {\mathcal{M}} 
\def \cN {\mathcal{N}} 
\def \cO {\mathcal{O}} 
\def \cP {\mathcal{P}} 
\def \cQ {\mathcal{Q}} 
\def \cR {\mathcal{R}} 
\def \cS {\mathcal{S}} 
\def \cT {\mathcal{T}} 
\def \cU {\mathcal{U}} 
\def \cV {\mathcal{V}} 
\def \cW {\mathcal{W}} 
\def \cX {\mathcal{X}} 
\def \cY {\mathcal{Y}} 
\def \cZ {\mathcal{Z}}

\def \ga {\alpha} 
\def \gb {\beta} 
\def \gg {\gamma} 
\def \gd {\delta} 
\def \ge {\epsilon} 
\def \gevar {\varepsilon} 
\def \gk {\kappa} 
\def \gl {\lambda} 
\def \gs {\sigma}
\def \gth {\theta} 
\def \gt {\tau} 
\def \gw {\omega} 
\def \gz {\zeta} 
\def \gG {\Gamma} 
\def \gD {\Delta} 
\def \gL {\Lambda} 
\def \gS {\Sigma} 
\def \gW {\Omega}
\def \frm {\mathfrak{m}}

\def \dim {{\rm dim}\,} 
\def \mod {{\rm mod}\;}

%%%\begin{document} 

\title[Medial/Skeletal Linking Structures] {Medial/Skeletal Linking 
Structures for Multi-Region Configurations} 
\author[James Damon and Ellen Gasparovic ]{James Damon$^1$ and Ellen 
Gasparovic$^2$} 
\thanks{(1) Partially supported by the Simons Foundation grant 230298, 
the National Science Foundation grant DMS-1105470 and DARPA grant 
HR0011-09-1-0055 (2) This paper contains work from this author's Ph. D. 
dissertation at Univ. of North Carolina} 
\address{Department of Mathematics \\ 
University of North Carolina \\ 
Chapel Hill, NC 27599-3250  \\ 
USA}
\curraddr{E. Gasparovic: Dept. of Mathematics \\
 Duke University \\
Durham, NC, 27708
}

\begin{abstract} 
We consider a generic configuration of regions, consisting of a collection 
of distinct compact regions $\{ \gW_i\}$ in $\R^{n+1}$ which may be 
either regions with smooth boundaries disjoint from the others or regions 
which meet on their piecewise smooth boundaries $\cB_i$ in a generic 
way.  We introduce a skeletal linking structure for the collection of 
regions which 
simultaneously captures the regions\rq\, individual shapes and geometric 
properties as well as the \lq\lq positional geometry\rq\rq\, of the 
collection.  The linking structure extends in a minimal way the individual 
\lq\lq skeletal structures\rq\rq\, on each of the regions.  This allows us 
to significantly extend the mathematical methods introduced for single 
regions to the configuration of regions.  \par
	We prove for a generic configuration of regions the existence of a 
special type of Blum linking structure which builds upon the Blum medial 
axes of the individual regions.  As part of this, we introduce the \lq\lq 
spherical axis\rq\rq , which is the analogue of the medial axis but for 
directions.  These results require proving several transversality theorems 
for certain associated \lq\lq multi-distance\rq\rq\, and 
\lq\lq height-distance\rq\rq\, functions for such configurations.  We 
show that by relaxing the conditions on the Blum linking structures we 
obtain the more general class of skeletal linking structures which still 
capture the geometric properties. \par
The skeletal linking structure is used to analyze the \lq\lq positional 
geometry\rq\rq\, of the configuration.  This involves using the 
\lq\lq linking flow\rq\rq\, to identify neighborhoods of the configuration 
regions which capture their positional relations.  As well as yielding 
geometric invariants which capture the shapes and geometry of individual 
regions, these structures are used to define invariants which measure 
positional properties of the configuration such as: measures of relative 
closeness of neighboring regions and relative significance of the 
individual regions for the configuration.   \par 
All of these invariants are computed by formulas involving \lq\lq skeletal 
linking integrals\rq\rq\, on the internal skeletal structures 
of the regions.  These invariants are then used to construct a \lq\lq tiered 
linking graph\rq\rq,  which for given thresholds of closeness and/or 
significance, identifies subconfigurations and provides a hierarchical 
ordering in terms of order of significance.  
\end{abstract}


\maketitle

\section{Introduction}  
\label{S:sec0} \par
We consider a collection of distinct compact regions $\{ \gW_i\}$ in 
$\R^{n+1}$ with piecewise smooth generic boundaries $\cB_i$, where we 
allow the boundaries of the regions to meet in generic ways (see Figure 
\ref{fig.1a}).  For example, in 2D and 3D medical images, we encounter 
collections of objects which might be organs, glands, arteries, bones, etc.  
Researchers have already begun to recognize the importance of using the 
relative positions of objects in medical images to aid in analyzing 
physical features for diagnosis and treatment (see especially the work of 
Stephen Pizer and coworkers in MIDAG at UNC for both time series of a 
single patient and for populations of patients \cite{CP}, \cite{LPJ}, 
\cite{GSJ}, \cite{JSM}, \cite{JPR}, and \cite{Jg}) and other approaches 
such as by e.g.  Pohl et al \cite{PFL}.
\par
These physical configurations in images can be modeled by such a 
configuration of regions (see Figure \ref{fig.1b}).  Now, the geometric 
properties of the configuration are determined by both the shapes of the 
individual regions and the positions of the regions in the overall 
configuration.  The \lq\lq shapes\rq\rq of the regions capture both the 
local and global geometry as well as the topology of the regions.  The 
overall \lq\lq positional geometry\rq\rq of the configuration involves 
such information as: the measure of relative closeness of portions of 
regions, characterization of \lq\lq neighboring regions\rq\rq, and the 
\lq\lq relative significance\rq\rq of an individual region within the 
configuration.  Such properties are not captured by single numerical 
values such as the Gromov-Hausdorff distance between such 
configurations nor by invariants that would be appropriate for a collection 
of points.  
\par
\begin{figure}[ht]
\begin{center} 
\includegraphics[width=4cm]{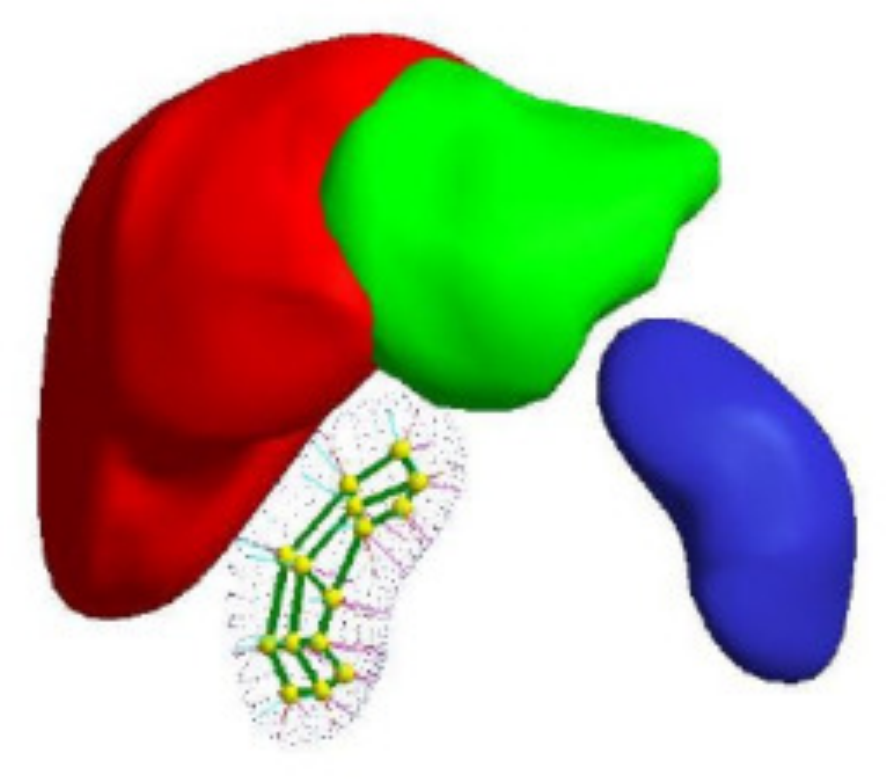}  \hspace{.25in} 
\includegraphics[width=5cm]{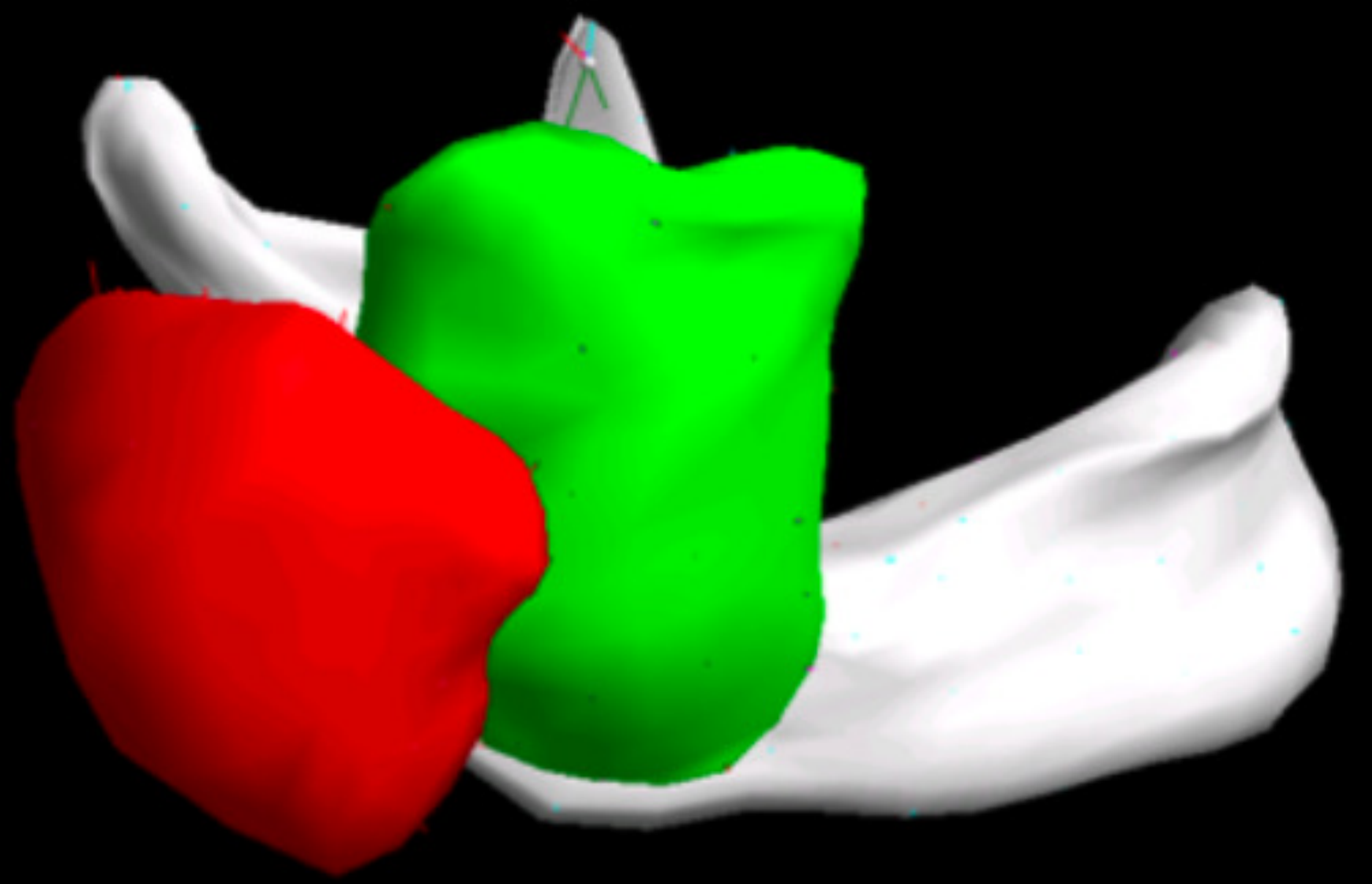}
\end{center}
\hspace{0.4in} (a) \hspace{1.8in} (b) 
 
\caption{\label{fig.1b} Examples of 3-dimensional medical images 
(obtained by MIDAG at UNC Chapel Hill) of a collection of physiological 
objects which can be modeled by a multi-region configuration.  
a)  Prostate, bladder and rectum in pelvic region and b) mandible, 
masseter muscle, and parotid gland in throat region.}  
\end{figure} 
\par

The goal of this paper is to introduce for such configurations \lq\lq 
medial and skeletal linking structures\rq\rq, which allow us to 
simultaneously capture shape properties of the individual objects and 
their \lq\lq positional geometry\rq\rq.  
\par
\begin{figure}[ht] 
\includegraphics[width=3.5cm]{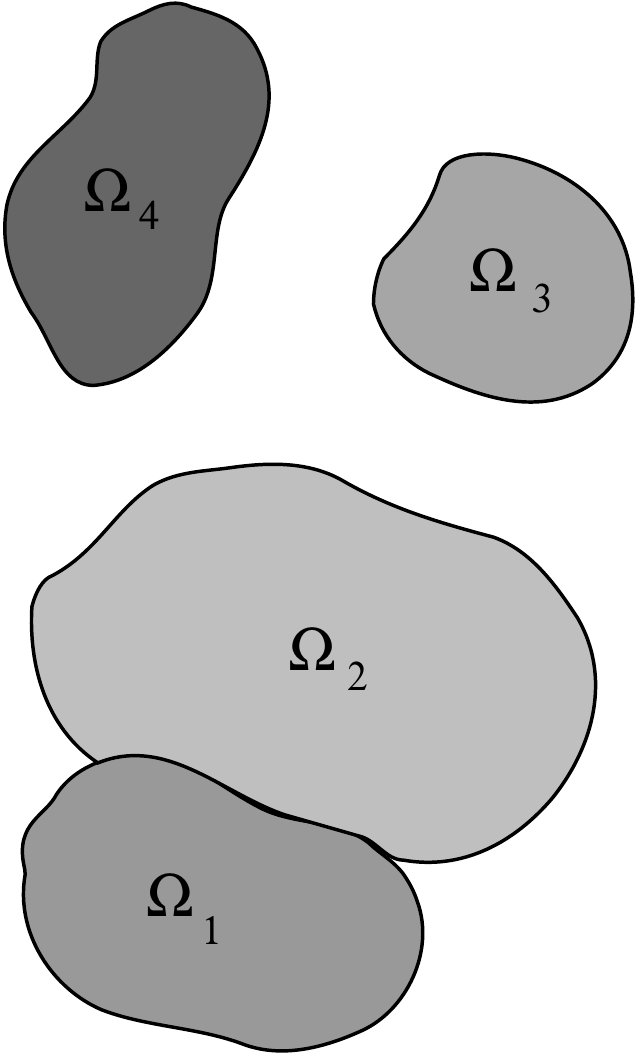} 
\caption{\label{fig.1a} Multi-region configuration in $\R^2$.} 
\end{figure} 
\par

Such structures build on earlier work in which the first author developed 
the notion of a \lq\lq skeletal structure\rq\rq for a single compact region 
$\gW$ with smooth boundary $\cB$ \cite{D1}.  It consists of a pair $(M, 
U)$, where the \lq\lq skeletal set\rq\rq $M$ is a Whitney stratified set in 
the region and $U$ is a multivalued \lq\lq radial vector field\rq\rq 
defined on $M$.  Skeletal structures generalize the notion of the Blum 
medial axis of a region with smooth boundary \cite{BN} (also called the 
\lq\lq central set\rq\rq, see \cite{Y}), which is the locus of centers of 
spheres in $\gW$ tangent to $\cB$ at two or more points (or having a 
single degenerate tangency).  The Blum medial axis is a special type of 
skeletal structure (with $U$ consisting of the vectors from points of $M$ 
to the points of tangency).  \par 
The Blum medial axis $M$ captures the shape of a region.  It has several 
alternate descriptions as the shock set of the \lq\lq grassfire/eikonal 
flow\rq\rq from the boundary as in Kimia-Tannenbaum-Zucker \cite{KTZ} 
and as the Maxwell set of the family 
of \lq\lq distance to the boundary functions\rq\rq, see Mather \cite{M2}.  
These multiple descriptions have allowed for the classification of the 
local structure of $M$ for regions with generic boundaries $\cB$ \cite{Y}, 
\cite{M2}, \cite{Gb}, \cite{GK}.   In addition, these have led to several 
different methods for computing the medial axis using properties of the 
grassfire/eikonal flow \cite{BSTZ}, \cite{DDS} and Voronoi methods (see 
\cite{PS} for a survey of these methods) and a recent method using 
B-spline representations to directly evolve the medial axis \cite{MCD}. 
\par
The skeletal structure relaxes several of the conditions in the Blum case, 
and allows more flexibility in applying skeletal structures to model 
objects including: using skeletal structures as deformable templates for 
modeling objects [P], overcoming the lack of $C^1$-stability of the Blum 
medial axis, allowing alternate models based on a region being swept out 
by a family of hyperplanes, \cite{D7} and the related \cite{GK2}, and 
providing discrete models to which statistical analysis can be applied 
\cite{P2}, \cite{PJG}.  This has enhanced their usefulness for modeling and 
computer imaging questions for medicine and biology (see e.g. \cite{PS} 
for a survey of results).  \par 
Furthermore, the structure enables both the local, relative, and global 
geometric properties of the region and its boundary to be computed from 
the \lq\lq medial geometry\rq\rq of the radial vector field on the skeletal 
structure \cite{D2}, \cite{D3}, \cite{D4}.  This includes conditions 
ensuring the nonsingularity of a \lq\lq radial flow\rq\rq from the skeletal 
set to the boundary.  This allows the region to be fibered with the level 
sets of the flow and implies the smoothness of the boundary \cite[Thm 
2.5]{D1}.  As the Blum medial axis is a special type of skeletal structure 
these results also apply to it, and to the related 
\lq\lq symmetry-set\rq\rq \cite{BGG}.  \par
\par 
In this paper, we introduce the \lq\lq medial and skeletal linking 
structures\rq\rq\, which build upon individual skeletal structures of the 
regions in a minimal way, but still enable us to analyze the positional 
geometry of the configuration along with the shapes of the individual 
regions.  The added structure consists of a multivalued \lq\lq linking 
function\rq\rq $\ell_i$ defined on the skeletal set $M_i$ for each region 
$\gW_i$ and a refinement of the Whitney stratification of $M_i$ on which 
it is stratawise smooth.  The linking functions $\ell_i$ together with the 
radial vector fields $U_i$ yield multivalued \lq\lq linking vector 
fields\rq\rq $L_i$, which satisfy certain linking conditions.  Even though 
the structures are defined on the skeletal sets within the regions, the 
linking vector fields allow us to capture geometric properties of the 
external region as well.  In addition, we identify the regions which are 
unlinked and classify their local generic structure by introducing the {\it 
spherical axis} which is the analog of the Blum medial axis but for the 
family of height functions on the region boundaries.  \par
The paper is divided into four parts.  In Part I, we define and give the 
basic properties of medial/skeletal linking structures and state two 
theorems assuring the existence of a skeletal linking structure for 
generic configurations.  First, for a general generic multi-region 
configuration with singular shared boundaries, we establish the existence 
of a generic \lq\lq full Blum linking structure\rq\rq for the configuration 
(Theorem \ref{Thm3.6}).  In the special case where all of the regions are 
disjoint with smooth generic boundaries, this yields a \lq\lq Blum medial 
linking structure\rq\rq (Theorem \ref{Thm3.5}).  This special case for 
disjoint regions was obtained in the thesis of the second author \cite{Ga}.  
Then, in the case of a generic configuration with singular boundaries, the 
Blum structure now contains the singular points of the boundaries in its 
closure and we give an \lq\lq edge-corner normal form\rq\rq near such 
singular points.  This requires providing an addendum (Theorem 
\ref{Thm3.2}) to the result of Mather \cite{M2} for a single region, where 
we now allow a singular boundary.  In Section \ref{S:sec8} we explain how 
to modify the resulting Blum structure to obtain a skeletal linking 
structure.  \par 
In Part II, we use the linking structure to determine properties of the  
\lq\lq positional geometry\rq\rq of the configuration.  We introduce and 
compute several invariants of the positional geometry, and deduce 
properties of these invariants.  These use the linking flow (which extends 
the radial flow into the external region) and allows the determination of 
the linking neighborhoods between regions.  The nonsingularity of the 
linking flow will follow from {\it linking curvature conditions} on the 
linking functions, having a form analogous to those given in \cite[Thm 
2.5]{D1} for the radial flow.  \par 
This allows us to introduce invariants which include measures of relative 
closeness and positional significance of the individual regions for the 
configuration.  These are given in terms of volumetric measurements on 
the regions themselves and on associated regions defined by the linking 
flow.  The measure of significance allows us to identify the central 
regions as well as outliers from among the regions.  We prove that using 
the linking structure, we can compute these volumetric invariants (which 
involve regions outside the configuration) as \lq\lq skeletal linking 
integrals\rq\rq on the internal skeletal sets.  These invariants are then 
used to construct a \lq\lq tiered linking graph\rq\rq.  When given 
thresholds of closeness and significance are applied to this graph, they 
yield subgraph(s) identifying subconfigurations and provide a hierarchical 
ordering of the regions.  The skeletal linking structure also allows for the 
comparison and statistical analysis of collections of objects in 
$\R^{n+1}$, extending the analyses given in earlier work for single regions. 
\par
In Parts III and IV, we prove the existence and derive the generic 
properties of the (full) Blum linking structure.  The properties of the Blum 
structure result from generic properties of several associated 
\lq\lq multi-functions\rq\rq which include \lq\lq multi-distance 
functions\rq\rq and \lq\lq height-distance functions\rq\rq.  Because 
these latter functions are examples of divergent diagrams of functions, 
the usual theorems of singularity theory do not apply (see e.g. \cite{DF}).  
Nonetheless, in Part III we prove a multi-transversality theorem (Theorem 
\ref{Thm5.1}) which applies to the multi-functions relative to a \lq\lq 
partial multijet space\rq\rq.  This yields the generic properties for the 
Blum linking structures for an open dense set of the space of embeddings 
of configurations of each given type.  This transversality theorem extends 
earlier transversality theorems for families of functions due to Looijenga 
\cite{L} and Wall \cite{Wa} and is based on a \lq\lq hybrid transversality 
theorem\rq\rq (Theorem \ref{Thm9.2}) using results from \cite{D5}.  \par 
In Part IV we construct the families of perturbations needed for applying 
the transversality theorems.  We then carry out the necessary derivative 
computations needed to prove the applicability of the transversality 
theorems to the space of embeddings of a configuration.  This allows us to 
deduce for an open dense set of mappings the existence of the Blum linking 
structures with their generic properties.  \par
The authors would like to thank Stephen Pizer for sharing with us his 
early work with his coworkers on medical imaging involving multiple 
objects in medical images.  This led us to seek a completely mathematical 
approach to these problems.  We also are very grateful to the several 
referees for their careful reading and multiple suggestions for improving 
the exposition in the paper, and to the editor Alejandro Adem for his 
considerable help in moving the process forward.  Hopefully the final 
version reflects all of this.
\par
\subsection*{Overview of the Genericity and Transversality Results} 
\hfill 
\par
To establish the results of this paper for the generic properties of the 
geometric structures, we will carry out extensions of earlier work of 
Mather \cite{M2}, Looijenga \cite{L}, and Wall \cite{Wa}.  We indicate 
exactly the form that these extensions will ultimately take.  \par
First, in the Blum case, rather than consider the disjoint Blum medial axes 
for different regions, we consider the \lq\lq generic linking 
properties\rq\rq\, for the distinct regions.  This forces us to consider the 
interplay between the stratifications on the boundaries that arise from 
the individual Blum medial axes and the stratification resulting from the 
family of height functions.  These interactions result from having two 
distance functions or a distance function and a height function at the 
same point.  This problem already arises in the case of distinct regions 
with smooth boundaries as the linking occurs via the complementary 
region.  To handle this situation we introduce transversality theorems for 
multi-functions, which will yield the generic interplay between the 
stratifications.  These \lq\lq hybrid transversality theorems\rq\rq\, 
allow us to prove transversality for the multijets of such 
multi-functions relative to partial jet spaces, which are subbundles of jet 
bundles.  We apply these theorems in the context of continuous mappings 
from Baire spaces of embeddings of configurations to the spaces of 
parametrized families of functions.  These theorems extend earlier 
relative and absolute transversality theorems in \cite{D5}. \par
Second, Mather\rq s results for the Blum medial axis concentrated on the 
local structure of the Blum medial axis by using a multi-germ versality 
theorem.  This by itself does not imply anything about the corresponding 
properties of points on the boundaries corresponding to the medial axis 
points.  Several partial results were obtained by Porteous \cite{Po} and 
Bruce-Giblin-Tari \cite{BGT} from the point of view of the geometry of 
the boundary as a surface.  We address this by establishing a general 
result for the resulting stratifications of the boundary by the \lq\lq 
generic linking type\rq\rq\, of the points.  We also apply the 
transversality theorem of Wall for the family of height functions and our 
extension for \lq\lq height-distance\rq\rq\, functions to give the generic 
properties of the \lq\lq spherical axis\rq\rq, the resulting properties of 
the stratification of the unlinked region, and its relation with the Blum 
stratifications on the boundaries.  \par
Third, one of the principal extensions is to collections of regions allowing 
boundaries and corners where the regions may share portions of their 
boundaries allowing specific generic local forms.  The methods we develop 
allow us to include these nonsmooth features in our analysis.  
For the global theory we prove special versions of the transversality 
theorem to overcome the problem on stratified sets in several ways.  This 
depends upon replacing the Seeley extension theorem \cite{Se} used by 
Mather with a more general extension theorem due to Bierstone \cite{Bi}.  
This then extends traditional transversality theorems so they can apply to 
this situation.  One consequence is to provide an addendum to Mather\rq s 
theorem on the local generic form of the Blum medial axis for a region 
with generic smooth boundary to the case where the region has a generic 
boundary with corners.  \par
Fourth, genericity proved from the multi-transversality theorems only 
yields the results for a residual set of mappings of configurations.  By 
contrast, Mather \cite{M2} asserts that the generic set of embeddings for 
the Blum medial structure form an open set, although he does not prove it 
in his paper.  We give a treatment in our general case to prove that the set 
of smooth embeddings of configurations which exhibit the generic linking 
properties forms an open set in the space of smooth embeddings (and 
hence smooth mappings).  We do so by relating the versality of the 
distance and height functions with the infinitesimal stability of 
associated mappings, and then applying Mather\rq s general theorem 
\lq\lq infinitesimal stability implies stability\rq\rq \cite{M5}.

%%%  \newpage
 \vspace{2ex} 
\centerline{CONTENTS}
 \vspace{2ex}
\begin{itemize}

\item[Part I] Medial/Skeletal Linking Structures
\vspace{2ex}

\item[(2)]	Multi-Region Configurations in $\R^{n+1}$
 \par \vspace{1ex}
\hspace{2em}Local Models for Regions at Singular Points on Boundary
\par
\vspace{1ex}
\hspace{2em}The Space of Equivalent Configurations via Mappings of a 
Model
\par
\vspace{1ex}
\hspace{2em}Configurations allowing Containment of Regions
\par
\vspace{2ex}

\item[(3)]	Skeletal Linking Structures for Multi-Region Configurations in 
$\R^{n+1}$ 
\par 
\vspace{1ex}
\hspace{2em} Skeletal Linking Structures for Multi-Region Configurations
\par 
\vspace{1ex}
\hspace{2em} Linking between Regions and between Skeletal Sets
%%%  \par \vspace{1ex}
%%%  \hspace{2em} Special Properties of Blum Linking Structure 
\par
\vspace{2ex}
\item[(4)]	Blum Medial Linking Structure for a Generic Multi-Region 
Configuration
\par \vspace{1ex}
\hspace{2em}Blum Medial Axis for a Single Region with Smooth Generic  
\par 
\hspace{2em}Boundary \par
\vspace{1ex}
\hspace{2em}Addendum to Generic Blum Structure for a Region with 
\par 
\hspace{2em}Boundaries and Corners
\par
\vspace{1ex}
\hspace{2em}Blum Medial Linking Structure \par
\vspace{1ex}
\hspace{2em}Generic Linking Properties 
\par
\vspace{1ex}
\hspace{2em}Existence of Blum Medial Linking Structure 
\vspace{2ex}
 \par
\hspace{2em}Addendum: Classification of Linking Types for Blum Medial 
\par
 \hspace{2em}Linking Structures in $\R^3$
\par
\vspace{2ex}
\item[(5)]	Retracting the Full Blum Medial Structure to a Skeletal Linking 
Structure \par
\vspace{2ex}

\item[Part II] Positional Geometry of Linking Structures 
 \par
\vspace{2ex}
\item[(6)] Questions Involving Positional Geometry of a Multi-region 
Configuration 
\par
\vspace{1ex}
\hspace{2em}Introduction
\par
\vspace{2ex}
\item[(7)] Shape Operators and Radial Flow for a Skeletal Structure
\par
\vspace{1ex}
\hspace{2em}The Radial Flow
\par
\vspace{1ex}
\hspace{2em}Radial and Edge Shape Operators
\par
\vspace{1ex}
\hspace{2em}Curvature Conditions and Nonsingularity of the Radial Flow 
\par
\vspace{1ex}
\hspace{2em}Evolution of the Shape Operators under the Radial Flow
\vspace{2ex}
\item[(8)] Linking Flow and Curvature Conditions  
\par
\vspace{1ex}
\hspace{2em}Nonsingularity of the Linking Flow 
\par
\vspace{1ex}
\hspace{2em}Special M\"{o}bius Transformations of Matrices and 
Operators
\par
\vspace{1ex}
\hspace{2em}Evolution of the Shape Operators under the Linking Flow
 \par
\vspace{2ex}
\item[(9)] Properties of Regions Defined Using the Linking Flow \par
\vspace{1ex}
\hspace{2em}Medial/Skeletal Linking Structures in the Unbounded Case
\par
\vspace{1ex}
\hspace{2em}Medial/Skeletal Linking Structures for the Bounded Case 
\par
\vspace{2ex}
\item[(10)] Global Geometry via Medial and Skeletal Linking Integrals
\par
\vspace{1ex}
\hspace{2em}Defining Medial and  Skeletal Linking Integrals
\par
\vspace{1ex}
\hspace{2em}Computing Boundary Integrals via Medial Linking Integrals
\par
\vspace{1ex}
\hspace{2em}Computing Integrals as Skeletal Integrals via the Linking 
Flow 
\par
\vspace{1ex}
\hspace{2em}Skeletal Linking Integral Formulas for Global Invariants
\par
\vspace{2ex}
\item[(11)] Positional Geometric Properties of Multi-Region 
Configurations \par
\par
\vspace{1ex}
\hspace{2em}Neighboring Regions and Measures of Closeness
\par
\vspace{1ex}
\hspace{2em}Measuring Significance of Objects Via Linking Structures
\par
\vspace{1ex}
\hspace{2em}Properties of Invariants for Closeness and Significance
\par
\vspace{1ex}
\hspace{2em}Tiered Linking Graph
\par
\vspace{1ex}
\hspace{2em}Higher Order Positional Geometric Relations via Indirect 
Linking
\par
\vspace{2ex}

\item[Part III] Generic Properties of Linking Structures via Transversality 
Theorems \par
\vspace{2ex}

\item[(12)]	Multi-Distance and Height-Distance Functions and Partial 
Multijet Spaces 
\par 
\vspace{1ex}
\hspace{2em}Multi-Distance and Height-Distance Functions
\par 
\vspace{1ex}
\hspace{2em}Partial Jet Spaces for Multi-Distance and Height-Distance 
\par 
\hspace{2em}Functions
\par 
\vspace{2ex}
\item[(13)]	Generic Blum Linking Properties via Transversality Theorems
\par  
\vspace{1ex}  
\hspace{2em}Transversality Theorem for Multi-Distance and 
Height-Distance   \par 
\hspace{2em}Functions \par
\vspace{1ex}
\hspace{2em}Strata for Generic Properties of Blum Linking Structure 
\par 
\vspace{2ex}
\item[(14)]	Generic Properties of  Blum Linking Structures
\par  
\vspace{1ex}  
\hspace{2em} Properties of Transversality and Whitney Stratified Sets
\par  
\vspace{1ex}  
\hspace{2em} Consequences of Transversality for Multi-Distance 
Functions
\par  
\vspace{1ex}  
\hspace{2em} Consequences of Transversality for Height-Distance 
Functions
\par  
\vspace{1ex}  
\hspace{2em} Proof of Theorem \ref{Thm3.5} for a Residual Set of 
Embeddings
\par
\vspace{2ex}
\item[(15)]	Concluding Generic Properties of Blum Linking Structures 
\par 
\vspace{1ex}
\hspace{2em} Blum Medial Structure near Corner Points
\par 
\vspace{1ex}
\hspace{2em} Openness of the Genericity Conditions \par
\par 
\vspace{2ex}

\item[Part IV] Proofs and Calculations for the Transversality Theorems
\vspace{2ex}

\item[(16)]	Reductions of the Proofs of the Transversality Theorems
\par 
\vspace{1ex}
\hspace{2em}Hybrid Transversality Theorem
\par 
\vspace{1ex}
\hspace{2em}Multi-Jet Properties Implying Stratification Properties 
\par
\hspace{2em}on the Boundaries
\par
\vspace{2ex}
\item[(17)]  Families of Perturbations and their Infinitesimal Properties
\par 
\vspace{1ex}
\hspace{2em} Construction of the Families of Perturbations
\par 
\vspace{1ex}
\hspace{2em}Computation of Derivatives for Families of Perturbations
\par 
\vspace{1ex}
\hspace{2em}Jet Properties of Parameter Deformations
\par
\vspace{2ex}
\item[(18)]  Completing the Proofs of the Transversality Theorems
\par 
\vspace{1ex}
\hspace{2em}Perturbation Transversality Conditions via  Fiber Jet 
Extension Maps 
\par 
\vspace{1ex}
\hspace{2em}Computing Derivatives of the Multi-functions at Critical 
Points
\par 
\vspace{1ex}
\hspace{2em}Concluding the Proofs
\par
\vspace{1ex}
\item[(19)] Appendix:  List of Frequently Used Notation
\end{itemize}

\newpage

\section{Multi-Region Configurations in $\R^{n+1}$}  
\label{S:sec1}
%%% \label{Sec:Reg.config}
\par
\subsection*{Local Models for Regions at Singular Points on Boundary}
We begin by defining what exactly we mean by a \lq\lq multi-region 
configuration\rq\rq.  First, we consider compact connected $n+1$-
dimensional regions $\gW \subset \R^{n+1}$ which are {\it smooth 
manifolds with boundaries and corners}, with boundaries denoted by $\cB$.  
We recall that $\gW$ is a manifold with boundaries and corners if each 
point $x \in \cB$ has a neighborhood $U \subset \gW$ diffeomorphic, with 
$x \mapsto 0$, to an open subset of $C_k = \R^k_{+} \times \R^{n+1-k}$, 
for some $0 \leq k \leq n$.  We refer to such $x$ as a {\it  
$k$-edge-corner point}.  Here $\R^k_{+} = \{ (x_1, \dots , x_k) \in \R : x_i 
\geq 0\}$. Then $\cB$ is stratified by the strata consisting of 
$k$-edge-corner points $x \in \cB$.  Those strata of dimension $n$ will be 
referred to as the {\it open (or regular) strata of the boundary}.  \par
Second, we require that the regions satisfy the {\it boundary intersection 
condition:} if two such regions intersect, they do so only on their 
boundaries; and their common intersection is a union of strata (defined 
above).  Third, the regions satisfy the {\it boundary edge condition:} if a 
point $x$ is contained in more than one region then the union of the 
boundaries of the regions containing $x$ is locally diffeomorphic in a 
neighborhood of $x$ to one of the following regions $P_k$ or $Q_k$ in 
$\R^{n+1}$ for $k = 1, \dots n+1$.  \par
For $y = (y_1, \dots , y_{k+1}) \in \R^{k+1}$, we let $g_{k+1}(y) = \sum_{i 
= 1}^{k+1} y_i$.  Then, we may identify $\R^{n+1}$ with $L_k \times 
\R^{n+1-k}$ where $L_k = \{ y \in \R^{k+1} : g_{k+1}(y) = 0\}$.  We then 
define
\begin{itemize}
\item[i)] 
$P_k = Y_k \times \R^{n+1-k}$, where   
$$Y_k = \{ y \in L_{k} : \text{for some } i \neq j,\,  y_i = y_j \leq y_{\ell} 
\text{ for } \ell \neq i, j\}  $$
and for $1 \leq k \leq n$, 
\item[ii)] $Q_k = Z_{n+1} \cup (H_{n+1} \cap P_k)$, 
\end{itemize}
%%%  for $k = 1, \dots n+1$, 
where for $(y, x) \in L_k \times \R^{n+1-k}$ with $x = (x_{k+1}, \dots , 
x_{n+1}) \in \R^{n-k+1}$, $Z_{n+1}$ is the hyperplane defined  by $x_{n+1} 
= 0$, and $H_{n+1}$ is the half-space defined by $x_{n+1} \geq 0$. 
\par
\par
The local forms for $\R^3$ consist of a smooth surface together with 
those shown in Figure \ref{fig.2b}.
\begin{figure}[ht]
\begin{center} 
\includegraphics[width=10cm]{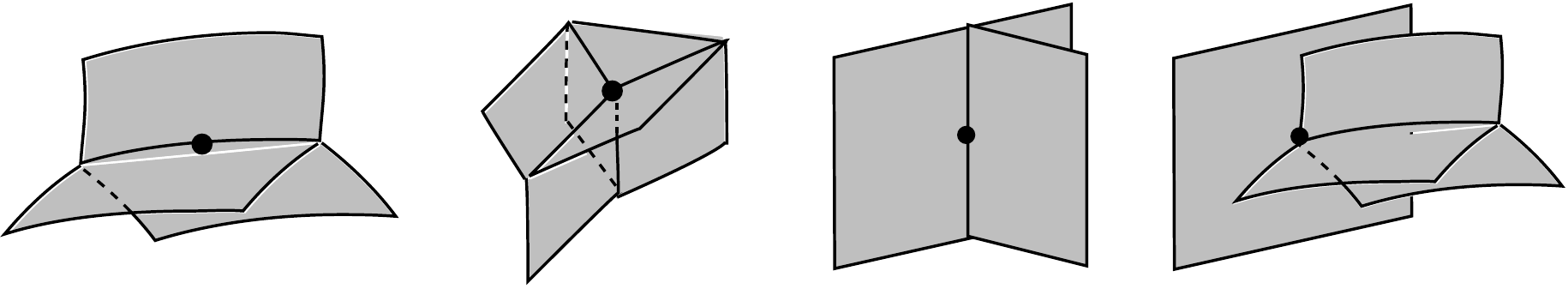} 
\end{center} 
\hspace{0.4in}(a) \hspace{0.8in} (b) \hspace {0.6in} (c) \hspace{0.8in} (d)
\caption{\label{fig.2b} Generic local forms for intersecting regions in 
$\R^3$:  a) $P_2$; b) $P_3$; c) $Q_1$; and d) $Q_2$.}  
\end{figure} 
\par
We note that $P_1$ denotes a smooth boundary region of each of two 
regions, so a point of the common boundary region will be in the closure of 
the open strata unless it is a point where more than two regions 
(including the exterior) meet.  For $k > 1$, a $P_k$-point of a region 
$\gW_i$ will be a singular point of the boundary, and hence lies in the 
singular set of $\cB_i$.  \par 
By contrast, there are two possibilities for $Q_k$-points.  First, a region 
may have for its boundary the hypersurface $Z_{n+1}$ in the local model at 
a $Q_k$-point.  We refer to the point as a {\it smooth 
$Q_k$-point}.  Second, the boundary in the local model at a $Q_k$-point 
may be formed from both part of $Z_{n+1}$ and a face from $P_k$ in the 
local model.  The $Q_k$-point will be a singular point of its boundary, and 
we refer to it as a {\it singular $Q_k$-point}.  For a region $\gW_i$, the 
set of smooth  $Q_k$-points for all $k$ defines a compact Whitney 
stratified set $\gS_{Q\, i}$.  This locally consists of the intersection of 
$Z_{n+1}$  with the other faces.  It is contained in the set of smooth 
points of the boundary $\cB_i$.  We let $\gS_{Q} = \cup_{i} \gS_{Q\, i}$.  

\flushpar
\begin{Remark}
\label{Rem.2.1}
For the local configurations in a) and b) in Figure \ref{fig.2b}, one of the 
complementary regions may denote the external complement of the union 
of regions.  For c) and d), only one of the regions to the right of the plane 
may be a complementary external region.  Physically such local 
configurations of type a) or b) arise when regions with flexible boundaries 
have physical contact.  This includes the singularities in soap bubbles.  For 
c) and d), the region to the left of the plane would represent a rigid region 
with which one or more flexible regions have contact.  
\end{Remark}
\par
\begin{Definition}
\label{Def.1.1}
%%% \label{Def.mul.reg}
A {\em multi-region configuration} consists of a collection of compact 
$(n + 1)$-dimensional regions $\gW_i \subset \R^{n+1}$, $i=1, \dots , m$, 
which have smooth boundaries with corners, with boundaries denoted by 
$\cB_i$, and satisfying the boundary intersection condition and the 
boundary edge condition. 
\end{Definition}
\par
  For a given configuration of regions $\bgW = \{ \gW_i\}_{i = 1}^{m}$, the 
union $\cup_{i =1}^{m} \gW_i$ has a Whitney stratification consisting of 
the interiors of the $\gW_i$ together with the strata of the boundaries of 
the $\gW_i$ formed from the $k$-edge-corner points for each given $k$, 
and their complement consisting of smooth points of each boundary (for 
detailed treatments of Whitney stratifications and their properties see 
e.g. \cite{Wh}, \cite{M1}, \cite{M3}, or \cite{GLDW}).  Also, we let 
$\gW_0$ denote the closure of the complement $\R^{n+1} \backslash 
\cup_{i =1}^{m} \gW_i$.

\subsection*{The Space of Equivalent Configurations via Mappings of a 
Model}
\par
In order to consider generic configurations, we describe how we will 
deform such a configuration of regions preserving the form of the 
intersections via mappings of the configuration.   Let $\bgD$ be a 
configuration  of multi-regions $\{\gD_i\}$ (in $\R^{n+1}$) satisfying 
Definition \ref{Def.1.1}.  
\begin{Definition}
\label{Def1.1b}
A {\em multi-region configuration} $\bgW$ {\em based on model 
configuration} $\bgD$ is given by a smooth embedding $\Phi : \bgD \to 
\R^{n+1}$ which restricts to diffeomorphisms $\gD_i \simeq \gW_i$ for 
each $i$.  Multi-region configurations with model configuration $\bgD$ 
will be said to be {\em configurations of the same type} as $\bgW$.  \par
The {\it space of configurations of type $\bgD$} is given by the space of 
embeddings $\emb (\bgD, \R^{n+1})$ with the $C^{\infty}$-topology.  
\end{Definition}
For such a configuration, we let the boundary of $\gD_i$ be denoted by 
$X_i$ and then $\cB_i = \Phi(X_i)$ is the boundary of $\gW_i$.  
%%% We also let $\gD_0$ denote the closure of the complement $\R^{n+1} 
%%% \backslash \cup_{i} \gD_i$.  
Even though the configuration varies with $\Phi$ we still use the notation 
$\bgW$ for the resulting (varying) image configuration (with a specific 
$\Phi$ understood).  We have the stratification of $\bgD$ as described 
above and by applying $\Phi$ we obtain a corresponding stratification of 
$\bgW$.  \par  
We note that the union of the regions can be viewed as a manifold with 
boundaries and corners except edges and corners are inward pointing.  
Nonetheless, we shall see that the basic properties that are used for 
mappings on manifolds with boundaries and corners will still be valid. 
\par 
\begin{figure}[ht]
\includegraphics[width=6cm]{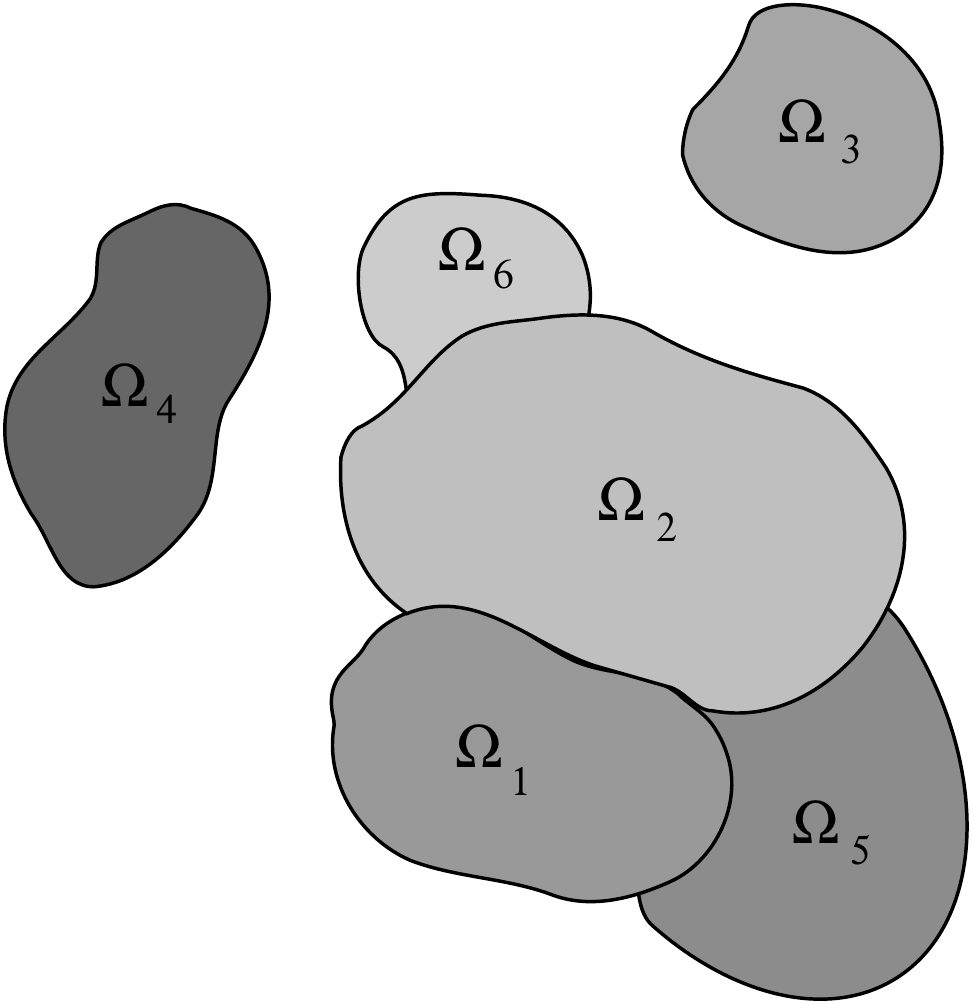}
\caption{\label{fig.1.4} A multi-region configuration in $\R^2$.  The 
stratification consists of the interiors of the regions, the open boundary 
curves, excluding branch points which are either of type $P_2$ or $Q_1$.}
\end{figure} 
\par
\begin{Definition}
\label{Def1.5}  The regions $\gD_i$ and  $\gD_j$ are  {\it adjoining 
regions} if $X_{i} \cap X_{j} \neq \emptyset$.  If $X_{i} \backslash \cup_{j 
\neq i} X_j \neq \emptyset$, we say $\gD_i$ {\it adjoins the complement}.  
These then have the corresponding meanings for the image regions 
$\gW_i$, $\gW_j$ and $\gW_0$.
\end{Definition} 
\par
\begin{Remark}
\label{Rem1.6}
In the special case of a configuration consisting of disjoint regions with 
smooth boundaries, 
%%% only the $X_{0 i} \neq \emptyset$, and $X_i = X_{0 i}$.  
each region is thus only adjoined to the complement.  Then, we shall see 
that the geometric relations between the regions will be captured only via 
linking behavior in the external region.  
\end{Remark}
\par
As all of the $\gD_i$ are compact, $\Phi(\bgD)$ is compact and hence 
bounded.  However, we introduce a stronger notion of being bounded.  If we 
are given an ambient region $\tilde \gW$ so that $\gW_i \subset \intr 
(\tilde \gW)$ for each $i$, then we say that $\bgW$ is a {\it configuration 
bounded by} $\tilde \gW$.  Then we may either consider bounded 
configurations given by an embedding $\tilde \Phi : \tilde \gD \to 
\R^{n+1}$, with $\tilde \Phi(\tilde \gD)$ denoting $\tilde \gW$, and  $\Phi 
= \tilde \Phi | \bgD$; or fix $\tilde \gW$ and consider embeddings $\Phi : 
\bgD \to \intr (\tilde \gW)$.  Such a $\tilde \gW$ might be a bounding box 
or disk or an intrinsic region containing the configuration, for examples 
see \S \ref{SII:sec.int}.  \par
\subsection*{Configurations allowing Containment of Regions} 
%%%  \hfill 
\par
In our definition of a multi-region configuration, we have explicitly 
excluded one region being contained in another.  However, given a 
configuration which allows this, we can easily identify such a 
configuration with the type we have already given.  To do so, we require 
that the boundaries of two regions still intersect in a union of strata 
which form the closure of strata of dimension $n$.  Then, if one region is 
contained in another  $\gW_i \subset \gW_j$, then we may represent 
$\gW_j$ as a union of two regions $\gW_i$ and the closure of $\gW_j 
\backslash (\intr(\gW_i) \cup (\cB_i \cap \cB_j))$, which we refer to as 
the {\it region complement} to $\gW_i$ in $\gW_j$.  By repeating this 
process a number of times we arrive at a representation of the 
configuration as a multi-region configuration in the sense of Definition 
\ref{Def.1.1}.  See Figure \ref{fig.1.5}.
\par
\begin{figure}[ht]
\includegraphics[width=7cm]{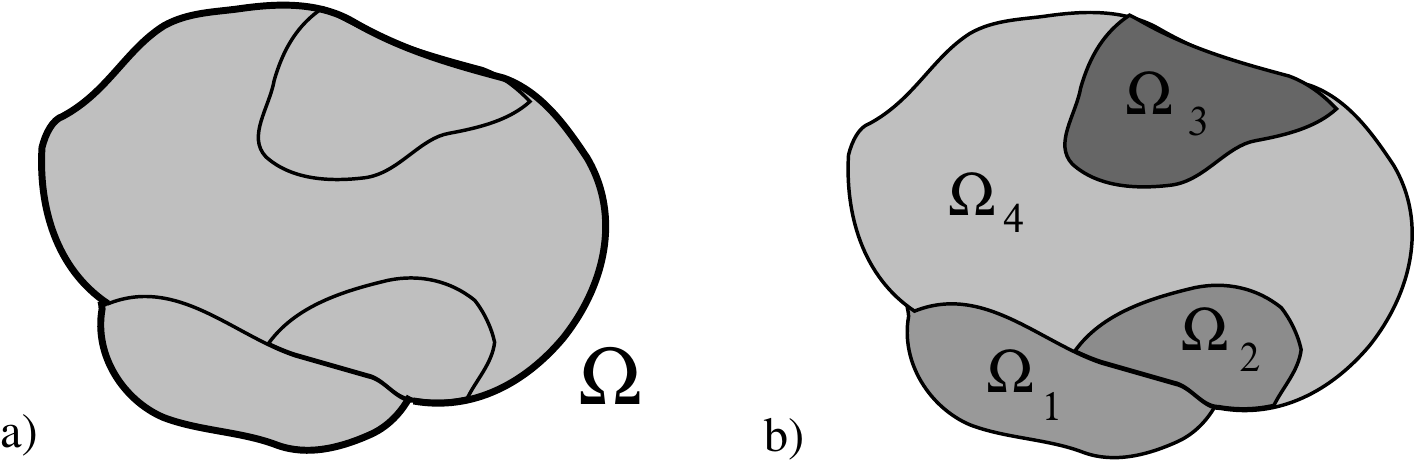}
\caption{\label{fig.1.5} a) is a configuration of regions contained in a 
region $\gW$.  It is equivalent to a multi-region configuration b) which is 
without inclusion.  Note that in this representation, the boundary of 
contained regions will meet in a transverse fashion, any boundary region 
of the containing region.}
\end{figure} 
\par

\section{Skeletal Linking Structures for Multi-Region Configurations in 
$\R^{n+1}$}
\label{S:sec2}
\par
The skeletal linking structures for multi-region configurations will build 
upon the skeletal structures for individual regions.  We begin by recalling 
their basic definitions and simplest properties.  
\subsection*{Skeletal Structures for Single Regions}
\par
We begin by recalling \cite{D1} that a \lq\lq skeletal structure\rq\rq $(M, 
U)$ in $\R^{n+1}$, \cite[Def. 1.13]{D1}, consists of: a Whitney stratified 
set $M$ and a multivalued \lq\lq radial vector field\rq\rq\, $U$ on $M$.  
We will not give every condition in detail, but instead point out the key 
features and just mention certain technical conditions that are needed so 
the structure has the correct global properties.  \par 
$M$ satisfies the conditions for being a \lq\lq skeletal set\rq\rq\, 
\cite[Def. 1.2]{D1}, which we briefly recall.  The skeletal set $M$ consists 
of smooth strata of dimension $n$, denoted by $M_{reg}$, and the set of 
singular strata $M_{sing}$, with $\partial M$ denoting the singular strata 
where $M$ is locally a manifold with boundary (for which we must use 
special \lq\lq boundary coordinates\rq\rq, see \cite[Def. 1.3]{D1}).  It 
satisfies the following properties. 
\begin{itemize}
\item[i)] For each point $x_0 \in M_{sing}$, and each connected component 
$M_{\ga}$ of $M_{reg}$ in a sufficiently small neighborhood of $x_0$, 
there is a unique limiting tangent space $\lim T_{x_i} M_{\ga}$ for any 
sequence $\{x_i\} \subset M_{\ga}$ converging to $x_0$. 
\item[ii)] Locally in a neighborhood of a singular point $x_0$, $M$ may be 
expressed as a union of (smooth) $n$--manifolds with boundaries and 
corners $M_{\ga_j}$, where two such intersect only on boundary facets 
(faces, edges etc.).  We will refer to this as a \lq\lq paved 
neighborhood\rq\rq\, of $x_0$ (see Figure \ref{fig.10.4a}).  
\item[iii)]  If $x_0 \in \barbdyM$ then those $M_{\ga_j}$ in (2) meeting 
$\bdyM$ meet it in an $n-1$ dimensional facet.  
\end{itemize}
Condition ii) is a simplified form of a local triangulation of a stratified 
set, see \cite{Go}, \cite{V}.  \par 
Second, the multivalued vector field $U$ has the following properties.
\begin{itemize}
\item[i)]  For each smooth point $x_0 \in M_{reg}$, there are two values of 
$U$ which point toward opposite sides of $T_{x_0}M$.  Moreover, on a 
neighborhood of a point of $M_{reg}$, the values of $U$ corresponding to 
one side form a smooth vector field. 
\item[ii)]  For a singular point $x_0 \notin \bdyM$, with $M_{\ga_j}$ a 
connected component containing $x_0$ in its closure, both smooth values 
of $U$ on $M_{\ga}$ extend smoothly to values $U(x_0)$ on the stratum of 
$x_0$.  If $M_{\ga_j}$ does not intersect $\bdyM$ in a neighborhood of 
$x_0$, then $U(x_0) \notin T_{x_0}M_{\ga_j}$.  In addition, for each local 
connected component $C_i$ of $B_{\gevar}(x_0) \cap (\R^{n+1} \backslash 
M)$ in a small ball $B_{\gevar}(x_0)$, there is a unique value of $U$ 
pointing into $C_i$ and the values at points in a neighborhood 
$B_{\gevar}(x_0) \cap M$ of $x_0$ which point into $C_i$ define a 
continuous vector field which is smooth on each stratum of $M$ (see a) in 
Figure \ref{fig.10.4a} or Figure \ref{fig.5c}). \par
\item[iii)]   At points $x_0 \in \bdyM$, there is a unique value for $U$ 
which is tangent to the stratum of $M_{reg}$ containing $x_0$ in the 
closure and points away from $M$.
\end{itemize}
When we refer to a {\em smooth value of $U$} at $x_0$ we mean either a 
smoothly varying choice of $U$ on one side of $M$ if $x_0 \in M_{reg}$ or 
one on $M_{\ga_j}$ extending smoothly to $x_0$ if $x_0 \in M_{sing}$.  
This allows various mathematical constructions on the smooth strata to 
be extended to the singular strata $M_{sing}$, see \cite[\S 2]{D1}.  \par 
Using the multivalued radial vector field $U$, we can define a stratified 
set (with smooth strata) $\tilde M$, called the \lq\lq double of $M$\rq\rq.  
Points of $\tilde M$ consist of all pairs $\tilde x = (U, x)$ where $U$ is a 
value of the radial vector field at $x$, and neighborhood of points $\tilde 
x$ are defined in $\tilde M$ using continuous extensions of $U$ near $x$.  
For example, the neighborhood in a) of Figure \ref{fig.10.4a} gives the 
neighborhoods in $\tilde M$  corresponding to b) and c).  There is also 
defined a finite-to-one stratified mapping $\pi : \tilde M \to M$ defined 
by $\pi (U, x) = x$, see \cite[\S 3]{D1}.  \par 
On $\tilde M$ is defined a \lq\lq normal line bundle\rq\rq $N$, such that 
over $\tilde x = (U, x)$, $N_{\tilde x} = \{ a\cdot U: a \in \R\}$.  It is a 
trivial line bundle, with a \lq\lq half-line bundle\rq\rq\, $N_{+}= \{ 
c\cdot U: c \geq 0 \}$.  We also define \lq\lq one-sided 
neighborhoods\rq\rq of the zero section $N_{a}= \{ c\cdot U: 0 \leq c \leq 
a\}$.   \par
\begin{figure}[ht]
\includegraphics[width=12cm]{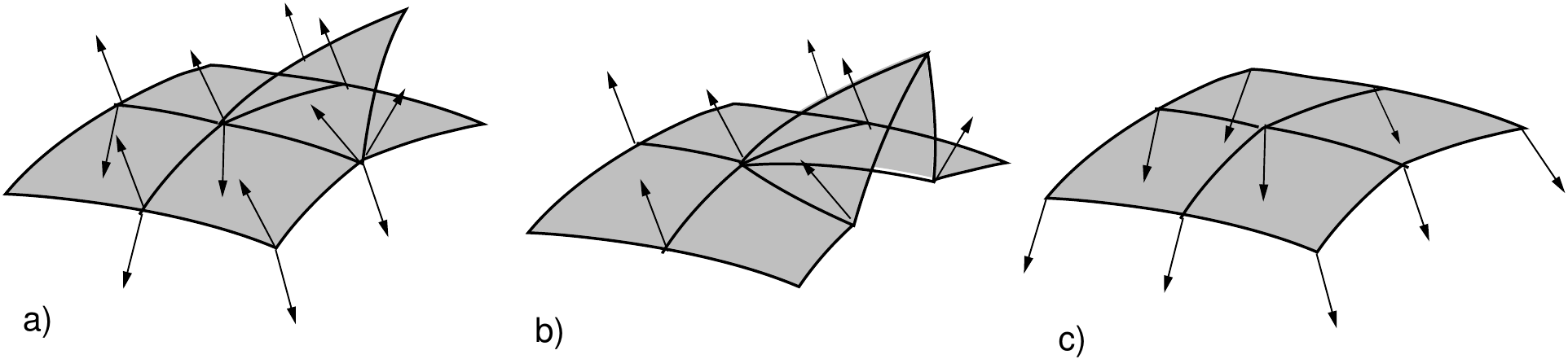}
%%%\centerline{\epsfxsize=12cm\epsffile{{Skel.Str.BdryIV.fig9}.pdf}}
\caption{\label{fig.10.4a} a) illustrates a \lq\lq paved neighborhood\rq\rq 
of a point in $M$ showing the paving by manifolds with boundaries and 
corners and the multivalued vector fields pointing to the local 
complementary components.  b) and c) illustrate the two corresponding 
paved neighborhoods in $\tilde M$.}
\end{figure}
\par

Then, using $\tilde M$, we can define the \lq\lq radial flow\rq\rq.  In a 
neighborhood $W$ of a point $x_0 \in M$ with a smooth single-valued 
choice for $U$, we define a local representation of the radial flow by 
$\psi_t (x) = x + t\cdot U(x)$.  It yields a {\it global radial flow} as a 
mapping $\Psi: N_{+} \to \R^{n+1}$.  Lastly, there are two technical 
conditions for a skeletal structure, the \lq\lq local initial 
conditions\rq\rq \cite[Def. 1.7]{D1} which ensure that the radial flow for 
small time remains one-to one (and stratawise nonsingular). \par 
We also recall from \cite{D1} that beginning with a skeletal structure 
$(M, U)$ in $\R^{n+1}$, we associate a \lq\lq region\rq\rq $\gW = 
\Psi(N_{1})$ and its \lq\lq boundary\rq\rq $\cB = \{ x + U(x) : x \in M\, 
\mbox{ all values of } U\}$.  Then, provided certain curvature and 
compatibility conditions are satisfied, which we will recall in \S 
\ref{S:secII.Rad.Flow}, it follows by \cite[Thm 2.5]{D1} that the radial 
flow defines a stratawise smooth diffeomorphism between $N_{1} 
\backslash \tilde M$ and $\gW \backslash M$, with the boundary.  From the 
radial flow we define the {\em radial map} $\psi_1 (x) = x + U(x)$ from 
$\tilde M$ to $\cB$.  We may then relate the boundary and skeletal set via 
the radial flow and the radial map.  \par
A standard example we consider will be the Blum medial axis $M$ of a 
region $\Omega$ with generic smooth boundary $\cB$ and its associated 
(multivalued) radial vector field $U$.  Then, the associated boundary $\cB$ 
we consider here will be the initial boundary of the object.   \par
\vspace{1ex}
\subsection*{Skeletal Linking Structures for Multi-Region Configurations} 
\hfill \par
We are ready to introduce skeletal linking structures for multi-region 
configurations.  These structures will accomplish multiple goals.  The 
most significant are the following.
\begin{itemize}
\item[i)]  Extend the skeletal structures for the individual regions in a 
minimal way to obtain a unified structure which also incorporates the 
positional information of the objects.  
\item[ii)] For generic configurations of disjoint regions with smooth 
boundaries, provide a Blum medial linking structure which incorporates 
the Blum medial axes of the individual objects. 
\item[iii)] For general multi-region configurations with common 
boundaries, provide for a modification of the resulting Blum structure to 
give a skeletal linking structure.
\item[iv)] Handle both unbounded and bounded multi-region configurations. 
\end{itemize} 
\par
Later we shall also see that the skeletal linking structure has a second 
important function allowing us to answer various questions involving the 
\lq\lq positional geometry\rq\rq of the regions in the configuration.  \par
We begin by giving versions of the definition for both the bounded and 
unbounded cases.  
\begin{Definition}
\label{Defmultlkgstr} A {\em skeletal linking structure for a multi-region 
configuration} $\{\gW_i\}$ in $\R^{n+1}$ consists of a triple $(M_i, U_i , 
\ell_i)$ for each region $\gW_i$ with the following two sets of 
properties.
\begin{itemize}
\item[S1)]
$(M_i, U_i)$ is a skeletal structure for $\gW_i$ for each $i$ 
with $U_i =  r_i\cdot \bu_i$ for $\bu_i$ a (multivalued) unit vector field  
and $r_i$ the multivalued radial function on $M_i$.
\item[S2)] $\ell_i$ is a (multivalued) {\em linking function} defined on 
$M_i$ (excluding the strata $M_{i\, \infty}$, see L4 below), with one value 
for each value of $U_i$, for which the corresponding values satisfy $\ell_i 
\geq r_i$, and it yields a (multivalued) {\em linking vector field} $L_i = 
\ell_i\cdot \bu_i$ .
\item[S3)] The canonical stratification of $\tilde M_i$ has a stratified 
refinement $\cS_i$, which we refer to as the {\em labeled refinement}.
\end{itemize}
By $\cS_i$ being a \lq\lq labeled refinement\rq\rq of the stratification  
$\tilde M_i$ we mean it is a refinement in the usual sense of 
stratifications in that each stratum of $\tilde M_i$ is a union of strata of 
$\cS_i$; and they are labeled by the linking types which occur on the 
strata. \par 
In addition, they satisfy the following four {\em linking conditions}. 
\flushpar
{\it Conditions for a skeletal linking structure} \par
\begin{itemize}
\item[L1)]   $\ell_i$ and $L_i$ are continuous where defined on $M_i$ and 
are smooth on strata of $\cS_i$.

\item[L2)]  The \lq\lq linking flow\rq\rq (see (\ref{Eqn2.1}) below) 
obtained by extending the radial flow is nonsingular and for the strata 
$S_{i\, j}$ of $\cS_i$, the images of the linking flow are disjoint and each 
$W_{i\, j} = \{x + L_i (x): x \in S_{i j}\}$ is smooth.  

\item[L3)]  The strata $\{W_{i\, j}\}$ from the distinct regions either 
agree or are disjoint and together they form a stratified set $M_0$, which 
we shall refer to as the {\em (external) linking axis}.
%%% refinement of a Whitney stratification of   
\item[L4)]  There are strata $M_{i\, \infty} \subset \tilde M_i$ on which 
there is no linking so the linking function $\ell_i$ is undefined.  On the 
union of these strata $M_{\infty} = \cup_{i} M_{i\, \infty}$, the global 
radial flow restricted to $N_{+} | M_{\infty}$ is a diffeomorphism with 
image the complement of the image of the linking flow.  In the bounded 
case, with $\tilde \gW$ the enclosing region of the configuration, it is 
required that the boundary of $\tilde \gW$ is transverse to the 
stratification of $M_0$ and where the linking vector field extends beyond 
$\tilde \gW$, it is truncated at the boundary of $\tilde \gW$ (this 
includes $M_{\infty}$).  
\end{itemize}
\end{Definition}

\par
We denote the region on the boundary corresponding to $M_{\infty}$ by 
$\cB_{\infty}$ and that corresponding to $M_{i\, \infty}$ by $\cB_{i\, 
\infty}$. 
\begin{Remark}
\label{Rem.2.1a}
By property L4), $\cB_{i\, \infty}$ does not exhibit any linking with any 
other region.  We will view it as either the {\em unlinked region} or 
alternately as being {\it linked to $\infty$}, where we may view the 
linking function as being $\infty$ on $M_{i\, \infty}$.  In the bounded case 
with $\tilde \gW$ the enclosing region of the configuration, we modify 
the linking vector field so it is truncated at the boundary of $\tilde \gW$.  
We can also introduce a \lq\lq linking vector field on $M_{\infty}$\rq\rq\,  
by extending the radial vector field until it meets the boundary of $\tilde 
\gW$.  Then, in the bounded case, we let $M_b$ denote the set of points in 
$\tilde M$ whose linking vector fields end at the boundary of $\tilde \gW$. 
Then, $M_{\infty} \subseteq M_b \subset \tilde M$, and in the generic 
bounded case, $M_b$ has a natural stratification, and it provides the 
linking to the boundary of $\tilde \gW$. 
\end{Remark}
\par
For this definition, we must define the \lq\lq linking flow\rq\rq\,  which 
is an extension of the radial flow.  We define the {\it linking flow} from 
$M_i$ by 
%%%  $\gl_{i}(x,t) = x + \chi_i(x,t)\bu_i(x)$, where 

\begin{align}
\label{Eqn2.1}
\gl_{i}(x,t) \, \, &= \,\, x \,+ \, \chi_i(x,t)\bu_i(x)  \qquad \text{where}  
\\
\chi_i(x,t) &= \left\{      
\begin{array}{lr}       
 2tr_i(x) &  \displaystyle 0 \leq t \leq \frac{1}{2}\\        2(1-t)r_i(x) + 
(2t-1)\ell_i(x) &  \displaystyle \frac{1}{2} \leq t \leq 1      
\end{array}    \right.  \, .   \notag
\end{align}
%%% \newpage
As with the radial flow it is actually defined from $\tilde M_i$ (or $\tilde 
M_i \backslash M_{\infty}$). The combined linking flows $\gl_i$ from all 
of the $M_i$ will be jointly referred to as the linking flow $\gl$.  For 
fixed $t$ we will frequently denote $\gl(\cdot, t)$ by $\gl_t$. \par
\vspace{1ex}
\flushpar
{\bf Convention: }  Because we will often view the collection of objects 
for the linking structure as together forming a single object, we will 
adopt notation for the entire collection.  This includes $M$ for the union of 
the $M_i$ for $i > 0$, and similarly for $\tilde M$.  On $M$ (or $\tilde M$) 
we have the radial vector field $U$ and radial function $r$ formed from 
the individual $U_i$ and $r_i$, the linking function $\ell$ and linking 
vector field $L$ formed from the individual $\ell_i$ and $L_i$; as well as 
the linking flow $\gl$ and $M_{\infty}$ already defined.
\par
\vspace{1ex}
\par
We see that for $0 \leq t \leq \frac{1}{2}$ the flow is the radial flow at 
twice its speed; hence, the level sets of the linking flow, $\cB_{i\, t}$, 
for time $0 \leq t \leq \frac{1}{2}$ will be those of the radial flow.  For 
$\frac{1}{2} \leq t \leq 1$ the linking flow is from the boundary $\cB_i$ 
to the linking strata of the external medial linking axis.   \par 
By the {\it linking flow being nonsingular} we mean it is a piecewise 
smooth homeomorphism, which for each stratum $S_{i\, j} \subset \tilde 
M_i$, is smooth and nonsingular on $S_{i\, j} \times \left[ 0, 
\frac{1}{2}\right]$.  Also, either: $S_{i\, j} \times \left[ \frac{1}{2}, 
1\right]$ is smooth and nonsingular if $S_{i\, j}$ is a stratum associated 
to strata in $\cB_{i\, 0}$; or $S_{i\, j}$ is not associated to strata in 
$\cB_{i\, 0}$, $\ell_i = r_i$ on $S_{i\, j}$, and so the linking flow on 
$S_{i\, j} \times \left[ \frac{1}{2}, 1\right]$ is constant as a function of 
$t$.  That the linking flow is nonsingular will follow from the analogue of 
the conditions given in \cite[\S 3]{D1} for the nonsingularity of the radial 
flow.  These will be given later, when we use the linking flow to establish 
geometric properties of the configuration.  
\par 
\begin{figure}[ht] 
\centerline{\includegraphics[width=3.5cm]{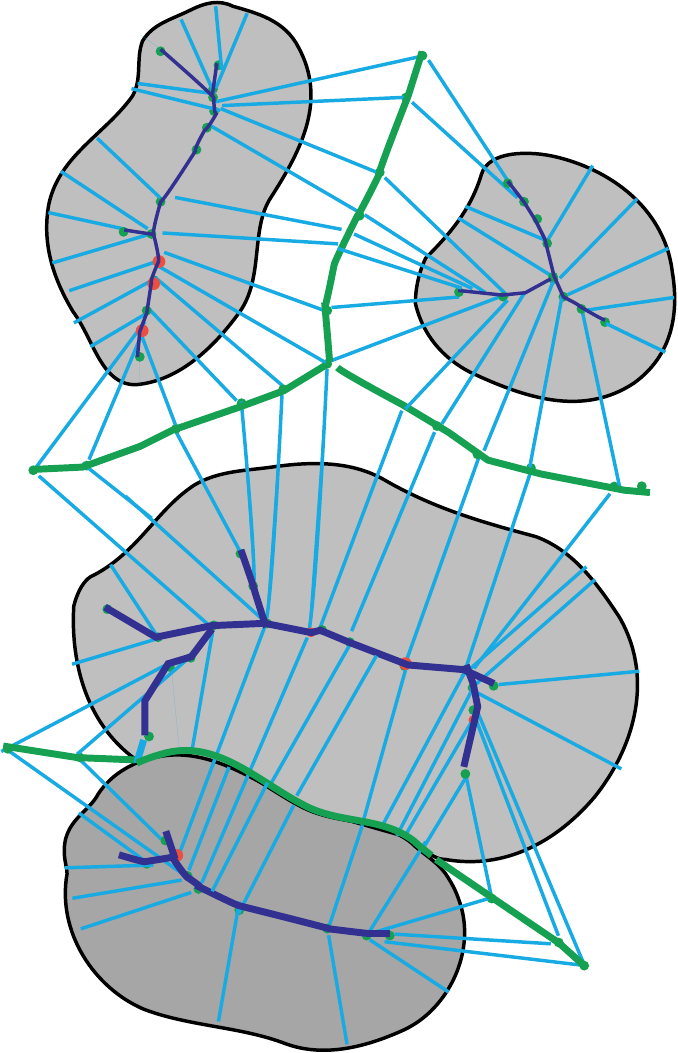}} 
\caption{\label{fig.5a} Skeletal linking structure for multi-region
configuration in $\R^2$.  Note that the linking axis includes the shared 
boundary curves of adjoined regions.} 
\end{figure} 
\par
In Figure \ref{fig.5a} we illustrate (a portion of) the skeletal linking 
structure for a configuration of four regions.  Shown are the linking vector 
fields meeting on the (external) linking axis.  The linking flow moves 
along the lines from the medial axes to meet at the linking axis.  \par

\subsection*{Linking between Regions and between Skeletal Sets}
\par
A skeletal linking structure allows us to introduce the notion of linking of 
points in different (or the same) regions and of regions themselves being 
linked.    We say that two points $x \in \tilde M_i$ and $x^{\prime} \in 
\tilde M_j$ are {\em linked} if the linking flows satisfy $\gl_i(x, 1) = 
\gl_j(x^{\prime}, 1)$. This is equivalent to saying that for the values of 
the linking vector fields $L_i (x)$ and $L_j(x^{\prime})$, $x + L_i(x) = 
x^{\prime} + L_j (x^{\prime})$.  Then, by linking property $L3)$, the set of 
points in $\tilde M_i$ and $\tilde M_j$ which are linked consist of a union 
of strata of the stratifications $\cS_i$ and $\cS_j$.  Furthermore, if the 
linking flows on strata from $S_{i\, k} \subset \tilde M_i$ and $S_{j\, 
k^{\prime}} \subset \tilde M_j$ yield the same stratum $W \subset M_0$, 
then we refer to the strata as being linked via the linking stratum $W$.  
Then, $\mu_{i\, j} = \gl_j(\cdot, 1)^{-1}\circ \gl_i(\cdot, 1)$ defines a 
diffeomorphic {\it linking correspondence} between $S_{i\, k}$ and $S_{j\, 
k^{\prime}}$.  
\par 
In Part II we will introduce a collection of regions which capture 
geometrically the linking relations between the different regions.  For 
now we concentrate on understanding the types of linking that can occur.  
There are several possible different kinds of linking.  More than two 
points may be linked at a given point in $M_0$.  Of these more than one 
may be from the same region.  If all of the points are from a single region, 
then we call the linking {\it self-linking}, which occurs at indentations of 
regions.  If there is a mixture of self-linking and linking involving other 
regions then we refer to the linking as {\it partial linking}, see 
Figure~\ref{fig.5d}.  \par 
\begin{figure}[ht] 
\centerline{\includegraphics[width=7cm]{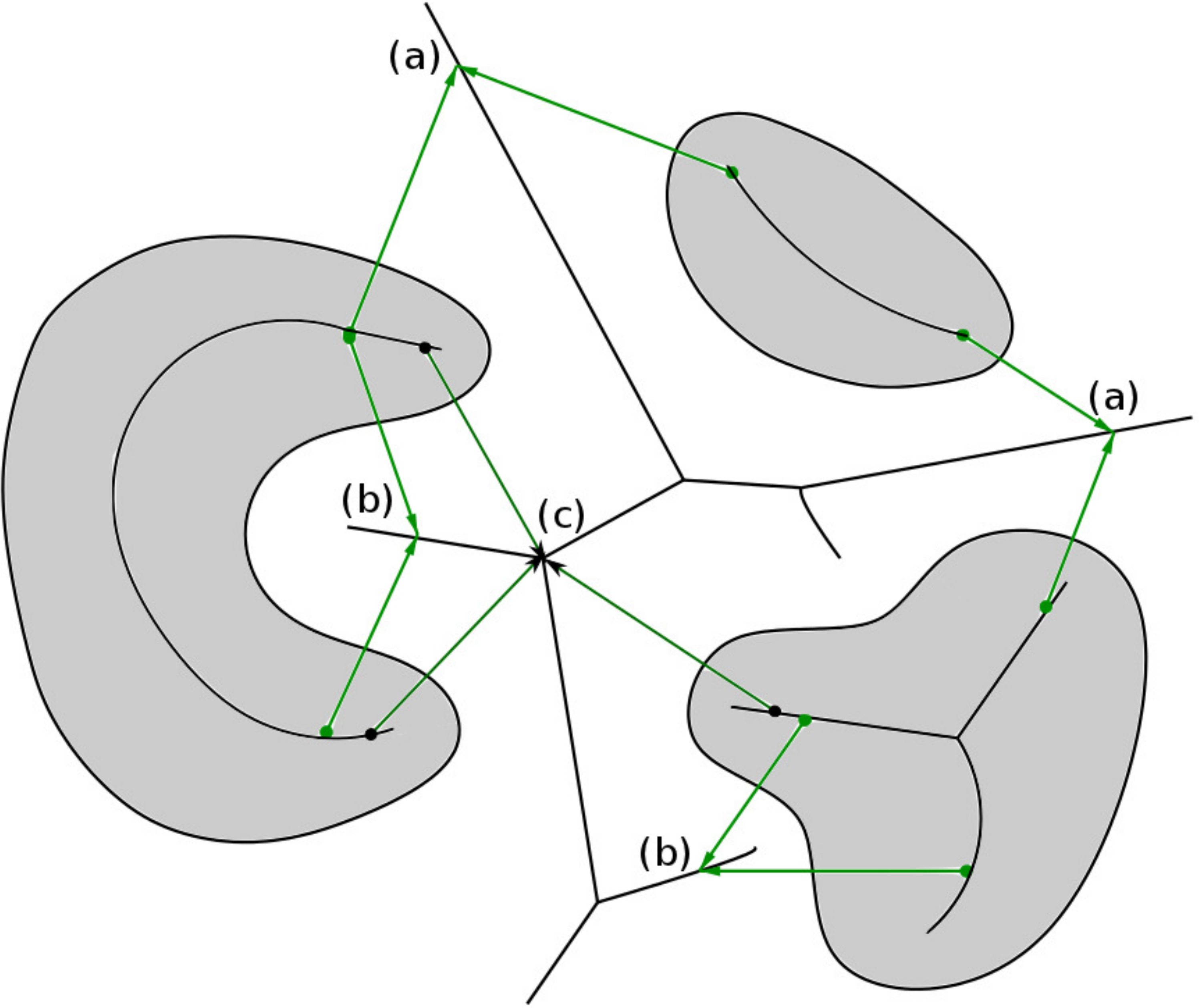}} 
\caption{\label{fig.5d} Types of linking for multi-region
configuration in $\R^2$: a) linking between two objects; b) self-linking; 
and c) partial linking.} 
\end{figure} 
\par
\begin{Remark}
\label{Rem2.2}  
If $\gW_i$ and $\gW_j$ share a common boundary region, then certain 
strata in $M_i$ and $M_j$ are linked via points in this boundary region, 
and for those $x \in M_i$, $\ell_i(x) = r_i(x)$; see Figure~\ref{fig.5a}.  
\end{Remark}

\par
\section{Blum Medial Linking Structure for a Generic Multi-Region 
Configuration}
\label{S:sec3}
%%%  \label{Sec: Blum.str}
\par
\subsection*{Blum Medial Axis for a Single Region with Smooth Generic 
Boundary} 
\par
In the case of a single region $\gW$ with smooth generic boundary $\cB$, 
the Blum medial axis $M$ is a special example of a skeletal structure 
which has special properties for both the medial axis $M$ and the radial 
vector field $U$.  The generic local structure of $M$ is given by normal 
forms defined in terms of properties of the family of distance-squared 
functions.  This family is the restriction of the \lq\lq distance-squared 
function\rq\rq\, $\gs(x, u) = \| x - u\|^2$ on $\R^{n+1} \times \R^{n+1}$.
We let $\rho = \gs | (\cB \times \intr(\gW))$.  Then, the Blum medial axis 
$M$ is the Maxwell set of $\rho$, which is the set of $u \in \intr(\gW))$ 
at which the absolute minimum of $\rho(\cdot, u)$ occurs at multiple 
points or is a degenerate minimum.  \par
To describe the generic structure of $M$, we begin by recalling a result of 
Mather.  In \cite{M2}, Mather is concerned with determining the generic 
structure of the distance function $\gd(x) = \min_{y \in \cB} \| x - y\|$  
from points $x \in \R^{n+1} \backslash \cB$ to $\cB$.  He gives such a 
classification theorem which for generic $\cB$ with $n+1 \leq 7$  gives 
the structure of $\gd$ off of a finite set of points.  He excludes a finite 
set of points for two reasons, which he does not really explain in the 
paper.  One is that his classification excludes the classification at 
critical points of $\gd$ on strata of $M$.  The second is because for $n+1 = 
7$ there can be isolated points where the structure theorem for generic 
germs in terms of versal unfoldings as explained below does not directly 
apply.   \par
We begin our discussion of his results stated in a form where they yield 
the generic structure of $M$, but do not consider the structure of the 
global distance function $\gd$ to $\cB$. 
\subsection*{Structure of Maxwell Set Described by $\cR^+$-Versal 
Unfoldings} \par
At a point $u_0 \in M$ of the Maxwell set, we let $S = \{x_1, \dots , x_k\} 
\subset \cB$ be the points with $r_0 = \rho(x_i) =  \rho(x_j)$ for all $1 
\leq i, j \leq k$, the minimum value for $\rho(\cdot, u_0)$ and consider 
the multigerm $\rho : \cB \times \R^{n+1}, S \times \{u_0\} \to \R, r_0$, 
We view the coordinates of $u$ as a set of parameters for $u \in 
\intr(\gW) \subset \R^{n+1}$.  The mapping $\rho$ is an \lq\lq 
unfolding\rq\rq  of the multigerm $\rho_0 = \rho(\cdot, u_0) : \cB, S \to 
\R, r_0$ on the parameters $u$.  Such multigerms and their unfoldings are 
studied using $\cR^+$-equivalence of multigerms $\rho_0$ and $\rho_1$  
via the action of the group of pairs $(\varphi, c)$ which consist of a 
multigerm of a diffeomorphism  $\varphi : \cB, S \to \cB, S$ and constant 
$c \in \R$, so $\rho_0$ is $\cR^+$-equivalent to $ \rho_1 = \rho_0\circ 
\varphi + c$ (or for unfoldings $\rho$, pairs $\Psi(x, u) = (\psi(x, u), 
\gl(u))  : \cB \times \R^{n+1}, S \times \{u_0\} \to \cB \times \R^{n+1}, S 
\times \{u_0\}$ and a smooth function germ $c(u)$ so that $\rho$ is 
$\cR^+$-equivalent to $\rho\circ \Psi + c(u)$).  Singularity theory applies 
to the classification of such multigerms and their unfoldings.  Provided 
$\rho_0$ has \lq\lq finite $\cR^+$-codimension\rq\rq in an appropriate 
sense, then there is an $\cR^+$-versal unfolding which is a finite 
parameter unfolding which captures all possible unfoldings of $\rho_0$ up 
to $\cR^+$-equivalence (this extends to multigerms Thom\rq s versal 
unfolding theorem for germs which he used in Catastrophe Theory 
\cite{Th}).  The minimum number of parameters needed for 
the versal unfolding is the $\cR^+_e$-codimension of $\rho_0$.  These 
results are discussed by Mather in \cite{M2}, and some details relating 
versality and transversality that were left out are treated in a more 
general context in e.g. \cite{D6}.  
\par
These are applied by Mather to classify the local structure of the Blum 
medial axis.  We view $\gW$ as the region enclosed by the boundary $\cB = 
\varphi(X)$, for $\varphi : X \to \R^{n+1}$ a smooth embedding and $X$ a 
smooth compact $n$-manifold.  Then, there is the following theorem of 
Mather \cite{M2} which describes the {\em generic properties} of the Blum 
medial axis of $\gW$.  The notion of \lq\lq genericity\rq\rq will be 
explained in more detail in \S \ref{S:sec5}. 
\begin{Thm}[Generic Properties of the Blum Medial Axis]
\label{Thm3.1}
If $n + 1 \leq 6$, there is an open dense set of embeddings 
$\varphi \in \emb(X, \R^{n+1})$ such that for any finite subset 
$S \subset \cB = \varphi(X)$ and $u \in M$, for which 
$x, x^{\prime} \in S$ satisfy $\rho(x, u) = \rho(x^{\prime}, u) (= r)$, then 
the multigerm $\rho : \cB \times \intr(\gW), S \times \{u\} \to \R, r$ is 
a versal unfolding for  $\cR^{+}$-equivalence of multigerms.  \par
If $n+1 = 7$, then there is a finite set of points $E_M \subset M$ which 
form strata of dimension $0$ such that for $u \in M \backslash E_M$ the 
same conclusions hold for multigerms of $\rho$; while for each point $u 
\in E_M$, there is a unique point $S = \{x\}$ in $\cB$ and the unfolding 
$\rho$ at $(x, u)$ defines a transverse section to the 
$\tilde E_7$-stratum (see below).
\end{Thm}
\begin{Remark}
\label{Rem3.1}
In the case $n+1 = 7$, there may be isolated points $u \in M$ with each 
having a unique corresponding point $x \in \cB$ such that the germ of 
$\rho(\cdot, u) | \cB$ at $x$ is $\cR^+$--equivalent in local coordinates 
$(y_1, \dots , y_6)$  to an $\tilde{E}_7$ singularity $y_1^4 + y_2^4 + 
ay_1^2y_2^2 + y_3^2 + \cdots + y_6^2$, where $| a | < 2$.  However the 
corresponding unfolding is not $\cR^+$--versal as $a$ is a modulus and the 
entire stratum in jet space formed by the $\cR^+$--orbits allowing $a$ to 
vary has $\cR^+$--codimension $7$, while the individual orbits have 
codimension $8$.  Generically a germ in the stratum may occur at an 
isolated point; and the resulting unfolding cannot be versal but it does 
define a transverse section to the 
$\tilde{E}_7$-stratum.  \par
	Now such unfoldings of $\tilde{E}_7$ germs are understood by a 
result of Looijenga \cite{L2}, from which it follows that the topological 
classification of $M$ near such $\tilde{E}_7$ points is independent of $a$ 
for $| a | < 2$.  Moreover, it follows that the associated mappings used in 
Lemma \ref{Lem7.6} are topologically stable by Looijenga\rq s theorem 
using the argument in \cite[Thm. 4] {D8}.  Thus, the arguments in \S 
\ref{S:sec7} still can be applied when appropriately modified in 
neighborhoods of the $\tilde{E}_7$ points using topological stability.  We 
refer to such points as {\em generic $\tilde E_7$ points}.\par 
\end{Remark}
\par
For $n +1 \leq 7$, Mather\rq s theorem then gives the list of possible 
multigerms and hence the corresponding local structure of $M$ resulting 
from the normal forms for the versal unfolding, except at the finite set of 
points in $E_M$ when $n + 1 = 7$.  \par
Excluding the case of $E_M$ when $n + 1 = 7$, since each germ represents 
a local minimum for points on the Blum medial 
axis, the only germs which occur are the $A_k$ singularities, which are 
$\cR^{+}$-equivalent in local coordinates $(y_1, \dots , y_n)$ to 
$y_1^{k+1} + \sum_{j = 2}^{n} y_j^2$ for $k$ odd.  If $\rho(\cdot, u)$ is a 
germ of type $A_{\ga_j}$ at the point $x_j$, then the multigerm is said 
to be of {\it generic type $A_{\bga}$} where $\bga = (\ga_1, \dots , 
\ga_k)$, and is denoted $A_{\ga_1}A_{\ga_2}\cdots A_{\ga_k}$ (where it 
is customary to denote $A_{\ga_i}$ repeated $r$ times using exponent 
notation $A_{\ga_i}^r$).   \par  
In the generic case, the set of points $u \in M$ of type $A_{\bga}$ forms a 
submanifold $\gS_{M}^{(\bga)}$ whose codimension in $\R^{n+1}$ equals 
the $\cR^{+}_e$ codimension of a multigerm of type $A_{\bga}$, which 
equals $|\bga | - 1$, with $|\bga | = \sum_{i = 1}^{k} \ga_j$.  In addition, 
the $\{\gS_M^{(\bga)}: |\bga | \leq n + 2\}$ form a Whitney stratification 
of $M$ (because by the versality theorem, the structure of $M$ is 
analytically trivial along the strata, while for different simple types the 
topological structure of the stratification differs).  The smooth points of 
$M$ are the points $u$ where $\rho(\cdot, u)$ has an $A_1^2$ singularity.  
The singular strata are of two types.  The stratum consisting of points $u$ 
with $\rho(\cdot, u)$ having a single $A_3$ singularity forms the edge of 
$M$, denoted $\bdyM$.  This is  part of the boundary of the regular stratum 
viewed as an $n$-manifold.  The closure $\barbdyM$ consists of strata 
$\gS_M^{(\bga)}$ with some $\ga_i \geq 3$ (and in the case $n+1 = 7$ the 
$\tilde E_7$ germs).  The second type of singular 
strata are the $A_1^k$ strata with $k > 2$ which are interior strata.   For 
$\R^3$ see Figure \ref{fig.5c} and the accompanying Example \ref{Ex3.1}.
\par
\subsubsection*{Stratification of the Boundary $\cB$} \hfill
\par
Mather does not explicitly refer to the stratification of the boundary.  
However, it plays an important role when we consider configurations of 
regions.  Corresponding to each stratum $\gS_M^{(\bga)}$ are strata of 
$\cB$ consisting of the individual $x_i$ which belong to the subset $S$ 
defining points in $\gS_M^{(\bga)}$, see Figures~\ref{fig.5c} and 
\ref{fig.5e}.  We denote the corresponding strata of $\cB$ by 
$\gS_{\cB}^{(\bga)}$.  These strata are the images under projection onto 
$\cB$ of  strata in the space of the versal unfolding.  In turn, those strata 
consist of subsets of points where the unfolding defines multigerms of a 
given type.  However, the unfolding theorem does not assert that the 
strata in $\cB$ need be smooth of the same dimension as corresponding 
strata in $M$.  We shall later prove that generically the 
$\gS_{\cB}^{(\bga)}$ are indeed smooth submanifolds of the same 
dimension as $\gS_{M}^{(\bga)}$.  As $S$ and $\bga$ have no intrinsic 
ordering, when we refer to one of these strata in $\cB$ which contains 
$x_j$ we shall place $\ga_j$ first in the ordering.  If $n + 1 = 7$, there 
are unique generic $\tilde E_7$ points on $\cB$ corresponding to each of 
(the finite number of) points of $E_M$; these form a set of 
$0$-dimensional strata $E_{\cB}$.  \flushpar 
{\it Stratification of $\cB$ by $\gS_{\cB}^{(\bga)}$: } \hfill
\par
The $\gS_{\cB}^{(\bga)}$ define a stratification of $\cB$ which we may 
view as consisting of three parts: \flushpar
\begin{enumerate}
\item  The first part consists of edge closure points and has a Whitney 
stratification consisting of the strata $\gS_{\cB}^{(\bga)}$ containing the 
point $x_1$ with $\ga_1 \geq 3$.  This is the subset of $\cB$ where 
$\rho(\cdot, u)$ has a local minimum of Thom-Boardman type $\gS_{n, 1}$ 
for some $u \in  \gW$; or if $n + 1 = 7$, it also contains the 
$0$-dimensional strata $E_{\cB}$.  
\item The second part has a Whitney stratification consisting of the 
strata $\gS_{\cB}^{(\bga)}$ with all $\ga_i = 1$.  
\item The third part consists of the $A_1$ points which belong to a tuple 
in $\gS_{\cB}^{(\bga)}$ with some $\ga_j \geq 3$ for 
$j > 1$ (thus, they are points associated via the medial axis  to 
edge-closure points in the Blum medial axis).  
\end{enumerate}
Although we will only show that the first two of these subsets of strata 
each separately form Whitney stratifications, it should be possible with 
considerably more work to show that together they form a Whitney 
stratification of $\cB$.  \par 
\begin{Example}
\label{Ex3.1}
For a generic region in $\R^2$, the refinement of $\tilde M$ only involves 
adding a finite number of isolated points to smooth curve segments, so the 
stratification is Whitney.  For a generic region in $\R^3$, the standard 
local Blum types are $A_1^2$, $A_3$, $A_1A_3$, $A_1^3$, and $A_1^4$ 
(the first consists of smooth points of $M$ and the others are shown in 
Figure~\ref{fig.5c}).  In Figure~\ref{fig.5e}, we show the corresponding 
local stratification types for the boundary.  The $A_3$ and $A_3A_1$ 
points correspond to the edge-closure points; the $A_1^3$ and $A_1^4$ 
points form the $A_1^k$-types, and the last $A_1$ point of type 
$A_1A_3$ is the third type.  In this case, a direct calculation shows that 
the closure of the $A_1^2$ points at an $A_3A_1$ point in the boundary is 
a smooth curve, so that the $\gS_{\cB}^{(\bga)}$ form a Whitney 
stratification of $\cB$.  
\end{Example}
\par
\begin{figure}[ht] 
\centerline{\includegraphics[width=10cm]{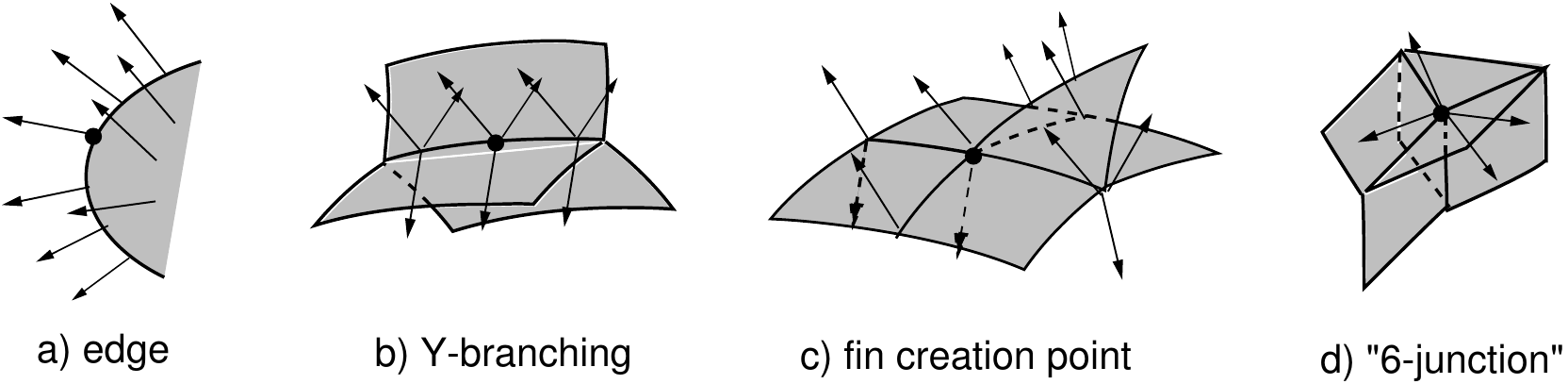}} 
\caption{\label{fig.5c} 
Generic local medial axis structures in $\R^3$: in a) the $A_3$ points form 
the edge curve and in b) the $A_1^3$ points form the \lq\lq Y-branch 
curve\rq\rq; for c) there is an isolated $A_1A_3$ point which is the fin 
point and for d), an isolated $A_1^4$ point which is the 
\lq\lq $6$-junction\rq\rq point.} 
\end{figure} 
\par
\begin{figure}[ht] 
\centerline{\includegraphics[width=10cm]{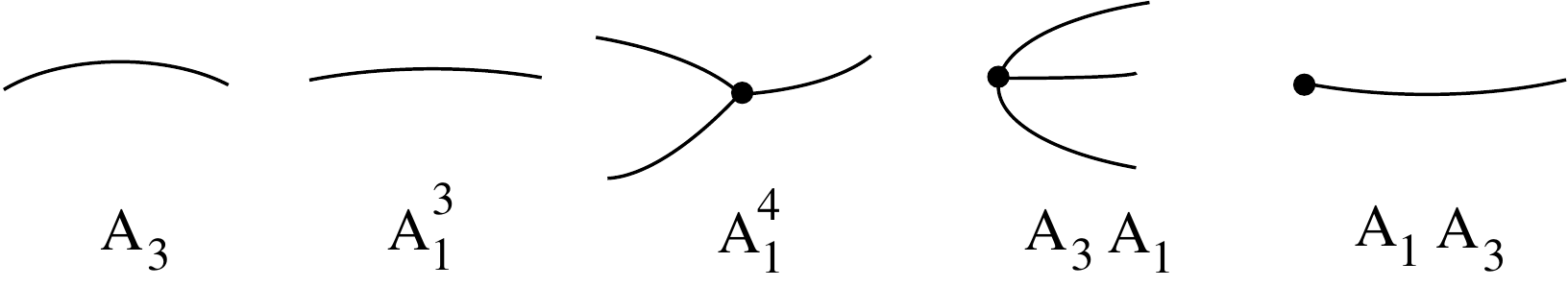}} 
\caption{\label{fig.5e} Generic local stratification of the boundary $\cB$ 
for a region in $\R^3$ in terms of the corresponding types on the medial 
axis listed in Figure \ref{fig.5c} in terms of their $A_{\bga}$-type:  $A_3$ 
and $A_1^3$ are the dimension $1$ strata in $\cB$ corresponding to edge 
and $Y$-branching curves (with $3$ $A_1^3$ strata in the boundary for 
each $Y$-branching curve), while the other three are dimension $0$ strata 
of $\cB$ corresponding to the $0$-dimensional strata of the medial axis 
in Figure \ref{fig.5c}; and the complement consists of open strata in $\cB$ 
corresponding to the $2$-dimensional strata in the medial axis consisting 
of $A_1^2$ points.  
%%% These illustrate the strata in $\cB$ which correspond to the types of 
%%%  strata given in Figure \ref{fig.5c}, of the same $A_{\bga}$-type.  
Note there are $4$ $A_1^4$ strata in the boundary for each $6$-junction 
point, and the last two figures represent strata on the two opposite 
boundary regions correponding to the \lq\lq fin creation point\rq\rq 
$A_1A_3$. } 
\end{figure} 
\par

\subsection*{Addendum to Generic Blum Structure for a Region with 
Boundaries and Corners}
\par
There is an addendum to Mather\rq s theorem in the case of regions $\gW$ 
which are manifolds with boundaries and corners.  It concerns the local 
form of the Blum medial axis in a neighborhood of a $k$-edge-corner point 
for regions with boundaries and corners, which we introduced at the 
beginning of \S \ref{S:sec1}.  This is described using the following normal 
form .  

\begin{Definition}
\label{Def3.3}
The {\it edge-corner normal form} for the Blum medial axis of a 
$k$-edge-corner point $x$ consists of a smooth diffeomorphism $\psi$ 
from the neighborhood of $0$ in $C_k = \R^k_{+} \times \R^{n+1-k}$, $k 
\geq 2$, 
to a neighborhood $W$ of $x$ with $0 \mapsto x$ such that the medial axis 
in $W$ is the image $\psi(E)$, where 
$$E_k = \{ u = (u_1, \dots , u_{n+1}) \in C_k : \makebox{ there are $1 \leq 
i, j \leq k$ such that $i \neq j$ and  $u_i = u_j$}\}  $$
with the canonical Whitney stratification of $E_k \cup \partial C_k$. 
\end{Definition}
For example in $\R^3$, a) and b) of Figure~\ref{fig.3.3b} illustrate the 
edge-corner normal forms for the Blum medial axis at $P_2$ and 
$P_3$-edge-corner points.  
\par
\begin{figure}[ht]
\includegraphics[width=8cm]{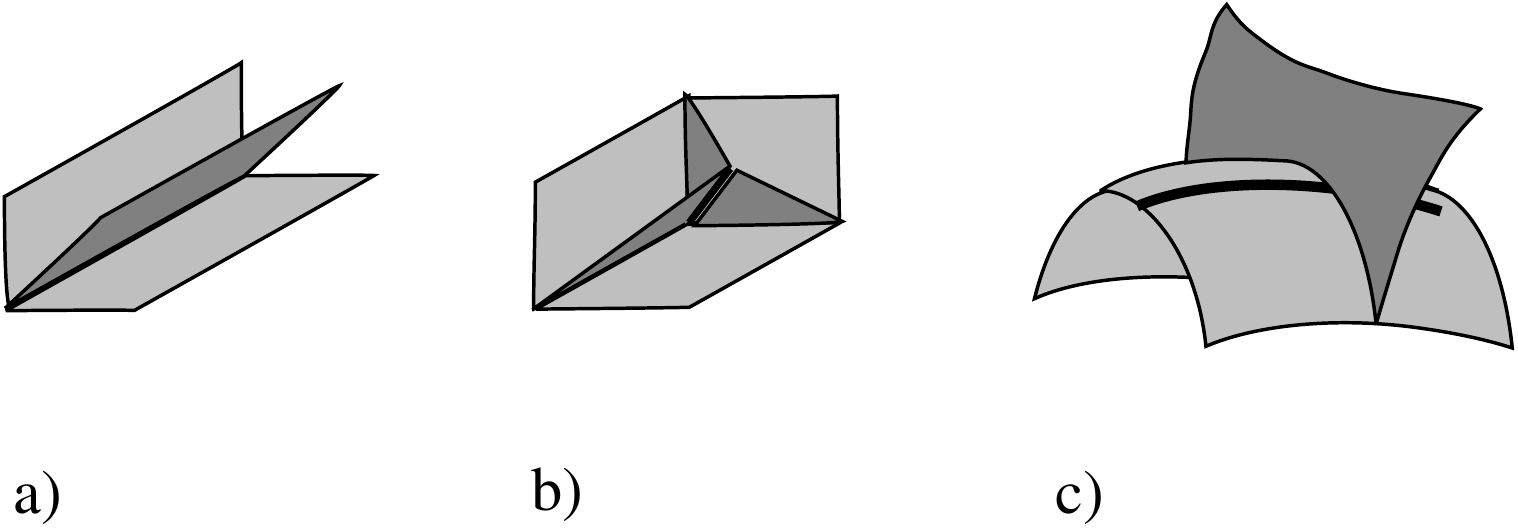}
\caption{\label{fig.3.3b} 
a) and b) edge corner normal forms for $P_2$ and $P_3$ points.  The 
darker shaded regions are the Blum medial axis.  c) Example at a $Q_1$ 
point of the transverse intersection of $\gS_Q$ with the $A_3$ curve in 
the boundary.}
\end{figure} 
 \par
We use similar notation for Mather\rq s Theorem, where now $X$ is the 
boundary of the compact region $\gD \subset \R^{n+1}$ which has 
boundaries and corners.  In addition, for an embedding $\varphi : \gD \to 
\R^{n+1}$, $\gW = \varphi(\gD)$ and $\cB = \varphi(X)$.  Then, the 
addendum is the following.
\begin{Thm}[Addendum to Generic Properties of the Blum Medial Axis]
\label{Thm3.2}
If $n + 1 \leq 7$, then there is an open dense set of embeddings 
$\varphi \in \emb(\gD, \R^{n+1})$, such that: \flushpar 
\begin{itemize}
\item[i)] the Blum medial axis has the same local properties in the 
interior of the region $\intr(\gW)$ as given in Theorem \ref{Thm3.1}. 
\item[ii)] if $n + 1 = 7$, then at a point $u_0 \in E_M$ with corresponding 
$x_0 \in E_{\cB}$, the unfolding germ $(x, u) \mapsto (\rho(x, u), u) : \cB 
\times \intr(\gW), (x_0, u_0) \to \R \times \intr(\gW)$ is a topologically 
stable unfolding of an $\tilde E_7$ germ; 
\item[iii)] the set of closure points in the boundary $\cB$ of the Blum 
medial axis consists of the edge-corner point strata of $\cB$;
\item[iv)] at these edge-corner points, the Blum medial axis satisfies the 
edge-corner normal form.  
\end{itemize}
\par 
For noncompact $\gW$, the same result holds on a given compact region of 
$\R^{n+1}$ for an open dense set of embeddings. \par
If there is a compact Whitney stratified set $\cS$ contained in the smooth 
strata of the boundary, then for an open dense set of embeddings, the 
strata of $\cS$ are transverse to the strata $\gS_{\cB}^{(\bga)}$.  
\end{Thm}
\par
This addendum will be proved in the process of establishing the generic 
Blum structure for general multi-region configurations (Theorem 
\ref{Thm3.6}) in \S  \ref{S:sec7}.  The specific information about $E_M$ 
follows using a result of Looijenga \cite{L2} and is explained in \S 
\ref{S:sec7}. 
\begin{Remark}
\label{Rem3.2}
The Blum medial structure $M$ for a region with boundary and corners 
contains the singular points of the boundary in its closure, and at such 
points the radial vector field $U = 0$.  Hence, $(M, U)$ does not define a 
skeletal structure in the strict sense.  However, it can still be used in 
exactly the same way to compute the local, relative, and global geometry 
and topology of both the region and its boundary, just as for skeletal 
structures.  Hence, we can view it as a \lq\lq relaxed skeletal 
structure\rq\rq, where \lq\lq relaxed\rq\rq means that $M$ includes the 
singular boundary points and $U = 0$ on these points.
\end{Remark}
\subsection*{Spherical Axis of a Configuration} 
%%%  \hfill 
\par
Along with the medial axis, we will also find use for its analog for the 
{\it family of height functions}.  This family is the restriction of the 
\lq\lq height function\rq\rq\, $\nu : \R^{n+1} \times S^n \to \R$, defined 
by $\nu(x, v) = x\cdot v$, to $\tau : \cB \times S^n \to \R$ ($S^n$ is the 
unit sphere in $\R^{n+1}$).  Here $\cB$ may denote the smooth generic 
boundary of a single region $\gW$, or more generally the boundary for a 
configuration defined by $\Phi : \bgD \to \R^{n+1}$.  We define the {\it 
spherical axis} $\cZ \subset S^n$ of $\gW$ or the configuration $\bgW$ 
defined by $\Phi$ to be the {\it Maxwell set of $-\tau$}, which is the set 
of $v \in S^n$ at which the absolute maximum of $\tau(\cdot, v)$ occurs 
at multiple points or is a degenerate maximum.  This consists of  
directions $v \in S^n$ for which the supporting hyperplanes $x\cdot v = c$ 
for the convex hull of $\gW$ or $\bgW$ have more than two tangencies 
with $\cB$ or a degenerate tangency, (and $v$ is normal to $\cB$ at these 
points) see e.g. Figure~\ref{fig.3.3d}.  Recall that $x\cdot v = c$ defines a 
supporting hyperplane for $\bgW$ if it meets the boundary $\cB$ of 
$\bgW$, which is contained in the half-space defined by $x\cdot v \leq c$.  
\par
\begin{Remark}
\label{Rem3.1b}
We note that while the height function depends on the choice of the origin, 
the spherical axis doesn\rq t.  Choosing a different point for the origin 
only changes the height function by adding a constant.  This will not 
change which points on $\cB$ are critical points, nor which type of 
singularity occurs at these critical points.  If there are multiple critical 
points at the same \lq\lq height\rq\rq, they will remain at the same 
height when we shift the point (but the height will change by the same 
amount for all).  
\end{Remark}
\par
Then, there is the following analog of the addendum to Mather\rq s 
Theorem.  

\begin{Thm}[Generic Properties of the Spherical Axis]
\label{Thm3.1b}
If $n + 1 \leq 7$, there is an open dense set of embeddings of 
configurations $\Phi \in \emb(\bgD, \R^{n+1})$ such that for any finite 
subset $S \subset \cB$ and $v \in M$, for which 
$x, x^{\prime} \in S$ satisfy $\tau(x, v) = \tau(x^{\prime}, v) (= r)$, then 
the multigerm $\tau : \cB \times S^n, S \times \{v\} \to \R, r$ is a 
versal unfolding for  $\cR^{+}$-equivalence of multigerms.  \par
If $n + 1 = 8$, then there is a finite set $E_{\cZ} \subset \cZ$ such that 
for any $v \in \cZ \backslash E_{\cZ}$, the same conclusions hold for the
multigerms of $\tau$; while at points $v \in E_{\cZ}$, $\tau$ has a 
generic $\tilde E_7$ singularity. 
\end{Thm}
This result will follow from the transversality results in Proposition 
\ref{Prop6.3a} as explained in \S  \ref{S:sec6}.  \par
Note: this holds for $n + 1\leq 8$, because the dimension of the parameter 
space $S^n$ is one 
less than that of $\R^{n+1}$.  This again gives the list of possible 
multigerms and 
hence the corresponding local structure of $M$ resulting from the normal 
forms for the versal unfolding.  
Again since each germ for points on the spherical axis represents a local 
minimum of $-\tau$ (or local maximum for $\tau$), the only germs which 
occur are the $A_k$ singularities, for $k$ odd, which are local minima (or 
local maxima for $\tau$).  \par  
In the generic case, it follows that the set of points $v \in \cZ$ of type 
$A_{\bga}$ forms a submanifold $\gS_{\cZ}^{(\bga)}$ whose codimension 
in $S^n$ equals the $\cR^{+}_e$ codimension of a multigerm of type 
$A_{\bga}$.  In addition, the $\{\gS_{\cZ}^{(\bga)}: |\bga | \leq n + 1\}$ 
form a Whitney stratification of $\cZ$.  This stratification has the same 
generic form as for the Blum medial axis except for one lower dimension.  
\par
\subsubsection*{Spherical Structure for a Configuration}
%%% \begin{Remark}
%%%  \label{Rem3.2b}
Just as for the Blum medial axis, we may associate to the spherical axis 
$\cZ$, both a height function $h$ and a multivalued vector field $V$.  To do 
this we need to initialize the origin.  Then, the height for a point $\bu$ on 
the spherical axis, which is a unit vector, defined for a point $x \in \cB$ 
is just the dot product $x \cdot \bu$.  Furthermore, there may be multiple 
points associated to $\bu$, all of which lie in the supporting hyperplane 
$H$ defined by $x \cdot \bu = h(\bu)$, for the maximum value $h = h(\bu)$ 
for $\cB$.  For each point $x_i \in H \cap \cB$, there is a vector 
$V = x_i - h\bu$ orthogonal to the line spanned by $\bu$.  This defines a 
multivalued vector field $V$.  \par

\begin{Definition}
\label{Def3.2b} The full {\em spherical structure} for the spherical axis of 
the multi-region configuration is the triple $(\cZ, h, V)$, consisting of the 
spherical axis $\cZ$, the height function $h$, and the multivalued vector 
field $V$.  This depends upon the choice of an origin on which all height 
functions are $0$. 
\end{Definition}
\par 
\begin{figure}[hb] 
\centerline{\includegraphics[width=7cm]{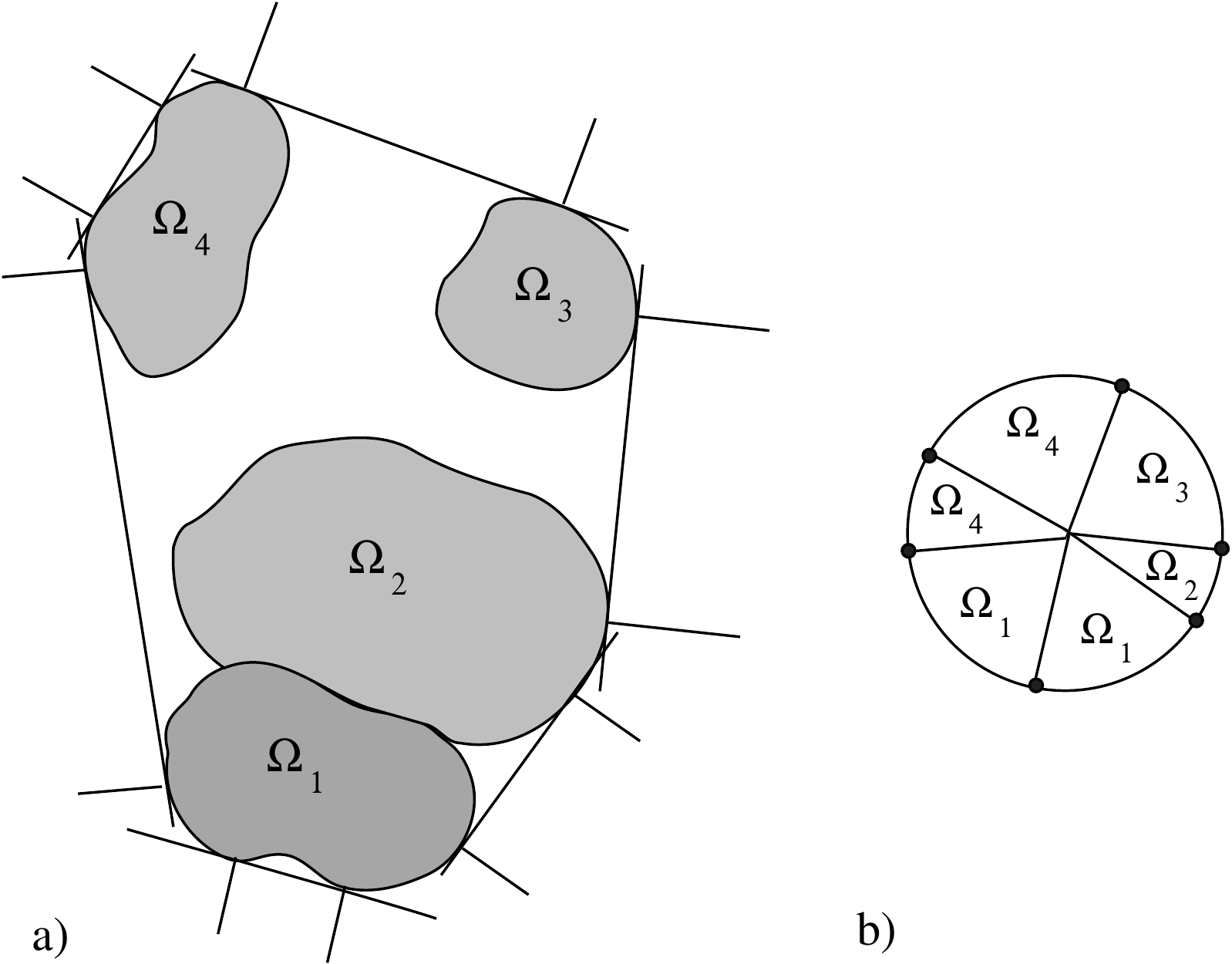}} 
\caption{\label{fig.5b} a) Configuration of four regions with the bitangent  
supporting lines, and b) the corresponding spherical axis, which consists 
of the points on $S^1$.  The regions between points represent subregions 
of $\cB$ in $\cB_{\infty}$.  If the same region $\gW_i$ is indicated on 
both sides of a radial line, then in $\cB_i$ is a region involving 
self-linking.} 
\end{figure} 
\par
From the spherical structure, we can reconstruct the boundary of 
$\cB_{\infty}$ by $x = V(\bu) + h(\bu) \bu$ for $\bu \in \cZ$ and the 
multiple values of $V$ at $\bu$.  Here $x$ denotes a collection of points 
corresponding to the values of $V$.  \par 
Then, the regions in $\intr(\cB_{\infty})$ are the regions in the 
complement of the boundary of $\cB_{\infty}$ which have supporting 
hyperplanes for at least one point in one of the corresponding 
complementary regions to the spherical axis.  If we have in addition the 
height function for the configuration defined on all of $S^n$, then we can 
construct the supporting hyperplanes for all $\bu \in S^n$, and the 
envelope of these hyperplanes yields $\cB_{\infty}$.  
%%% \end{Remark}
\par 
\begin{Example}
\label{Exam3.2b}
For generic configurations in $\R^2$, the spherical axis consists of a 
finite number of points in $S^1$ representing bitangent supporting lines.  
This is illustrated in Figure~\ref{fig.5b}. \par
In the case of generic configurations in $\R^3$, the spherical axis consists 
of curve segments which may either end or three may join at a \lq\lq 
$Y$-branch point\rq\rq.  These curves divide the complement in $S^2$ into 
connected components which correspond to regions of the $\cB_i$ which 
will not be linked in the Blum case to follow.  In Figure~\ref{fig.3.3d} are 
illustrated the three configurations in $\R^3$ which exhibit the basic 
properties for the spherical axis: a) smooth curve segment, b) 
$Y$-branch point, and c) end point at an $A_3$ singularity.
\end{Example}
\par
\begin{figure}[ht]
\begin{center}
\includegraphics[width=10cm]{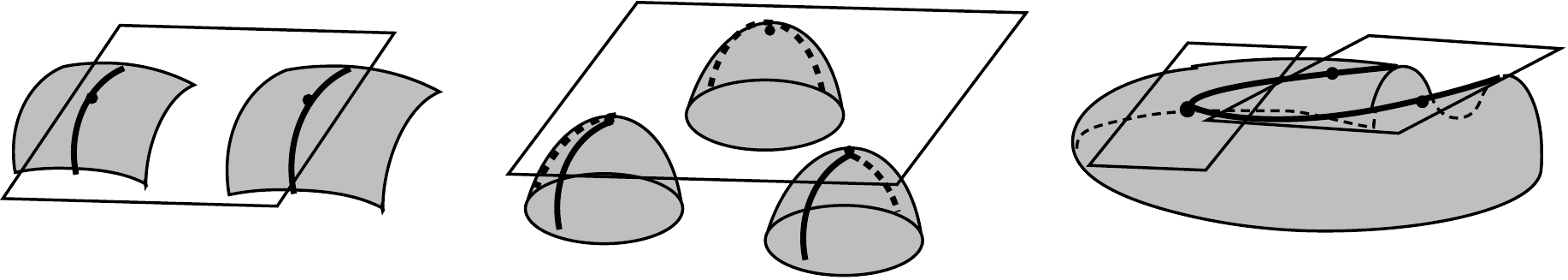} 
\end{center}
\hspace{0.5in} a) \hspace{1.0in} b) \hspace{1.5in} c) \hspace{0.5in}
\caption{\label{fig.3.3d} The generic multigerms for the height function.  
These determine the local structure for the boundary strata of 
$M_{\infty}$ for a multi-region configuration in $\R^3$ : a) of type  
$A_1^2$, b) of type $A_1^3$, and c) of type $A_3$.  The tangent planes 
shown are to points of multiple tangencies (except for the single $A_3$ 
point on the left in c).  The darker curves (including the darker dashed 
curves) denote the boundary strata bounding regions of $M_{\infty}$ 
(consisting of points whose outward pointing normals point away from the 
other regions).}
\end{figure} 
\par

\subsection*{Blum Medial Linking Structure} 
%%%  \hfill 
\par
We now consider the analogous Blum medial linking structure for a generic 
multi-region configuration.  First, we note that if the configuration has 
regions with boundaries and corners, then the Blum medial axes of the 
individual regions will not define skeletal structures. This is because the 
Blum medial axis will actually meet the boundary at the edges and 
corners.  However, in the case of disjoint regions $\{\gW_i\}$ in 
$\R^{n+1}$ with smooth generic boundaries (which do not intersect on 
their boundaries) there is a natural Blum version of a linking structure, 
which we introduce.  \par

\begin{Definition} 
\label{Def3.3a}
Given a multi-region configuration of disjoint regions $\bgW = \{\gW_i\}$ 
in $\R^{n+1}$, for $i = 1, \dots , m$, with smooth generic boundaries 
(which do not intersect on their boundaries), a {\em Blum medial linking 
structure} is a skeletal linking structure for which: 
\begin{itemize} 
\item[B1)] the $M_i$ are the Blum medial axes of the regions $\gW_i$ 
with $U_i$ the corresponding radial vector fields;
\item[B2)] the linking axis $M_0$ is the Blum medial axis of the exterior 
region $\gW_0$; and
\item[B3)] $M_{i\, \infty}$, which consists of the points in $\tilde M_i$ 
which are unlinked to any other points, corresponds to the points in 
$\cB_{i\, \infty} \subset \cB_i$ for which a height function has an 
absolute maximum on $\cB$ (or minimum for the height function for the 
opposite direction).  
\end{itemize} 
\end{Definition}
Because of B2), we will frequently in the Blum case refer to $M_0$ as the 
{\em linking medial axis}.  
\par
\begin{Remark}
It follows from B2) that if $x \in M_i$ and $x^{\prime} \in M_j$ are linked, 
the corresponding values of the radial and linking functions satisfy 
$\ell_i(x) - r_i(x) = \ell_j(x^{\prime}) - r_j(x^{\prime})$. 
\end{Remark}
\vspace{1ex}
\subsection*{\it Generic Linking Properties} \hfill
\par
Consider a configuration $\{ \gW_i\}$ with $\cB_i$ the boundary of 
$\gW_i$, and $M_i$ the Blum medial axis of $\gW_i$.  We let $\cB = \cup 
\cB_i$ and let $\gs(x, u) = \| x - u\|^2$ denote the distance squared 
function on $\R^{n+1} \times \R^{n+1}$ as earlier.  We consider a 
collection of smooth boundary points $S = \{x_1, \dots, x_r\}$ with $x_i 
\in \cB_{j_i}$.  
\begin{Definition}
\label{Def3.4}
The set of points $S = \{x_1, \dots, x_r\}$ exhibits {\em generic Blum 
linking} of type $(A_{\bga} : A_{\bgb_1}, \dots, A_{\bgb_r})$ if:
\begin{itemize}
\item[i)]  there is a $u_0 \in \gW_0$ so that $\gs(\cdot, u_0) | \cB$ has a 
common minimum on $S$ with value $\gs(x_i, u_0) = y_0$, and $\gs: \cB 
\times \R^{n+1},(S,w) \rightarrow \R, y_0$ is of generic type $A_{\bga}$;
\item[ii)]  for each $i$, there is a subset $S_{i} \subset \cB_{j_i}$ and a 
point $u_i \in M_i$, so that $\gs(\cdot, u_i) | \cB_{j_i}$ has a common 
minimum on $S_{i}$ with value $r_i$ and  $\gs : \cB \times \R^{n+1}, S_{i} 
\times \{u_i\} \to \R, r_i$ is of generic type $A_{\bgb_i}$;  
\item[iii)]  the singular sets $\gS_{\cB}^{(\bga)} \subset \cB$ and 
$\gS_{\cB_{j_i}}^{(\bgb_i)} \subset \cB_{j_i}$ intersect transversely in 
$\cB_{j_i}$; and 
\item[iv)] the images of the strata $\gS_{\cB}^{(\bga)} \cap 
\gS_{\cB_{j_i}}^{(\bgb_i)}$ in $\gS_{M_0}^{(\bga)}$ intersect in general 
position.
\end{itemize}
If $n + 1 = 7$, then we may also have generic Blum linking involving 
generic $\tilde E_7$ points for either self-linking $(\tilde E_7 : A^2_1)$ 
or simple linking $(A^2_1 : \tilde E_7, A^2_1)$.
\end{Definition}
 \par
We will often abbreviate the above notation with $\bgb =\{\bgb_1, \dots, 
\bgb_r\}$, by $(A_{\bga}:A_{\bgb})$.  For a given $(\bga: \bgb)$, we let 
$\gS_{M_0}^{(\bga: \bgb)}$ denote the 
set of points $u_0 \in \gW_0$ exhibiting the properties in i) - iv) in 
Definition \ref{Def3.4}.  We shall prove that it is a smooth submanifold of 
the stratum $\gS_{M_0}^{(\bga)}$ of the Blum medial axis $M_0$ of 
$\gW_0$.  Along with the stratum $\gS_{M_0}^{(\bga: \bgb)}$, we also 
have the corresponding stratum $\gS_{\cB_{j_i}}^{(\bga: \bgb)} \subset 
\cB_{j_i}$ consisting of those points $x_{j_i} \in \gS_{\cB_j}^{(\bgb_i)}$ 
which correspond to points in $\gS_{M_0}^{(\bga: \bgb)}$.  We shall prove 
for a generic configuration, that the $\{\gS_{\cB_{j}}^{(\bga: \bgb)}\}$ 
give a refinement of the stratification of $\cB_{j}$ by the 
$\{\gS_{\cB_{j}}^{(\bgb_i)}\}$, and  $\{\gS_{M_0}^{(\bga: \bgb)}\}$ gives a 
refinement of the Whitney stratification of $M_0$.  
\par 
For a given $\cB_j$ we may divide the strata into groups based on whether  
for $x_{j\, 1} \in \cB_j$, $\ga_j \geq 3$ or $\ga_i = 1$ for all $i$ and 
whether $\gb_{j\, 1} \geq 3$ or $\gb_{j\, i} = 1$ for all $i$.  The 
intersection of the strata for each of the pairs of groups will satisfy the 
Whitney conditions.  
In the low dimensional cases of multi-region configurations in $\R^2$ or 
$\R^3$, the $\gS_{\cB_{j}}^{(\bga: \bgb)}$ give a Whitney stratification of 
$\cB_{j}$.  \par  
For a configuration in $\R^3$ the stratifications of a boundary 
$\cB_j$ given by either $\gS_{\cB_j}^{(\bga)}$ or $\gS_{\cB_j}^{(\gb_j)}$ 
locally have the forms in Figure~\ref{fig.5e}.  Their transverse 
intersection implies that $0$ dimensional strata of one will lie in a 
smooth strata of the other.  Also, the $1$ dimensional strata will 
intersect transversally giving four possibilities.  For a point in one of 
these intersections, any other point in another $\cB_i$ associated to the 
corresponding point in $\gS_{M_0}^{(\bga)}$ must be of type $A_1^2$ by 
property iv) of Definition \ref{Def3.4}.  If instead the point in 
$\gS_{M_0}^{(\bga)}$ is in a singular stratum of dimension $0$, then all 
associated points in some $\cB_i$ are of type $A_1^2$.  A third 
possibility for a smooth point of $M_0$, which is an $A_1^2$ point for say 
$\cB_i$ and $\cB_j$, is that the images of the corresponding 
$1$-dimensional strata $\gS_{\cB_i}^{(\gb_i)}$ and 
$\gS_{\cB_j}^{(\gb_j)}$ intersect transversally in a smooth point of 
$M_0$.  This gives a corresponding analysis.  Together these yield all of 
the generic linking types listed in Table \ref{Table1} in the Addendum.
\par
Examples of some linking types and the strata are illustrated in 
Figure~\ref{fig.3.3c}.
\par
\begin{figure}[ht]
\begin{center}
\includegraphics[width=6.0cm]{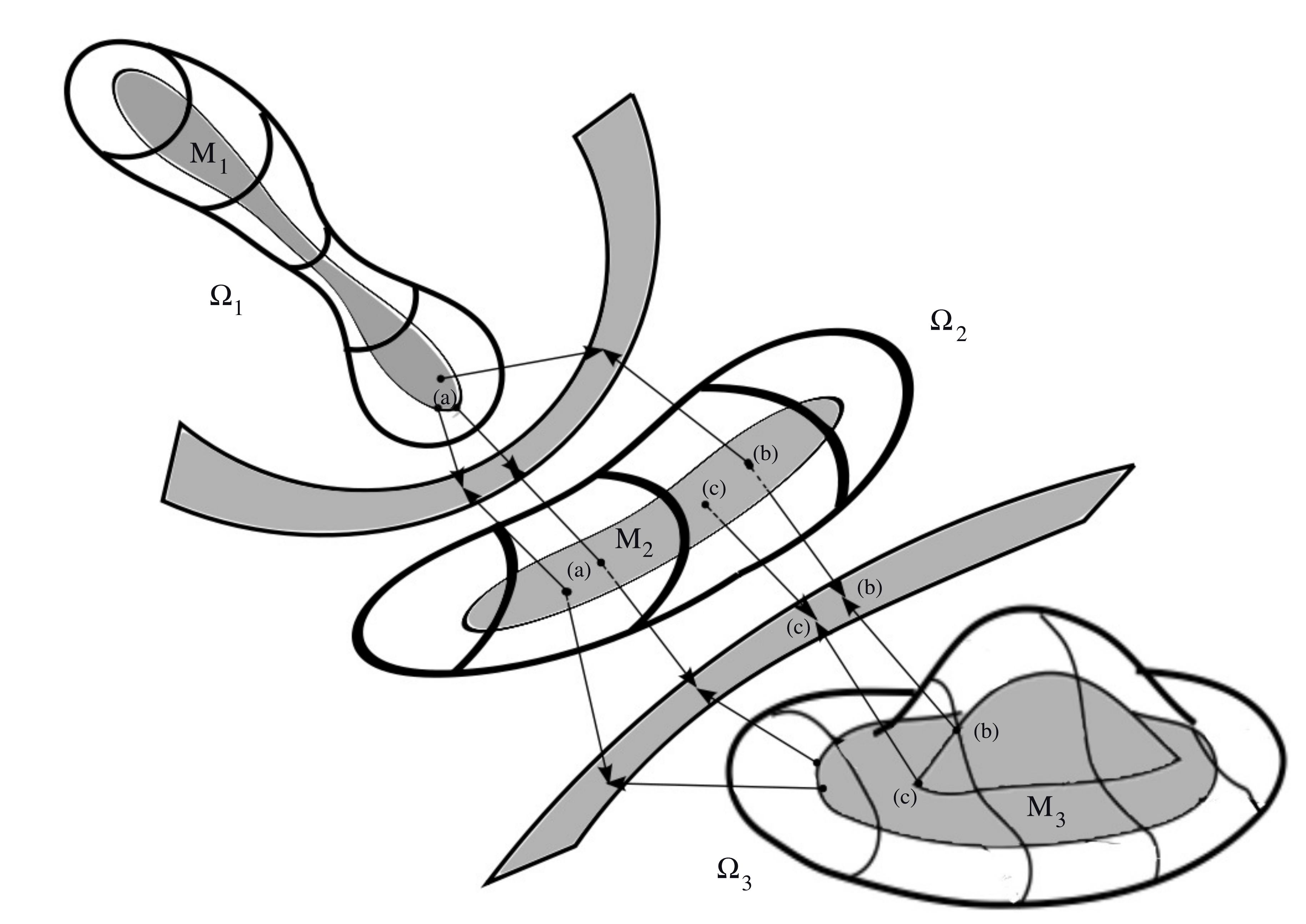}\hspace{.2in}
\includegraphics[width=6.0cm]{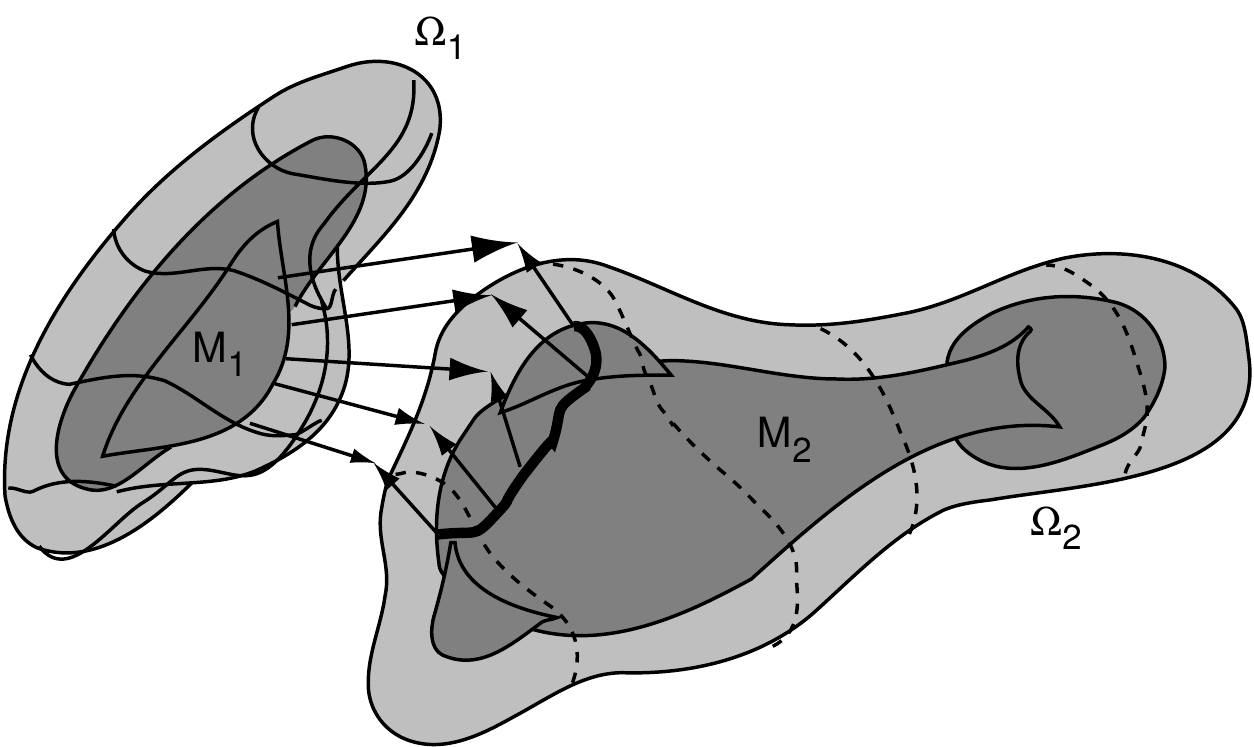}
\end{center} 
\hspace{1in} a) \hspace{2.5in} b) \hspace{1.0in}
\caption{\label{fig.3.3c} Examples of linking types.  In a) points (a) and (b) 
illustate linking of type $(A_1^2: A_3, A_3)$, and (c) of type $(A_1^2: 
A_3A_1, A_1^2)$.  In  b) is illustrated via the dark curve consisting of 
$1$-dimensional strata of type $(A_1^2: A_3, A_1^2)$ and isolated points 
of types $(A_1^2: A_3, A_1^3)$ and $(A_1^2: A_3, A_3)$ where the curve 
crosses the $Y$-branch curve of the medial axis $M_2$ of $\gW_2$, resp. 
where it meets the edge of $M_2$.}
\end{figure} 
\par
\begin{Remark}
\label{Rem3.5}
For a general multi-region configuration, we can substitute in place of 
$\gW_0$ a region $\gW_i$ which has multiple adjoining regions (including 
possibly the complement $\gW_0$) and the definition of \lq\lq generic 
linking\rq\rq of the adjoining regions relative to $\gW_i$ has the same 
form as in Definition \ref{Def3.4}.  We shall see that generically they have 
the same properties as for $\gW_0$. \par
If we consider the double $\tilde M_i$, only part of the stratification of 
$M_i$ coming from one side of $\gW_i$ appears and this simplifies the 
resulting stratification of $\tilde M_i$.  Only for the complement 
$\gW_0$ and $M_0$ does the stratification of $M_0$ itself play an 
important role. 
\end{Remark}

\subsection*{\it Generic Structure for $\cB_{\infty}$ and $M_{\infty}$}  
\hfill
\par
We recall that $\cB_{\infty}$ and $M_{\infty}$ denote the unions of the 
$\cB_{i\, \infty}$, respectively $M_{i\, \infty}$.  We will show in the 
generic case for $n + 1 \leq 7$ that the stratification of $\cB_{\infty}$ 
has the following properties.  The interior points of 
$\cB_{i\, \infty}$ are those points where a height function has a unique 
absolute minimum on $\cB$.  In addition, the boundary of $\cB_{i\, 
\infty}$ consists of strata $\gS^{(\bga)}_{\infty}$ defined by the $\cR^+$-
versal unfolding of a multigerm of the height function of type 
$\bA_{\bga}$ of 
$\cR^+_e$--codimension $\leq n$.  These strata lie in the smooth strata of 
the $\cB_i$ and correspond to the strata $\gS^{(\bga)}_{\cZ}$ of the 
spherical axis $\cZ$ and are of the same dimensions.  The strata 
$\gS^{(\bga)}_{\infty}$ again forms a collection of three stratifications 
with the same properties listed for the distance function for types 1) - 3).  
For $\R^2$ and $\R^3$, these together form Whitney stratifications for 
$\cB_{\infty}$ for elementary reasons.  
\par 
Furthermore, we will show that generically this stratification intersects 
tranversally the stratification $\gS^{(\bgb)}_{\cB_i}$ for the Blum medial 
axis of $\gW_i$.  
Then, the strata of $M_{\infty}$ are the images in $\tilde M_i$ of the 
transverse intersections 
$\gS^{(\bga)}_{\infty} \cap \gS^{(\bgb)}_{\cB_i}$.  
 \par
\begin{Definition}
\label{Def3.4b}
By $M_{\infty}$ and $\cB_{\infty}$ having {\em generic structure} we 
mean that each $\cB_{i\, \infty}$ has the above local structure with the 
resulting generic structure for each $M_{i\, \infty}$.
\end{Definition}
\par
\begin{Remark}
\label{Rem3.5b}
In the Blum case for disjoint regions with smooth boundaries, if we 
remove the self-linking set from the external linking axis, then we obtain 
what is classically called the \lq\lq conflict set\rq\rq, see e.g. 
Siersma\cite{Si}, Sotomayor-Siersma-Garcia\cite{SSG}, and Van Manen 
\cite{VaM}.  
\end{Remark}
\subsection*{Existence of a Blum Medial Linking Structure} 
%%%  \hfill
\par
We use the notation of \S \ref{S:sec1} and consider a model configuration 
$\bgD$ but with the $\gD_i$ disjoint regions with smooth boundaries 
$X_i$.  Then, we let $\bgW$ be a configuration based on $\bgD$ via the 
embedding $\varphi : \bgD \to \R^{n+1}$ so that $\gW_i = \varphi(\gD_i)$ 
for each $i$ (in this case each $\gG_i = \gW_i$).  As earlier, the space of 
configurations of type $\bgD$ is given by the space of embeddings $\emb 
(\bgD, \R^{n+1})$. \par

Then, the existence of Blum medial linking structures is guaranteed by the 
following, which in addition ensures generic linking (see also \cite{Ga}).

\begin{Thm}[Existence of Blum Medial Linking Structure]
\label{Thm3.5}
For $n +1 \leq 7$, we consider multi-region configurations $\bgW$ 
modeled 
by $\bgD = \{\gD_i\}$, consisting of disjoint regions with smooth 
boundaries (which do not intersect on their boundaries).  Then, for any 
compact region $\tilde \gW \subset \R^{n+1}$ with nonempty interior, 
there is an open dense set of embeddings $\Phi \in \emb (\bgD, 
\intr(\tilde \gW))$ such that : 
\begin{enumerate}
\item  the resulting configuration $\{\gW_i = \Phi(\gD_i)\}$ has a Blum 
medial linking structure such that each $M_i$ (including $M_0$) has 
generic local properties given by Theorem \ref{Thm3.1};
\item  the linking structure exhibits generic linking as in Definition 
\ref{Def3.4}; and 
\item  $M_{\infty}$ and $\cB_{\infty}$ have generic structure as given in 
Definition \ref{Def3.4b}.
\item In the case that $\tilde \gW$ is convex, the properties for a linking 
structure in the bounded case hold.  
\end{enumerate} 
\end{Thm}
\par
We shall prove Theorem \ref{Thm3.5} as a special case of the following 
more general genericity result for the {\it full Blum linking structure} for 
a multi-region configuration.  \par
\begin{Thm}[Full Blum Linking Structures]
\label{Thm3.6}
For $n + 1 \leq 7$, let $\bgD = \{\gD_i\}$ be a model multi-region 
configuration.  For a compact region $\tilde \gW \subset \R^{n+1}$ with 
nonempty interior, there is an open dense set of embeddings $\Phi \in 
\emb (\bgD, \intr(\tilde \gW))$ such that the resulting configuration 
$\bgW$ modeled by $\bgD$, with $\{\gW_i (= \varphi(\gD_i))\}$, has the 
following properties: 
\begin{itemize}
\item[i)] each $\gW_i$ has a Blum medial axis $M_i$ exhibiting the 
generic local properties in $\intr(\gW_i)$ given by Theorem \ref{Thm3.5}; 
\item[ii)] the complement $\gW_0$ exhibits the generic local properties 
on $M_0$ in $\intr(\gW_0) \cap \tilde \gW$;  
\item[iii)] the local structure of $M_i$ (including $i = 0$) near a boundary 
point of type $P_k$ or singular $Q_k$ point has the local generic 
edge-corner normal form given by Definition \ref{Def3.3};
\item[iv)]  at a smooth $Q_k$ boundary point of a region $\gW_i$, 
$\gS_{Q\, i}$ intersects the strata $\gS_{\cB_i}^{(\bga : \bgb)}$ 
transversally (and if $n + 1 = 7$ it does not contain $\tilde E_7$ points);
\item[v)] generic linking occurs between the smooth points of the regions 
and no linking occurs at edge-corner points, and this holds as well for 
generic linking between adjoining regions of a given region $\gW_i$ 
relative to the region $\gW_i$; and
\item[vi)]  $\cB_{i\, \infty}$ is contained in the smooth strata of the 
$\cB_i$ (as a piecewise smooth manifold), and $\cB_{i\, \infty}$ and the 
corresponding strata of $M_{i\, \infty}$ exhibit the generic properties 
given in Definition \ref{Def3.4b}.
\end{itemize}  
\end{Thm}
The last two parts of this paper will be devoted to developing the 
necessary transversality theorems, associated computations, and 
auxiliary results for proving this theorem.
\begin{Remark}
\label{Rem3.6}
For a bounded region $\tilde \gW$ whose boundary $\partial \tilde \gW$ is 
transverse to the linking vectors of $M_0$ (i.e. the extension of the radial 
lines from $M_{\infty}$ are tranverse to the limiting tangent spaces at 
points of $\partial \tilde \gW$), we can modify the linking vectors that 
extend beyond $\tilde \gW$, or the extension of the radial vectors from 
$M_{\infty}$, by truncating them at $\partial \tilde \gW$.  They will be 
stratawise smooth when we add the strata corresponding to the 
intersection  $M_0 \cap \partial \tilde \gW$.  For example if $\tilde \gW$ 
is convex, then, for almost all small translations of the configuration, the 
resulting $M_0$ is transverse to $\partial \tilde \gW$.  \par
An alternate approach for a more general region $\tilde \gW$ is to let the 
exterior region be $\gW_0 \cap \tilde \gW$ and allow linking from 
$\partial \tilde \gW$ with the other internal regions.  Then, including 
$\tilde \gW$ as part of the configuration, we may apply the existence 
theorem to give the result. 
\end{Remark}
\par
As an example, c) of Figure~\ref{fig.3.3b} illustrates iv) of the Theorem.  
By property iii) we see that the intersection of the closure of $M_i$ with 
the boundary consists of the singular points on the boundary.  This holds 
equally well for the complement with the closure of $M_0$ containing the 
singular strata of the $\gG_i$ in $\intr(\gW_0) \cap \tilde \gW$.  
\par 
One consequence of the characterization of generic linking in terms of the 
versality of the distance functions from Theorem \ref{Thm3.1} and the 
transversality of the stratifications is the determination of the 
codimensions of the strata.  
\begin{Corollary}
\label{Cor3.5}
For a generic embedding $\Phi : \bgD \to \R^{n+1}$ as in Theorems 
\ref{Thm3.5} or \ref{Thm3.6}, for $(\bga : \bgb_1, \dots , \bgb_r)$, the 
codimensions of the strata 
$\gS_{\cB_i}^{(\bga : \bgb)}$ and $\gS_{M_i}^{(\bga : \bgb)}$ satisfy 
\begin{equation}
\label{Egn3.4a}  
\codim_{\cB_i}(\gS_{\cB_i}^{(\bga : \bgb)}) \,\, = \,\, 
\codim_{\R^{n+1}}(\gS_{M_i}^{(\bga : \bgb)}) \, - \, 1 , 
\end{equation}
and
\begin{equation}
\label{Egn3.4b}
\codim_{\R^{n+1}}(\gS_{M_i}^{(\bga : \bgb)}) \, \, = \, \, \cR_e^+\text{-
}\codim(\bA_{\bga}) \, + \, \sum_{p=1}^{r} 
\cR_e^+\text{-}\codim(\bA_{\bgb_{j_p}}) \, - \, r. 
\end{equation}
\end{Corollary}
This will be derived in \S \ref{S:sec6}.  \par
As a consequence of the corollary, we can immediately list the generic 
linking types which can occur for a given  $\R^{n+1}$ as 
$\codim(\gS_{M_i}^{(\bga : \bgb)}) \leq n+1$.  
\par
\subsection*{Addendum: Classification of Linking Types for Blum Medial 
Linking Structures in $\R^3$} 
\par 
For a generic configuration $\bgW$ in $\R^3$, the Blum medial linking 
structure exhibits generic linking properties given by the Table 
\ref{Table1}.  We first briefly explain the features of the table.  
Dimension refers to the dimension in $\R^3$ of the strata where the given 
linking type occurs.  There are three 
types of linking: 1) linking between points on distinct medial axes; \lq\lq 
partial linking\rq\rq\, involving more than one point from one medial axis 
and point(s) from another; and 3) \lq\lq self-linking\rq\rq\, where linking 
is between points from a single medial axis.  \lq\lq Pure linking 
type\rq\rq\, refers to cases only occurring for self linking.  The $A_1^2$ 
and $A_1A_3$ linking can occur for either linking or self-linking; while 
$A_1^3$ and $A_1^4$ linking can occur for any of the three linking types.
\par

%%%%\begin{table}
%%%  \label{Table1} 
%%%%  \begin{tabular}{lccr}     &   &   \\

\begin{longtable}{lccl}
\caption{Classification of Linking Types for Blum Medial Linking 
Structures in $\R^3$} \\
  &  Linking Type  &   Dimension   &   Description of Linking  \\  
\hline
\endhead
%%%[1ex]
%%%   &   &   &   \\
   &  {\bf $A_1^2$ Linking}  &  &   \\ 
 [1ex] 
i) & $(A_1^2 : A_1^2, A_1^2)$   &   2   &   between $2$ smooth points \\  
[1ex]
ii) & $(A_1^2 : A_1^3,  A_1^2)$  &   1   &   between a smooth point \\
  &  &     & and a Y-junction point \\  [1ex]

iii) &  $(A_1^2 : A_3,  A_1^2)$  &  1  &   between a smooth point  \\
  &   &    &   and an edge point   \\  [1ex]

iv) &  $(A_1^2 : A_3A_1,  A_1^2)$   &  0   & between a fin point \\
  &   &     &   and a smooth point  \\   [1ex]

v) &  $(A_1^2 : A_1A_3,  A_1^2)$   &  0   & between a smooth point  \\
  &   &     &  associated to a fin point and \\
  &    &    &    another smooth point  \\   [1ex]

vi) &  $(A_1^2 : A_1^4,  A_1^2)$   &  0   &  between a smooth point \\
  &   &    &  and a $6$-junction point \\   [1ex]

vii) & $(A_1^2 : A_1^3,  A_1^3)$ &  0  & between $2$ Y-junction points \\
   &    &     &  \\  [1ex]

viii) & $(A_1^2 : A_3,  A_3)$ &  0  & between $2$ edge points \\
   &    &     &   \\  [1ex] 
ix) &  $(A_1^2 : A_1^3,  A_3)$ &  0  & between a Y-junction point  \\
   &   &     &   and an edge point \\  [2ex]

 $\text{ }$  & {\bf $A_1^3$, $A_1^4$ and $A_1A_3$ Linking}  &   &  \\  
[1ex]
%%%    &    &   &  \\

x)  &  $(A_1^3 : A_1^2,  A_1^2, A_1^2)$   &  1   &  between $3$ smooth 
points \\   [1ex]
 xi)  &  $(A_1^3 : A_1^3,  A_1^2,  A_1^2)$  &   0   &  between $2$ smooth 
points \\
   &  &     & and a Y-junction point \\ [1ex]

xii)  &  $(A_1^3 : A_3,  A_1^2,  A_1^2)$  &  0  &  between $2$ smooth 
points  \\
   &  &    &   and an edge point   \\  [1ex]

xiii)  &  $(A_1^4 : A_1^2,  A_1^2, A_1^2, A_1^2)$   &  0   &  between $4$ 
smooth points \\   [1ex]

xiv) &  $(A_1A_3 : A_1^2, A_1^2)$   &  0  &  $A_1A_3 $ linking between 
$2$\\ [1ex]     &  &    &   smooth points  \\ [2ex]

   & {\bf Pure Self-Linking}  &  &  \\ [1ex]

xv)  &  $(A_3 : A_1^2)$   &  1   & edge-type self-linking with \\
   &   &     &   a smooth point \\  [1ex]

xvi)  &  $(A_3 : A_1^3)$   &  0   & edge-type self-linking with \\
   &   &     &   a Y-junction point \\  [1ex]

xvii)  &  $(A_3 : A_3)$   &  0   & edge-type self-linking with \\
  &   &     &   an edge point \\  [1ex]
\label{Table1} 
\end{longtable}

%%%  &    &   &   \\
%%%    [1ex]  &    &   &  
%%%\end{tabular} 
%%%  \centerline 
%%%\caption{ Types of Linking.  } 
%%%  \end{table}  

\section{Retracting the Full Blum Medial Structure to a Skeletal Linking 
Structure}
\label{S:sec8}
We know by Theorem \ref{Thm3.6} that a generic multi-region 
configuration has a Blum medial structure.  If the regions are disjoint 
with smooth boundaries, then the Blum linking structure is a skeletal 
linking structure.  However, if the configuration contains regions which 
adjoin, then the Blum linking structure does not satisfy all of the 
conditions for being a skeletal linking structure.  Specifically the 
individual Blum medial axes of both the regions and the complement will 
extend to the singular points of the boundaries. 
\par
There are two perspectives on this.  On the one hand, as mentioned in 
\ref{Rem3.2}, we may view this as a \lq\lq relaxed form of a skeletal 
linking structure\rq\rq.  We shall see that from this structure we still 
obtain all of the local, relative, and global geometry of the individual 
regions and the positional geometry of the configuration.  However, if we 
consider the stability and deformation properties, such a structure does 
not behave well.  \par 
There are two approaches to modifying the full Blum linking structure to a 
skeletal linking structure.  One approach is when the configuration with 
adjoined regions can be viewed as a deformation of a configuration with 
disjoint regions.  The second is to modify the full Blum linking structure 
by a process of \lq\lq smoothing the corners\rq\rq of the regions.  We 
consider each of these.  \par

\subsection*{Example of Evolving Skeletal Linking Structure for Simple 
Generic Transition} \par
We will not attempt to handle the most general case but illustrate the 
method for an adjoining of two regions.  We assume that initially we have 
a configuration of two disjoint compact regions $\gW_i$ in $\R^{n+1}$ 
with smooth boundaries $\cB_i$ defined by the model $\Phi : \bgD \to 
\R^{n+1}$.  We consider a {\it simple generic transition} in the 
configuration defined by a smooth map $\Psi : \bgD \times [0, 1] \to 
\R^{n+1}$, with $\Psi_0 = \Phi$, where disjoint regions becoming adjoined 
causes a transition in the full Blum medial linking structure as in Figure~ 
\ref{fig.8.1}.  
\par
\begin{figure}[ht]
\begin{center} 
\includegraphics[width=10cm]{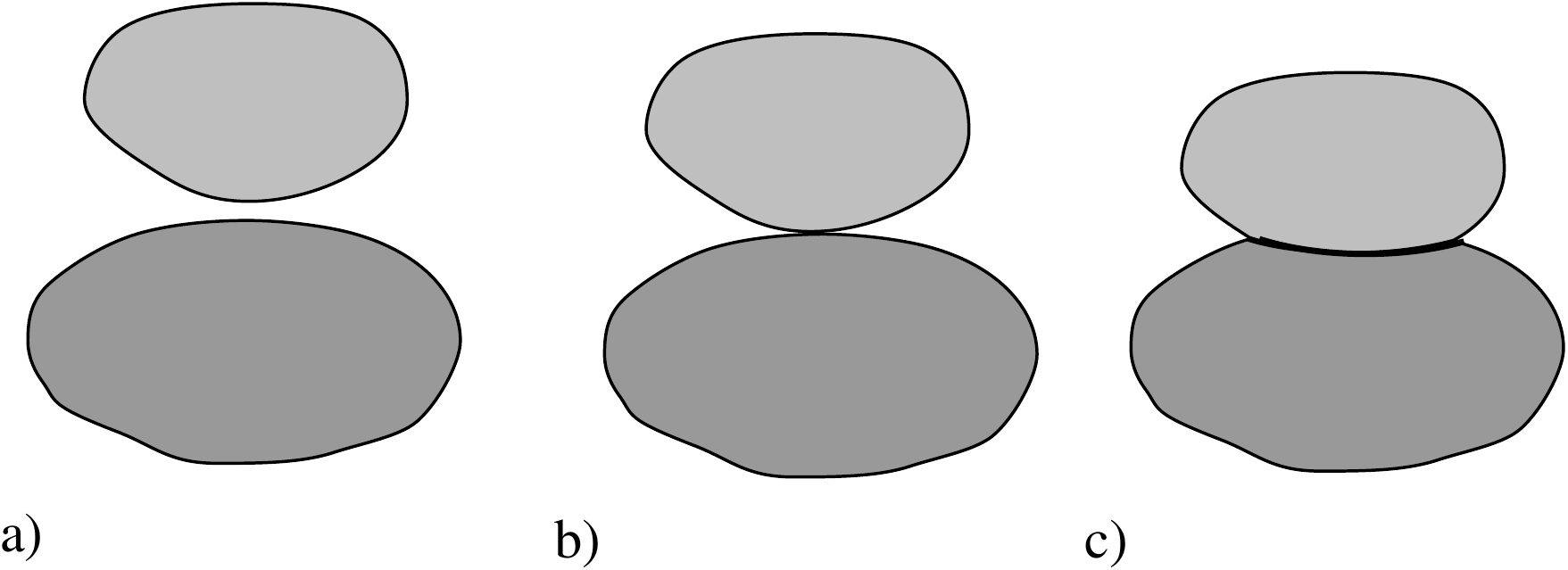} 
\end{center} 
\caption{\label{fig.8.1} The stages for a simple generic transition of two 
evolving regions $\gW^{\prime}_{i\, t}$ becoming adjoined: a) disjoint 
regions, b) simple tangency at $t = t_0$, and c) regions adjoined along 
$Z_t$ at $t > t_0$.}  
\end{figure} 
\par
We denote $\Psi_t = \Psi(\cdot, t)$ and suppose the restriction $\Psi_t | 
X_i$, $i = 1, 2$, is an embedding for each $t$.  Then, let $\gD_{i\, t} = 
\Psi_t(\gD_i)$ be the individual regions bounded by $\cB_{i\, t} = 
\Psi_t(X_i)$.  Suppose 
\begin{itemize}
\item[i)]  The individual $\gD_{i\, t}$ have generic Blum medial axes for 
all $t$.  
\item[ii)]  There is a $0 \leq t_0 \leq 1$ such that $\gD_{1\, t}$ and 
$\gD_{2\, t}$ are disjoint for $t < t_0$, and at $t_0$ there is a generic 
transition of tangency occurring at a single point corresponding to smooth 
points on the medial axes, so that for $t > t_0$, $\cB_{1\, t}$ and 
$\cB_{2\, t}$ intersect transversally and each is transverse to the radial 
lines of the other. 
\item[iii)] The external Blum medial axis of $\bgD_t = \gD_{1\, t} \cup 
\gD_{2\, t}$ is generic for all $t > t_0$.  
\item[iv)]  For $t > t_0$, there is a smooth submanifold $Z_t \subset 
\gD_{1\, t} \cap \gD_{2\, t}$ so that $\partial Z_t = \cB_{1\, t} \cap 
\cB_{2\, t}$ and $Z_t$ is transverse to the radial lines from smooth 
points of the Blum medial axes for each $\gD_{i\, t}$ (see Figure 
\ref{fig.8.1} and \ref{fig.8.2}).  
\end{itemize}
We then can form for $t >t_0$ new configurations consisting of $\gW_{i\, 
t}^{\prime}$ bounded by $(\cB_{i\, t} \backslash \gD_{i^{\prime}\, t}) \cup 
Z_t$ for $(i, i^\prime) = (1, 2), (2, 1)$.  The $\gW_{i\, t}^{\prime}$ are 
adjoined along $Z_t$.  
 The Blum medial structure for each $\gW_{i\, t}^{\prime}$ would now 
extend to $\partial Z_t$.  However, we can modify the Blum medial 
structure as shown in c) of Figure \ref{fig.8.2} by: 
\begin{itemize} 
\item[a)] retaining the Blum medial axes $M_{i\, t}$ of each $\gD_{i\, t}$; 
\item[b)] shortening the radial vectors which extend into $\gD_{1\, t} 
\cap \gD_{2\, t}$ so they end at $Z_t$; 
\item[c)] refine the stratification of each $\tilde M_{i\, t}$ by adding as a 
stratum $S_{i\, t} \subset \tilde M_{i\, t}$ the submanifold of each 
$\tilde M_{i\, t}$ which extends radially to $\partial Z_t$; 
\item[d)] extend those radial vectors which still meet the original 
$\cB_{i\, t}$ until they meet the external Blum medial axis. Those 
together with those shortened radial vectors give the linking vector field; 
and
\item[e)] retain the evolving external Blum medial axis for $\bgD_t$ as 
the linking medial axis. 
\end{itemize}
 \par 
These now define a family of skeletal linking structures for the varying 
configurations $\bgW_{t}^{\prime}$, which evolve continuously (and 
stratawise differentiably on the added strata $S_{i\, t} $), see 
Figure~\ref{fig.8.2}.

\begin{figure}[ht]
\begin{center} 
\includegraphics[width=10cm]{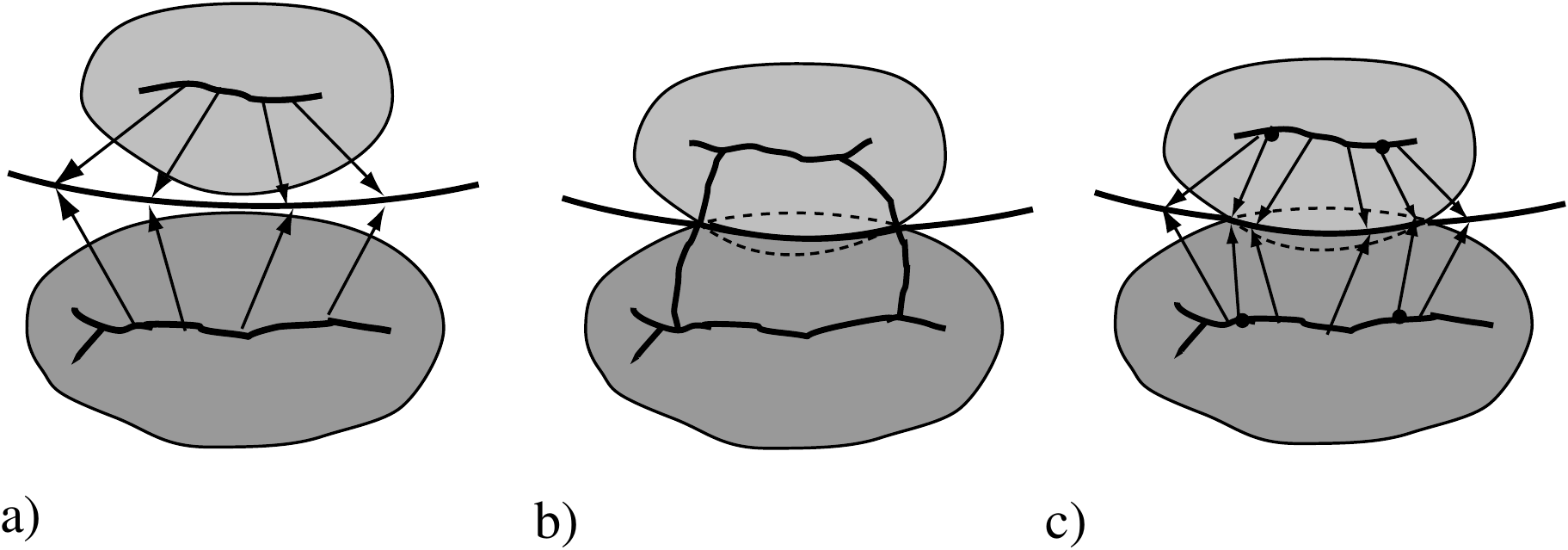} 
\end{center} 
\caption{\label{fig.8.2} Comparison of generic bifurcation of the full Blum 
linking structure versus evolution of the retracted skeletal linking 
structure for two evolving regions $\gW^{\prime}_{i\, t}$ becoming 
generically adjoined as in Figure~\ref{fig.8.1}.  Full Blum linking structure 
bifurcates by adding branches from a) unjoined regions to b) after 
becoming adjoined.  By contrast, the retracted skeletal linking structure 
evolves while retaining the structure of the skeletal sets, from a) 
unjoined regions to c) after being adjoined.}  
\end{figure} 
\par
\subsection*{Retracting Full Blum Linking Structure via Smoothing} 
\par
We consider a second situation where we modify the full Blum linking 
structure of a configuration with adjoining regions to a skeletal linking 
structure.  We do this using a \lq\lq smoothing of the corners of the 
regions\rq\rq, in a small neighborhood of the singular set. To precisely 
describe this we suppose we have a multi-region configuration $\bgW$, 
which includes regions which adjoin, modeled by $\Phi : \bgD \to 
\R^{n+1}$, so that it exhibits a generic full Blum medial structure, with 
$(M_i, U_i, \ell_i)$ the Blum medial structure for each $\gW_i$, and 
$M_0$ the external medial linking axis.  Given a neighborhood $W$ of the 
singular set of $\bgW$, i.e. the union of the singular strata of the 
boundaries, the goal is to modify the regions $\gW_i$ in the neighborhood 
$W$ so that the region boundaries are smooth and agree with the original 
boundaries outside of $W$, and such that the resulting Blum medial axes 
extend to a skeletal structure for $\bgW$.
\begin{Definition}
\label{Def8.1}
A {\em smoothing of the configuration} $\bgW$ defined by $\Phi$, with a 
neighborhood $W$ of the singular set of $\bgW$, consists of a disjoint 
configuration $\bgD^{\prime}$, with a region $\gD^{\prime}_i$ for each 
$\gD_i$, and an embedding $\Phi^{\prime} : \bgD^{\prime} \to \R^{n+1}$ 
satisfying the following conditions.  We let $\gW_i^{\prime} = 
\Phi^{\prime}(\gD^{\prime}_i)$ and $\cB_i^{\prime} = 
\Phi^{\prime}(X^{\prime}_i)$.  There is a neighborhood $W^{\prime} \subset 
W$ such that:
\begin{itemize}
\item[i)]  Each $\gW_i^{\prime} \subset \gW_i$, and $\gW_i^{\prime}$ has 
a generic Blum medial axis $M_i^{\prime}$ such that $\cB_i^{\prime} 
\backslash W^{\prime} = \cB_i \backslash W^{\prime}$. Moreover, the 
portion of the medial axis of $\gW_i$ defined from $W^{\prime} \cap 
\cB_i$ is contained in $W$ so that $M_i^{\prime} = M_i$ off $W$.  
\item[ii)]  The radial flow of $\gW_i^{\prime}$ from $M_i^{\prime} \cap 
W$ is nonsingular and remains transverse to the radial lines out to and 
including its first intersection with $M_0,$ and the radial lines intersect 
both $\cB_i$ and $M_0$  transversally (including the limiting tangent 
spaces of it at singular points).  This agrees with the radial flow for the 
full Blum structure off $W^{\prime}$.  
\item[iii)] The pull-back of the singular set of $M_0$ by the radial flow in 
ii) refines the stratification of $\tilde M_i^{\prime}$ on $W^{\prime}$. 
\end{itemize}
\end{Definition}
\begin{Remark}
\label{Rem8.3}
The verification that the radial flow is nonsingular and transverse to the 
radial lines is done using the radial and edge curvature conditions in 
\cite[Thm. 2.5]{D1}
\end{Remark}
\flushpar
An example of such a smoothing is given in Figure~\ref{fig.8.3}.
\par
\begin{figure}[ht]
\begin{center} 
\includegraphics[width=8cm]{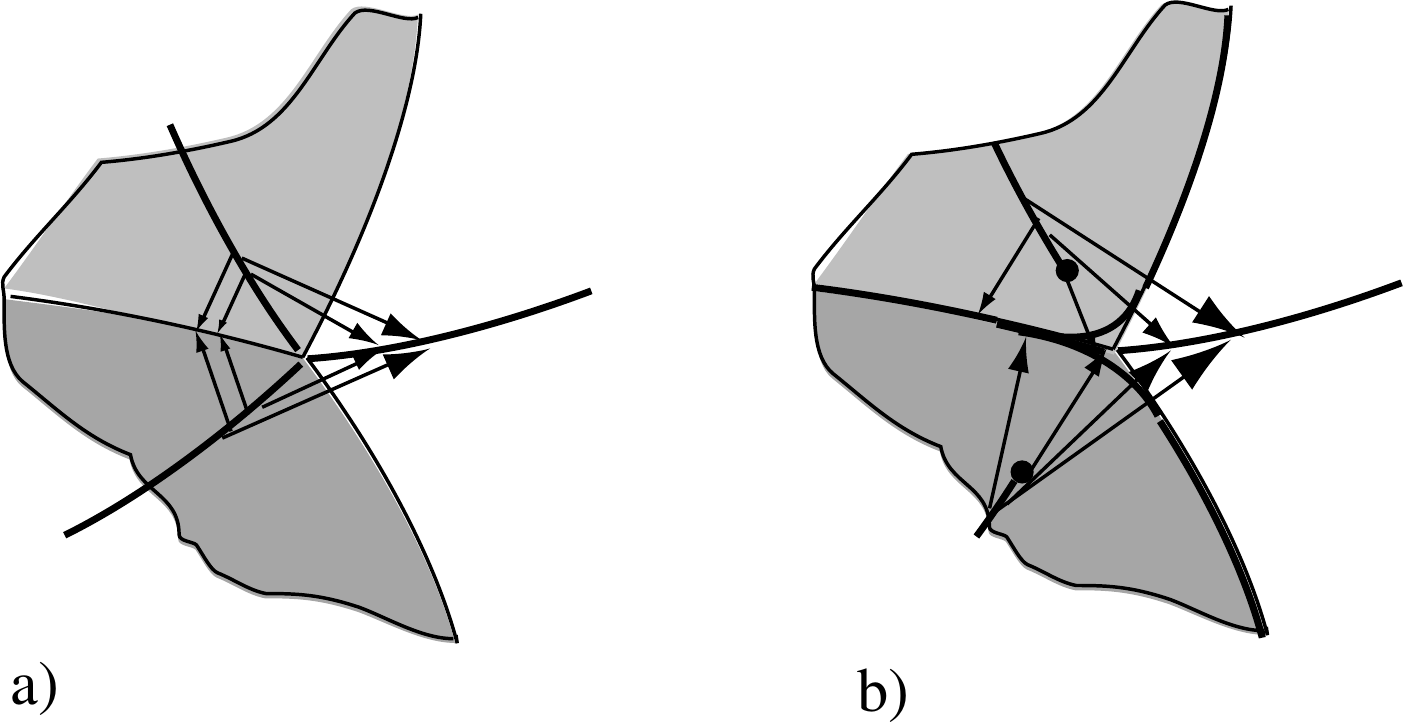} 
\end{center} 
\caption{\label{fig.8.3} An example of a smoothing of a configuration with 
adjoining regions $\gW_1$ and $\gW_2$ in the neighborhood of a corner 
point with linking vector fields shown:  a)  Blum linking structure and b) 
Smoothing and resulting retracted skeletal linking structure.}  
\end{figure} 
\par
The smoothing gives rise to a skeletal linking structure as follows.
\par
\begin{Proposition}
\label{Prop8.4}
For a configuration of regions with generic full Blum linking structure, a 
smoothing of the configuration allows the Blum linking structure to be 
replaced by a skeletal linking structure. 
\end{Proposition}
\begin{proof}
To define the skeletal structure, given the smoothing, we begin by letting 
the skeletal set in each $\gW_i$ be the Blum medial axis $M_i^{\prime}$ 
for $\gW_i^{\prime}$.  By ii) the radial lines from $M_i^{\prime}$ extend 
to the boundary $\cB_i$ and the vectors from points $x \in M_i^{\prime}$ 
to the intersection points of the radial lines from $x$ with $\cB_i$ are 
the radial vectors $U_i^{\prime}$ for the skeletal structure.  The 
stratification of $\tilde M_i^{\prime}$ is given by the refinement $\tilde 
M_i^{\prime\prime}$ of $\tilde M_i^{\prime}$ given by condition iii).  
Since the singular set of each $\cB_i$ is a portion of the singular set of 
$M_0$, condition iii) guarantees that the radial flow is nonsingular and 
smooth on the strata out to the boundary $\cB_i$.  \par
To extend the individual skeletal structures on the $\gW_i$ to a skeletal 
linking structure, we use $M_0$ for the external linking medial axis.  
We use the extensions of the radial lines of $\gW_i^{\prime}$ until they 
meet $M_0$.  These extensions define the linking vector fields $L_i$, 
which agree with the Blum linking vector fields off $W^{\prime}$.  The 
linking flow will agree with that for the Blum structure off $W^{\prime}$.   
On $M_i^{\prime} \cap W^{\prime}$ the linking flow will agree with the 
radial flow in $\cB_i$.  From $\cB_i$ to $M_0$, the radial flow will be 
nonsingular and transverse to the radial lines.  It follows by the radial, 
edge, and linking curvature conditions (see Propositions \ref{PropII.3.1} 
and \ref{PropII.3.1b} in \S\S \ref{S:secII.Rad.Flow} and 
\ref{S:secII.Link.Flow}), that the linking flow is also nonsingular and 
transverse to the radial lines.  Lastly, since both the radial flow and 
linking flow have the radial lines as flow lines, the inverse image of the 
singular set of $M_0$ under the linking flow will be the same as for the 
extended radial flow.  By the definition of the refinement of $\tilde 
M_i^{\prime}$ in iii), the $L_i^{\prime}$ (and hence $\ell_i^{\prime} = \| 
L_i^{\prime}\|$) will be smooth on the strata of the refinement.   \par
Thus, the collection of $(M_i^{\prime}, U_i^{\prime}, \ell_i^{\prime})$ 
define the resulting skeletal linking structure.
\end{proof}
 \par
We indicate an approach to constructing a smoothing of the corners of a 
configuration.  However, we will not give the details here to verify that 
the conditions are satisfied.  First, we construct for each region $\gW_i$ 
a smooth function $f_i$ defined on a neighborhood of $\gW_i$ such that:  
$f_i \geq 0$ on $\gW_i$, $f_i \equiv 0$ on $\cB_i$, with $\grad(f_i)$ non 
zero on the smooth points of $\cB_i$; and second, in a neighborhood of a 
$k$-edge-corner point the normal lines to the level sets of $f_i$ meet the 
limiting tangent planes of both the boundary and the external medial axis 
transversally.  Next, we use a bump function $\gd_i$ which does not 
vanish on $\sing(\cB_i)$ and has its support in a neighborhood $W$ of 
$\sing(\cB_i)$, and so that it and its first derivatives (in the corner 
coordinates) are bounded by $\gevar > 0$.  Then, the hypersurface defined 
by $f_i(x) = \gd_i(x)$ will satisfy the conditions for a sufficiently small 
$W$ and appropriate $\gd_i$.  To construct $f_i$ we use the local model 
coordinates for a $k$-edge-corner point.  We then use a partition of unity 
to piece together functions of the form $h = g\cdot \prod_{j = 1}^{k} x_j$, 
with $\grad(g)$ pointing into $\gW_i$ and $g$ nonvanishing on the 
boundary.  This gives the desired smoothing.  \par

\newpage
\section{Questions Involving Positional Geometry of a Multi-region 
Configuration}
\label{S:secII.1}
\par  
\subsection*{Introduction}  
 
\par
Having introduced medial/skeletal linking structures for a multiregion 
configuration $\bgW = \{ \gW_i\}$ in $\R^{n+1}$ in Part I, we now develop 
an approach to the \lq\lq positional geometry\rq\rq\, of the configuration 
using mathematical tools defined in terms of the linking structure.   We do 
so by building upon the methods already developed in the case of a single 
region with smooth boundary \cite{D1}, \cite{D2}.  Moreover, we will see 
that certain constructions and operators used for determining the 
geometry of single regions can be combined to give geometric 
properties of the configuration.  \par
There are several possible aspects to this.  One approach is to measure the 
differences between two configurations.  More generally given a collection 
of configurations, we may ask what are the statistically meaningful 
shared geometric properties of the collection of configurations, and how 
do the geometric properties of a particular configuration differ from 
those for the collection.  
To provide quantitative measures for these properties, we will directly 
associate geometric invariants to a configuration.  Such invariants may be 
globally defined depending on the entire configuration or locally defined 
invariants depending on local subconfigurations associated to each region.
\par
For example, if we view the union of the regions as a topological space, 
then we can measure the Gromov-Hausdorff distance between two such 
configurations.  We may also use the geodesic distance  between the two 
configurations measured in a group of global diffeomorphisms mapping one 
configuration to another.  Such invariants give a single numerical global 
measure of differences between two configurations.  Instead, we will use 
skeletal linking structures associated to the configurations to directly 
associate both global and local geometric invariants which can be used to 
measure the differences between a number of different features of 
configurations.  
\par
\begin{figure}[b] 
\begin{center} 
\includegraphics[width=3.5cm]{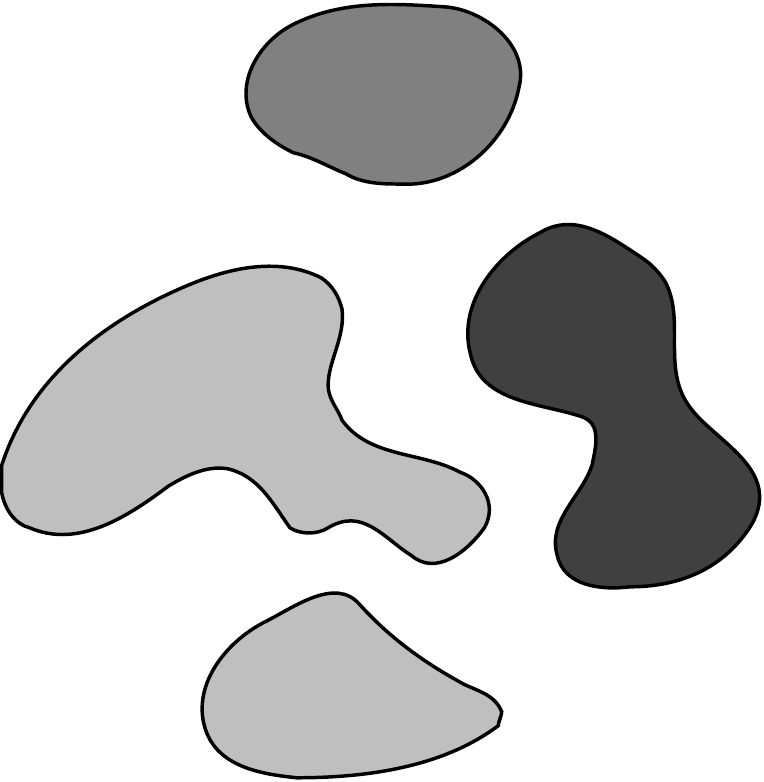}
\hspace*{0.05cm}
\includegraphics[width=3.5cm]{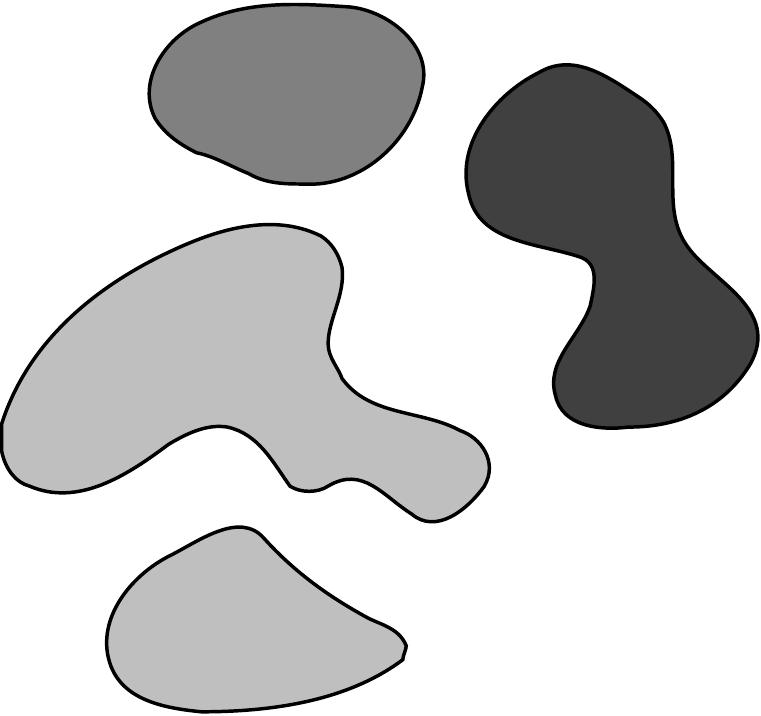} 
\hspace*{0.05cm}
\includegraphics[width=3.5cm]{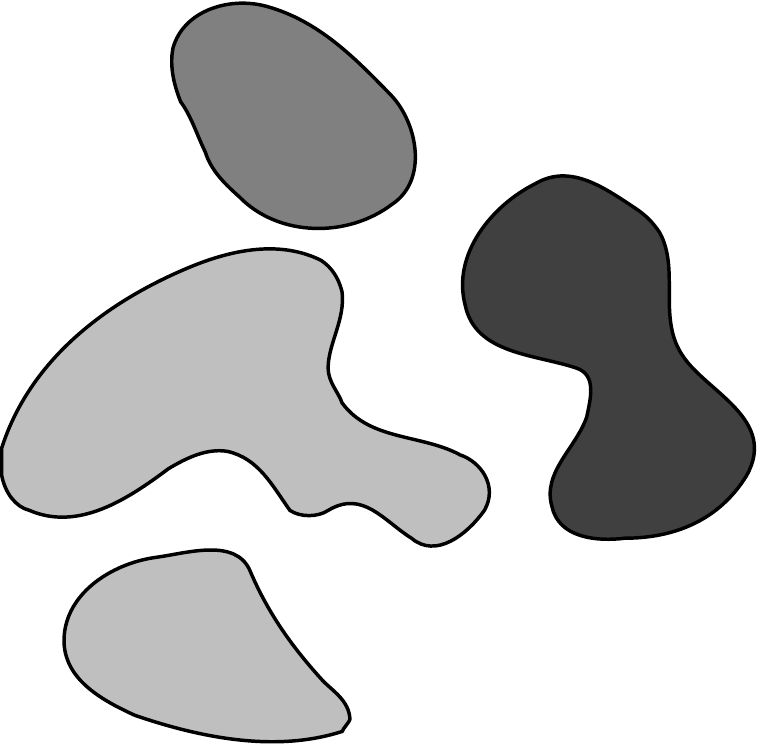} 
\end{center}  
(a) \hspace{1.5in} (b) \hspace {1.5in} (c) \hspace{1in} 

\caption{\label{fig.II1} Images exhibiting configurations of four regions.  
A basic problem is to determine the differences between the 
configurations that are due to changes in shapes of the regions versus 
those due to changes in positions.  Furthermore, one would want to find 
invariants which capture these differences.}  
\end{figure}
\par
In introducing these invariants, we will be guided by several key 
considerations.  The first involves distinguishing between the differences 
in the shapes of individual regions versus their positional differences and 
how each of these contributes to the differences in the configurations.  As 
illustrated in Figure~\ref{fig.II1}, the three configurations differ; but it 
is not clear to what extent this is due to shape differences of the regions 
versus changes in position, nor how these different contributions should 
be measured.  
\par
In measuring the relative positions of the regions, more than just the 
minimum distance between region boundaries is required.  As illustated in 
Figure~\ref{fig.II2}, while region $\gW_3$ touches $\gW_1$, it does so 
only on a small portion of $\gW_3$.  By comparison, $\gW_2$ does not 
touch $\gW_1$, but it remains close over a larger region.  A goal is to 
define a numerical measure of {\it closeness of regions} which takes into 
account both aspects.  Related to this is the issue of which regions that do 
not touch should be considered neighbors and what should be the 
criterion?  \par
\begin{figure}[ht]
\includegraphics[width=4cm]{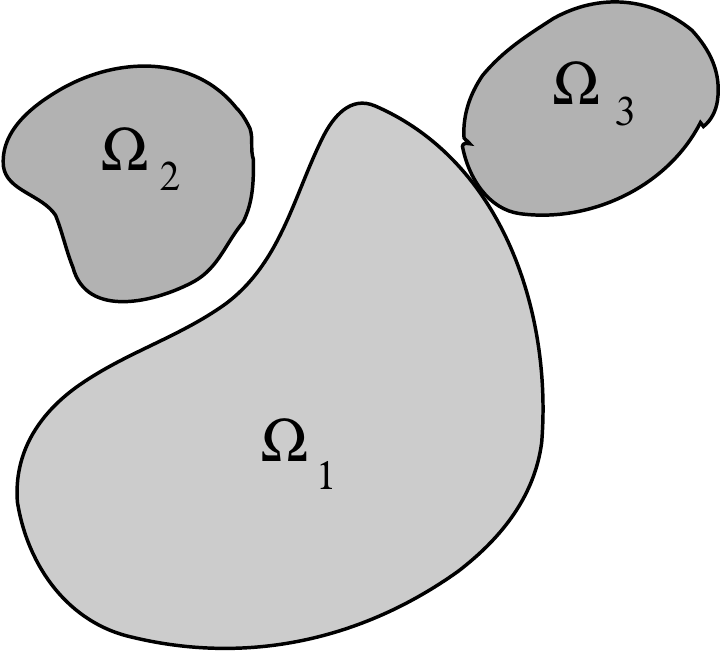}
\caption{\label{fig.II2} Measuring closeness of regions in a configuration.  
Although $\gW_3$ touches $\gW_1$, only a small portion of $\gW_3$ is 
close to $\gW_1$.  In contrast, $\gW_2$ does not touch $\gW_1$, but it 
remains close over a larger region.}
\end{figure} 
\par
A third issue for a configuration is viewing it as a {\it hierarchy of the 
regions}, indicating which regions are most central to the configuration 
and which are less {\it geometrically significant}.  For example in 
Figure~\ref{fig.II3}, the position of region $\gW_1$ makes it more 
important for the overall configuration in b) than in a), where it is more 
of an \lq\lq outlier\rq\rq.  For example, a small movement of $\gW_1$ in 
a) would be less noticeable and have a smaller effect to the overall 
configuration than in b).  By having a smaller effect we mean that the 
deformed configuration could be mapped to the original by a 
diffeomorphism which has smaller local distortions near the 
configuration in the case of a) versus b).  We will also provide a numerical 
measure of {\it geometric significance} of a region for the configuration.  

\begin{figure} 
\begin{center} 
\includegraphics[width=4.5cm]{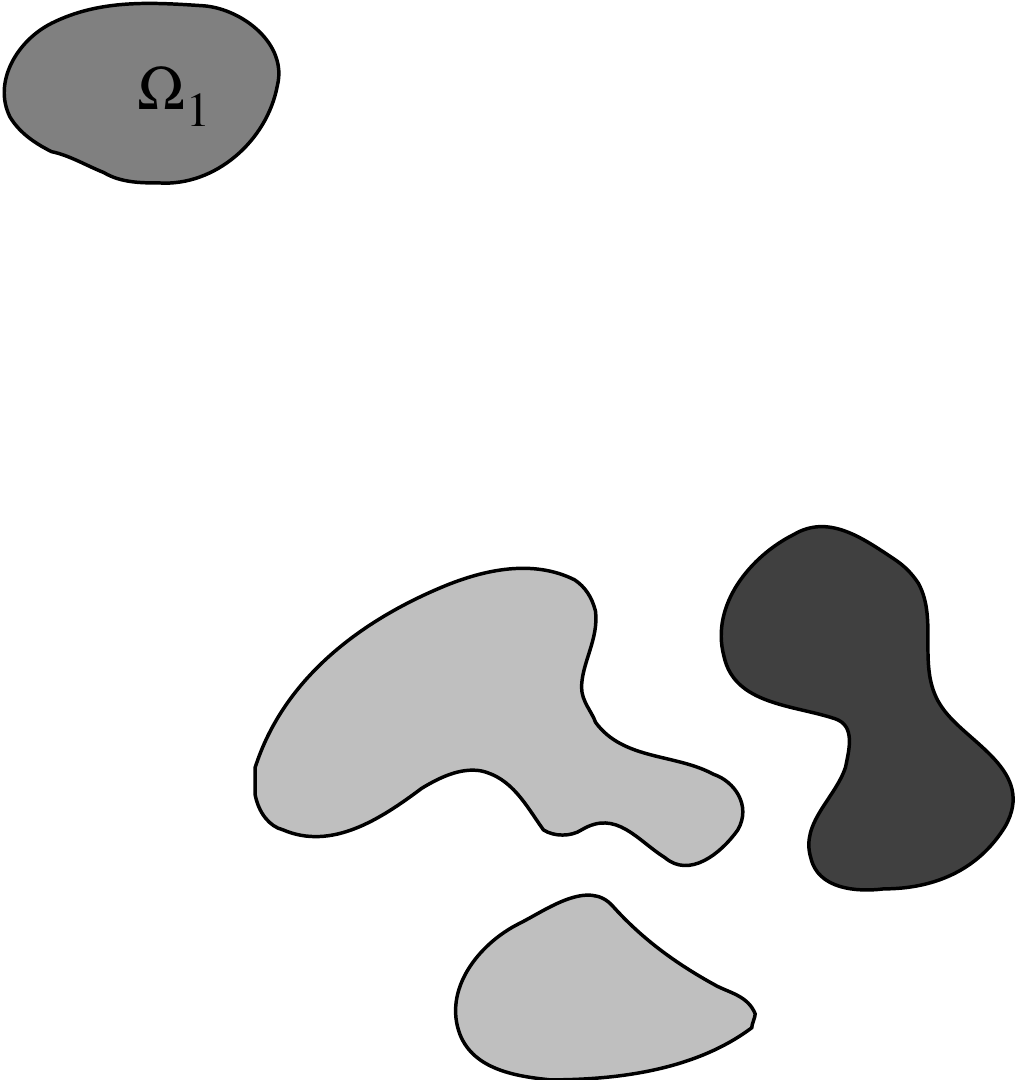}
\hspace*{0.5cm}
\includegraphics[width=3.5cm]{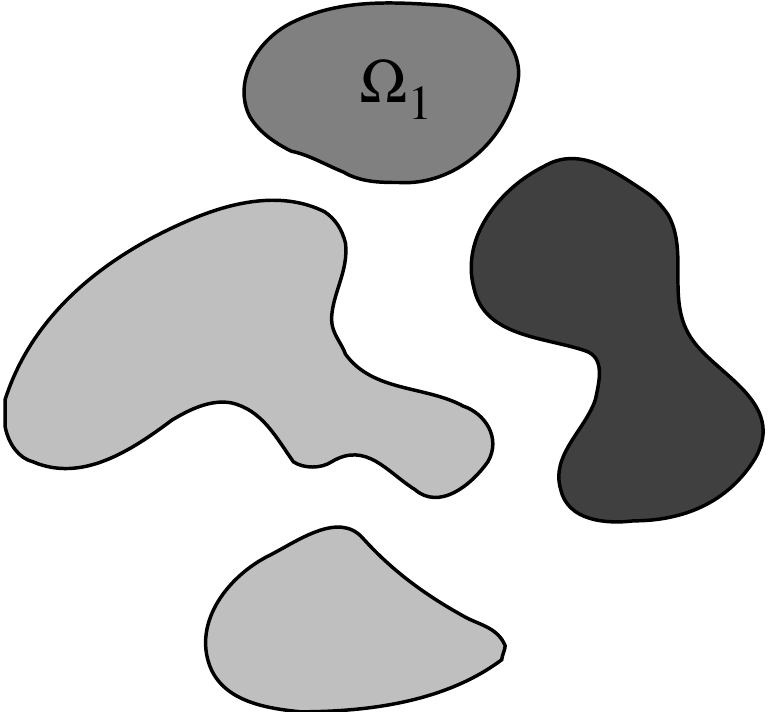}  
\end{center}  
\hspace {1.0in}  (a) \hspace{1.5in} (b)   

\caption{ \label{fig.II3} Images exhibiting configurations of four regions.  
In a), region $\gW_1$ is a greater distance from the remaining regions, 
and hence is less significant when modeling the configuration.  In b), the 
closeness of region $\gW_1$ to the other regions makes it more 
significant for the configuration.}  
\end{figure}
\par 
Along with issues of closeness, significance and hierarchy, there is also 
the question of when there are natural subconfigurations.  An example of 
this is seen in Figure~\ref{fig.II4}.  In a) is shown a configuration with 
distinct subconfigurations.  As these subconfigurations are moved 
together, a point is reached when they are no longer distinguished by 
geometric features.  This raises the question of whether there are 
numerical invariants which can be used to determine when there are 
identifiable subconfigurations. 
\par
We will combine the invariants which measure these geometric features  
to provide a \lq\lq tiered graph structure\rq\rq, which is a graph with 
vertices representing the regions, edges between vertices of neighboring 
regions, and values of significance assigned to the vertices, and closeness 
assigned to the edges.  Then, as thresholds for closeness and significance 
vary, the resulting subgraph satisfying the conditions will exhibit the 
central regions of the configuration, various subgroupings of regions and 
the hierarchy of relations between the regions.  
\par
Our goal in the subsequent sections is to first use the skeletal linking 
structure to identify neighboring regions and introduce linking 
neighborhoods between the regions.  Second, we will introduce numerical  
\lq\lq volumetric invariants\rq\rq\, for the indicated measures.  In turn, 
we will use the skeletal linking structure to compute the volumetric 
invariants via \lq\lq skeletal linking integrals\rq\rq computed on the 
skeletal sets.  
\begin{figure} 
\begin{center} 
\includegraphics[width=3.5cm]{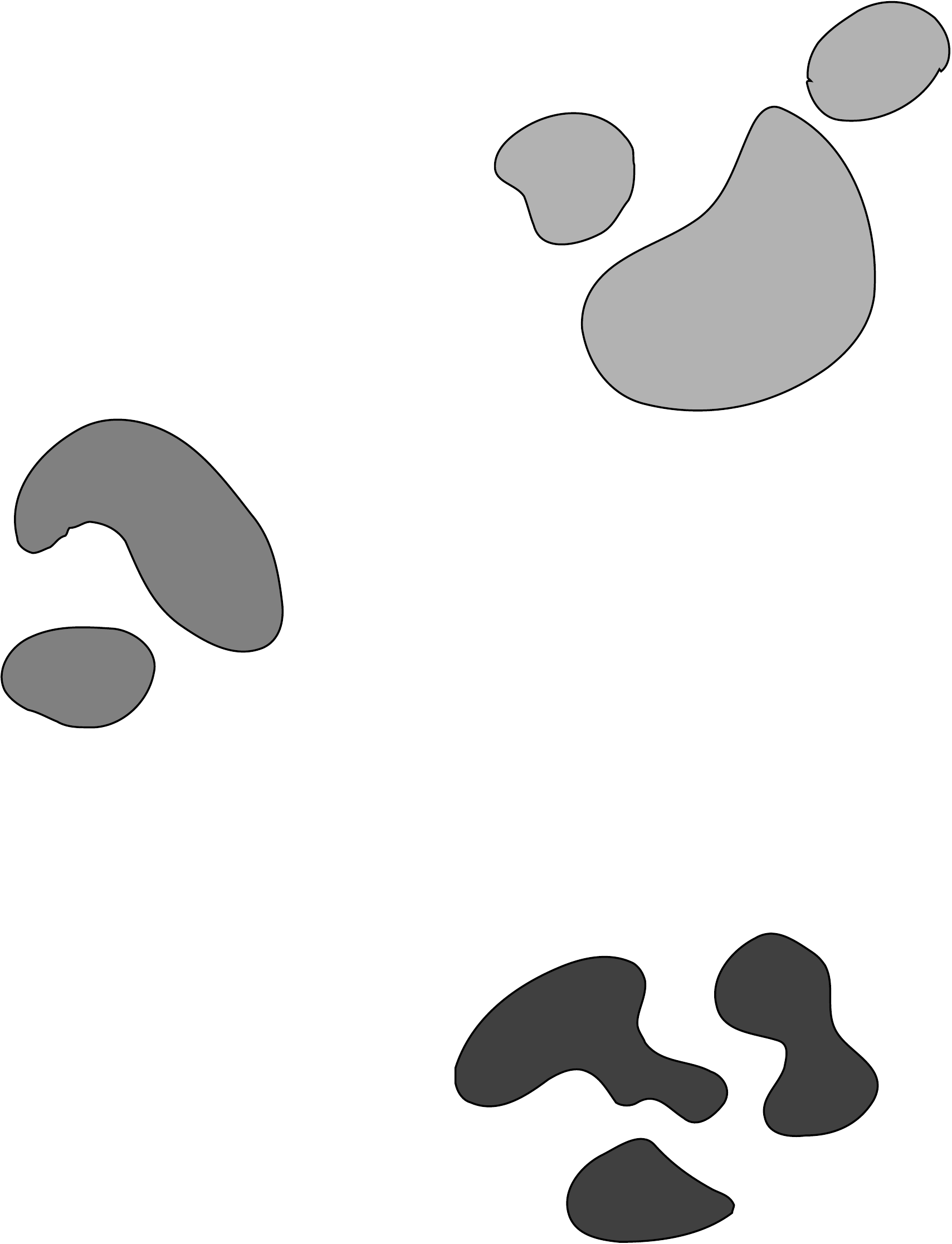}
\hspace*{1.0cm}
\includegraphics[width=3.5cm]{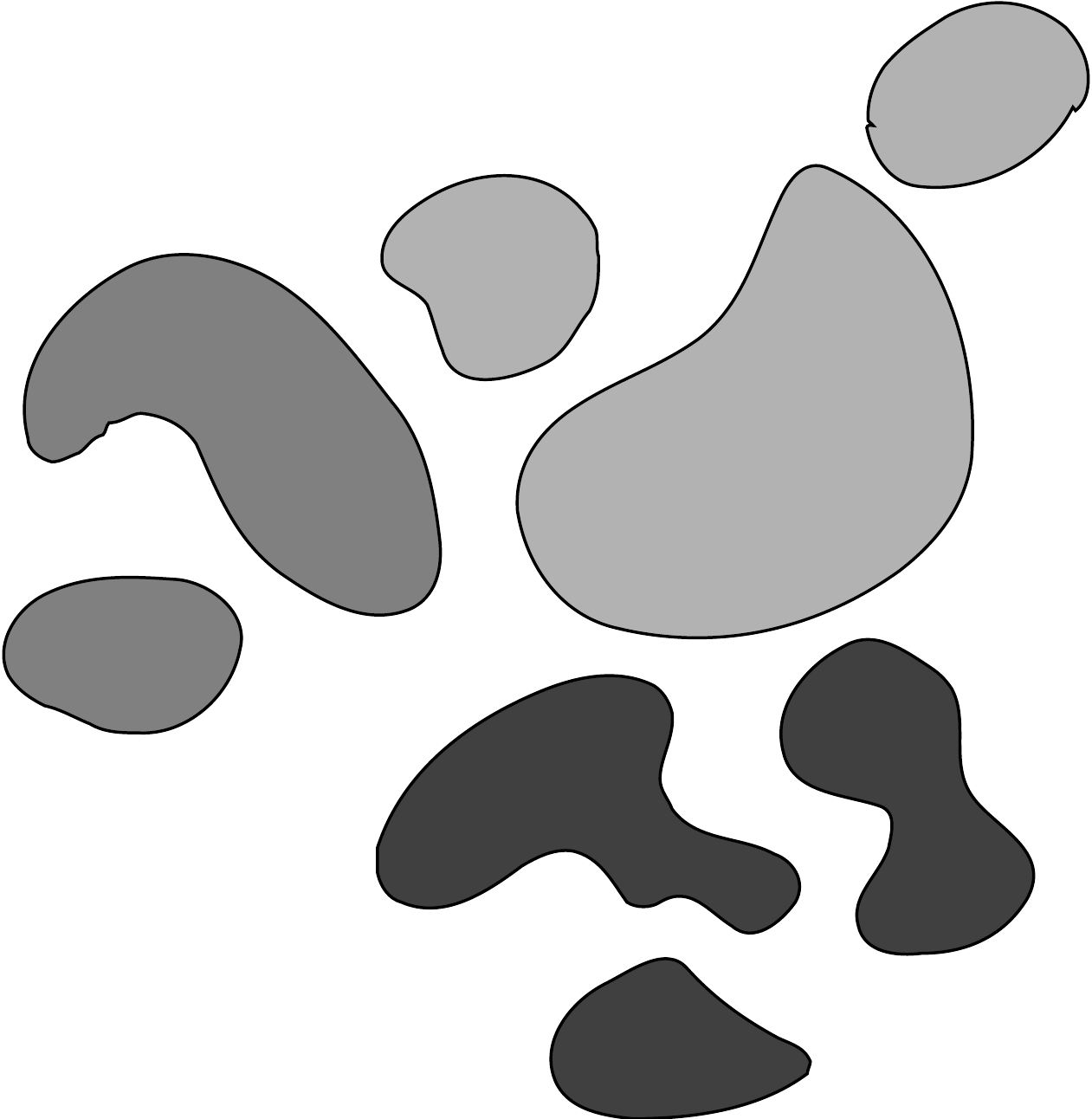}  
\end{center}  
\hspace {1.0in}  (a) \hspace{1.5in} (b)   

\caption{\label{fig.II4} Subconfigurations of regions.  In a) a configuration 
is formed from three groups of regions.  In b), the groups of regions have 
been moved closer; based on geometric position, the groups are no longer 
clearly distinguished.}  
\end{figure}
\par

\section{Shape Operators and Radial Flow for a Skeletal Structure}
\label{S:secII.Rad.Flow}
\par
 In preparation for obtaining the desired properties of the linking flow, we 
recall how for the skeletal structure $(M, U)$ for a single region $\gW$ we 
may introduce radial and edge shape operators. Using them, we have  
sufficient conditions for the nonsingularity of the radial flow and 
smoothness of the associated boundary (see \cite{D1}).  Furthermore, 
these operators can be used to compute both the intrinsic differential 
geometry of the boundary and the global geometry of the region (using 
\lq\lq skeletal integrals\rq\rq\, on $\tilde M$, see e.g. \cite{D2} and 
\cite{D4}).  Our eventual goal is to show using the skeletal linking 
structure that these results can be extended to the entire multi-region 
configuration.  \par

\subsection*{The Radial Flow} 
\par
In \S \ref{S:sec2}, we briefly recalled how from a skeletal structure $(M, 
U)$ in $\R^{n+1}$, we can define the \lq\lq associated boundary\rq\rq\, 
$$\cB \,\, = \,\, \{ x + U(x) : x \in M, \, \mbox{ all values of } U\}\, .$$  
However, conditions are required on the skeletal structure to ensure that 
the boundary is smooth.  This is achieved using the radial flow and the 
radial map which is the time one map of the radial flow. This allows us to 
establish additional properties of the region $\Omega$ bounded by $\cB$.  
The radial flow is defined as a map from the \lq\lq normal line 
bundle\rq\rq\, $N$ on $\tilde M$, the \lq\lq double of $M$\rq\rq\, which 
has a finite-to-one stratified mapping $\pi : \tilde M \to M$.  In the 
neighborhood $W$ of a point $x_0 \in M$ with a smooth single-valued 
choice for $U$, we define a local representation of the {\it radial flow} by 
$\psi_t (x) = x + t\cdot U(x)$, $1 \leq t \leq 1$, and the radial map 
$\psi_1 (x) = x + U(x)$.  Together the local $\psi_t$ define a global map 
$\psi: N \to \R^{n+1}$.  We will next recall the role of the radial and edge 
shape operators in guaranteeing the nonsingularity of the radial flow and 
the smoothness of the level sets $\cB_t = \{x + t\cdot U(x): x \in M, \, 
\mbox{ all values of } U\}$.  \par

\subsection*{Radial and Edge Shape Operators}  
\par
We define \lq\lq shape operators\rq\rq for the skeletal structure $(M, U)$.  
At a smooth point $x_0$ of $M$ we choose a \lq\lq smooth value\rq\rq of 
$U$.  We recall from \S \ref{S:sec2} that in a neighborhood of any smooth 
point of $M$ (i.e. a point in $M_{reg}$), values of $U$ on one side form a 
smooth vector field.  By a {\em smooth value} of $U$ we mean such a 
smooth choice of $U$ values.  Then, $U = r\cdot \bu$ for a unit radial 
vector field $\bu$.  \flushpar
{\it Radial Shape Operator:  }  We define for $v \in T_{x_0}M$
\begin{equation*}
S_{rad}(v) \quad = \quad - \proj_U(\pd{\bu}{v})
\end{equation*}
where $\proj_U$ denotes projection onto $T_{x_0}M$ along $U$ (in 
general, this is not orthogonal projection, see a) in Figure \ref{fig.6.3c}).   
The {\it principal radial curvatures} $\gk_{r\, i}$ are the  eigenvalues of 
$S_{rad}$.  For a basis $\bv = \{v_1, \dots , v_n\}$ for $T_{x_0} M$, we let 
$S_{\bv}$ denote the matrix representation of $S_{rad}$ with respect to 
$\bv$. \par
For a non--edge point $x_0 \in M$, a value of $U$ extends to be smooth on 
some local neighborhood of $x_0$ on any regular stratum containing $x_0$ 
in its closure.  For this {\em smooth value} of $U$, we may likewise define 
the radial shape operator at $x_0$.  Thus, the radial shape operator is also 
multivalued in that at a non--edge point $x_0$ there will be a radial shape 
operator for each value of $U$ at $x_0$, which at smooth points of $M$ 
means a value for each side of $M$.  
\par
\begin{figure}[ht]
\begin{center}
\includegraphics[width=6.0cm]{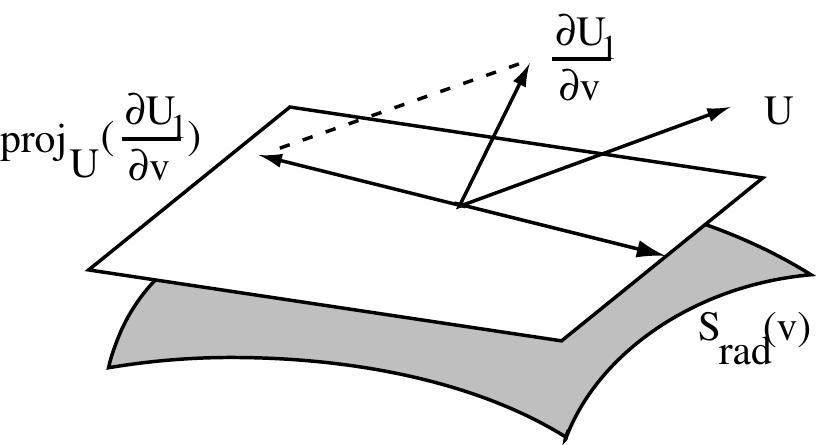}\hspace{.2in}
\includegraphics[width=6.0cm]{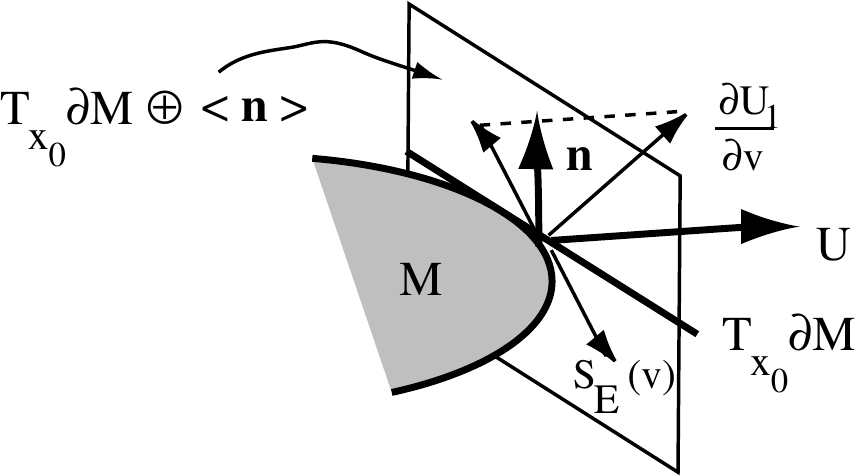}
\end{center} 
\hspace{1in} a) \hspace{2.5in} b) \hspace{1.0in}
\caption{\label{fig.6.3c} Illustrated in a) is the radial shape operator  and 
in (b) the edge shape operator.}
\end{figure} 
\par

\flushpar
{\it Edge Shape Operator:  } For an edge (closure) point $x_0$, a {\em 
smooth value} of $U$ is a smoothly varying choice of values for $U$, 
defined on a neighborhood of $x_0$ one side of $M$ for edge coordinates. 
Then for a normal vector field $\bn$ to $M$ on the same side of $M$ as 
$U$, we define \par
\begin{equation*}
S_{E}(v) \quad = \quad - \proj^{\prime} (\pd{\bu}{v}) \, .
\end{equation*}
Here $\proj^{\prime}$ denotes projection onto $T_{x_0}\partial M \oplus 
<\bn >$ along $U$ (again this is not orthogonal, see b) in Figure 
\ref{fig.6.3c}).  Let $\bv = \{v_1, \dots , v_n\}$ be a basis of $T_{x_0}M$ 
so that $\{v_1, \dots , v_{n-1}\}$ is a basis $T_{x_0}\bdyM$, and $v_n$ 
maps under the edge parametrization map to $c\cdot U$ for $c \geq 0$.  
We refer to $\bv$ as a {\it special basis} for $T_{x_0}M$.  Then, $S_{E\, 
\bv}$ is a matrix representation of $S_{E}$ with respect to the basis 
$\bv$ in the source and $\{v_1, \dots , v_{n-1}, \bn\}$ in the target. We 
let $I_{n-1, 1}$ be the 
$n \times n$--diagonal matrix with $1$\rq s in the first $n-1$ diagonal 
positions and $0$ in the last.  The {\it principal edge curvatures} 
$\gk_{E\, i}$ are the generalized eigenvalues of  $(S_{E\, \bv}, I_{n-1, 1})$ 
(i.e. $\gl$ such that $S_{E\, \bv} - \gl \cdot I_{n-1, 1}$ is singular).  \par

\subsection*{Curvature Conditions and Nonsingularity of the Radial Flow} 
\par
We now recall conditions guaranteeing the smoothness of the radial flow 
on smooth strata and the resulting smoothness of the level sets.  We 
emphasize that this is a local result as neither the flow defined on a 
stratified set nor its level sets are smooth, but rather they are stratified.   
Along with the radial and edge shape operators, an important role is also 
played by the multi-valued 
%%% compatibility $1$-form $\eta_U$, defined in \S \ref{S:sec2}.  
%%% Lastly we mention the {\it compatibility condition} which involves 
%%% the 
{\it compatibility $1$-form} $\eta_U = dr + \gw_U$, where $\| U\| = r$ is 
the {\it radial function} and $\gw_U(v) = v\cdot \bu$ for $U = r\cdot \bu$, 
with $\bu$ the multivalued unit vector field along $U$.  Then, $\eta_U$ is 
also multivalued with one value for each value of $U$.  
\par 
Then, we consider the following three conditions for a skeletal structure 
$(M, U)$:
\begin{enumerate}
\item ({\it Radial Curvature Condition} )  for a point $x_0 \in M\backslash 
\partial M$ and all values $U(x_0)$
\begin{equation*}
r < \min \{ \frac{1}{\gk_{r\, i}}\} \quad \mbox{ for all positive principal 
radial curvatures } \gk_{r\, i} \, ;
\end{equation*}
\item ({\it Edge Condition} )  for a point $x_0 \in \overline{\bdyM}$ 
(closure of $\bdyM$) ;
\begin{equation*}
r < \min \{ \frac{1}{\gk_{E\, i}}\} \quad \mbox{ for all positive principal 
edge curvatures } \gk_{E\, i}
\end{equation*}
\item ({\it Compatibility Condition} )  for a point $x_0 \in M$ (which 
includes edge points), $\eta_U \equiv 0$. 
\end{enumerate}
\begin{Remark}
\label{RemII.2.1} The radial curvature, edge, and compatibility conditions 
involve choices of values for $U$ and hence are multi-valued conditions at 
each point.  In the radial curvature condition it is to be understood that 
the $r$ value associated to a given value of $U$ satisfies the inequality 
for the shape operator associated to that value.  Thus, at smooth points of 
$M$, we have inequalities corresponding to each side of $M$.  
\end{Remark}  
\par  
The first two conditions allow us to control the local behavior of the 
radial flow, ensuring that singularities do not develop from smooth 
points, nor further singularities from singular or edge points.  \par
While the first two conditions are open conditions and hence robust, the 
third compatibility condition is not and reveals an essential feature about 
the level sets of the flow.  For any time $t <1$ the level sets are singular 
at points coming from the singular points of $M$ (including edge points).  
The compatibility condition at singular points of $M$ ensures that only at 
$t = 1$ when the flow reaches the boundary do the singularities 
simultaneously disappear so the boundary becomes weakly $C^1$ at the 
points corresponding to singular points of $M$ (which means there are 
unique limiting tangent spaces at these points, and as a consequence the 
boundary is $C^1$ at points coming from the strata of codimension $1$, 
see 
\cite[Lemma 5.4]{D1}).  However, if the compatibility condition does not 
hold on a stratum, then the boundary may have edges and corners at these 
image points.  Hence, skeletal structures can also be used for regions with 
boundaries and corners.  
\par

\subsubsection*{\it Local Nonsingularity of the Flow from Smooth and 
Edge Points} \par
We begin by stating how the radial curvature and edge conditions imply 
the local nonsingularity of $\psi$ and $\psi_t$ (these are given in 
\cite[Props. 4.1 and 4.4]{D1}).  
\begin{Proposition}
\label{PropII.2.1}
Let $U$ be a smooth value of the radial vector field defined in a 
neighborhood $W$ of $x_0 \in M$.  Suppose either that $x_0 \in M_{reg}$ 
and satisfies the radial curvature condition in the neighborhood $W$; or 
$x_0 \in \bdyM$ (or is an edge closure point) and satisfies the edge 
condition on the neighborhood $W$.  Then, 
\begin{enumerate}
\item $\psi : W \times \R \to \R^{n+1}$ is a local diffeomorphism at 
$(x_0, t)$ for $0 < t \leq 1$ (and also $t = 0$ for smooth points);
\item $\psi_t : W \to \R^{n+1}$ is a local embedding at $x_0$ for any $0 
\leq t \leq 1$; and
\item $\psi_t(W)$ is transverse to the line spanned by $U$ for each $0 
\leq t \leq 1$. \par
Conversely, suppose $\psi$ is nonsingular at $(x_0, t)$ for $0 < t \leq 1$, 
which is equivalent to $\psi_t(W)$ being nonsingular at $\psi_t(x_0)$ and 
transverse to the radial line for $0 < t \leq 1$.  Then the radial, resp. edge, 
curvature condition is satisfied for $x_0$ and the value $U$, depending on 
whether $x_0$ is a nonedge closure point, resp. an edge closure point.  
\end{enumerate}
\end{Proposition}
Although only one direction for this proposition was proven in the 
propositions just cited  in \cite{D1}, an examination of the proof given 
there also establishes the converse. $\Box$ \par

\subsubsection*{\it Global Nonsingularity of the Radial Flow}
The preceding local results at all points of $M$ can be combined with the 
compatibility condition at singular points to yield the following global 
result (given in \cite[Thm. 2.5]{D1}).
\begin{Thm}
\label{ThmII.1} 
Let $(M, U)$ be a skeletal structure which satisfies: the radial curvature 
condition at all nonedge points, the edge condition at edge points and edge 
closure points, and the compatibility condition at all singular points.  
Then (see Figure \ref{fig.II2c}):
\begin{enumerate}
\item The associated boundary $\cB$ is an immersed topological manifold 
which is smooth at all points except those corresponding to points of 
$M_{sing}$.  
\item  At points corresponding to points of $M_{sing}$, it is weakly $C^1$ 
(this implies that it is $C^1$ on the points which are in the images of 
strata of codimension $1$).  
\item   At smooth points, the projection along the lines of $U$ will locally 
map $\cB$ diffeomorphically onto the smooth part of $M$.  
\item   Also, if there are no nonlocal intersections, then $\cB$ will be an 
embedded manifold and $\gW \backslash M$ is fibered by the level sets 
$\cB_t$, $0 < t \leq 1$.  
\end{enumerate}
\end{Thm}
\par
\begin{figure}[ht]
\includegraphics[width=7cm]{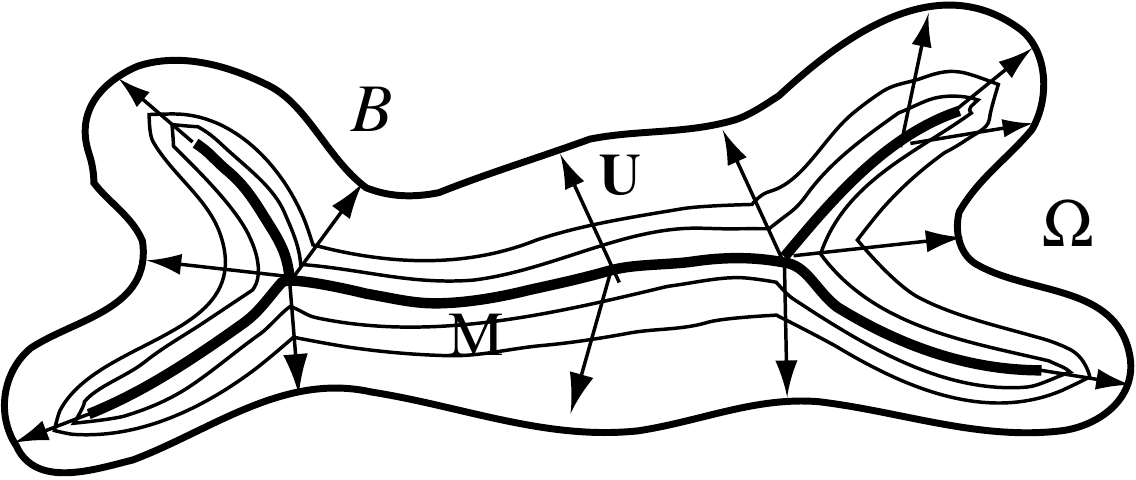}
\caption{\label{fig.II2c} Illustrating the radial flow from the skeletal set 
$M$, with the radial vector field $U$, flowing in the region $\gW$ to the 
boundary $\cB$, where the level sets are stratified sets forming a 
fibration of $\gW\backslash M$.}
\end{figure} 
\par

This allows us to completely describe the structure of the interior of the 
regions in a multi-region configuration. This extends to the geometry via 
the next results.
\par
\subsection*{Evolution of the Shape Operators under the Radial Flow}  
 \par
The last feature of the radial flow which we recall is the evolution of the 
radial and edge shape operators under the radial flow, with the eventual 
goal of extending these results to the linking flow.  This allows us to 
relate in the Blum case the radial geometry on the medial axis with the  
differential geometry of the boundary.  
\par
We first consider the evolution of the radial shape operator under the 
radial flow.  Let $x_0 \in M_{reg}$, and let $\{ v_1, \dots , v_n\}$ be a 
basis for $T_{x_0}M$.  We suppose we have chosen a smooth value of $U$ 
in a neighborhood of $x_0$.  If $x_0$ is a non--edge singular point, then 
we can carry out an analogous argument on a local component for $x_0$.  
\par  
For a given $t$, let 
$$ v_i^{\prime} = d\psi_t(v_i) = \pd{\psi_t}{v_i} \quad \mbox{for } i = 1, 
\dots , n. $$
With $r$ denoting the radial function as in \S \ref{S:sec2}, we suppose 
that  $\frac{1}{t r}$ is not an eigenvalue of $S_{rad}$ (at $x_0$).  Then, by 
the proof of Proposition 4.1 of \cite{D1}, $\psi_t$ maps a neighborhood 
$W$ of $x_0$ diffeomorphically to a smooth submanifold transverse to 
$U(x_0)$.  Thus, the image of $U$ along $\psi_t$ remains transverse in 
some neighborhood of $x_0^{\prime} = \psi_t(x_0)$ to $W^\prime = 
\psi_t(W) \subset \cB_t$.  Hence, it has a well-defined radial shape 
operator, which we denote by $S_{rad\, t}$.  We will compute 
$S_{\bv^{\prime}\, t}$, the matrix representation of $S_{rad\, t}$ with 
respect to the basis $\{ v_i^{\prime}\}$.  
\par  Then, the following two results are consequences of Proposition 2.1 
and Corollary 2.2 of \cite{D2}.

\par 
\subsubsection*{\it Evolution of the Radial Shape Operator from Smooth 
Points}

\begin{Proposition}
\label{PropII.3.1}
Suppose that at a smooth point $x_0 \in M_{reg}$, we have a smooth value 
of $U$ and a basis $\{ v_i\}$ for $T_{x_0}M$.  Let $\{ v_i^{\prime}\}$ 
denote the image of $\{ v_i\}$ under $d\psi_t(x_0)$.  If $\frac{1}{t r}$ is 
not an eigenvalue of the radial shape operator $S_{\bv}$ at $x_0$, then the 
radial shape operator $S_{\bv^{\prime}\, t}$ for $\cB_t$ at $x_0^{\prime} 
= \psi_t(x_0)$ for the corresponding smooth value of $U$ is given by
\begin{equation}
\label{EqnII.3.5}
S_{\bv^{\prime}\, t}  \quad = \quad  (I - t r\cdot S_{\bv})^{-1}S_{\bv} \, .
\end{equation}
\end{Proposition} 
\par
With the identification of the tangent spaces $d\psi_t(x) : T_xM \simeq 
T_{\psi_t(x)}\cB$, we may write the preceding in the basis independent 
form
\begin{equation}
\label{Eqn.II3.5b}
  S_{rad, t} \,\, = \,\,  (I - t r\cdot S_{rad})^{-1}S_{rad} \, .  
\end{equation}
\begin{Remark}
\label{Rem7.3}
If the compatibility $1$-form $\eta_U$ vanishes on an open subset of 
$\tilde M_{reg}$, then $U$ will be orthogonal to $\cB$ at each point of its 
the image under the radial flow (see \S \ref{S:secII.Rad.Flow}) as happens 
for the Blum case.  It follows that the formula (\ref{EqnII.3.5}) gives the 
differential geometric shape operator for $\cB$; thus, capturing the 
geometry of the boundary, see \cite{D2}.  \par 
\end{Remark}
\subsubsection*{\it Principal Radial Curvatures for $\cB_t$} \par
Then, we can deduce information about the principal radial curvatures at 
$x_0^{\prime}$ in terms of those at $x_0$.  
\begin{Corollary}
\label{CorII.3.2} Under the assumptions of Proposition \ref{PropII.3.1}, 
there is a correspondence (counting multiplicities) between the principal 
radial curvatures $\gk_{r\, i}$ of $M$ at $x_0$ and $\gk_{r\, t\, i}$ of 
$\cB_t$ at $x_0^{\prime}$ given by 
\begin{equation*}
\gk_{r\, t\, i} = \frac{\gk_{r\, i}}{(1-t r \gk_{r\, i})}
\quad \mbox{or equivalently} \quad
\gk_{r\, i} = \frac{\gk_{r\, t\, i}}{(1+t r \gk_{r\, t\, i})} \, .
\end{equation*}
Furthermore, if $e_i$ is an eigenvector for the eigenvalue $\gk_{r\, i}$ of 
$S_{rad}$, then $e_i^{\prime}$, which has the same coordinates with 
respect to $\bv^{\prime}$ as $e_i$ has with respect to $\bv$, is an 
eigenvector with eigenvalue $\gk_{r\, t\, i}$ of $S_{rad, t}$.
\end{Corollary}

\subsubsection*{\it Evolution of the Radial Shape Operator from Edge 
Points} \par
We can carry out an analogous line of reasoning for the evolution of the 
radial shape operator for points corresponding to an edge point $x_0$.  A 
smooth value of $U$ in a neighborhood of $x_0$ corresponds to one side of 
$M$.  Although $\cB_t$ is not smooth at $\psi_t(x_0)$ if $t < 1$,  we note 
that by Proposition 4.4 of \cite{D1}, provided $\frac{1}{tr}$ is not a 
generalized eigenvalue for $(S_{E\, \bv}, I_{n-1, 1})$, the one side of 
$\cB_t$ corresponding to $U$ is smooth and is transverse to $U$ at $x_0$ 
when $t > 0$.  Thus, the radial shape operator is defined for $\cB_t$ at 
points corresponding to edge points when $t > 0$.  Hence, we may compute 
the radial shape operator $S_{\bv^{\prime} t}$ for this one side using 
\cite[Proposition 2.3]{D2} as follows.  
\begin{Proposition}
\label{PropII.E3.1}
Suppose that at an edge point $x_0 \in \overline{\bdyM}$, we have a 
smooth value of $U$ (corresponding to one side of $M$) and a special basis 
$\{ v_i\}$ for $T_{x_0}M$.  Let $\{ v_i^{\prime}\}$ denote the image of $\{ 
v_i\}$ under $d\psi_t(x_0)$.  If $\frac{1}{t r}$ is not a generalized 
eigenvalue of $(S_{E\, \bv}, I_{n-1, 1})$, then the radial shape operator 
$S_{\bv^{\prime}\, t}$ for $\cB_t$ at $x_0^{\prime} = \psi_t(x_0)$ is 
given by
\begin{equation}
\label{EqnII.E3.5}
S_{\bv^{\prime}\, t}  \quad = \quad  (I_{n-1, 1} - t r\cdot S_{E\, \bv})^{-
1}S_{E\, \bv} \, .
\end{equation}
\end{Proposition}  
We note that unlike the situation for the radial shape operator, $S_{E\, 
\bv}$ does not necessarily commute with $I_{n-1, 1}$ so the order of the 
factors is important.  \par
Unlike the case of non--edge points, we cannot in general deduce a simple 
formula for the principal radial curvatures for $S_{\bv^{\prime}\, t}$ in 
terms of the principal edge curvatures.  However, in certain special cases, 
we can carry out the calculations (see e.g. \cite{D2} or \cite{D7}).  
\par

\section{Linking Flow and Curvature Conditions}
\label{S:secII.Link.Flow}
\par
We consider a skeletal linking  structure $\{(M_i, U_i, \ell_i)\}$.  If each 
$(M_i, U_i,)$ satisfies the conditions of Theorem \ref{ThmII.1}, then we 
obtain regions $\gW_i$ which, provided the regions are disjoint or 
intersect appropriately, give a multi-region configuration $\bgW = \{ 
\gW_i\}$.  There are additional conditions to be  satisfied for $\{(M_i, 
U_i, \ell_i)\}$ to be the skeletal linking structure for $\bgW$.  One part of 
this requires properties of the linking flow.  Using analogs of the 
preceding results for the radial flow, we will give sufficient conditions 
for nonsingularity of the linking flow.  We also will derive analogous 
formulas for the evolution of the radial and edge shape operators under 
the linking flow.   \par

\subsection*{Nonsingularity of the Linking Flow} 
\par
We first establish the conditions for the nonsingularity of the linking 
flow $\gl$ (or $\gl_t$) introduced in the definition of the skeletal linking 
structure (see (\ref{Eqn2.1}) in \S \ref{S:sec2}).  For the skeletal linking  
structure $\{(M_i, U_i, \ell_i)\}$, we have the following two conditions:
\begin{enumerate}
\item ({\it Linking Curvature Condition} )  For all points $x_0 \in M_i 
\backslash \partial M_i$ and all values $U_i(x_0)$,
\begin{equation*}
\ell_i < \min \{ \frac{1}{\gk_{r\, j}}\} \quad \mbox{ for all positive 
principal radial curvatures } \gk_{r\, j}\, ;
\end{equation*}
\item ({\it Linking Edge Condition} )  For all points  $x_0 \in 
\overline{\bdyM_i}$ (the closure of $\bdyM_i$)
\begin{equation*}
\ell_i < \min \{ \frac{1}{\gk_{E\, j}}\} \quad \mbox{ for all positive 
principal edge curvatures } \gk_{E\, j}\, .
\end{equation*}
\end{enumerate}
Then, given a skeletal structure $\{(M_i, U_i, \ell_i)\}$ the smoothness of 
the linking flow is guaranteed by the following analog of Proposition 
\ref{PropII.2.1}. \par

\begin{Proposition}
\label{PropII.3.1b}
Let $U$ be a smooth value of the radial vector field defined in a 
neighborhood $W$ of $x_0 \in M_i$.  Suppose either that $x_0 \in M_i 
\backslash \bdyM_i$ and satisfies the linking curvature condition in the 
neighborhood $W$; or $x_0 \in \overline{\bdyM_i}$  and satisfies the 
linking edge condition on the neighborhood $W$.  Then, 
\begin{enumerate}
\item $\gl : W \times \R \to \R^{n+1}$ is a local diffeomorphism at $(x_0, 
t)$ for $0 < t \leq 1$ (and also $t = 0$ for nonedge closure points);
\item $\gl_t : W \to \R^{n+1}$ is a local embedding at $x_0$ for any $0 
\leq t \leq 1$; and
\item $\gl_t(W)$ is transverse to the line spanned by $U$ for each $0 \leq 
t \leq 1$. \par
Conversely, suppose $\gl$ is nonsingular at $(x_0, t)$ for $0 < t \leq 1$, 
which is equivalent to $\gl_t(W)$ being nonsingular at $\gl_t(x_0)$ and 
transverse to the line spanned by $U$ for $0 < t \leq 1$.  Then the linking 
curvature, resp. linking edge condition is satisfied at $x_0$, depending on 
whether it is a nonedge closure point, resp. an edge closure point.  
\end{enumerate}
\end{Proposition}
\par
Then, we can combine these conditions to conclude the following.
\begin{Corollary}
\label{CorII.3.2b}
For a skeletal structure $\{(M_i, U_i, \ell_i)\}$, if the linking curvature 
condition or linking edge condition is satisfied at all points of a stratum 
$S_j$ of $\tilde M_i$, then the linking flow from $S_j$ is nonsingular and 
remains transverse to the radial lines.  
\end{Corollary}
\par
Thus, the images of strata of the labelled refinement of each $\tilde M_i$ 
under the linking flow are immersed submanifolds for all $0 < t \leq 1$, 
and provided distant points from the same strata are not linked, then the 
images of strata are submanifolds.  
\begin{proof}[Proof of Proposition \ref{PropII.3.1b}]
First, we may view the linking flow from $M_i$ for $t > \frac{1}{2}$ as 
the composition of the flow for $t = \frac{1}{2}$, with level set the 
boundary $\cB_i$, followed by a flow from $\cB_i$ for time $t^{\prime} = 
2(t -  \frac{1}{2})$.  Thus, for $\frac{1}{2} \leq t \leq 1$, if we denote the 
radial flow by $\psi_t$, then $\gl_t(x) = \psi_1(x) + (\chi(t) - 
r_i(x))\bu_i(x)$ (with $\chi(t)$ as defined in (\ref{Eqn2.1})).  Hence, the 
linking flow can itself be viewed as a radial flow for $(\cB_i, 
U_i^{\prime})$ at the point $x_0^{\prime} = \gl_{\frac{1}{2}}(x_0)$ 
reparametrized at twice the speed.   Here the vector field $U_i^{\prime}$ 
is the translate along the radial lines to $x_0^{\prime}$ of the vector 
field $(\ell_i - r_i)\bu_i$, where $\bu_i$ is the unit vector field in the 
direction of $U_i$. \par
As $r_i \leq \ell_i$, the linking radial curvature and linking edge 
conditions imply the corresponding radial curvature and edge conditions 
for the radial flow.  Hence, by Proposition \ref{PropII.2.1} we have 
nonsingularity of the linking flow and the level surfaces for $0 < t \leq 
\frac{1}{2}$, with the level surfaces transverse to the radial lines.  Also, 
at $t = \frac{1}{2}$, the level surface is the boundary $\cB_i$ of the 
associated region $\gW_i$.  \par  
Then, we can further verify the nonsingularity of the linking flow for 
$\frac{1}{2} < t \leq 1$, and the corresponding level sets by again applying 
the radial curvature conditions to $(\cB_i, U_i^{\prime})$ . \par 
Now, $t^{\prime} = 2(t - \frac{1}{2})$, and 
$$\chi(t)-r_i \,\, =\,\,  2(t - \textstyle{\frac{1}{2}}\displaystyle)(\ell_i 
- r_i) \,\, = \,\, t^{\prime}(\ell_i - r_i). $$
  Thus, the linking flow for $\frac{1}{2} \leq t \leq 1$ gives the radial 
flow for $(\cB_i, U_i^{\prime})$ for $0 \leq t^{\prime} \leq 1$.
Hence, by Proposition \ref{PropII.E3.1} we obtain 
\begin{equation}
\label{EqnII3.1}
S_{\bv^{\prime\prime}, t} \,\,  = \,\,   (I - (\chi(t)-r_i) S_{\bv^{\prime}, 
\frac{1}{2} })^{-1}S_{\bv^{\prime}, \frac{1}{2}} \, ,
\end{equation}
where $S_{\bv^{\prime}, \frac{1}{2}}$ is the matrix representation for the 
radial shape operator for $(\cB_i, U_i^{\prime})$ at $x_0^{\prime}$ with 
respect to the basis $\bv^{\prime}$ corresponding to the basis $\bv$ for 
$T_{x_0}M$ under the radial flow.  \par
However, by Proposition \ref{PropII.2.1}, 
%%% 4.1 and 4.4 of \cite{D1}), 
the radial curvature condition is equivalent to the nonsingularity of 
$I - (\chi(t)-r_i) S_{\bv^{\prime}, \frac{1}{2}}$ for $\frac{1}{2} < t \leq 
1$; or equivalently the existence of $S_{\bv^{\prime\prime}, t}$ given by 
(\ref{EqnII3.1}) for $\frac{1}{2} < t \leq 1$.  We now claim there is an 
alternate formula for $S_{\bv^{\prime\prime}, t}$ given for $\frac{1}{2} < 
t \leq 1$ in the case of non-edge closure points $x_0$ by
\begin{equation}
\label{EqnII3.2} 
S_{\bv^{\prime\prime}, t} \,\, = \,\, (I - \chi(t) S_{\bv})^{-1}S_{\bv}\, ,
\end{equation}
 where $S_{\bv}$ is the matrix repesentation for the shape operator with 
respect to the basis $\bv$.  For edge closure points $x_0$, it is given by
\begin{equation}
\label{EqnII3.3} S_{\bv^{\prime\prime}, t} \,\, = \,\, (I_{n-1, 1} - \chi(t) 
S_{E, \bv})^{-1}S_{E, \bv}\, ,
\end{equation}
with $S_{E, \bv}$ denoting the matrix representation for the edge shape 
operator.  \par
It then follows in the first case, that $S_{\bv^{\prime\prime}, t}$ will be 
defined for $0 \leq t \leq 1$ provided $\frac{1}{\chi(t)}$ is not an 
eigenvalue for $S_{\bv}$ for $0 \leq t \leq 1$.  This says that the 
eigenvalues of $S_{\bv}$ do not lie in the interval $[\frac{1}{\ell_i}, 
\infty)$, or that the positive eigenvalues of $S_{\bv} < \frac{1}{\ell_i}$, 
which is equivalent to the linking curvature condition.  Likewise, in the 
second case we require $\frac{1}{\chi(t)}$ is not a generalized eigenvalue 
for $(S_{E,\bv}, I_{n-1, 1})$, for $0 \leq t \leq 1$.  Again, this is 
equivalent to the linking edge condition.  
This establishes the result.  
\end{proof}
It remains to verify the assertions regarding (\ref{EqnII3.2}) and 
(\ref{EqnII3.3}).  We do this by expressing the linking flow in either case 
via a semigroup of special M\"{o}bius tranformations on matrices and 
operators.
\par
\subsection*{Special M\"{o}bius Transformations of Matrices and 
Operators} 
%%%  \hfill 
\par
We will represent the evolved radial and edge shape operators under the 
linking flow as resulting from applying M\"{o}bius transformations to 
matrices and operators.  
We begin by considering for either $n \times n$ matrices $A$ and $B$ or 
operators $A, B : V \to V$, for $V$ a vector space of dimension $n$, a 
family of {\it special M\"{o}bius transformations} defined by
$$  \Xi_{B, t}(A) \,\, = \,\, (B - tA)^{-1} A .   $$
Here we view $B$ as fixed and either view $\Xi_{B, t}(A)$ as a function of 
$t \in \R$ for fixed $A$, or as a function of $A$ for fixed $t$.  We observe 
that the evolution equations for both the radial and edge shape operators 
under the radial flow have this form for either $(A, B) = (S_{\bv}, I)$ or 
$(S_{E, \bv}, I_{n-1, 1})$.  
We shall concentrate on the case of matrices, but the corresponding 
statements for operators are similar. \par
In general, if $\ker(A) \cap \ker(B) = 0$, then $\Xi_{B, t}(A)$ is defined 
for all but the finite set of $t$ such that $\frac{1}{t}$ is not a generalized 
eigenvalue of $(A, B)$ (i.e. $\gl$ so that $A- \gl B$ is singular).  Viewed as 
a function of $t$, we see
\begin{align} 
\pd{(\Xi_{B, t}(A))}{t} \,\,  &= \,\, - ((B - tA)^{-1} (-A) (B - tA)^{-1}) A 
\notag   \\
&= \,\,  ((B - tA)^{-1} A)^2  = \,\,  (\Xi_{B, t}(A))^2 \, .
\end{align}
Thus, $\Xi_{B, t}(A)$ is a solution of the simplest matrix Riccati equation
\begin{equation}
\label{EqnII.3.4} 
\pd{\Xi_{B, t}(A)}{t} \,\, = \,\, (\Xi_{B, t}(A))^2 .  
\end{equation}
If $B$ is nonsingular, $\Xi_{B, t}(A)$ is the solution to the basic matrix 
Riccati equation (\ref{EqnII.3.4}) with initial condition $\Xi_{B, 0}(A) = 
B^{-1} A$.  \par 
We consider the special cases of $\Xi_{B, t}(A)$ relevant for the evolution 
equations
\begin{align}
\mu_t(A) \,\, &= \,\, (I - tA)^{-1} A  \qquad \text{ and } \notag  \\
\nu_t(A) \,\, &= \,\, (I_{n-1, 1} - tA)^{-1} A \, .
\end{align}
For these to be well-defined we require for $\mu_t$ that $\frac{1}{t}$ is 
not an eigenvalue of $A$; while for $\nu_t$, we require that $\frac{1}{t}$ 
is not a generalized eigenvalue of $(A, I_{n-1, 1})$.  \par 
If $B$ is nonsingular, then there is the relation
\begin{equation}
\label{EqnII.3.4b}  \Xi_{B, t}(A) \,\, = \,\, ((I - tB^{-1}A)^{-1} (B^{-1}A) \, ,  
\end{equation}
so that  
$$  \Xi_{B, t}(A) \,\, = \,\, \mu_t(B^{-1} A)  \, .   $$
In particular, $\mu_t(A)$ is the basic solution of the matrix Riccati 
equation with initial condition $\mu_0(A) = A$.  
We next observe two further properties of $\mu_t(A)$: i) it is a \lq\lq 
one-parameter\rq\rq group of transformations (where we understand that 
it is defined except for a finite set of values $t$), and ii) $\mu_t$ acts on 
the curves defined by $\Xi_{B, t}(A)$.  These basic properties of these 
transformations are given by the following two lemmas.
\begin{Lemma}
\label{LemII.3.6}
If $\frac{1}{t}$ and $\frac{1}{(t+s)}$ are not eigenvalues of $A$, then
$$  \mu_s (\mu_t (A)) \,\,  =  \,\,\mu_{t + s} (A)\, .   $$
\end{Lemma}
\par
\begin{Lemma}
\label{LemII.3.7}
If $\frac{1}{t}$ and $\frac{1}{(t+s)}$ are not generalized  eigenvalues of 
$(A, I_{n-1, 1})$, then
$$  \mu_s (\Xi_{B, t}(A)) \,\,  =  \,\,\Xi_{B, t + s} (A)\, .   $$
\end{Lemma}
Although Lemma \ref{LemII.3.6} gives $\mu_t(A)$ as a one-parameter 
group, we will principally be interested in the semi-group for $t \geq 0$.  
Also, Lemma \ref{LemII.3.7} has as a consequence, 
$$  \mu_s (\nu_t(A)) \,\,  =  \,\,\nu_{t + s} (A) \, .  $$
We give the proof for Lemma \ref{LemII.3.6}.  That for Lemma 
\ref{LemII.3.7} is similar. 
\begin{proof}[Proof of Lemma \ref{LemII.3.6}]
Let $B = \mu_t (A)$.  First we observe that if $\frac{1}{t}$ is not an 
eigenvalue of $A$, then $v$ is an eigenvector of $A$ with eigenvalue $\ga$ 
if and only if $v$ is an eigenvector for $B$ with  eigenvalue $\ga (1 -t 
\ga)^{-1}$.  This follows because we may solve the equation $B = \mu_t 
(A)$ to obtain $A = B (I + t B)^{-1}$.  Then, we directly see that $v$ is an 
eigenvector of $A$ if and only if it is an eigenvector of $B$, with the 
eigenvalues related as claimed.  \par
Next, we claim that $\frac{1}{s}$ is not an eigenvalue for $B$.  Otherwise, 
there is an eigenvalue $\ga$ of $A$ such that $\frac{1}{s} = \gl (1 -t 
\ga)^{-1}$.  However, solving for $\ga$ we obtain $\ga = \frac{1}{(t+s)}$, a 
contradiction.  As  $\frac{1}{s}$ is not an eigenvalue for $B$, $\mu_s (B)$ 
is defined, and 
\begin{align}
\label{EqnII.3.8}
 \mu_s (B) \,\, &= \,\, (I - s B)^{-1} B  \notag  \\
&= \,\, (I - s ((I - tA)^{-1}) A))^{-1} ((I - tA)^{-1}) A)   \notag  \\
&= \,\, ((I - tA) (I - s ((I - tA)^{-1}) A))^{-1}  A   \notag  \\
&= \,\, ((I - tA) I - s  A))^{-1}  A  \notag  \\
&= \,\, ((I - ( t + s) A )^{-1}  A \,\, = \,\, \mu_{s + t} (A) \, .
\end{align}
\end{proof}
\par
\subsubsection*{Proof of (\ref{EqnII3.2}) and (\ref{EqnII3.3})} \hfill
\par
In terms of these M\"{o}bius transformations, we express the evolution of 
the 
radial and edge shape operators under the radial flow and rewrite 
(\ref{EqnII3.2})  and (\ref{EqnII3.3}) as
\begin{equation}
\label{EqnII.3.9}  
 S_{\bv^{\prime}, t} \,\, = \,\, \mu_t (S_{\bv})   
\end{equation} 
and 
\begin{equation}
\label{EqnII.3.10}  
 S_{E, \bv^{\prime}, t} \,\, = \,\, \nu_t (S_{E, \bv})\,  .  
\end{equation} 
We next use the Lemmas to represent the evolution of the radial and edge 
shape operators under the linking flow.  \par
We use the same argument explained at the beginning of the proof of 
Proposition \ref{PropII.E3.1}.  To be specific, we first consider the radial 
shape operator in the neighborhood of a non-edge point $x_0 \in M_i$ for 
some $i$.  We consider the linking flow $\gl_t$ on each half interval.  
Hence, by Proposition \ref{PropII.3.1}, provided $\frac{1}{2t}$ is not an 
eigenvalue of $S_{rad}$ for $0 \leq t \leq \frac{1}{2}$, then the evolved 
radial shape operator at time $t = \frac{1}{2}$ has matrix representation 
\begin{equation}
\label{EqnII.3.11}  
S_{\bv^{\prime},  \frac{1}{2}} \,\,  = \,\, \mu_{r_i}(S_{\bv}) \,\, = \,\, (I - 
\chi(\textstyle{\frac{1}{2}}\displaystyle) S_{\bv})^{-1}S_{\bv}\, .
\end{equation}   
Now, as explained earlier, the linking flow beginning at time $t = 
\frac{1}{2}$ can itself be viewed as a radial flow from $\cB_i$ with 
radial vector field $U_i^{\prime} = (\ell_i - r_i) \bu_i$, except again 
traveled at twice the speed $t^{\prime} = 2(t - \frac{1}{2})$.  Hence, we 
can apply Proposition \ref{PropII.3.1b} to conclude for $\frac{1}{2} \leq t 
\leq 1$, that if $\frac{1}{t^{\prime}(\ell_i - r_i)}$ is not an eigenvalue of 
$\mu_{r_i}(S_{\bv})$, which is equivalent to $\frac{1}{\chi(t)} = 
\frac{1}{r_i + 2(t - \frac{1}{2})(\ell_i - r_i)}$ not being an eigenvalue of 
$S_{\bv}$, then
\begin{align}
\label{Eqn.II3.3}
  S_{\bv^{\prime\prime}, t} \,\,  &= \,\,  \mu_{2(t - \frac{1}{2})(\ell_i - 
r_i)}( S_{\bv^{\prime}, \frac{1}{2}})  \notag  \\
\intertext{which by Lemma \ref{LemII.3.6} equals}
\,\,  &= \,\, \mu_{r_i + 2(t - \frac{1}{2})(\ell_i - r_i)}( S_{\bv})  \notag  
\\
\,\,  &= \,\, (I - \chi(t) S_{\bv})^{-1}S_{\bv} \, ,
\end{align} 
 yielding (\ref{EqnII3.2}).  \par
A similar argument, but replacing (\ref{EqnII.3.11})
by 
\begin{equation}
\label{EqnII.3.11E}  
S_{\bv^{\prime},  \frac{1}{2}} \,\,  = \,\, \nu_{r_i}(S_{E, \bv}) \,\, = \,\, 
(I_{n-1, 1} - \chi(\textstyle{\frac{1}{2}}\displaystyle) S_{E, \bv})^{-
1}S_{E, \bv}\, ,
\end{equation}
and then applying Lemma \ref{LemII.3.7} yields instead (\ref{EqnII3.3}).
  \par
\vspace{1ex}
\subsection*{Evolution of the Shape Operators under the Linking Flow} 
\par
We obtain the following corollaries describing the evolution of the radial 
and edge shape operators under the linking flow from $M$ (which we recall 
is the union of the $M_i$).

\begin{Corollary}  
\label{Cor.II3.5}
Let $x_0 \in M \backslash \overline{\bdyM}$ be a point on the closure of a 
manifold component of $M$ with $U(x_0)$ a smooth value in a 
neighborhood of $x_0$ on this component.  If  $\frac{1}{\chi(t)}$ is not an 
eigenvalue of $S_{\bv}$ for $0 \leq t \leq 1$, then the linking flow is 
nonsingular and the evolved radial shape operator on the level surface 
$\cB_t = \gl_t(M_i)$  in a neighborhood of $\gl_t(x_0)$ is given by
$$  S_{\bv^{\prime\prime}, t} \,\,  = \,\,   (I - \chi(t) S_{\bv})^{-1}S_{\bv} 
\, .  $$
Here $\bv^{\prime\prime}$ is the image of the basis $\bv$ under 
$d\psi_t(x_0)$. 
\end{Corollary}
\hspace{11.5cm} $\Box$. \par
%%%\vspace{2ex}
\begin{Corollary}  
\label{Cor.II3.6}
Let  $x_0 \in \overline{\bdyM}$ be a point on the closure of an edge 
manifold component of $\bdyM$ with $U(x_0)$ a smooth value in a 
neighborhood of $x_0$ on this component. If $\frac{1}{\chi(t)}$ is not a 
generalized eigenvalue of $(S_{E, \bv}, I_{n-1,1})$ for $0 < t \leq 1$, then 
the linking flow is nonsingular and the radial shape operator on the level 
surface $\cB_t = \gl_t(M_i)$  in a neighborhood of $\gl_t(x_0)$ is given 
by
$$  S_{\bv^{\prime\prime}, t} \,\,  = \,\,   (I_{n-1,1} - \chi(t) S_{E, 
\bv})^{-1}S_{E, \bv} \, .  $$
Here $\bv^{\prime\prime}$ is the  image of the basis $\bv$ of $T_{x_0} 
\partial M$ under $d\psi_t(x_0)$.   
\end{Corollary}
\hspace{11.5cm} $\Box$. \par
\subsubsection*{Shape Operator on the Linking Axis} 
\par
As a consequence of the corollaries, we can deduce the shape operator for 
the linking axis $M_0$.  Let $x \in M_0$, and suppose $x^{\prime} \in M_i$ 
is a point for which the linking flow ends at $x$.  Then, $x = 
\gl_1(x^{\prime}) = x^{\prime} + L_i(x^{\prime})$, for some choice of the 
linking vector field at $x^{\prime}$.  Then, to such a point there is the 
corresponding point $x^{\prime\prime} = \psi_1(x^{\prime}) \in \cB_i$.   
We then have a value of a radial vector field $U_0 = -(\ell_i - r_i)\bu_i$ 
at $x \in M_0$, with $\ell_i$, $r_i$, and $\bu_i$ associated to the value 
$L_i(x^{\prime})$.  This vector ends at $x^{\prime\prime}$.  We obtain a 
value of $U_0$ at $x$ for each such point in some $M_j$ linked at $x$.  
This defines a multi-valued vector field on $M_0$.  \par
Then, the associated unit vector field at $x$ is $\bu_0 = -\bu_i$.  Thus, 
the radial shape operator for $(M_0, U_0)$ at $x \in \tilde M_0$, the 
\lq\lq double of the linking axis\rq\rq\, is the negative of that for the 
stratum of $M_0$, viewed as a level set of the linking flow from 
$x^{\prime} \in M_i$.  Hence, by the corollaries we obtain the following 
calculation of the  corresponding radial shape operator.
\begin{Corollary}  
\label{Cor.II3.7}
If $x \in M_0$ is as in the above discussion, then the radial shape operator 
for the skeletal structure $(M_0, U_0)$ at $x$ is given by either: if 
$x^{\prime}$ is a non-edge closure point, then with the notation of 
Corollary \ref{Cor.II3.5}, 
$$  S_{\bv^{\prime\prime}, t} \,\,  = \,\,  - (I - \ell_i S_{\bv})^{-1}S_{\bv} 
\, ;  $$
or if $x^{\prime}$ is an edge closure point, then with the notation of 
Corollary \ref{Cor.II3.6},
$$  S_{\bv^{\prime\prime}, t} \,\,  = \,\,  - (I_{n-1,1} - \ell_i 
 S_{E, \bv})^{-1}S_{E, \bv} \, .  $$
\end{Corollary}
\hspace{11.5cm} $\Box$. \par
Hence, any calculation which could be performed using the radial shape 
operators on $M_0$ could be performed using the radial or edge shape 
operators on the appropriate $M_i$.  

\section{Properties of Regions Defined Using the Linking Flow} 
\label{SII:sec.int}
\par
We will next consider how the medial/skeletal linking structure $\{(M_i, 
U_i, \ell_i)\}$ for a multi-region configuration $\bgW = \{ \gW_i\}$ in 
$\R^{n+1}$ allows us to decompose the region external to the 
configuration into sub-regions reflecting the positional relations between 
the individual regions.  We referred to this in \S \ref{S:sec2}.  We now 
provide more details and terminology which we will employ in the next 
sections.  We will separately consider the unbounded and bounded cases.
\subsection*{Medial/Skeletal Linking Structures in the Unbounded Case} 
\par
The principal difference for the \lq\lq unbounded case\rq\rq, for which 
the configuration is considered in $\R^{n+1}$, is that the associated 
external regions which naturally reflect the positional information are 
also usually unbounded.  We begin by defining regions $\gW_i$ and 
$\gW_j$ in the configuration that are linked via the linking structure, and 
introduce basic notation using the linking flow $\gl_i$ on $\gW_i$.  We 
recall that a point $x \in \tilde M_i$ is linked to a point $x^{\prime} \in 
\tilde M_j$ if for the corresponding values of the linking vector fields $x 
+ L_i(x) = x^{\prime} + L_j(x^{\prime})$, or equivalently $\gl_i(x) = 
\gl_j(x^{\prime})$.  Then, by property L2) in Definition 
\ref{Defmultlkgstr}, entire strata of $\tilde M_i$ are linked to entire 
strata of $\tilde M_j$, or the images of the strata remain disjoint under 
the linking flow.  \par
Then, we introduce regions defined using the linking flow
\begin{Definition}  
\label{DefII4.1} 
For a skeletal linking structure $\{(M_i, U_i, \ell_i)\}$ for the 
configuration $\bgW = \{ \gW_i\}$, we define regions associated to the 
structure as follows: 
\flushpar

\begin{itemize}
\item[i)] $M_{i \to j}$ will denote the union of the strata of $\tilde M_i$ 
which are linked to strata of $\tilde M_j$, and we refer to it as the {\it 
strata where $M_i$ is linked to $M_j$} (the strata being in $\tilde M_i$ 
indicate on which \lq\lq side\rq\rq of $M_i$ the linking occurs).  
\item[ii)] $\gW_{i\to j} = \gl_{i}(M_{i\to j} \times \left[0, 
\frac{1}{2}\right])$ denotes the {\it region of $\gW_i$ linked to $\gW_j$}.   
\item[iii)]  $\cN_{i\to j} = \gl_{i}(M_{i\to j} \times \left[\frac{1}{2}, 
1\right])$ denotes the {\it linking neighborhood} of $\gW_i$ linked to 
$\gW_j$.  
\item[iv)]  $\cB_{i \to j} = \gW_{i \to j} \cap \cN_{i \to j}$ is the {\it 
boundary region of $\cB_i$ linked to $\cB_j$}. 
\item[v)]  $\cR_{i \to j} = \gW_{i \to j} \cup \cN_{i \to j}$, is the {\it 
total region for $\gW_i$ linked to $\gW_j$}.
\end{itemize}
In the case that the configuration in bounded within $\tilde \gW$, the 
regions will be those for the corresponding bounded linking structure. 
\end{Definition} 
We illustrate these regions in Figure~\ref{fig.II4.1}.  
\par 
\begin{figure}[ht] 
\begin{center}
\includegraphics[width=5.5cm]{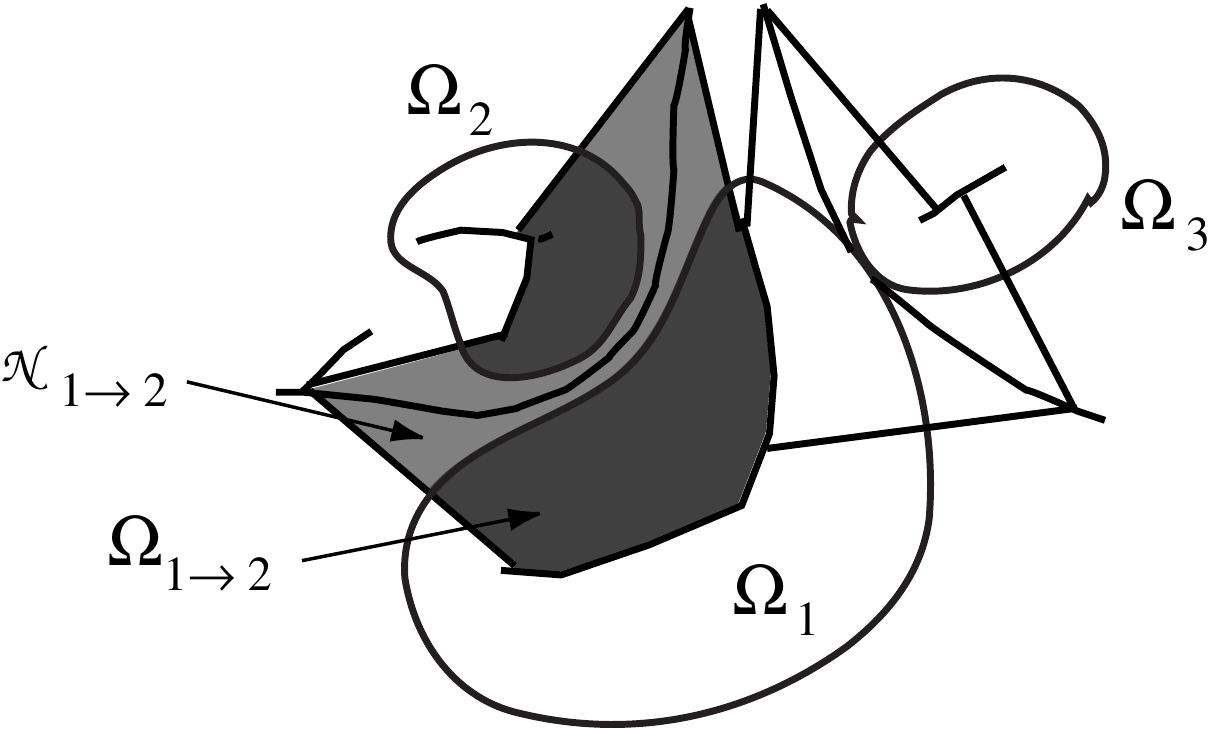} 
\hspace*{0.10cm}
\includegraphics[width=5.5cm]{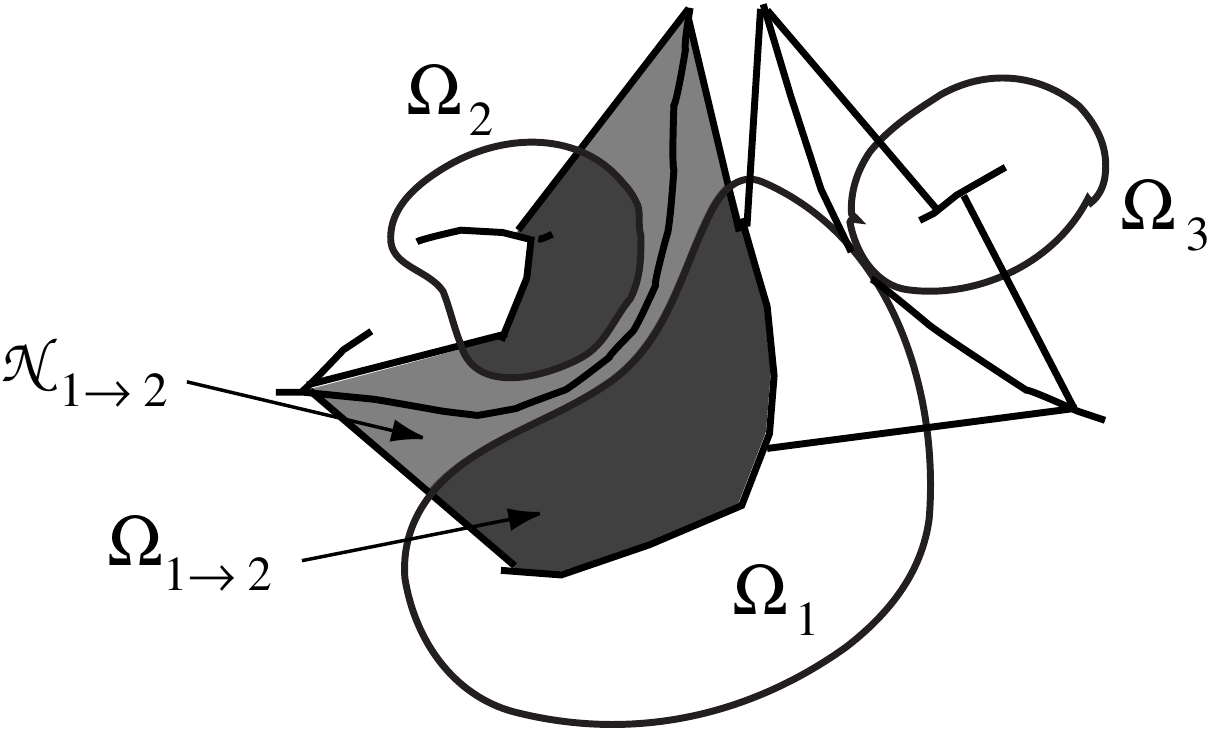}
\end{center}
(a) \hspace{2.5in} (b) \hspace {2.0in} 

%%%\centerline{\includegraphics[width=7cm]{{Multi.Obj.fig.16}.pdf}} 
\caption{\label{fig.II4.1} Configuration of three regions with portions of 
the regions $\gW_1$ and  $\gW_2$ which are linked to each other (darkly 
shaded regions are parts of $\gW_{1 \to 2}$ and $\gW_{2 \to 1}$), and 
their linking neighborhoods (grey shaded regions are parts of $\cN_{1 \to 
2}$ and $\cN_{2 \to 1}$).  Then, $\cB_{1 \to 2}$ is the portion of the 
boundary $\cB_{1}$ where $\cN_{1 \to 2}$ meets $\gW_{1 \to 2}$, while 
$\cR_{1 \to 2}$ is the union of the two regions $\cN_{1 \to 2}$ and 
$\gW_{1 \to 2}$.  Note that in the unbounded case in a), much of the 
linking in the infinite region occurs between small parts of $\gW_1$ and 
$\gW_2$, and this would not occur for a bounded linking structure in a 
bounded region as in b) where a threshold is imposed. 
%%% The full region $\cN_{1 \to 2}$ would be unbounded and share 
%%%  boundaries with $\cN_{1 \to 3}$ and $\cN_{1 \, \infty}$.
} 
\end{figure} 
\par

Then, we make a few simple observations: first, by property L2) in 
Definition \ref{Defmultlkgstr} of the skeletal linking structure, $\cN_{i 
\to j} \cap \cN_{j \to i}$ will consist of a union of strata of the linking 
axis $M_0$ where the linking between $\gW_i$ and $\gW_j$ occurs; and 
second, the regions for a fixed $i$ but different $j$ may intersect on the 
images under the linking flow of strata where there is linking between 
$\gW_i$ and two or more other regions.  \par
Next, strata of $M_i$ may involve self-linking.  We will still use the 
notation $M_{i \to i}$, $\gW_{i \to i}$, etc for the strata, regions etc. 
involving self-linking.  Then, $\cN_{i \to i}$ will intersect $\cN_{i \to j}$ 
on strata where partial linking occurs. 
Finally the remaining strata in $\tilde M_i$ lie in $M_{i\, \infty}$, which 
consists of the union of strata which are unlinked.  Then, property $L4)$ 
of Definition \ref{Defmultlkgstr} of skeletal linking structure requires 
that the global radial flow from the union of the $N_{+} | M_{i\, \infty}$ 
defines a diffeomorphism to the complement of the linking flows (see e.g. 
Figure~\ref{fig.5b}).  We then also let $\gW_{i\, \infty}$, $\cN_{i\, 
\infty}$, $\cB_{i\, \infty}$ and $\cR_{i\, \infty}$ denote the 
corresponding regions for $M_{i\, \infty}$.  \par
Then, we have the decompositions 
\begin{equation}
\label{EqnII4.2}
\gW_i \,\, =\,\,  (\cup_{j \neq i} \gW_{i \to j}) \cup \gW_{i \to i} \cup 
\gW_{i\, \infty} \quad \text{ with } \quad (\cup_{j \neq i} \gW_{i \to j} 
\cup \gW_{i \to i}) \cap \gW_{i\, \infty} \,\, =\,\,  \emptyset \, ,
\end{equation} 
but the various $\gW_{i \to j}$ and/or $\gW_{i \to i}$ may have 
non-empty intersections, as explained above.  There are analogous 
decompositions for $\tilde M_i$ and $\cB_i$.  Also, we denote the {\it 
total linking neighborhood} by $\cN_i = \cup_{j \neq i} \cN_{i \to j}$.
Then, $\cN_i \cup \cN_{i \to i} \cup \cN_{i\, \infty}$ is the {\it total 
neighborhood of $\gW_i$} (in the complement of the configuration), whose 
interior consists of points external to the configuration which are closest 
to $\gW_i$.  \par 
Using the notation of \S \ref{S:sec2}, $\cB_{i\, 0}$ denotes the portion of 
the boundary $\cB_i$ not shared with any other region.  We recall that the 
strata in $\cB_i\backslash \cB_{i\, 0}$, are of the form $\cB{i\, j}$, 
which are boundary strata shared with the boundary of another $\gW_j$.  
Then, on the strata of $\tilde M_i$ corresponding to these strata, the 
linking flow is constant in $t$ for $\frac{1}{2} \leq t \leq 1$, remaining 
at the shared boundary region.  Hence, $\cN_i$ at these points only 
consists of boundary points of $\cB_i\backslash \cB_{i\, 0}$.  Off of these 
points we can describe the structure of $\cN_i$ using the linking flow. 
We summarize the consequences of the properties of the linking flow for 
the various regions.  \par
\begin{Corollary}
\label{CorII4.1}
For a configuration of regions $\bgW = \{ \gW_i\}$ with skeletal linking 
structure $\{(M_i, U_i, \ell_i)\}$, there are the following properties for 
each region associated to $\gW_i$:
\begin{itemize}
\item[1)]  $\gW_i \backslash M_i$ is fibered by the level sets of the 
linking flow for $0 < t \leq \frac{1}{2}$;
\item[2)] $\cN_i \backslash (\cB_i \backslash \cB_{i\, 0})$ is fibered by 
the level sets of the linking flow for $\frac{1}{2} \leq t \leq 1$;
\item[3)] $\cN_{i \to i}\backslash M_0$ is fibered by the level sets of the 
linking flow for $\frac{1}{2} \leq t < 1$; and 
\item[4)] $\cN_{i\, \infty}$ is fibered by the level sets of the radial flow 
for $1 \leq t < \infty$\, .
\end{itemize}
For each of these fibrations, the fibers are stratified manifolds.  
\par
Furthermore, the evolving radial shape operator for $0 < t \leq 1$ under 
the linking flow is given by Corollaries \ref{Cor.II3.5} and \ref{Cor.II3.6}.
\end{Corollary}
\begin{proof}
First, the nonsingularity of the linking flow (linking condition (L2)) gives 
two cases: i) for each stratum $S_{i\, k}$ of $\tilde M_i$, which 
corresponds to a stratum of $\cB_{i\, 0}$ the linking flow defines a 
piecewise smooth diffeomorphism $\gl_{i} : S_{i\, k} \times (0, 1] \to 
\cR_{i\, k}$, where $\cR_{i\, k}$ denotes the image; and ii) if $S_{i\, k}$ 
corresponds to a stratum of $\cB_i \backslash \cB_{i\, 0}$, then only 
$\gl_{i} : S_{i\, k} \times (0, \frac{1}{2}] \to \gW_{i}$ is a smooth 
diffeomorphism onto its image; and $\gl_{i} : S_{i\, k} \times [\frac{1}{2}, 
1] \to \gW_{i}$ is constant in $t$ and has as image a stratum in the 
boundary $\cB_i$.  \par
Thus, in either of the two cases, the restriction of the flow to the subset 
$S_{i\, j} \times (0, \frac{1}{2}]$ is a smooth diffeomorphism, and we 
obtain the radial flow (at double speed), so the image is the union of the 
level sets for each stratum giving a level set of the radial flow $\cB_{i\, 
t}$.  This is a global stratawise diffeomorphism and gives a fibration of 
$\gW_i \backslash M_i$ by Theorem 2.5 on \cite{D1}.  \par 
Again in the first case, for $S_{i\, k}$ a stratum of $M_{i \to j}$ which is 
associated to $\cB_{i\, 0}$, the restriction of the linking flow to $S_{i\, 
k} \times  [\frac{1}{2}, 1]$ is a smooth diffeomorphism.  The union of 
these two parts of the linking flows then defines a global stratawise 
piecewise smooth diffeomorphism with image $\cN_i \backslash (\cB_i 
\backslash \cB_{i\, 0})$.   However, if we consider $M_{i \to i}$, then 
more than one stratum will have as image at $t = 1$ the same stratum of 
$M_0$, so it will only be a diffeomorphism for $\frac{1}{2} \leq t < 1$.    
\par 
Lastly, by Proposition \ref{Prop6.8}, on strata of $M_{i\, \infty}$ the 
radial flow will satisfy the radial or edge curvatures conditions for all $t 
> 0$ and defines a global diffeomorphism $\psi_{i} : S_{i\, k} \times [1, 
\infty)$ onto its image.  By the same proposition, these fit together to 
form a global stratawise diffeomorphism $M_{i\, \infty} \times [1, 
\infty) \simeq \cN_{i\, \infty}$.  
\par
Furthermore, by Proposition \ref{PropII.3.1b}, the linking curvature or 
linking edge condition is satisfied at each point as appropriate.  Then, this 
also implies that the radial curvature or edge curvature conditions, as 
appropriate, are also satisfied.  Hence, Corollaries \ref{Cor.II3.5} or 
\ref{Cor.II3.6} apply, yielding the stated form.  
\end{proof}
\subsubsection*{Regions for the Full Blum Linking Structure}
\par 
If we use instead the full Blum linking structure for a general 
configuration, we point out the modification of Corollary \ref{CorII4.1}.  
By Theorem \ref{Thm3.2}, the Blum medial axis $M_i$ in the interior of 
each $\gW_i$ has closure containing edge-corner points of $\gW_i$ 
($P_k$ or the singular $Q_k$ points for all $k$).  At these points, the Blum 
medial axis has the local edge-corner normal form given in Definition 
\ref{Def3.3}.  Because of this normal form, the radial flow still gives a 
local fibration structure in a neighborhood of the edge-corner points, 
provided we exclude those points from the boundary.  Hence, the radial 
flow from the medial axis in the interior defines a fibration structure for 
$\gW_i \backslash Cl(M_i)$.  The remaining properties of Corollary 
\ref{CorII4.1} then remain true because for 2), $\cB_i \backslash \cB_{i\, 
0}$ contains the edge corner points and is removed.  For both 3) and 4), the 
edge-corner points do not play a role. \par
For the special case of a configuration of disjoint regions with smooth 
boundaries, the Blum medial linking structure is a skeletal linking 
structure, so no modifications of the conclusions of Corollary 
\ref{CorII4.1} are required.  
\par
We next consider the bounded case.
\subsection*{Medial/Skeletal Linking Structures for the Bounded Case}
 \par
For the bounded case we suppose that the configuration lies in the interior 
of a region $\tilde \gW$ whose boundary $\partial \tilde \gW$ is 
transverse to the stratification of $M_0$ and to the linking vectors on $M$ 
(i.e. in the linking region, and including the extension of the radial lines 
from $M_{\infty}$, the linking line segments or extended radial lines are 
tranverse to the (limiting) tangent spaces at points of $\partial \tilde 
\gW$).  As explained in Remark \ref{Rem3.6}, we can alter the linking 
vector field either by truncating it or defining it on $M_{\infty}$ and then 
refining the stratification so that on appropriate strata the linking vector 
field ends at $\partial \tilde \gW$.  It is now defined on all of $M_i$ for 
all $i > 0$.  Because we are either reducing $\ell_i$, or defining $L_i$ on 
$M_{i\, \infty}$, the linking flow is still nonsingular, so we have 
corresponding properties from Corollary \ref{CorII4.1}, except that for 
properties 2), 3), and 4) the linking flow and corresponding regions 
$\cN_{i \to j}$, $\cN_{i \, \infty}$, and $\cR_{i \to j}$ may only extend to 
$\partial \tilde \gW$.  Also, we still obtain the same formulas for the 
evolution of the radial shape operators for those level sets of the linking 
flow while they remain within $\intr(\tilde \gW)$.  \par
Thus, we have compact versions of the regions defined for the unbounded 
case.  To construct these regions, there are a number of different 
possibilities.  \par
%%%\vspace{2ex} \flushpar
\subsubsection*{\it Possibilities for a Bounded Region $\tilde \gW$}:  
\hfill \par
%%%\begin{itemize}
%%%  \item[ 
\vspace{1ex}
\flushpar
{\it Bounding Box or Bounding Convex Region: } \hfill 
\par
A bounding box requires a center and directions and sizes for the edges of 
the box.  For this, we would need to first normalize the center and 
directions for the sides of the box and then normalize the sizes of the 
edges either using a fixed size or one based on features sizes of the 
configuration.  An example of a bounding box and the resulting linking 
structure is shown in Figure~\ref{fig.II4.2}~(a).  For another convex 
bounding region such as a bounding sphere, because of the symmetry, it is 
only necessary to normalize the center, and then either fix the radius or 
base it on the feature sizes.  For any  convex region $\tilde \gW$ with 
piecewise smooth boundary, the limiting tangent planes of $\partial \tilde 
\gW$ are supporting hyperplanes for $\tilde \gW$.  Hence any line in the 
tangent plane lies outside $\intr(\tilde \gW)$.  Thus, the radial lines from 
regions in the configuration will meet the boundary transversely 
(including the limiting tangent planes at singular points). 
\par
\begin{figure}[ht] 
\begin{center}
\includegraphics[width=5.0cm]{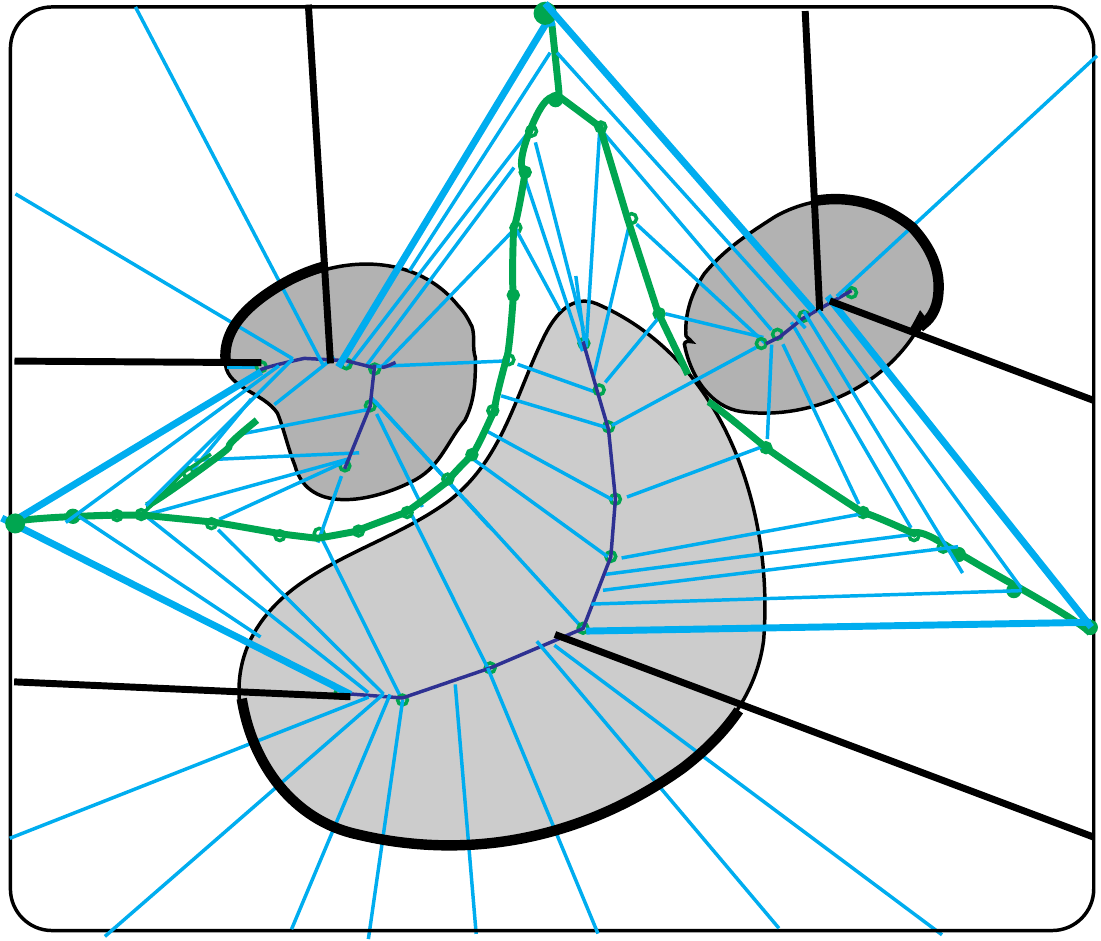} 
\hspace*{0.10cm}
\includegraphics[width=5.0cm]{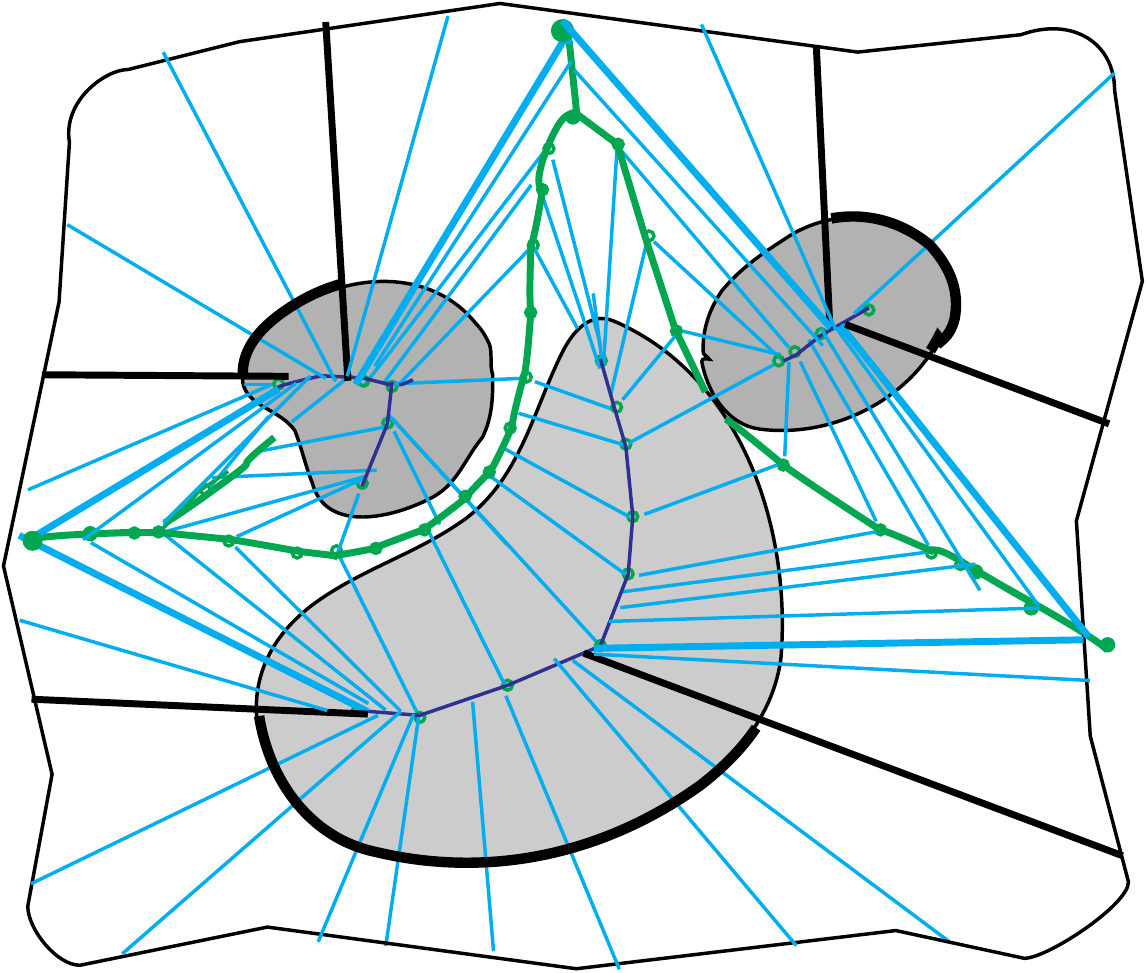}
\end{center}
(a) \hspace{2.5in} (b) \hspace {2.0in}  
\caption{\label{fig.II4.2} a) {\em Bounding box} with curved corners 
containing a configuration of three regions and b) (a priori given) {\em 
intrinsic bounding region} for the same configuration.  The linking 
structure is either extended to the boundary (in the region bounded by the 
darker lines on the boundary for $M_{i\, \infty}$) or truncated at the 
boundary.} 
\end{figure} 
\par
%%%\item[ ]
\vspace{2ex}
\flushpar
{\it Convex Hull: } \hfill 
\par  
The smallest convex region which contains a configuration is the convex 
hull of the configuration.  For a generic configuration, the convex hull 
consists of the regions $\cB_{i\, \infty}$ together with the truncated 
envelope of the family of degenerate supporting hyperplanes which meet 
the configuration with a degenerate tangency or at multiple points.  \par 
For a) of Figure~\ref{fig.5b} in \S~\ref{S:sec3}, which is a configuration 
in $\R^2$, the envelope consists of line segments joining the doubly 
tangent points corresponding to the points in the spherical axis (b) of 
Figure~\ref{fig.5b}.  \par
In $\R^3$, the envelope consists of a triangular portion of a triply tangent 
plane, with line segments joining pairs of points with a bitangent 
supporting plane, and a decreasing family of segments ending at a 
degenerate point (see Example~\ref{Exam3.2b} in \S~\ref{S:sec3}).  
\par 
%%%\item[ ]
\vspace{2ex}
\flushpar
{\it Intrinsic Bounding Region: } \hfill 
\par
If the configuration is naturally contained in an (a priori given) intrinsic 
region, which is modeled by $\tilde \gW$, then provided that the 
extensions of the radial lines intersect $\partial \tilde \gW$ 
transversely, then we can modify the linking structure as in the convex 
case to have it defined on all of the $M_i$, and terminating at $\partial 
\tilde \gW$, if linking has not already occurred (see e.g. 
Figure~\ref{fig.II4.2}~(b)).  \par
%%%\item[ ]
\vspace{2ex}
\flushpar
{\it Threshold for Linking: } \hfill 
\par
The external region can be bounded by placing a threshold $\tau$ on the 
$\ell_i$ so that the $L_i$ remain in a bounded region.  This can be done in 
two different ways.  An {\it absolute threshold} restricts the external 
regions to those arising from linking vector fields with $\ell_i \leq \tau$; 
and a {\it truncated threshold} restricts to a bounded region formed by 
replacing $\ell_i$ by $\ell_i^{\prime} = \min\{\ell_i, \tau\}$.  For the 
first type, only part of the region would have an external linking 
neighborhood; while in the second, the entire region would.  As for the 
convex case, since we are replacing $\ell_i$ by smaller values, the linking 
flow will remain nonsingular.  We obtain modified versions of the regions 
lying in a bounded region (see e.g. Figures~\ref{fig.II4.1} or \ref{figII5.9}).  
%%% \end{itemize}
\par

\begin{Remark}
\label{RemII4.3}
If the configuration is subjected to a transformation formed from a 
translation, rotation, and scaling, then we would like the associated 
bounded region $\tilde \gW$ to be sent to the corresponding bounded 
region for the resulting image configuration.  Provided we also scale the 
length of vectors, a skeletal linking structure will be sent by such 
mappings to a skeletal linking structure for the image configuration.  This 
is also true for a Blum linking structure.  Thus, it is only necessary that 
the associated bounding region is constructed from geometric data from 
the configuration.  For the convex hull this is true.  For an intrinsic 
bounding region, we would require that the transformations send one 
intrinsic bounding region to another.  This will also happen for bounding 
boxes, etc or thresholds provided any reference points, directions of 
edges, and lengths are determined by the configuration.  
\end{Remark}
\par

\section{Global Geometry via Medial and Skeletal Linking Integrals} 
\label{SII:SkelLnkInt}
\par
	We are now ready to use the regions we have defined for a 
configuration to introduce quantitative invariants measuring positional 
geometry for configurations.  We do so in terms of integrals which are 
defined on the skeletal sets for the individual regions using the shape 
operators defined from the linking structure.  \par 
\subsection*{Defining Medial and  Skeletal Linking Integrals}
\par
We begin by considering a medial or skeletal linking structure $\{(M_i, 
U_i, \ell_i)\}$ for a multi-region configuration $\bgW = \{ \gW_i\}$ in 
$\R^{n+1}$.  We recall from \S \ref{S:sec2} that we introduced the 
notation that $M = \coprod_{i > 0} M_i$ denotes the disjoint union of the 
$M_i$ for each region $\gW_i$ for $i > 0$.  Each $M_i$ has its double 
$\tilde M_i$, and we also introduced the double $\tilde M$ for the 
configuration by $\tilde M = \coprod_{i > 0} \tilde M_i$.  For each $i > 0$, 
there is a canonical finite-to-one projection $\pi_i : \tilde M_i \to M_i$, 
mapping $(x, U(x)) \mapsto x$.  The union of these defines a canonical 
projection $\pi : \tilde M \to M$, such that $\pi | \tilde M_i = \pi_i$ for 
each $i > 0$. 
 \par

Even though $\bgW$ is defined via an embedding $\Phi : \bgD \to \R^{n+1}$, 
we will define the integrals directly using the skeletal linking structure 
for $\bgW$.  We will define the skeletal integral on $\tilde M$ for a 
multi-valued function $g : M \to \R$, by which we mean for any $x \in 
M_i$, $g$ may have a different value for each different value of $U_i$ at 
$x$.  Such a $g$ pulls-back via $\pi$ to a well-defined map $\tilde g : 
\tilde M \to \R$ so that $g \circ \pi = \tilde g$.  \par 
To define the integrals we use an especially suited positive Borel measure 
on $\tilde M$.  The Borel measures are defined separately for each $\tilde 
M_i$ and then together they give the Borel measure on $\tilde M$.  The 
existence of the corresponding positive Borel measures $dM_i$ on each 
$\tilde M_i$  for each $i > 0$ follows from Proposition 3.2 of \cite{D4}.  
We briefly recall how they are defined.  We cover $M_i$ by a finite union 
of paved neighborhoods $\{ W_k\}$.  We recall from \S \ref{S:sec2} 
(see  also \cite[\S 3]{D4}) that for a \lq\lq paved neighborhood\rq\rq\, 
we mean an 
open $W_k$ whose closure may be decomposed into a finite number of 
connected smooth manifolds with boundaries and corners $M_{\ga_j}$ 
which only meet along boundary facets, see e.g. Figure~\ref{fig.10.4a}, and 
each value of $U$ is defined smoothly on $M_{\ga_j}$. 
\par
%%%\begin{figure}[ht]
%%%\centerline{\epsfxsize=12cm\epsffile{{Skel.Str.BdryIV.fig9}.pdf}}
%%%\caption{\label{fig.10.4} a) Paved neighborhood of a point in $M$ and  
%%%corresponding  paved neighborhoods b) and c) in $\tilde M$}
%%%\end{figure}
%%%\par
%%% The existence of such local paved regions in neighborhoods of 
%%% singular points of $M_i$ follows from the local properties of skeletal 
%%% sets (see of \cite{D4})  that on a connected stratum near a singular 
%%% point there is a unique limiting tangent space.  
\par
The measure is defined by defining the integrals of continuous functions 
on $\tilde M$ and using the Riesz representation theorem.  By a partition 
of unity argument as in e.g. \cite[Chap. 10]{LS} or 
\cite[Vol. 1, Chap. 8]{Sp}, it is enough to define the integral for one such 
region 
$M_{\ga_j}$.  We let $M_{\ga_j}^{(k)}$, $k = 1, 2$, denote the inverse 
images of $M_{\ga_j}$ under the canonical projection map $\pi_i : \tilde 
M_i \to M_i$.  Each $M_{\ga_j}^{(k)}$ is a copy of $M_{\ga_j}$, with the 
smooth value of $U_{i}$ associated to the copy.  For each copy we let 
$dM_{j\, i} = \rho_{j\, i} dV_i$, where $dV_i$ is the Riemannian volume on 
$M_{\ga_j}$ and $\rho_{j\, i} = \bu_{k\, j}\cdot \bn_{k\,j}$ where 
$\bu_{k\, j}$ is a value of the unit vector field corresponding to the 
smooth value of $U_{k\, j}$ for $M_{\ga_j}^{(k)}$ and $\bn_{k\, j}$ is the 
normal unit vector pointing on the same side as $U_{k\, j}$. \par  
By a partition of unity argument, we obtain integrals of continuous 
functions $\tilde h$ on $\tilde M_i$.  Then, by the Riesz representation 
theorem, there is a regular positive Borel measure $dM_i$ on $\tilde M_i$ 
so that integration extends to Borel measurable functions $\tilde h$ on 
$\tilde M_i$, given by the integral with respect to this measure, see 
\cite[\S 3]{D4}.  \par
Then, for a multivalued function $g$ on $M_{\ga_j}$, the integral of $g$ is, 
by definition, the sum of the integrals of the corresponding values of $g$ 
over each copy $M_{\ga_j}^{(k)}$ with respect to the medial measure 
$dM_{k\, j}$.  For Borel measurable functions $\tilde g$ on $\tilde M_i$ 
obtained (as above) by composing $g$ with $\pi_i$, this gives a 
well-defined integral with respect to the Borel measure $dM_i$ on 
$\tilde M_i$.  
\par
Then, the distinct Borel measures on $dM_i$ on $\tilde M_i$ together 
define a regular positive Borel measure $dM$ on $\tilde M$; and the 
integral of a Borel measurable multi-valued function $g$ on $M$ is defined 
to be
\begin{equation}
\label{EqnII5.1}
 \int_{\tilde M} g \, dM \,\, = \,\, \sum_{i> 0} \int_{\tilde M_i} \tilde g\, 
dM_i , 
\end{equation}
where each integral on the RHS is the integral of $\tilde g$ over $\tilde 
M_i$ with respect to the Borel measure $dM_i$, and it can be viewed as an 
integral of $g$ over \lq\lq both sides of $M_i$\rq\rq. \par
A Blum medial linking structure for a configuration of disjoint regions is 
itself a skeletal structure, so the integral makes sense for such 
configurations.  We refer to the integrals in (\ref{EqnII5.1}) as {\it medial 
or skeletal linking integrals}, depending on whether the linking structure 
is a Blum medial linking structure or a skeletal linking structure.  \par
In the results that follow, just as in the preceding definition, we will be 
adapting the arguments given in \cite{D4} for a single region to the 
individual regions of the configuration.  Consequently, we will frequently 
outline the arguments and refer to the specific proofs in \cite{D4} for the 
details.  \par 
\subsubsection*{\it Blum Medial Linking Integrals for General 
Configurations} \hfill
\par
For a general configuration with a full Blum linking structure, the Blum 
medial axis extends to the boundary; and its closure contains the 
edge-corner points of each $\gW_i$.  We briefly explain how the preceding 
definition of the integral can be extended to this case.  We now include the 
edge-corner points as points of the medial axis $M_i$, so that it is still 
compact.  The double $\tilde M_i$ is defined as before, except that there 
is only a single point lying over each edge-corner point, just as for the 
edge points of the medial axis in the smooth case.  Again $\tilde M_i$ is a 
compact (and locally compact) Hausdorff space.  Also, because there is the 
edge-corner normal form for the medial axis in the neighborhood of any 
edge-corner point, we can again give a paved neighborhood of any such 
point, in the sense of \cite[\S 3]{D4}.  
\par 
We can then define an integral of a multi-valued function $g$ on $M_i$ 
just as for skeletal structures using a partition of unity and paved 
neighborhoods.  As in the preceding, there is a unique regular positive 
Borel measure $dM_i$ on $\tilde M_i$ so that the integral of $g$ on $M_i$ 
is the integral of $\tilde g$ with respect to the measure $dM_i$ on $\tilde 
M_i$.  We shall refer to this measure as the {\em skeletal (or medial) 
measure} on $\tilde M_i$.  Locally the measure still has the form $dM_{i} = 
\rho_{i} dV_i$, except now $\rho_{i}$ vanishes on the edge-corner points.  
In this case, we still refer to the integrals in (\ref{EqnII5.1}) as {\it 
medial linking integrals}.  \par

\subsection*{Computing Boundary Integrals via Medial Linking Integrals} 
\par
So far the integrals which we have just defined only depend on the 
skeletal structures for the individual regions.  We now show how, by using 
the linking flow, we may express integrals of functions on $\cB$ or  
functions on $\R^{n+1}$ as medial or skeletal integrals over the skeletal 
sets.  \par
\subsubsection*{Representing Integrals of Functions on $\cB$ as Medial 
Integrals}  First, we consider a Borel measurable function $g : \cB \to \R$, 
integrable for volume measure on $\cB$, which is allowed to be 
multi-valued in the sense that for any $k$-edge-corner point $x \in \cB$, 
$g$ may take distinct values for each region $\gW_i$, $i > 0$, containing 
$x$ on its boundary.  Thus, for a shared boundary region $\cB_{i\, j} = 
\cB_i \cap\cB_j$, $g$ may take different values for each region $\gW_i$ 
or $\gW_j$ (so the values on $\cB_{i\, j}$ for each $\gW_i$ define a Borel 
measurable and integrable function on $\cB_{i\, j}$ and on $\cB_{i\, 0} = 
\cB_i \backslash \cup_{j \neq i}\cB_j$).  For example, such a function 
might represent a physical quantity such as a density of the boundary 
region which may differ for each region $\gW_i$ and hence give two 
possible values on shared regions such as $\cB_{i\, j}$.  By the integral of 
such a 
multi-valued function $g$ over $\cB$ we mean 
$$  \int_{\cB} g\, dV \,\, = \,\,  \sum_{i \neq j, \, i, j \geq 0} 
\int_{\cB_{i\, j}} g_{i\, j}\, dV $$
where $g_{i\, j}$ denotes the values of $g$ on $\cB_{i\, j}$ for $\gW_i$ 
and $dV$ denotes the $n$-dimensional Riemannian volume on each $\cB_i$.  
This includes both $\cB_{i\, 0}$ and $\cB_{i\, \infty}$ so the integral is 
over the complete boundaries of all regions.  \par

Then, for the time-one radial flow map $\psi_{i\, 1} : \tilde M_i \to 
\cB_i$, we define $\tilde g : M_i \to \R$ by $\tilde g = g \circ \psi_{i\, 
1}$, where the value on $\cB_{i\, j}$ is the value associated to $\gW_i$.  
Then, $\tilde g$ is a multi-valued Borel measurable function on $M_i$.  
\par
We may compute the integral of $g$ over $\cB$ by the following result.
\begin{Thm}
\label{ThmII5.1}
Let $\bgW$ be a multi-region configuration with (full) Blum linking 
structure.  If $g : \cB \to \R$ is a multi-valued Borel measurable and 
integrable function, then
\begin{equation}
\label{EqnII5.2}
\int_{\cB} g\, dV \,\, = \,\, \int_{\tilde M} \tilde g \det(I - r_i S_{rad})\, 
dM\, 
\end{equation}
where $r_i$ is the radius function of each $\tilde M_i$. 
\end{Thm}
\begin{proof}
In the case of a configuration of disjoint regions with smooth boundaries, 
for each region $\gW_i$,  $g$ is a well-defined function on $\cB_i$ and 
we may apply Theorem 4.1 of \cite{D4} to conclude
\begin{equation}
\label{EqnII5.3}
\int_{\cB_i} g\, dV \,\, = \,\, \int_{\tilde M_i} \tilde g \det(I - r_i 
S_{rad})\, dM_i\, .
\end{equation}
Summing (\ref{EqnII5.3}) over $i > 0$ gives the result in this case.  \par
Next for a general configuration, we first consider a single region 
$\gW_i$ in it.  We may apply the proof of \cite[Thm 4.1]{D4} to one of the 
regions $M_{i\, \ga}$ with a given smooth value of $U_i$ on it.  If $M_{i\, 
\ga}$ meets the edge-corner strata, it does so on its boundary.  Let $B_{i\, 
\ga}^{(j)} = \psi_{i\, 1}(M_{i\, \ga}^{(j)})$ for the given values for $U_i$.  
Then, $\psi_{i\, 1} | M_{i\, \ga}^{(j)}$ is a homeomorphism to $B_{i\, 
\ga}^{(j)}$, and $\psi_{i\, 1}$ is nonsingular on the interior of $M_{i\, 
\ga}$.  Thus, we can still apply the change of variables formula for 
multiple integrals.  
Hence, applying the proof of \cite[Thm 4.1]{D4} we obtain
\begin{equation}
\label{EqnII5.4}
 \int_{B_{i\, \ga}} g\, dV \,\, = \,\, \int_{\tilde M_{i\, \ga}} \tilde g \det(I 
- r_i S_{rad})\, dM_i\, .  
\end{equation}
Then, we can first sum (\ref{EqnII5.4}) over the $M_{i\, \ga}^{(j)}$ which 
form a paved neighborhood, to obtain the result on paved neighborhoods, 
and then use the partition of unity to obtain (\ref{EqnII5.3}).
\end{proof}
\par 
In the case of a skeletal structure, there is a form of Theorem 
\ref{ThmII5.1} which still applies.  For each region $\gW_i$, with $i > 0$, 
let $\tilde R_i$ denote a Borel measurable region of $\tilde M_i$ which 
under the radial flow maps to a Borel measurable region $ R_i$ of $\cB_i$.  
Let $R = \cup_{i} R_i$ and $\tilde R = \cup_{i} \tilde R_i$.  We suppose 
that the skeletal structure satisfies the \lq\lq partial Blum 
condition\rq\rq\, on $\tilde R$, by which we mean: for each $i$, the 
compatibility $1$-form $\eta_{U_i}$ vanishes on $\tilde R_i$ (recall this 
means that the radial vector $U_i$ at points of $x \in \tilde R_i$ is 
orthogonal to $\cB_i$ at the point where it meets the boundary).  Note 
that for a skeletal structure this forces $R$ to be contained in the 
complement of $\cB_{sing}$.  \par 
Then, there is the following analogue of Theorem \ref{ThmII5.1}.
\begin{Corollary}
\label{CorII5.1}
Let $\bgW$ be a multi-region configuration with skeletal linking 
structure which satisfies the partial Blum condition on the region $\tilde 
R \subset \tilde M$, with image $R$ under $\psi_1$.  If $g : R \to \R$ is a 
multi-valued Borel measurable and integrable function, then
\begin{equation}
\label{EqnII5.3b}
\int_{ R} g\, dV \,\, = \,\, \int_{\widetilde{R}} \tilde g \det(I - r_i 
S_{rad})\, 
dM\, .
\end{equation} 
\end{Corollary}
\begin{proof}
We again follow the proof of Theorem 4.1 in \cite{D4}.  We let $\chi_{ R}$ 
denote the characteristic function of the region $R$ on $\cB$ (so it equals 
$1$ on $R$ and $0$ off $R$), and we replace $g$ by $g^{\prime} = \chi_{ 
R}\cdot g$.  Then the integral of $g$ over $R$ equals that of $g^{\prime}$ 
over $\cB$.  \par 
We may cover $R$ by paved neighborhoods whose images lie in the smooth 
strata of $\cB$, and paved neighborhoods which miss $R$.  For each 
$M_{i\, \ga}^{(j)}$ which appears in one of the paved neighborhoods 
meeting $R$, we apply the proof of Theorem 4.1 in \cite{D4}.  By the 
partial Blum condition $U_i^{(j)}$ is normal to $\cB_{i\, \ga}^{(j)}$ at 
$\psi_{i\, 1}(x)$ for $x \in \tilde R \cap M_{i\, \ga}^{(j)}$; thus 
$$ g^{\prime}\, d\psi_{i\, 1}^*(dV) \,\, = \,\, 
\begin{cases} 0& x \notin \tilde R, \\ 
\tilde g \det(I - r_i S_{rad})\, dM & x \in \tilde R . 
\end{cases}
$$
Thus, 
\begin{equation}
\label{EqnII5.3c}
\int_{R \cap \cB_{i\, \ga}^{(j)}} g\, dV \,\, = \,\, \int_{\widetilde{R} \cap 
M_{i\, 
\ga}^{(j)}} \tilde g \det(I - r_i S_{rad})\, dM\, ,
\end{equation} 
yielding the result on $M_{i\, \ga}^{(j)}$.  Then, we may follow the 
reasoning in the proof of \cite[Thm. 4.1]{D4} to obtain the conclusion.
\end{proof}
\par
\subsubsection*{Volumes of Regions in $\cB$} \par
We give one application of Theorem \ref{ThmII5.1} to computing the 
$n$-dimensional volume of $\cB$.  There are two possible meanings for 
this.  The \lq\lq complete volume\rq\rq\, includes the volume of each 
$\cB_i$, counting that of each shared boundary $\cB_{i\, j}$ for both 
$\cB_i$ and $\cB_j$.  The other \lq\lq partial volume\rq\rq counts the 
shared boundaries only once.  We may compute the complete volume by 
applying Theorem \ref{ThmII5.1} to $g \equiv 1$.
\begin{Corollary}
\label{CorII5.2}
Let $\bgW$ be a multi-region configuration with (full) Blum medial 
linking structure.  Then the complete volume of $\cB$ is given by
\begin{equation}
\label{EqnII5.4b}
\text{complete $n$-dim $\vol(\cB)$} \,\, = \,\, \int_{\tilde M} \det(I - r_i 
S_{rad})\, dM\, .
\end{equation} 
\end{Corollary}
\hspace{11.5cm} $\Box$ 
\par
For the partial volume, we linearly order the boundaries $\cB_i$, letting 
\lq\lq $\succ$\rq\rq\, denote the ordering, and define the multi-valued 
function $\check{1}$ on $\cB$ so that for $\cB_i$, $\check{1}$ equals $1$ 
except on the strata $\cB_{i\, j}$ with $\cB_j \succ \cB_i$, where 
instead it equals $0$.  This is the characteristic function for a region in 
$\cB$ that corresponds under the radial flow to a region $\check{M} 
\subset \tilde M$, which is the union of $\check{M}_i \subset \tilde M_i$.  
Thus, by applying Theorem \ref{ThmII5.1} to $\check{1}$ we obtain the 
following.
\begin{Corollary}
\label{CorII5.3}
Let $\bgW$ be a multi-region configuration with (full) Blum medial 
linking structure.  Then, the partial volume of $\cB$ is given by
\begin{equation}
\label{EqnII5.5}
\text{partial $n$-dim $\vol(\cB)$} \,\, = \,\, \int_{\check{M}} \det(I - r_i 
S_{rad})\, dM\, .
\end{equation} 
\end{Corollary}
\hspace{11.5cm} $\Box$ 
\par
\begin{Remark}
\label{RemII5.5}
There is an analogous result for a region $R \subset \cB$ which is the 
image of a region $\tilde R \subset \tilde M$ under the radial flow.  If the 
configuration is modeled by a skeletal linking structure which satisfies 
the partial Blum condition on $\tilde R$, then the (complete) volume of 
$R$ is given by $\int_{\tilde R} \det(I - r_i S_{rad})\, dM$. 
\end{Remark}
We may expand $\det(I -t S_{rad}) = \sum_{j = 0}^n (-1)^{j}\gs_j t^j$, 
where $\gs_j$ is the $j$-th elementary symmetric function in the 
principal radial curvatures $\gk_i$.  Applying the formulas, we may write 
the integrals as a sum of integrals $(-1)^{j}\int \tilde g \gs_j\, r_i^j \, 
dM$, which are \lq\lq moment integrals\rq\rq\, of the functions $\tilde g 
\gs_j$ with respect to the radial functions.  Examples of these explicit 
integrals for a single region can be found in \cite[\S 6.1]{D4}; for 
configurations of regions the integrands have the same forms.  \par
%%%  TO DO [mention the definition of $K_{rad}$ and the Gauss-Bonnet 
%%%  Theorem for a configuration with distinct regions with smooth 
%%%  boundaries.]  

\subsection*{Computing Integrals as Skeletal Linking Integrals via the 
Linking Flow} 
\par
Next we turn to the problem of computing integrals over regions which 
may be partially or completely in the external region of the configuration.  
Quite generally we consider a Borel measurable and Lebesgue integrable 
function $g : \R^{n+1} \to \R$ with compact support.  We shall see that we 
can compute the integral of $g$ as a skeletal linking integral of an 
appropriate related function.  \par 
Since we are in the unbounded case, we first modify the skeletal linking 
structure by defining $\ell_i$ on $M_{i\, \infty}$ to be $\ell_i = \infty$.  
By Proposition \ref{Prop6.8}, the linking flow on $M_{i\, \infty}$ is a 
diffeomorphism for $0 \leq t < \infty (= \ell_i)$.  \par 
Next, we first replace the linking flow by a simpler {\it elementary 
linking flow} defined by $\gl^{\prime}_t(x) = x + t \bu_i$, for $0 \leq t 
\leq \ell_i$ (or $< \infty$ if $\ell_i= \infty$).  The elementary linking 
flow is again along the lines determined by the linking vector field $L_i$; 
however, the rate of flow differs from that for the usual linking flow.  
This means that the level surfaces will differ, although the image of 
strata under the elementary linking flow agrees with that for the linking 
flow.  In addition, as the linking flow is nonsingular, Proposition 
\ref{PropII.3.1b} implies that the linking curvature and edge conditions 
are satisfied.  Then, for the linking vector field, viewed as a radial vector 
field, the radial curvature and edge curvature conditions are satisfied, and 
hence imply the nonsingularity of the elementary linking flow.  
\par 
Then, using the elementary linking flow, we can compute the integral of 
$g$ as a skeletal linking integral.  We define a multi-valued function 
$\tilde g$ on $M$ as follows: for $x \in M_i$ with associated smooth value 
$U_i$ and linking vector $L_i$ in the same direction as $U_i$ (so $(x, U_i) 
\in \tilde M_i$),  
\begin{equation}
\label{EqnII5.6}
 \tilde g(x) \,\, \overset{def}{=} \,\, \int_{0}^{\ell_i} g (\gl^{\prime}_t(x)) 
\det(I - t S_{rad})\, dt  
\end{equation}
provided the integral is defined.  
Then, we have the following formula for the integral of $g$ as a skeletal 
linking integral. 
\begin{Thm}
\label{ThmII5.5}
Let $\bgW$ be a multi-region configuration in $\R^{n+1}$ with a skeletal 
linking structure.  If $g : \R^{n+1} \to \R$ is a Borel measurable and 
Lebesgue integrable function with compact support, then $\tilde g(x)$ is 
defined for almost all $x \in \tilde M$, it is integrable on $\tilde M$, and 
\begin{equation}
\label{EqnII5.7}
\int_{\R^{n+1}} g\, dV \,\, = \,\, \int_{\tilde M} \tilde g \, dM\, .
\end{equation} 
\end{Thm}
\par
\begin{Remark}
\label{RemII5.8}
If we compare this formula with that given for a single region in Theorem 
6.1 of \cite{D4}, we notice they have a slightly different form.  However, 
as noted in Remark 6.2 of that paper, it is possible to use a change of 
coordinates $t^{\prime} = \ell_i t$ to rewrite 
\begin{equation}
\label{EqnII5.8a}
  \tilde g(x) \,\, = \,\, \ell_i \int_{0}^{1} g (x + t^{\prime} L_i) \det(I - 
t^{\prime} \ell_i S_{rad})\, dt^{\prime}\, , 
\end{equation}
which agrees with the form given there.  We will use the same change of 
coordinates to rewrite the formula differently below.  In proving the 
theorem we will find it useful to first reduce it to the integral for a 
modified linking structure over a bounded region.
\end{Remark}
\par
\subsubsection*{Reducing to Integrals for Bounded Skeletal Linking 
Structures} \par
We relate the unbounded skeletal structure to a bounded one.  We let $Q = 
\supp(g)$, which is compact.  We may find a compact convex region $\tilde 
\gW$ with smooth boundary containing both the configuration $\bgW$ and 
$Q$.  Then, we can modify the linking structure by reducing the $L_i$ so it 
is truncated by $\partial \tilde \gW$ and defining $L_i$ on $M_{i\, 
\infty}$ as the extensions of the radial vectors to where they meet 
$\partial \tilde \gW$.  Because the extended radial lines are transverse to 
$\partial \tilde \gW$, the new values of $\ell_i$, which we denote by 
$\ell_i^{\prime}$, remain smooth on the strata of each $M_i$.  We also let 
$L_i^{\prime} = \ell_i^{\prime}\bu_i$, with $\bu_i$ denoting the unit 
vector in the same direction as $L_i$, be the corresponding linking vector.  
As $\ell_i^{\prime} \leq \ell_i$, the linking flow remains nonsingular on 
all strata of $\tilde M$.  \par 
First, we may express 
\begin{equation}
\label{EqnII5.9}
  \int_{\R^{n+1}} \, g \,dV  \,\, = \,\, \int_{\tilde \gW} \, g \,dV 
\end{equation}
This allows us to reduce the proof of Theorem \ref{ThmII5.5} to the case 
of the bounded integral on the RHS of (\ref{EqnII5.9}) using the bounded 
linking structure on $\tilde \gW$.
\begin{proof}[Proof of Theorem \ref{ThmII5.5}]
For the proof, we first also reduce the integral on the RHS of 
(\ref{EqnII5.7}) as an integral over the bounded region.  Since the 
intersection of $\supp(g)$ with the linking line segment $\{ t\bu_i : 0 \leq 
t < \ell_i\}$ lies in $\{ t\bu_i : 0 \leq t < \ell_i^{\prime}\}$, 
\begin{equation}
\label{EqnII5.10}
 \tilde g(x) \,\, = \,\, \int_{0}^{\ell_i^{\prime}} g (\gl^{\prime}_t(x)) \det(I 
- t S_{rad})\, dt  \, .
\end{equation}
By the change of coordinates $t = \ell_i^{\prime} t^{\prime}$, we obtain 
\begin{equation}
\label{EqnII5.11}
  \tilde g(x) \,\, = \,\, \ell_i^{\prime} \int_{0}^{1} g (x + t^{\prime} 
L_i^{\prime}) \det(I - t^{\prime} \ell_i^{\prime} S_{rad})\, dt^{\prime}\, . 
\end{equation}
Then, (\ref{EqnII5.7}) reduces to computing the integral in the RHS of 
(\ref{EqnII5.7}) using the bounded structure as an integral of 
(\ref{EqnII5.11}).  
Also, as the elementary linking flow is still one-one on each linking line, 
the elementary linking flow $\gl_i^{\prime}$ is still a homeomorphism on 
any of the regions $M_{i\, \ga}^{(j)}$.  Also, by the preceding discussion 
the elementary linking flow is also nonsingular on the interior of $M_{i\, 
\ga}^{(j)} \times [0, 1]$ with image in $\R^{n+1}$, denoted $R_{i\, 
\ga}^{(j)}$.  Thus, using the change of variables formula for multiple 
integrals and the argument proving Theorem 6.1 in \cite{D4}, we obtain
\begin{equation}
\label{EqnII5.12}
  \int_{R_{i\, \ga}^{(j)}} \, g \,dV  \,\, = \,\, \int_{M_{i\, \ga}^{(j)}} \tilde 
g(x) \, dt^{\prime}\, . 
\end{equation}
where $\tilde g(x)$ is given by (\ref{EqnII5.11}), which is the same as 
(\ref{EqnII5.6}).  \par
Thus, we may follow the reasoning from the proof of Theorem 6.1 in 
\cite{D4} to complete the proof.
\end{proof}
\par
As an immediate consequence of Theorem \ref{ThmII5.5} we can deduce   
the following.  
Let
$$  g_Q(x) =  \int_{0}^{\ell_i^{\prime}} g(\gl_t^{\prime}(x)) 
\chi_Q(\gl_t^{\prime}(x)) \det(I - t S_{rad, i}) \,dt \, .$$
\begin{Corollary}
\label{CorII5.4}
Let $\bgW$ be a multi-region configuration with skeletal linking 
structure.  Suppose that $Q \subset \R^{n+1}$ is a compact subset and $g : 
\R^{n+1} \to \R$ is a Borel measurable and Lebesque integrable function on 
$Q$.  Then,
\begin{equation}
\label{EqnII5.13}
\int_{Q} \, g \,dV  \,\, = \,\,  \int_{\tilde {M}} g_Q \, dM\, .
\end{equation} 
\end{Corollary}
The proof is an immediate consequence of Theorem \ref{ThmII5.5} applied 
to $\chi_Q \cdot g$.  \quad $\Box$
\par
In the special case where $g \equiv 1$, we obtain an analogue of the 
Crofton formula.  For a compact subset $Q \subset \R^{n+1}$ and each $x 
\in \tilde M_i$, we define the multi-valued function 
$$ m_Q(x) \,\, =\,\, \int_{0}^{\ell_i} \chi_Q(x + t L_i(x)) \det(I - 
t S_{rad})\, dt  \, .$$
We can view $m_Q(x)$ as a weighted $1$-dimensional measure of the 
intersection of $Q$ with the linking line from $x$ determined by $L_i(x)$.
\begin{Corollary}[Crofton Type Formula]
\label{CorII5.5}
Let $\bgW$ be a multi-region configuration with skeletal linking 
structure.  Suppose $Q \subset \R^{n+1}$ is a compact subset. Then,
\begin{equation}
\label{EqnII5.13b}
\vol(Q) \,\, = \,\,  \int_{\tilde M} m_Q(x)\, dM\, .
\end{equation} 
\end{Corollary}
\hspace{11.5cm} $\Box$
\par 
\subsubsection*{Decomposition of a Global Integral using the Linking 
Flow} 
\par  
We next decompose the integral on the RHS of (\ref{EqnII5.13}) using the 
alternative integral representation of $\tilde g$ using the linking flow.  
We do so by applying the change of variables formula to relate the 
elementary linking flow $\gl^{\prime}$ with the linking flow $\gl$, both 
of which flow along the linking lines but at different linear rates.  \par
We define 
\begin{align}
\label{EqnII5.14}
  \tilde g_{int}(x) \,\, &= \,\, \int_{0}^{r_i } g(x + t\bu_i) 
\det(I - t S_{rad})\, dt\, \quad \text{ and } \quad \notag \\
\tilde g_{ext}(x) \,\, &= \,\, \int_{r_i }^{\ell_i} g(x + t\bu_i) \det(I - t 
S_{rad})\, dt. 
\end{align}
These may be alternately written using a change of coordinates as
\begin{equation}
\label{EqnII5.15a}
  \tilde g_{int}(x) \,\, = \,\, r_i \int_{0}^{1} g(x + t r_i\bu_i) 
\det(I - t r_i S_{rad})\, dt\, . 
\end{equation}
and 
\begin{equation}
\label{EqnII5.15b}
  \tilde g_{ext}(x) \,\, = \,\, (\ell_i - r_i) \int_{0}^{1} g(x + (r_i + t(\ell_i 
- r_i))\bu_i) \det(I - (r_i +t (\ell_i - r_i)) S_{rad})\, dt\, . 
\end{equation}
Then, we may decompose $\int g$ as follows.
\begin{Corollary}
\label{CorII5.6}
Let $\bgW$ be a multi-region configuration in $\R^{n+1}$ with a skeletal 
linking structure.  If $g : \R^{n+1} \to \R$ is a Borel measurable and 
Lebesgue integrable function with compact support, then, $\tilde 
g_{int}(x)$ and $\tilde g_{ext}(x)$ are defined for almost all $x \in \tilde 
M$, they are integrable on $\tilde M$, and 
\begin{equation}
\label{EqnII5.7b}
\int_{\R^{n+1}} g\, dV \,\, = \,\, \int_{\tilde M} \tilde g_{int} \, dM\, \, + 
\,\, \int_{\tilde M} \tilde g_{ext} \, dM \, ,
\end{equation} 
where 
\begin{equation}
\label{EqnII5.7c}
\int_{\tilde M} \tilde g_{int} \, dM \,\, = \,\, \sum_{i, j >0} \int_{ M_{i \to 
j}} \tilde g_{int} \, dM\, \, + \,\, \sum_{i >0} \int_{M_{i \, \infty}} \tilde 
g_{int} \, dM \, ,
\end{equation} 
with an analogous formula with $g_{int}$ replaced by $g_{ext}$ 
everywhere in (\ref{EqnII5.7c}).  
\end{Corollary}
The first integral on the RHS of (\ref{EqnII5.7b}) is the \lq\lq interior 
integral\rq\rq\, of $g$ within the configuration using the radial flow, and 
the second integral is the\lq\lq external integral\rq\rq\, computed using 
the linking flow outside of the configuration.  Then we may decompose 
each of these integrals using (\ref{EqnII5.7c}) into integrals over the 
distinct linking regions as illustrated in Figure~\ref{figII5.8}
\par
\begin{figure}[ht]
\includegraphics[width=8cm]{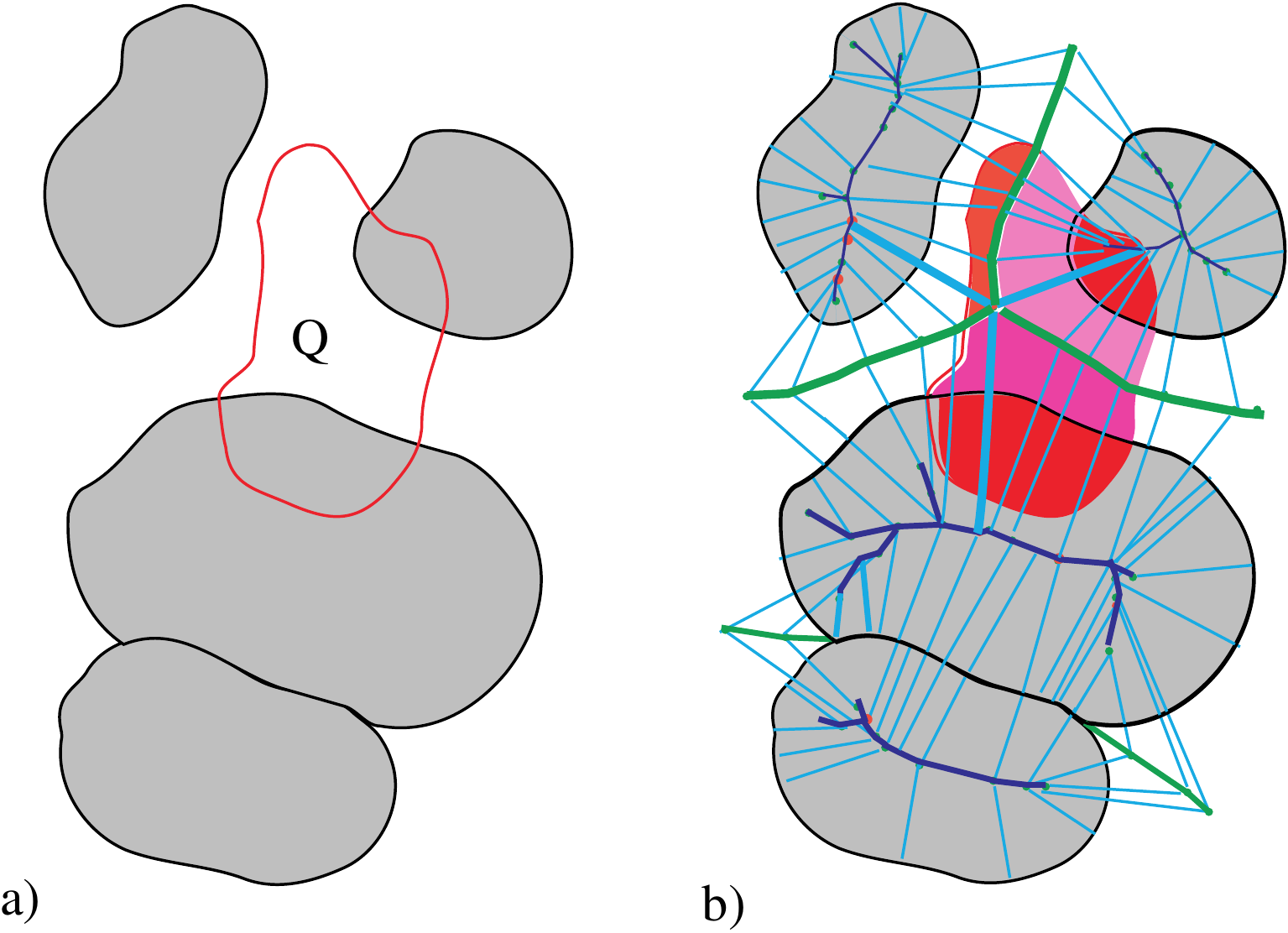}
\caption{ \label{figII5.8} The decomposition of the integral over a region 
$Q$ outlined in a) is given as the sum of integrals over regions in b)  
obtained by the subdivision of $Q$ by the linking axis and the three linking 
lines to the branch point of the linking axis.  Each $Q_{i\, j} \subset 
\cR_{i\to j}$ (or in general including $Q_{i\, \infty} \subset \cR_{i\, 
\infty}$) in the figure consists of the darker region inside the subregion 
$\gW_{i \to j}$ together with the portion of $Q$ in the linking 
neighborhood $\cN_{i\to j}$.  The integral can then be expressed by 
Corollary \ref{CorII5.6} as sums of internal and external integrals over 
the $M_{i\to j}$ and $M_{i\, \infty}$.}
\end{figure} 
\par
\vspace{2ex}
\begin{proof}
For (\ref{EqnII5.7b}) we may just apply Theorem \ref{ThmII5.5} and 
compute the RHS of (\ref{EqnII5.8a}) as a sum of two integrals by writing 
(\ref{EqnII5.6}) as a sum of two integrals to obtain
$\tilde g = \tilde g_{int} + \tilde g_{ext}$.  That they are defined for 
almost all $x \in \tilde M$ follows from an application of Fubini\rq s 
Theorem as in the proof of Theorem 6.1 in \cite{D4}.  
This decomposes the linking flow into the parts both before and after the 
flow reaches the boundary.  
 \par
For (\ref{EqnII5.7c}), we use that $\tilde M$ is a union of $M_{i \to j}$ for 
$i, j > 0$ and $M_{i \, \infty}$ for $i > 0$.  These are unions of strata 
which form $n$-dimensional stratified sets, and hence are Borel sets of 
dimension $n$.  Also, any two of these intersect on a union of strata of 
dimension $< n$ which have measure $0$ in $\tilde M$.  Thus, 
$$  \int_{\tilde M} g_{int}\, dV \,\, = \, \, \sum_{i, j >0} \int_{M_{i \to j}} 
g_{int} \, dV\, \, + \,\, \sum_{i >0} \int_{M_{i \, \infty}} g_{int} \, dM \, .  
$$
This yields (\ref{EqnII5.7c}) with an analogous formula for $g_{ext}$.
\end{proof}
\par
\subsection*{Skeletal Linking Integral Formulas for Global Invariants} 
\par
We now express the volumes of regions associated to the linking structure 
as skeletal linking integrals.  
We may apply the same reasoning as in Corollary \ref{CorII5.5}, using 
Corollary \ref{CorII5.6} to compute the volumes of compact measurable 
regions $Q \subset \R^{n+1}$ as a sum of internal and external integrals.  
We consider a specific application for linking regions. \par
For these calculations we will use the expression
\begin{equation}
\label{EqnII5.20a}
\cI(t) \,\, \overset{def}{=} \,\, \int_{0}^{t} \det(I - t S_{rad}) \, dt \,\, = 
\,\,  \sum_{j = 0}^{n} \frac{(-1)^j}{j+1} \gs_j t^{j+1} \, . 
\end{equation}

\subsubsection*{Volumes of Linking Regions as Skeletal Integrals} 
\par
We consider a multi-region configuration with a bounded skeletal linking 
structure.   
\begin{Corollary}
\label{CorII5.8}
Let $\bgW \subset \tilde \gW$ be a multi-region configuration with a 
bounded skeletal linking structure.  Then,
\begin{align}
\label{EqnII5.20}
\vol(\cN_{i \to j}) \,\, &= \,\,  \int_{\tilde M_{i \to j}} \cI(\ell_i) - 
\cI(r_i)\, dM\, \quad \text{ and } \quad \vol(\gW_{i \to j}) \,\, &= \,\,  
\int_{M_{i \to j}} \cI(r_i)\, dM\, \notag \\
\vol(\cN_{i\, \infty}) \,\, &= \,\,  \int_{M_{i\, \infty}} \cI(\ell_i) - 
\cI(r_i)\, dM\, \quad \text{ and } \quad \vol(\gW_{i\, \infty}) \,\, &= \,\,  
\int_{M_{i\, \infty}} \cI(r_i)\, dM\, .
\end{align} 
\end{Corollary}
It then follows that we can compute the volumetric invariants of the 
various linking regions such as $\cR_{i \to j}$, $\cN_{i}$, etc. using 
skeletal linking integrals of the polynomials $\cI(\ell_i)$ or $\cI(r_i)$.  
For example,
\begin{equation}
\label{EqnII5.20b}
 \vol(\cR_{i \to j}) = \int_{M_{i \to j}} \cI(\ell_i) \, dM \, , 
\end{equation}
 with an analogous formula for $\vol(\cR_{i \, \infty})$.  \par
As a consequence, we obtain generalizations of both Weyl\rq s formula for 
volumes of tubes \cite{W} and Steiner\rq s formula, see e.g. \cite{Gr}.  
\begin{Corollary}[Generalized Weyl\rq s Formula]
\label{CorII5.9}
Let $\bgW \subset \tilde \gW$ be a multi-region configuration with a 
bounded skeletal linking structure.  Then,
\begin{equation}
\label{EqnII5.21}
\vol(\gW_i) \,\, = \,\,  \int_{\tilde M_i} \cI(r_i)\, dM\, . 
\end{equation} 
\end{Corollary}
%%% \hspace{11.5cm} $\Box$ \par
The sense in which this generalizes Weyl\rq s formula is explained for the 
case of a single region with smooth boundary in \cite[\S 6, 7]{D4}.  For 
Steiner\rq s formula, we note that as explained in \S \ref{SII:sec.int},  
$\cN_i \cup \cN_{i\to i} \cup \cN_{i\, \infty}$ represents the total 
neighborhood of $\gW_i$, which is the region about $\gW_i$ extending 
along the linking lines.  This is a generalization of a partial tubular 
neighborhood about a region which depends on the specific type of 
bounding region (see Figure~\ref{figII5.9}). 
\begin{Corollary}[Generalized Steiner\rq s Formula]
\label{CorII5.10}
Let $\bgW \subset \tilde \gW$ be a multi-region configuration with a 
bounded skeletal linking structure.  Then,
\begin{equation}
\label{EqnII5.22}
\vol(\cN_i \cup \cN_{i\to i} \cup \cN_{i\, \infty} ) \,\, = \,\,  \int_{\tilde 
M} \cI(\ell_i) - \cI(r_i)\, dM\, . 
\end{equation} 
\end{Corollary}
 
\par
\begin{figure}[ht]
\includegraphics[width=12cm]{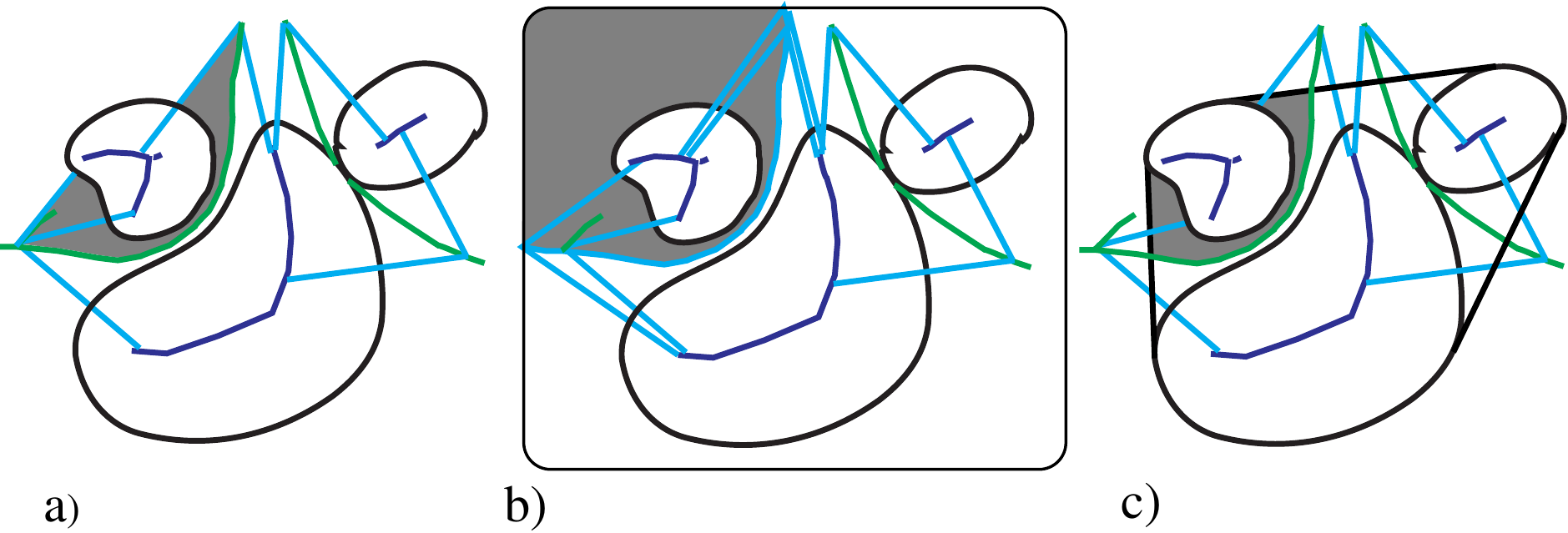}
\caption{\label{figII5.9}  Examples of a total neighborhood $\cN_i \cup 
\cN_{i\to i} \cup \cN_{i\, \infty}$ for a region $\gW_i$, to which the 
generalized Steiner\rq s formula applies: a) absolute threshold; b) 
bounding box; and c) convex hull.}
\end{figure} 
\par
Both of these results are immediate consequences of Corollary 
\ref{CorII5.8}. \hspace{1cm} $\Box$
\par
\begin{proof}[Proof of Corollary \ref{CorII5.8}]
In each case, it is sufficient to apply Corollary \ref{CorII5.6}.  For 
example, for $\cN_{i\to j}$, we apply the theorem to $g = \chi_{\cN_{i\to 
j}}$, whose integral equals $\vol(\cN_{i \to j})$.  The internal integral is 
zero and to compute the external integral, the characteristic function is 
$\equiv 1$ on the external part of the linking lines so we obtain using 
(\ref{EqnII5.20a}) that $\tilde g_{ext} = \cI(\ell_i) - \cI(r_i)$ at points of 
$M_{i \to j}$ and $0$ otherwise.  Thus, the integral is as stated.  
\end{proof}
\par 
We next turn to defining the positional geometric invariants in terms of  
volumes of regions.  
\par

\section{Positional Geometric Properties of Multi-Region Configurations}
\label{SecII:PosGeom}
\par 
We consider a configuration $\bgW = \{\gW_i\} \subset \tilde \gW$, in the 
bounding region $\tilde \gW$ with a bounded skeletal linking structure.  
We will now use the linking structure to investigate the positional 
geometry of the configuration.  First, we use it to determine which of the 
regions should be regarded as neighboring regions.  Then, we use the 
regions associated to the linking structure to define invariants which 
measure the closeness of such neighboring regions.  We further introduce 
invariants measuring \lq\lq positional significance\rq\rq of regions for 
the configuration.  These allow us to identify which regions are central to 
the configuration and which ones are peripheral.  Then, we construct a {\it 
tiered linking graph}, with vertices representing the regions, and edges 
between neighboring regions, with the closeness and significance values 
assigned to the edges, resp. vertices.  By applying threshold values to this 
structure we can exhibit the subconfigurations within the given 
thresholds.  \par
We show that these invariants are unchanged under the action of the 
Euclidean group and scaling acting on the configuration and bounding 
region. Furthermore, because these invariants are defined in terms of 
volumes of regions associated to the linking structure, the results of the 
previous section allow us to compute them as skeletal linking integrals. 
\par

\subsection*{Neighboring Regions and Measures of Closeness} 
\par
We begin by using previous results to identify neighboring regions and 
measuring their closeness.  We use linking between regions as a criterion 
for their being neighbors.  Then, a potential first measure of closeness 
would be to determine, for a fixed $\gW_i$, the minimum value of the 
radial linking function $\ell_{i}$. If $\| L_i(x)\|$ is the value of $\ell_i$ at 
this minimum, then we determine any vector $L_j(x^{\prime})$ so that 
$x^{\prime}$ is linked to $x$, and conclude that $\gW_j$ and $\gW_i$ are 
in some sense close. However, it is possible that two regions are only a 
small distance away from one another at some point, so that 
$\| L_i(x) \| = \| L_j(x^{\prime}) \|$ is small, but are otherwise at most 
points a much greater distance apart.  By contrast, there may be a third 
region $\gW_k$ whose linking vector of smallest norm linking 
$\gW_i$ to $\gW_k$ has larger norm than $\| L_j(x^{\prime}) \|$, 
but possibly other nearby vectors of $L_i$ and $L_k$ maintain 
approximately the same length along a much greater region.  Such a 
situation is illustrated in Figure~\ref{fig.II4.1} or \ref{fig.II6.8}, where 
$\gW_3$ is close to $\gW_1$ for a small region but $\gW_2$ is close to 
$\gW_1$ over a larger region.  \par
\par
\begin{figure}[ht]
\includegraphics[width=4cm]{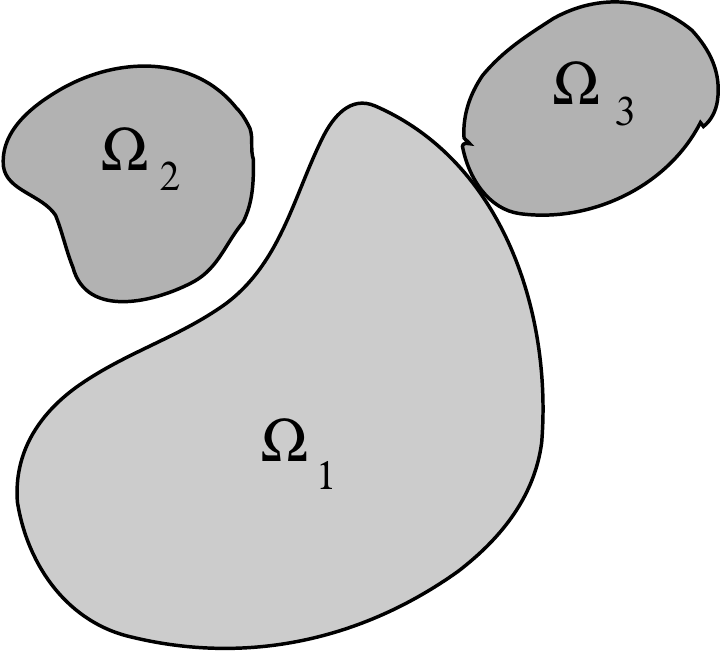}
\caption{\label{fig.II6.8} Measure of Closeness. Deciding whether $\gW_2$ 
or $\gW_3$ is \lq\lq closer\rq\rq to $\gW_1$.}
\end{figure} 
\par
Thus, an appropriate measure of closeness should take into account how 
large is the external region where the regions are close, compared with 
the size of the regions themselves.  We will do so by using volumetric 
measures of appropriate regions.  
For a configuration $\bgW$ with a skeletal linking structure, we 
introduced in \S~\ref{SII:sec.int}, Definition \ref{DefII4.1}, the regions 
$\gW_{i \to j}$, $\cN_{i \to j}$, and $\cR_{i \to j}$.  Because we will use 
volumetric measures of these regions, we assume that we have a bounded 
configuration $\bgW = \{ \gW_i\}$ in a bounding region $\tilde \gW$ with 
corresponding bounded skeletal linking structure.  Thus, all of the regions 
will have finite volume.  \par
The regions $\gW_{i \to j}$ and $\gW_{j \to i}$ capture the neighbor 
relations between $\gW_i$ and $\gW_j$.  As we explained in \S 
\ref{SII:sec.int}, $\cN_{i \to j}$ and $\cN_{j \to i}$ share a common 
boundary region in $M_0$, so they are both empty if one is, and then both 
$\gW_{i \to j}$ and $\gW_{j \to i}$ are empty.  In that case $\gW_i$ and 
$\gW_j$ are not linked.  Otherwise, we may introduce a measure of 
closeness.  \par
 There are two different ways to do this, each having a probabilistic 
interpretation.  First, we let 
$$ c_{i \to j} = \frac{\vol(\gW_{i \to j})}{\vol(\cR_{i \to j})}\qquad \text{ 
and } \qquad c_{i\, j} \, = \, c_{i \to j}\cdot c_{j \to i}\, .$$  
Then, $c_{i \to j}$ is the probability that a point chosen at random in 
$\cR_{i \to j}$ will lie in $\gW_i$ (see Figure~\ref{fig.II6.9}); so $c_{i\, 
j}$ is the probability that a pair of points, one each in $\cR_{i \to j}$ and 
$\cR_{j \to i}$ both lie in the corresponding regions $\gW_i$ and $\gW_j$.  
\par 
\par
\begin{figure}[ht]
\includegraphics[width=10cm]{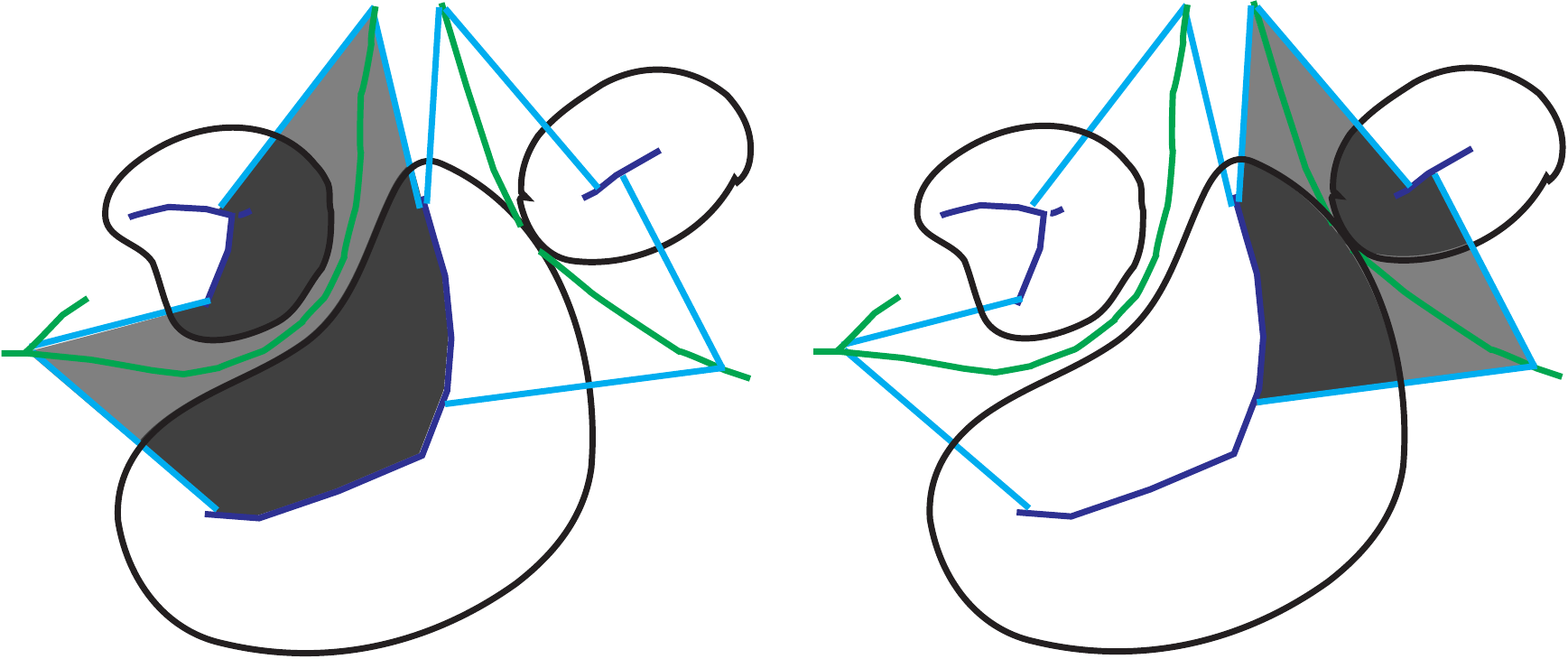}
\caption{\label{fig.II6.9} Measure of closeness for regions in the 
configuration in Figure \ref{fig.II6.8} bounded via a threshold with the 
bounded skeletal linking structure (see \S~\ref{SII:sec.int}).  For a pair of 
neighboring regions $\gW_i$ and $\gW_j$, $c_{i\to j}$ denotes the ratio 
of the volume of the darker region $vol(\gW_{i \to j})$ and the volume of 
the total shaded region $vol(\cR_{i \to j}) = vol(\gW_{i \to j}) + vol(\cN_{i 
\to j})$.}
\end{figure} 
\par

Note that $c_{i\, j}$ contains much more information than the closest 
distance between $\gW_i$ and $\gW_j$, and even the
\lq\lq $L^1$-measure\rq\rq of the region between $\gW_i$ and $\gW_j$.  
It compares this measure with how much of the regions $\gW_i$ and 
$\gW_j$ are closest as neighbors.  If both $\gW_{i \to j}$ and $\gW_{j \to 
i}$ are empty, we let $c_{i \to j}$, $c_{j \to i}$, and  $c_{i\, j} = 0$.  Also, 
we let $c_{i\, i} = 1$.  Thus, from the collection of values $\{c_{i\, j}\}$ 
we can compare the closeness of any pair of regions.  
\par 
Since these invariants depend on a bounded skeletal linking structure, one 
way to introduce a parametrized family $c_{i\, j}(\tau)$ is by considering 
the varying threshold values $\tau$.  For example, $\tau$ may represent 
the maximum allowable values for $\ell_i$ or the maximum value of 
$\ell_i$ relative to some intrinsic geometric linear invariant of $\gW_i$.  
As $\tau$ increases, the bounded region increases and how $c_{i\, 
j}(\tau)$ varies indicates how the closeness of the regions varies when 
larger linking values are taken into account. \par
A second way to introduce a measure of closeness is to use an \lq\lq 
additive\rq\rq contribution from each region and define
$$ c_{i\, j}^a \,\, = \,\, \frac{\vol(\gW_{i \to j}) + \vol(\gW_{j \to 
i})}{\vol(\cR_{i \to j}) + \vol(\cR_{j \to i})} \, . $$
Here $c_{i\, j}^a$ is the probability that a point chosen in the region 
$\cR_{i \to j} \cup \cR_{j \to i}$ lies in the configuration, i.e. in $\gW_i 
\cup\gW_j$.  We also let $c_{i\, j}^a = 0$ if $\gW_i$ and $\gW_j$ are not 
linked; and we let $c_{i\, i}^a = 1$.  Again, to obtain a more precise 
measure of closeness, we can vary a measure of threshold $\tau$ and 
obtain a varying family $c_{i\, j}^a(\tau)$.  
The invariants satisfy $0 \leq c_{i\, j}, c_{i\, j}^a \leq 1$.  The value $0$ 
indicates no linking, for values near $0$, the regions are neighbors but 
distant so they are \lq\lq weakly linked\rq\rq, and for values close to 
$1$, the regions are close over a large boundary region and are \lq\lq 
strongly linked\rq\rq.  \par
There is a simple but crude relation between $c_{i\, j}^a$ and the pair 
$c_{i \to j}$ and $c_{j \to i}$:
$$  c_{i\, j}^a \,\, \leq  \,\, c_{i \to j} \, + \, c_{j \to i} \, . $$
As $c_{i\, j}^a \leq 1$, this is only useful when the two regions are 
weakly linked.  This inequality is a special case of the following simple
lemma whose proof follows easily by induction.  
\begin{Lemma}
\label{LemII6.1}
If $a_i \geq 0$ and $b_i > 0$ for $i = 1, \dots, k$, then
$$  \frac{\sum_{i = 1}^{k} a_i}{\sum_{i = 1}^{k} b_i} \,\, \leq  \,\, \sum_{i 
= 1}^{k} \frac{a_i}{b_i} \, .$$
\end{Lemma}
\hspace{11.5cm} $\Box$ \par

\subsection*{Measuring Positional Significance of Objects Via Linking 
Structures} 
\par
Positional significance of an object among a collection of objects should 
be a measure of the object\rq s centrality in the configuration versus it 
being an outlier object.  We can measure positional significance in both 
absolute and relative terms.  In each case, we emphasize that we are 
considering geometric significance relative to the configuration, rather 
than some 
other notion such as statistical significance.  We begin with the relative 
version.  Given $\gW_i$, we define the {\it positional significance}
$$  s_i =  \frac{\sum_{j \neq i} \vol(\gW_{i \to j})}{\sum_{j \neq i} 
\vol(\cR_{i \to j})} \, . $$
It may take values $0 \leq s_i \leq 1$.  For values near $0$, the region of 
$\gW_i$ linked to some other region is a small fraction of the external 
region between $\gW_i$ and the other regions. Thus, it is a peripheral 
region of the configuration.  We would have the value $s = 0$ if $\gW_i$ is 
not linked to any other region in the bounding region $\tilde \gW$, which 
may occur if there is a threshold for which the region is not linked to 
another region with a linking vector of length less than the threshold.   By 
contrast, if $s_i$ is close to $1$, then there is very little external region 
between $\gW_i$ and the other regions.  Thus, $\gW_i$ is central for the 
configuration.  \par 
By lemma \ref{LemII6.1} it follows
$$  s_i \,\, \leq \,\, \sum_{j \neq i} c_{i\to j} \, , $$
so that $\gW_i$ being weakly linked to the other regions implies it has 
small positional significance for the configuration.  If we would like to 
further base the positional significance of the region $\gW_i$ on its 
absolute size, we can alternatively use an absolute measure of positional 
significance defined by $\tilde s_i = s_i \vol(\gW_i)$.  Then, the effect of 
the smallness of $s_i$ can be partially counterbalanced by the size of 
$\gW_i$.  
\begin{Example}
\label{ExamII.6.2} 
In Figure~\ref{fig.II6.10} is illustrated how the measure of positional 
significance of a region which is central for the configuration decreases 
as the region is moved away from the remaining configuration.  
\end{Example}
\par
\begin{figure}[ht]
\includegraphics[width=10cm]{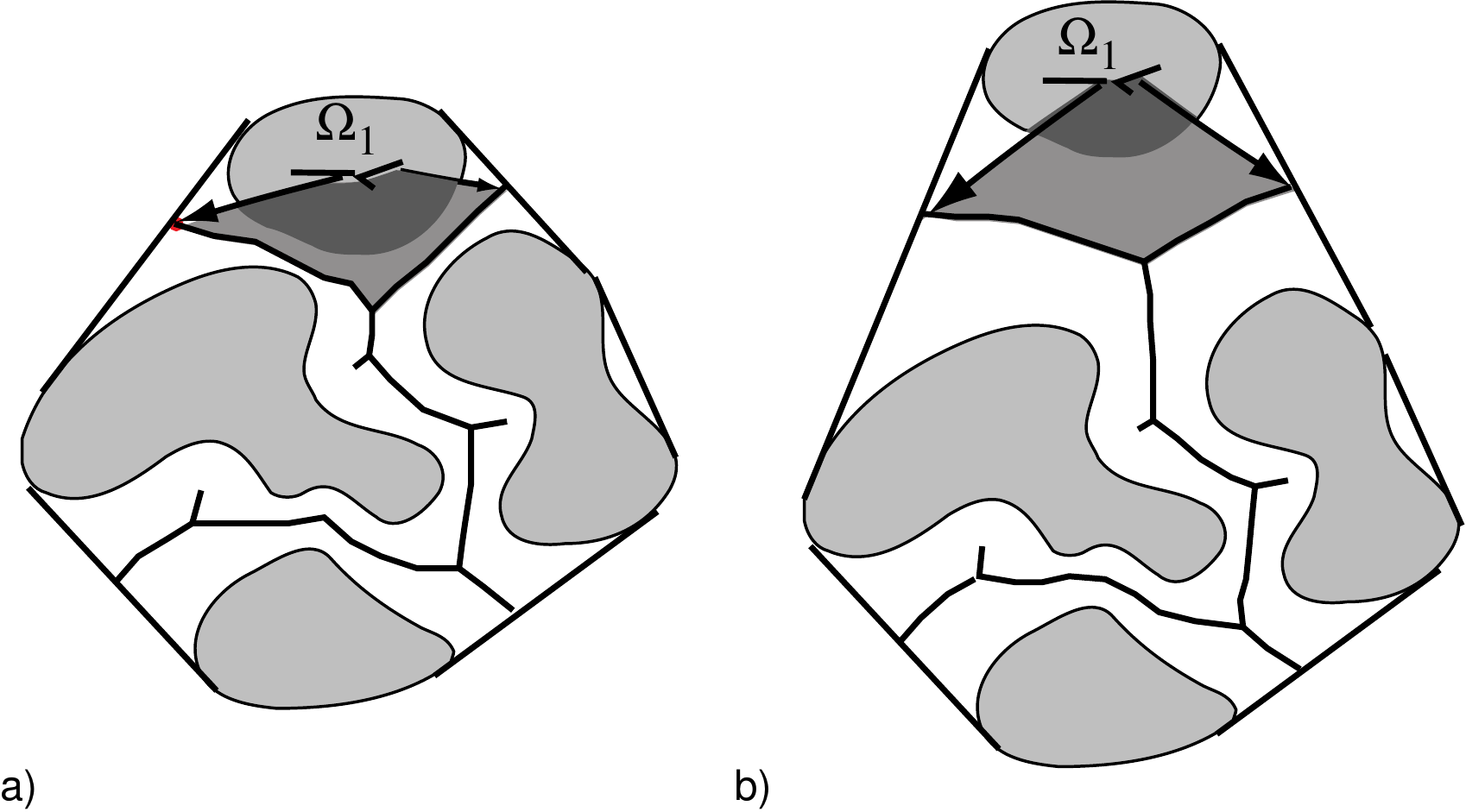}
\caption{\label{fig.II6.10} For $\gW_1$ in configurations using convex 
hull bounding regions, measure of positional significance is the ratio of 
the volume of the darkest region to the volume of the union of the two 
more darkly shaded regions.  In a) $\gW_1$ is central, while in b) when 
$\gW_1$ is moved away from the remaining regions, it becomes less 
significant for the configuration as indicated by the decreasing ratio.}
\end{figure} 
\par
\vspace{2ex}
\flushpar
\subsection*{Properties of Invariants for Closeness and Positional 
Significance} \par
We consider three properties of these invariants: \par
\begin{itemize} 
\item[1)] computation of all of the invariants as skeletal linking 
integrals;
\item[2)] invariance under the action of the Euclidean group and scaling; 
and 
\item[3)] continuity of the invariants under small perturbations of generic 
configurations.
\end{itemize}
\par
\subsubsection*{\it Computation of the Invariants as Skeletal Linking 
Integrals}  
We can use the results from the previous section to compute as skeletal 
linking integrals the above volumes of regions associated to $\bgW$.  This 
is summarized by the following.
\begin{Thm}
\label{ThmII6.3}
If $\bgW = \{\gW_i\} \subset \tilde \gW$ is a multi-region configuration, 
with a skeletal linking structure, then all global invariants of the 
configuration which can be expressed as integrals over regions in 
$\R^{n+1}$ can be computed as skeletal linking integrals using Theorem 
\ref{ThmII5.5}.  In particular, the invariants $c_{i\to j}$, $c_{i\, j}$, 
$c^a_{i\, j}$, and $s_i$ are given as the quotients of two skeletal linking 
integrals using (\ref{EqnII5.20}) and (\ref{EqnII5.20b}). 
\end{Thm}  
\hspace{11.5cm} $\Box$ 
\par

\begin{Remark}
\label{RemII5.2} We could try to alternatively use boundary measures for 
the regions to define closeness and positional significance.  There are two 
problems with this approach.  From a computational point of view, the 
skeletal structures could only be used where the partial Blum condition is 
satisfied.  Moreover, boundary measures do not capture how much of the 
regions are close to each other (only where their boundaries are close).  
For these reasons we have concentrated on (ratios of) volumetric 
measures to capture positional geometry of the configuration.   
\end{Remark}
\flushpar
\subsubsection*{\it Invariance under the action of the Euclidean Group and 
Scaling}  \par
Second, we establish the invariance of the invariants defining closeness 
and positional significance under Euclidean motions and scaling.  Let 
$\bgW = \{\gW_i\} \subset \tilde \gW$ be a multi-region configuration, 
with a skeletal linking structure $\{(M_i, U_i, \ell_i)\}$.  If $f$ is a 
Euclidean motion and $a > 0$ is a scaling factor, then we may let 
$\bgW^{\prime} = \{\gW_i^{\prime}\} \subset \tilde \gW^{\prime}$, where 
$\gW_i^{\prime} = f(\gW_i)$ and $\tilde \gW^{\prime} = f(\tilde \gW)$.  
We also let $\{(M_i^{\prime}, U_i^{\prime}, \ell_i^{\prime})\}$ be a 
skeletal linking structure for $\bgW^{\prime}$ defined by $M_i^{\prime} = 
f(M_i)$, $U_i^{\prime} = f(U_i)$, and $\ell_i^{\prime} = \ell_i$.  As $f$ 
preserves distance and angles, we have $r_i^{\prime} = r_i$, and the image 
of the linking flow for $\bgW$ is the linking flow for $\bgW^{\prime}$. 
Then, $\{(M_i^{\prime}, U_i^{\prime}, \ell_i^{\prime})\}$ satisfies the 
conditions for being a skeletal linking structure for $\bgW^{\prime}$.  As 
$\tilde \gW^{\prime} = f(\tilde \gW)$, the corresponding bounded linking 
structure for $\bgW^{\prime}$ using $\tilde \gW^{\prime}$ is the image of 
that for $\bgW$ for $\tilde \gW$.   Then, the associated linking regions 
for $\bgW^{\prime}$ are the images of the corresponding associated 
linking regons for $\bgW$.  Since $f$ preserves volumes, the invariants 
for closeness and positional significance are preserved by $f$.  \par 
If instead we consider a scaling by the factor $a > 0$, then we let $g_a(x) 
= a\cdot x$.  Now the images of $\bgW$ and $\tilde \gW$ under $g_a$ 
define a configuration $\bgW^{\prime}$ in $\tilde \gW^{\prime}$.  We 
likewise let $\{(M_i^{\prime}, U_i^{\prime}, \ell_i^{\prime})\}$ be defined 
by $M_i^{\prime} = g_a(M_i)$, $U_i^{\prime} = a U_i$, and $\ell_i^{\prime} 
= a \ell_i$ (and $r_i^{\prime} =a r_i$).  As before this is a skeletal 
structure for $\bgW^{\prime}$.  Everything goes through except that $g_a$ 
multiplies volume by $a^{n+1}$.  However, as the invariants are ratios of 
volumes, they again do not change.  We summarize this with the following.
\begin{Proposition}
\label{PropII6.4}
If $\bgW = \{\gW_i\} \subset \tilde \gW$ is a multi-region configuration, 
with a skeletal linking structure, then $c_{i\to j}$, $c_{i\, 
j}$, $c^a_{i\, j}$, and $s_i$ are invariant under the action of a Euclidean 
motion and scaling applied to both $\bgW$ and $\tilde \gW$ for the image 
of the skeletal linking structure for the image configuration and bounding 
region. 
\end{Proposition}
\hspace{11.5cm} $\Box$
\par 
We note that if we consider the absolute positional significance $\tilde 
s_i$, then it is still invariant under Euclidean motions. However, under 
scaling by $a > 0$, it changes by the factor $a^{n+1}$; but this would not 
alter the hierarchy based on absolute positional significance, as all 
$\tilde s_i$ would be multiplied by the same factor.
\begin{Remark}
\label{RemII.6.5}
Importantly, the invariance in Proposition \ref{PropII6.4} crucially 
depends on also applying the Euclidean motion and/or scaling to the 
bounding region $\tilde \gW$.  If the region is either fixed, or depends 
upon an external condition which prevents it from transforming along with 
the configuration, then the invariance does not hold.  This has important 
consequences when properties of the configuration are measured, and is a 
problem for many methods in imaging where the imaged region first has to 
be normalized in some way.  In our case, the measurements should also be 
taking into account how their relation with the bounding region changes.
\end{Remark}
\par  
\subsubsection*{\it Continuity and Changes under Small Perturbations} 
\par
Lastly, suppose that $\bgW = \{\gW_i\} \subset \tilde \gW$ is a 
multi-region configuration, with a bounded skeletal linking structure.  We 
ask how the invariants will change under small perturbations.  This is a 
generally challenging question which we are not attempting to answer in 
this paper.  Thus, for what we do say, we will not give precise proofs, but 
indicate what we expect to happen.  \par
First, if the linking structure is a Blum linking structure, then only if the 
objects undergo a sufficiently small deformation for say the $C^{\infty}$ 
topology will there be the stability of the Blum linking structure, so the 
Blum medial axes will deform in a smooth fashion.  Then, the associated 
regions will also deform in a piecewise smooth fashion.  Hence, the 
volumes of these regions will vary continuously.  Thus, the quotients of 
the volumes will also vary continuously.  It then follows that the 
invariants, which are quotients of such volumes will also vary 
continuously.  \par 
If instead we only require that the deformations be $C^1$ small then the 
Blum linking structure will be unstable; however, it is possible to deform 
it to a skeletal linking structure which changes continuously under this 
small deformation.  Under these changes the skeletal sets will deform 
without changing their basic structure, and the associated regions will 
change in a continuous fashion so the resulting invariants, which are 
volumetric based, will also change in a continuous fashion.  
\par
How exactly they will change will depend on the particular deformation 
and this will demand a precise analysis that we are not prepared to carry 
out here.  However, as a simple example, suppose we enlarge one of the 
regions $\gW_i$ by increasing the radial vectors by a factor $a > 1$, so 
that $a r_i < \ell_i$, and without altering the remainder of the skeletal 
structure.  If the region remains in the bounding region and doesn\rq t 
intersect itself or other regions, then the ratio $\vol(\gW_{i \to j})$ to 
$\vol(\cR_{i \to j})$ will increase for each $j$ so the $s_i$ will increase, 
as will the $c_{i\to j}$.  If instead $0 < a < 1$, then $s_i$ and $c_{i\to j}$ 
will decrease.  \par
If instead we move the region $\gW_i$ in a direction away from all of the 
other regions without altering its size, then in general $s_i$ will 
decrease, and conversely if we move it toward the other regions, generally 
$s_i$ will increase. Thus, the invariants capture the type of geometric 
properties that we would hope.

\subsection*{Tiered Linking Graph}
%%%\label{Sec: Tier.grph}
\par
Now that we have obtained invariants appropriately measuring closeness 
and positional significance among objects, we can combine these into a 
graph theoretic structure.  We define exactly the type of graph structure 
we consider.  For us a graph $\gG$ is defined by a finite set of vertices $V 
= \{v_i: i = 1, \dots , m\}$, and a set of unordered edges $E = \{e_{i\, j}\}$ 
with at most one edge $e_{i\, j}$ between any pair of distinct vertices 
$v_i$ and $v_j$. 
\begin{Definition}
A {\em tiered graph} consists of a graph $\gG$ together with a discrete 
nonnegative function $f: V \cup E \to \R_{+}$ which we shall more simply 
denote by $f: \gG \to \R_{+}$.  The discrete function $f$ has values $f(v_i) 
= a_i \geq 0$ for each vertex $v_i$, and $f(e_{i\, j}) = b_{i\, j} \geq 0$ for 
each edge $e_{i\, j}$.  
\end{Definition}
\par
Given such a tiered graph, we can view its values on vertices and edges as 
height functions assigning weights to the vertices and edges; and then 
apply \lq\lq thresholds\rq\rq to $f$ to identify subgraphs, consisting of 
distinguished vertices and edges.  First, given a value $b > 0$, we can 
consider the subgraph $\gG_b$ consisting of all vertices, but only those 
edges where $f \geq b$.  $\gG_b$ decomposes into connected subgraphs 
consisting of vertices which have edges of weights $> b$.  As $b$ 
decreases from $B = \max \{b_{i\, j}\}$, then we see the smaller graphs 
begin to merge as edges are added, until we reach $\gG$ for $b = \min 
\{b_{i\, j}\}$.  In fact, as pointed out by the referee this approach is 
similar to the known method of \lq\lq clustering\rq\rq for collections of 
points.  \par 
If instead we consider the threshold $a$ for $f$ on vertices, then instead 
we define $\gG^a$ to consist of those vertices with $f \geq a$, and only 
those edges joining two vertices within this set.  This identifies a 
subgraph consisting of the most important vertices as measured by 
weights, along with the edges between these vertices.  Then as $a$ 
decreases from $A = \max \{a_i\}$, we again see the small graphs 
being supplemented by additional vertices with edges being added from 
these vertices until we reach the full graph when $a = \min \{a_i\}$.  This 
gives a hierarchical structure to the graph $\gG$.  Along with the 
subgraphs and the hierarchical structure, we can also identify vertices 
which are joined by strongly weighted edges, and important vertices with 
large weights $a_i$, and less significant ones with small weights $a_i$.  
It would also be possible to use this to turn $\gG$ into a directed graph 
where the direction of an edge is from lower positional significance to 
one of higher. 
\par
\begin{figure}[ht]
\includegraphics[width=8.0cm]{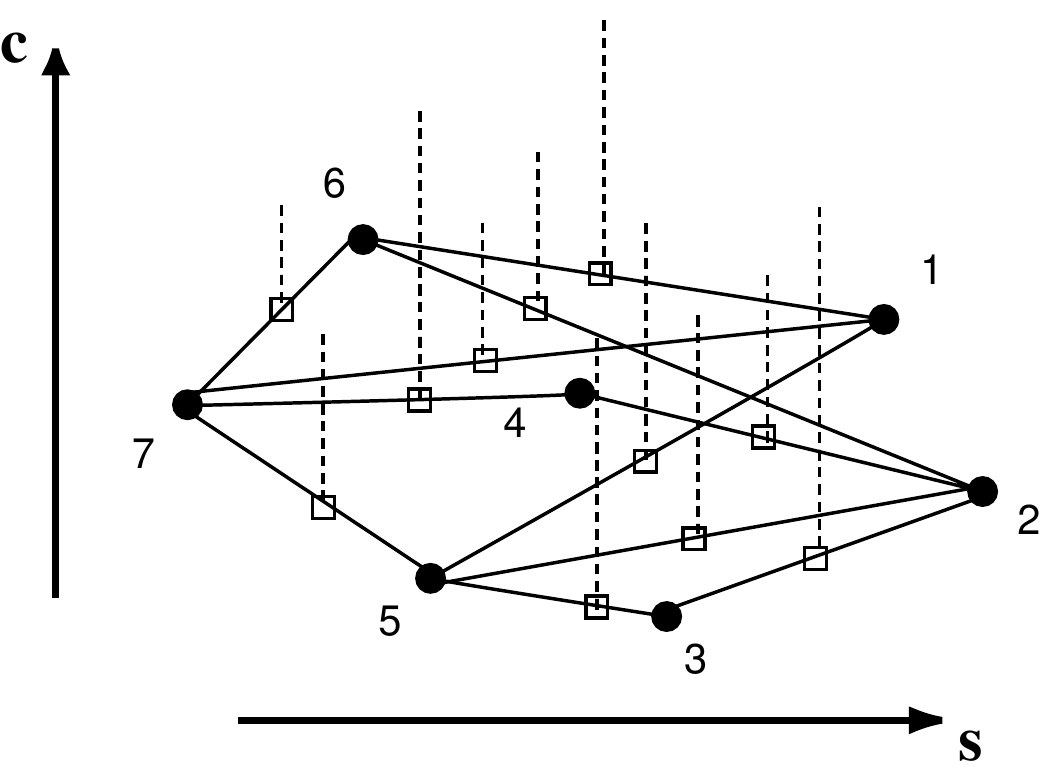}
\caption{Example of a tiered graph structure which we view as lying in a 
horizontal plane with one horizontal axis indicating positional 
significance, and the values of the height function giving values $c$ on the 
edges indicated by the heights of the dotted lines above the edges.}
\label{fig.II6.11}
\end{figure}
\par
%%% \vspace{3ex}
This approach, using the tiered graph structure, applies to a configuration 
of multi-regions with a skeletal linking structure.  We define the 
associated {\it tiered linking graph} $\gL$ as follows.  For each region 
$\gW_i$, we assign a vertex $v_i$ in $\gL$, and to each pair of 
neighboring regions $\gW_i$ and $\gW_j$, we assign an edge $e_{i\, j}$ 
joining the corresponding vertices.  If the regions are not neighbors, there 
is no edge.  We define the height function $f$ by: $f(v_i) = s_i$ and 
$f(e_{i\, j}) = c_{i\, j}$ (or $c^a_{i\, j}$).  
\par
An example is shown in Figure~\ref{fig.II6.11}.  Then, when we apply the 
thresholds, we remove vertices to the left of some vertical line or edges 
whose heights are below some height.  We see how subconfigurations 
associated to the subgraphs merge into larger configurations as the 
vertical line indicating $s$ moves to the left, or the height moves 
downwards, with the resulting graphs based on closeness of the regions or 
their positional significance.  Also, the position along the $s$-axis 
identifies the hierarchy of regions in the configuration.  
\par
\begin{Remark}
\label{RemII6.6}
Although we have used the term \lq\lq threshold\rq\rq\, here to identify 
graph theoretic features, this differs considerably from the notion of 
threshold $\tau$ that we referred to in \S \ref{SII:sec.int} when we 
discussed the bounding region, and how the threshold $\tau$ would 
introduce a parametrized family of invariants that would offer more 
detailed information about the positional geometry of the configuration.  
The height function that we have defined here would become a 
parametrized family of height functions depending on $\tau$.  
\end{Remark}

\subsection*{Higher Order Positional Geometric Relations via Indirect 
Linking}
\par
We conclude this introduction to positional geometry by observing that the 
skeletal linking structures we have introduced only relate neighboring 
regions.  We have shown that the skeletal linking structures (and 
specifically the full Blum medial linking structures) capture both the 
shape properties of individual regions (and their geometry), and the 
geometry of the external region yielding positional information.  However, 
this approach allows regions to hide other regions from the geometric 
influence of non-neighbors.  One way to overcome this is to work with 
\lq\lq indirect linking\rq\rq\, where two regions are linked via a third 
region $\gW_3$.  For example, if we add a small elliptical region  
$\gW_4$ to the configuration in a) of Figure~\ref{fig.II6.12}, we obtain b) 
of the figure.  The small region alters the closeness of regions $\gW_1$ 
and $\gW_2$.  However, we see in b) that the linking neighborhood for the 
$\gW_4$ lies in the linking region between $\gW_1$ and $\gW_2$ in a).  
Hence, if we allowed indirect linking through $\gW_4$, then the closeness 
of $\gW_1$ and $\gW_2$ would not be altered.  This is just one example 
of various aspects of configurations that could be better understood by 
deriving indirect linking properties from the linking structure.  While the 
numerical measures could not be deduced just using the values from 
linking, the occurence of indirect linking could already be seen from the 
tiered graph structure.  However, we will not attempt to develop this 
further in this paper.
\par
\begin{figure}[ht]
\includegraphics[width=10.0cm]{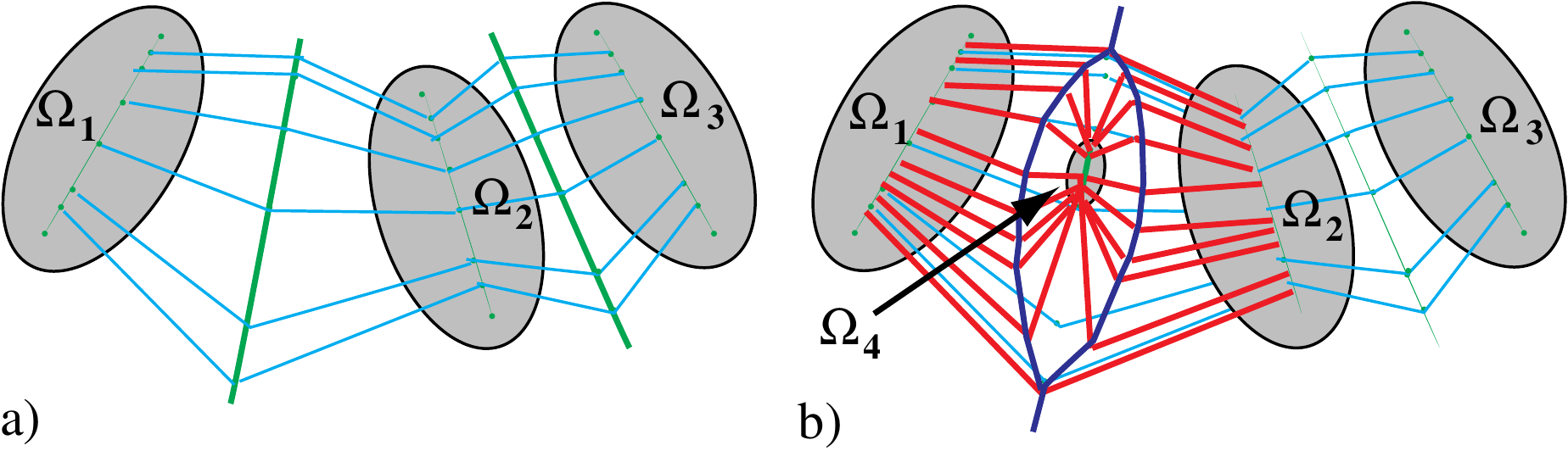}
\caption{Indirect linking when a region $\gW_4$ is added between regions 
in a) yielding b).  The Blum linking structure in b) has the neighborhood of 
$\gW_4$ contained in the neighboring regions of $\gW_1$ and $\gW_2$, so 
if indirect linking through $\gW_4$ is allowed, then there is no change in 
the closeness of $\gW_1$ and $\gW_2$.}
\label{fig.II6.12}
\end{figure}

\newpage
\section{Multi-Distance and Height-Distance Functions and Partial 
Multi-Jet Spaces} 
\label{S:sec4}
\par
	In order to prove the existence of a medial linking structure for 
a generic multi-region configuration it will be necessary to prove a 
generalization of the transversality theorem used by Mather \cite{M1} for 
the case of single regions.  This version will be applied in several 
different forms to a configuration $\bgW = \{ \gW_i\}$ in $\R^{n+1}$, 
via a model configuration $\bgD = \{ \gD_i\}$ and embedding $\Phi : \bgD 
\to \R^{n+1}$ as in \S \ref{S:sec1}.  

\par
For such a configuration, we will consider a collection of functions 
including: \lq\lq distance functions\rq\rq, \lq\lq height functions\rq\rq, 
\lq\lq multi-distance functions\rq\rq, and \lq\lq height-distance 
functions\rq\rq.    The transversality theorem applied to each one will 
yield for each compact set containing the configuration in its interior a 
residual set of embeddings $\Phi$ such that appropriate genericity 
properties hold.  Then, we use the special properties of the regions and the 
functions to see that the set of such embeddings is actually open.  Hence, 
the set of generic embeddings will be open and dense.  
\par

\subsection*{ Multi-Distance and Height-Distance Functions} 
\label{Multdist} 
\par  
For a multi-region configuration given by an embedding $\Phi : \bgD \to 
\R^{n+1}$, we define the associated multi-distance and height-distance 
functions.  However, we change the emphasis (and notation) from that of 
functions defined on subsets of the ambient space $\R^{n+1}$ to functions 
defined via the embedding $\Phi$ on the model configuration $\bgD$.  

\subsubsection*{\it  Stratification of the Configuration} \hfill 
\par  
For $\bgD$ we introduce notation to keep track of various strata of the 
boundaries of each $\gD_i$ and their corresponding images $\gW_i$ under 
$\Phi$.  As in \S \ref{S:sec1}, the boundary of $\gD_i$ is denoted by $X_i$; 
and we let $X = \cup_{i} X_i$ denote the stratified set which is the 
union of the boundaries.  We also let $X_{i j}$ denote the union of the 
smooth strata of $X_i \cap X_j$ (i.e. those of dimension $n$).  We also let 
for each $i > 0$, $X_{i 0} = X_i \backslash (\cup_{j >0} Cl(X_{i j}))$.  For 
regions satisfying the boundary edge condition, 
$\gD_i \cap \gD_j = Cl(X_{i j})$; 
thus, $X_{i 0}$ is an open stratum of $X_i$ and the points of $X_{i 0}$ do 
not lie in any region other than $X_i$.  For the configuration $\bgW$, with 
complement $\gW_0$, the smooth strata consist of the image of the union 
$\cup_i X_{i\, 0}$.  Just as for $X_{i\, j}$, we let $X_{0\, i} = X_{i\, 0}$ to 
emphasize that the image of $X_{0\, i}$ is a smooth stratum of $\gW_0$.   
Lastly, the remaining strata are formed from the $k$-edge-corner points 
for each given $k$.
\par  
For any $\gD_i$, we define an {\it index set} $\cJ_i = \{ j \neq i : X_{i j} 
\neq \emptyset\}$.  For $i = 0$, we let $\cJ_0 = \{ j > 0 : X_{0 j} \neq 
\emptyset\}$.  Also, for $i \geq 0$ we let $q_i = | \cJ_i |$.  
For each $i \geq 0$ we let $\cirX_i = \coprod_{j \in \cJ_i} X_{i j}$, which 
is the set of points in smooth strata of $X_i$.  Next, we let $X_{\cJ_i} = 
\coprod_{j \in \cJ_i} \cirX_j$, which is the disjoint union of the smooth 
strata of those $X_j$ for which  $X_{i\, j} \neq \emptyset$ (i.e. are 
adjoined to $X_i$).  We note that each point in $\cirX_i$ belongs to exactly 
one of the components of $X_{\cJ_i}$; however if there are $j, j^{\prime} 
\in \cJ_i$ with $X_{j\, j^{\prime}} \neq \emptyset$, then any point $x \in 
X_{j\, j^{\prime}}$ has representatives which belong to both $\cirX_j$ and  
$\cirX_{j^{\prime}}$.  \par
This is the space on which the models for \lq\lq simple linking\rq\rq will 
be defined.  However, to model both partial and self-linking, we must 
allow multiple copies of some of the $\cirX_j$.  We do so by extending the 
above definition using an {\em assignment function}. Given $m > 0$, we 
consider an {\em assignment}, which is a discrete function for integers 
$1 \leq p \leq m$, sending $p \mapsto j_p \in \cJ_i$.  Then, with this 
assignment, we define $X_{\cJ_i} = \coprod_{p = 1}^{m} \cirX_{j_p}$.  Note 
that because of the complexity of the notation which we will need to 
introduce, we do not include the assignment in the notation for 
$X_{\cJ_i}$ 
but will instead refer specifically to the assignment used for defining 
$X_{\cJ_i}$.  The basic model above for $X_{\cJ_i}$ uses the assignment 
for $m = q_i$ bijectively mapping $\{ p \in \Z : 1 \leq p \leq q_i\}$ to 
$\cJ_i$.  
\par
Example \ref{Exam4.7} 
illustrates $X_{\cJ_i}$ for a multi-region configuration given in Figure 
\ref{fig.4.4}.  We will let the images of these subspaces of the $X_i$ under 
$\Phi$ inherit the same notation, e.g. $\Phi(X_{i\, 0}) = \cB_{i\, 0}$, etc.
\par
We will also use the notation $\gS_{Q \, i}\subset X_i$ to denote the 
stratification by smooth $Q_k$-points in $X_i$.  Then, we note that while 
$\cirX_i$ consists of points in the smooth strata, $\cirX_i \cap \gS_{Q\, 
i} = \emptyset$, and the points in $\gS_{Q\, i}$ are also smooth points on 
the boundary $X_i$.  
\par
\begin{figure}[ht]
\includegraphics[width=6cm]{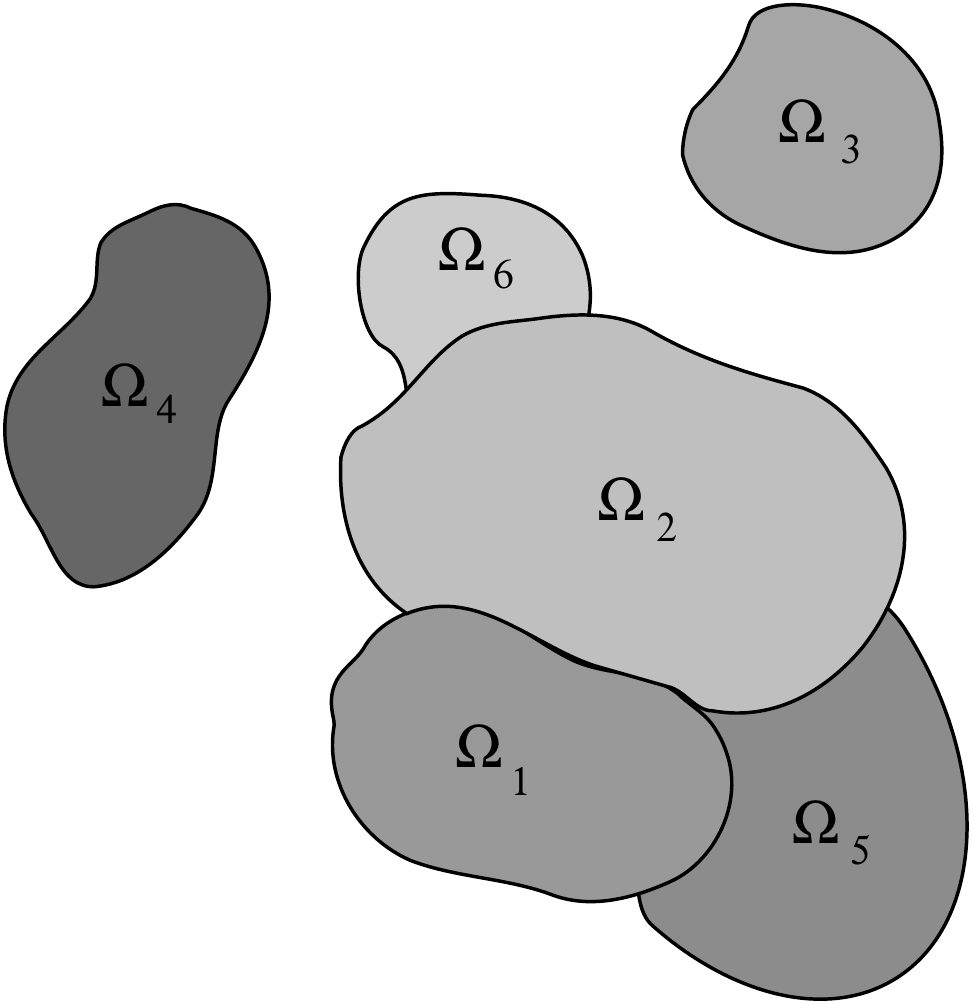}
\caption{\label{fig.4.4} A multi-region configuration in $\R^2$ given 
earlier with the decomposition explained in Example \ref{Exam4.7}.}
\end{figure} 
\begin{Example}
\label{Exam4.7}
In Figure \ref{fig.4.4} is a multi-region configuration $\bgW$ defined by 
the model $\Phi : \bgD \to \R^{n+1}$.  We describe the properties of the 
image regions $\{ \gW_i\} $ in terms of the model $\bgD$.  Both $X_3 = 
\partial \gD_3$ and $X_4 =\partial \gD_4$ are smooth so $\cJ_3 = \cJ_4 
= \{0\}$.  However, for $\gD_2$, $\cJ_2 = \{0, 1, 5, 6\}$ and for $\gD_1$, 
$\cJ_1 = \{0, 2, 5\}$.  For each $i = 1, 2, 5, 6$, $X_{\cJ_i}$ without 
repetitions consists of 
disjoint unions of smooth boundary curves of those regions $\gD_j$ which 
share a portion of their boundary with $\gD_i$.  The images of these 
curves give the corresponding curves of shared boundary regions between 
$\gW_j$ and $\gW_i$.  For example, the regions $\gW_1$, $\gW_5$, 
$\gW_6$ are adjoined to $\gW_2$ along the boundary curves which are 
images of $X_{2 1}$; $X_{2 5}$, and $X_{2 6}$.  All regions are adjoined to 
the complement.
\end{Example}
%%% \par
%%%Recall that in \S \ref{S:sec1}, we introduced $X$ to be the union of 
%%% the boundaries $X_i$ of the regions $\bgD_i$ for the model 
%%%configuration, and gave the stratification of $X$ as a subset in 
%%% $\R^{n+1}$.  We let $\cirX_j$ denote the union of smooth strata in 
%%% $X_j$, and let $X_{\cJ_i}$ be the disjoint union of the $\cirX_j$ 
%%% which share strata with $X_i$. 
\par
Our future computations shall concern transversality conditions on 
mappings associated to $\Phi \in \emb(\bgD, \R^{n+1})$.  For this reason 
we shall reintroduce the distance and height functions so they are defined 
on $X$.  The {\it distance function} on $X$ via the embedding $\Phi$ is the 
function 
$\gs : X \times \R^{n+1} \to \R$ defined by $\gs (x, u) = \| \Phi(x) - u\|^2$.  
Likewise, the {\it height function} via the embedding $\Phi$ is the 
function  $\nu : X \times S^{n} \to \R$ defined by $\nu(x, v) = \langle 
\Phi(x), v\rangle$, where $S^n$ is the unit sphere in $\R^{n+1}$.  
Using these basic functions we define both multi-distance functions and 
height-distance functions.  \par
We also introduce a variant form of the distance function for regions with 
singular boundaries to allow it to be defined at the nonsingular $Q_k$ 
points.  We let $X_i^*$ denote the set of smooth points of $X_i$ obtained 
by removing the points of type $P_k$ and singular $Q_k$ points, for all 
$k$.  We note that $\gS_{Q \, i} \subset X_i^*$, and that $X_i^*$ differs 
from the set $\cirX_i$ by including the nonsingular $Q_k$ points of $X_i$.  
Then we define the distance function $ \gs_i = \gs | (X_i^* \times 
\R^{n+1})$.  In addition to obtaining the generic properties of the distance 
function,  we shall also see that it generically satisfies transversality 
conditions relative to the strata of $\gS_{Q \, i}$. \par
We will use the standard notation that $Y^r = Y \times \cdots \times Y$ 
($r$ factors), and let $\gD^{(r)} Y \subset Y^r$ be the {\it generalized 
diagonal} consisting of $(y_1,\cdots,y_r) \in Y^r$ such that $y_i = y_j$ for  
some $i \neq j$.  We also let $\gD^{r} Y \subset Y^r$ denote the {\it exact 
diagonal} consisting of $(y,\cdots,y) \in Y^r$.  Then, we let $Y^{(r)} 
\overset{def}{=} Y^r \backslash \gD^{(r)} Y$.  \par
For each $j = 0, 1,  \dots , m$, we let $\R^{n+1}_j$ denote a copy of 
$\R^{n+1}$ indexed by $j$, and we abbreviate $\prod_{j = 0}^{m}\R^{n+1}_j 
= (\R^{n+1})^{m+1}$.  Then $(\R^{n+1})^{(m+1)}$ denotes the complement of 
the generalized diagonal in $(\R^{n+1})^{m+1}$. Quite generally we let 
$\pi_j : \prod_{i = 0}^{k} \R^{n+1}_i \to  \R^{n+1}_j$ denote both the 
projection on the $j$--th factor, as well as its restriction to 
$(\R^{n+1})^{(k+1)} \subset \prod_{i = 0}^{k}\R^{n+1}_i$.
%%% Using the notation for index sets from \S 
%%% \ref{S:sec2}, we let $(\R^{n+1})^{(q_i)}$ denote the complement of the 
%%% generalized diagonal in $\prod_{j \in \cJ_i}\R^{n+1}_j$. 
\par
Given an {\em assignment} for $1 \leq p \leq m$, $p \mapsto j_p \in 
\cJ_i$.   We denote points in $(\R^{n+1})^{(m)}$ by 
$(u^{(j_1)}, \dots , u^{j_{(m)}})$, where we let $u^{(j_p)}$ have coordinates 
$(u^{(j_p)}_1, \dots , u^{(j_p)}_{n+1})$.  This will allow us to consider 
multiple centers $u^{(j_p)}$ for distance functions on boundaries of given 
regions $\cirX_j$ with $j = j_p$.    

\par
\begin{Definition}
\label{Def4.4}
A {\it multi-distance function} associated to the $i$-th region $\gW_i$ 
for the configuration $\bgW$ defined by the model $\Phi : \bgD \to 
\R^{n+1}$, together with $m > 0$ and an assignment $p \mapsto j_p$, is 
given by 
\begin{align}
\label{Eqn4.2}
\rho_i: X_{\cJ_i} \times (\R^{n+1})^{(m + 1)} &\rightarrow \R^2, \\
(x, (u^{(j_1)}, \dots , u^{(j_{m})}, u^{(i)})) &\mapsto (\gs(x, u^{(i)}), \gs(x, 
u^{(j_p)}))  \quad \makebox{ for } x \in \cirX_{j_p}\, .  \notag
\end{align}
\end{Definition}

\par  
These functions are \lq\lq multi-distance\rq\rq\, functions in the sense 
that they incorporate multiple functions capturing the squared distance 
from distinct points $u^{(j)}$ in the ambient space, and ultimately we are 
interested in the case where the $u^{(q)} \in \intr(\gW_j)$ for $j \in 
\cJ_i$.  \par
Next we define height-distance functions associated to each region 
$\gD_i$.  
\begin{Definition}
\label{Def4.5}
The {\it height-distance function} associated to 
the configuration $\bgW$ defined by the model $\Phi : \bgD \to \R^{n+1}$  
together with $m > 0$ and an assignment $p \mapsto j_p$ and is given by
\begin{align}
\label{Eqn4.3}
\tau: X_{\cJ_0} \times (\R^{n+1})^{(m)}  \times S^n  &\rightarrow \R^2, 
\\
(x, (u^{(j_1)}, \dots , u^{(j_{m})}), v) &\mapsto (\nu(x, v), \gs(x, 
u^{(j_p)}))  \quad \makebox{ for } x \in \cirX_{j_p} \notag
\end{align}
where $\cJ_0 = \{ j : X_{0, j} \neq \emptyset\}$
%%%  and $q_0 = | \cJ_0 |$.
\end{Definition}
\flushpar
{\bf Notation: }  In what follows, we shall frequently abbreviate $\gs(x, 
u^{(j)})$ as $\gs_{j}(x)$. 
\par
The height-distance function will be used to relate the distance squared 
functions for different points $u^{(j)}$ in the ambient space and properties 
of the height function on $X_{\cJ_0}$.  
We shall apply a transversality theorem to these two families of 
multi-functions for various $m$ and assignments $p \mapsto j_p$ to 
deduce the genericity of properties of linking (in the case of $\rho_0$) 
and those relating linking with the boundary of the unlinked region. 
Since the transversality theorem is a result on the level 
of jet bundles, our focus in the next section is on defining a special type 
of \lq\lq partial multijet space\rq\rq\, and a special kind of multijets of 
such functions which map into these spaces. 

\begin{figure}
\begin{center}
\includegraphics[width=3.5cm]{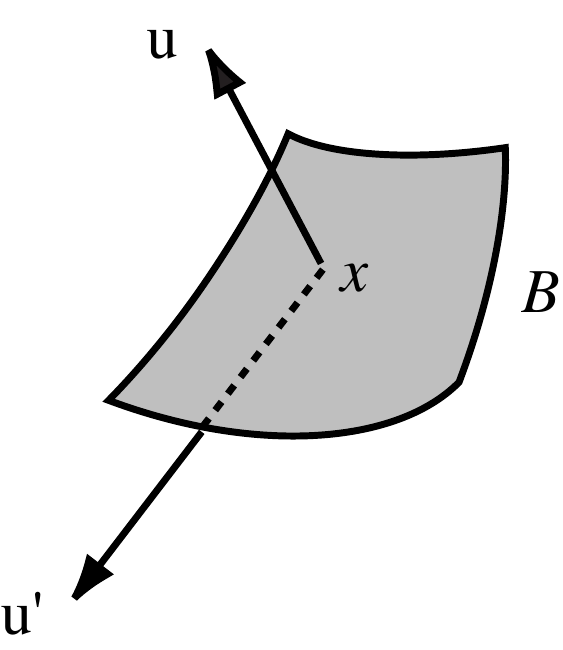}
\hspace*{0.30cm}
\includegraphics[width=5.0cm]{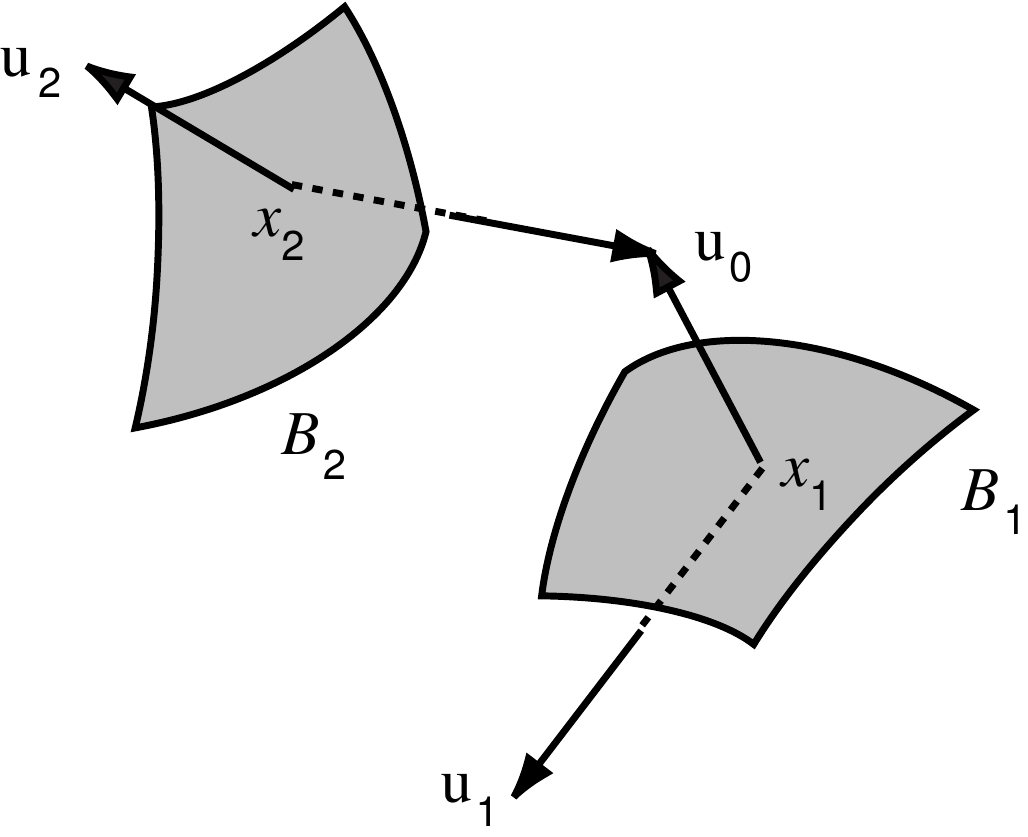}
\end{center}
\hspace {0.5in}(a) \hspace {2.0in} (b) \hspace{0.5in} 
\caption{a) Multi-distance function - pair of families of distance 
functions from a region boundary to distinct points $u$ and $u^{\prime}$ in 
the complement. b)  Linking via multi-distance functions - families of 
distance functions from region boundaries $\cB_i$ to a common point 
$u_0$ in the complement.} 
\label{fig7}
\end{figure}
\par

\subsection*{ Partial Jet Spaces for Multi-Distance and Height-Distance 
Functions} 
\label{Partjetspaces} 
\par
To examine the generic properties of linking between regions, we 
introduce the {\em partial multijet space} and the corresponding {\em 
partial multijet extension maps} for both multi-distance and 
height-distance functions.  
\par
Given an integer $s > 0$, $i$, an integer $m > 0$ with assignment $p 
\mapsto j_p$, we let 
$\bl=(\ell_1,\dots,\ell_{q_i})$ denote an ordered partition of the form 
$s=\ell_1 + \cdots + \ell_{m}$ with each integer $\ell_p > 0$.  For each $1 
\leq p \leq m$, we denote points $x^{(j_p)} = (x^{(j_p)}_1, \dots , 
x^{(j_p)}_{\ell_p}) \in \cirX_{j_p}^{(\ell_p)}$.  
Define
\begin{align}
\label{Eqn4.4}
X_{\cJ_i}^{(\bl)} \,\, &= \,\,  \{ (x^{(j_1)}, \dots , x^{(j_{m})}) \in 
\cirX_{j_1}^{(\ell_1)}\times \cdots \times \cirX_{j_{m}}^{(\ell_{m})} :  
x_1^{(j_p)} \in X_{i\, j_p} \text{ for all $p$ and } \\ 
& \qquad \qquad \text{ if } j_p = j_{p^{\prime}}, \text{then } x_q^{(j_p)} 
\neq x_{q^{\prime}}^{(j_{p^{\prime}})} \text{for any } q, q^{\prime}\}\, . 
\notag
\end{align}
%%% where $\cJ_i = \{ j_1, \dots , j_{q_i}\}$ 
%%% and if $\ell_p = 0$, then there is no factor $\cirX_{j_p}$.  
Observe that $X_{\cJ_i}^{(\bl)}$ is an open subset of the product space in 
(\ref{Eqn4.4}).  
We denote a point in $X_{\cJ_i}^{(\bl)}$ by 
$(x^{(j_1)}, \dots , x^{(j_{m})})$, where each 
$x^{(j_p)} = (x^{(j_p)}_1, \dots , x^{(j_p)}_{\ell_p})$. 
%%%  \in \cirX_{j_p}^{(\ell_p)}$.  \par 
%%% As all of the $\cirX_{j_p}$ are disjoint in $X_{\cJ_i}$, we have 
%%% $X_{\cJ_i}^{(\bl)} \subset (X_{\cJ_i})^{(s)}$.  
By definition, for any $s$-tuple in $X_{\cJ_i}^{(\bl)}$, any two points in it 
belonging to the same $\cirX_{j}$ are distinct.  However, we still note 
that not every $s$-tuple need belong to $X^{(s)}$.  This is because if there 
are $j_p \neq  j_{p^{\prime}}$ with $X_{j_p\, j_{p^{\prime}}} \neq 
\emptyset$, then any point $x \in X_{j_p\, j_{p^{\prime}}}$ is represented 
in the disjoint copies of both $\cirX_{j_p}$ and $\cirX_{j_{p^{\prime}}}$. 
Nonetheless the multi-function $\rho_i$ will be defined for $x \in 
\cirX_j$ by $(\gs( x, u^{(i)}), \gs(x, u^{(j_p)}))$ versus for $x \in 
\cirX_{j_{p^{\prime}}}$, by $(\gs(x, u^{(i)}), \gs(x, u^{(j_{p^{\prime})}}))$, 
with $u^{(j_p)} \neq u^{(j_{p^{\prime}})}$.  A consequence is that the 
transversality conditions on $\cirX_{j_p}$ and $\cirX_{j_{p^{\prime}}}$ 
are different and will be verified independently.  This motivates our 
definition of the partial multijet spaces to follow. 
\par
We consider the usual $k$-multijet space ${_s}J^k(X,\R^2)$, which 
consists of $s$ $k$-jets of germs at distinct points of $X$ mapping to 
$\R^2$, and we define a subspace using the partition $\bl$.  
\begin{Definition}
\label{Def4.6}
For $X_{\cJ_i}$ (defined for $m > 0$ with an assignment $p \mapsto j_p$, 
let $\bl =(\ell_1,\dots,\ell_{m})$ be an ordered partition 
of $s > 0$.  Then, the {\em partial $\bl$-multi $k$-jet space} is the 
subspace of ${_s}J^k(X_{\cJ_i},\R^2)$ defined by the restriction
\begin{equation}
\label{Eqn4.4b}
 {_{\bl}}E^{(k)}(X_{\cJ_i},\R^2)  \,\, = \,\, \left( \prod_{p=1}^{m} \, 
{_{\ell_p}}J^k(\cirX_{j_p},\R^2)\right) |\, (X_{\cJ_i})^{(\ell)}.
\end{equation}
%%%where $\cJ_i = \{ j_1, \dots , j_{q_i}\}$ (and we exclude any factor 
%%% for which $\ell_p = 0$).  
\end{Definition}
The two basic properties of multijet spaces which are shared by the 
partial multijet spaces are summarized in the following straightforward  
lemma.
\begin{Lemma}
\label{Lem4.4}
The partial multijet spaces have the following properties:
\begin{itemize}
\item[(a)] ${_{\bl}}E^{(k)}(X_{\cJ_i},\R^2)$ is a smooth submanifold of 
${_s}J^k(X_{\cJ_i},\R^2)$; and 
\item[(b)] ${_{\bl }}E^{(k)}(X_{\cJ_i},\R^2)$ is a locally trivial fiber bundle 
over $X_{\cJ_i}^{(\bl)}$ with fiber \par 
$\prod_{p=1}^{m} \, (J^k(n,2) \times \R^2)^{\ell_{p}}$.
\end{itemize}
\end{Lemma}
\hspace{11.5cm} $\Box$
\par 
Both the multi-distance and height-distance functions defined for 
 $m > 0$ with an assignment $p \mapsto j_p \in \cJ_i$, are given as 
mappings of the form $\psi : X_{\cJ_i} \times U \to \R^2$, where $U$ is 
either $(\R^{n+1})^{(m+1)}$ or $(\R^{n+1})^{(m)} \times S^n$.  
For such mappings we define an associated {\it partial multijet map} 
\begin{align}
\label{Eqn4.5}
_{\bl}j^k(\psi): &\, X_{\cJ_i}^{(\bl)} \times U 
\longrightarrow  \, {_{\bl }}E^{(k)}(X_{\cJ_i},\R^2) \notag \\
((x^{(j_1)}, \cdots ,  x^{(j_{m})}),  & u) 
\mapsto  ({_{\ell_1}}j_1^k(\psi(x^{(j_1)}, u), \dots ,  
{_{\ell_{m}}}j_1^k(\psi(x^{(j_{m})}, u))) 
\end{align}
where $_{\ell_p}j_1^k(\psi(x^{(j_p)}, u))$ denotes the multijet 
$_{\ell_p}j^k(f)$ for the function $f = \psi(\cdot, u)$ (for a fixed $u$) on 
$\cirX_{j_p}$.  
In the special cases of the multi-distance and height-distance functions 
we obtain partial multijet maps
\begin{equation}
\label{Eqn4.6a}  
_{\bl}j^k(\rho_i) : \, X_{\cJ_i}^{(\bl)} \times (\R^{n+1})^{(m+1)} 
\longrightarrow  \, {_{\bl}}E^{(k)}(X_{\cJ_i},\R^2)  
\end{equation}
and 
\begin{equation}
\label{Eqn4.6b}
  _{\bl}j^k(\tau): \, X_{\cJ_0}^{(\bl)} \times (\R^{n+1})^{(m)} \times S^n 
\longrightarrow  \, {_{\bl}}E^{(k)}(X_{\cJ_0},\R^2)\, .  
\end{equation}

\section{Generic Blum Linking Properties via Transversality Theorems}
\label{S:sec5} 
\par
In this section, we state a transversality theorem for the multi-distance 
and height-distance functions associated to a multi-region configuration.   
We then list the submanifolds of partial multijet spaces which capture 
the generic linking properties.  These are the submanifolds to which it 
will be applied.  \par
\subsection*{Transversality Theorem for Multi-Distance and 
Height-Distance Functions} 
%%%  \hfill 
\par
We begin with a variant of the transversality theorem of Looijenga for the 
distance functions for multi-region configurations defined by $\Phi : \bgD 
\to \R^{n+1}$. The first transversality theorem concerns submanifolds of 
the $s$-multijet space that are $ _{s}\cR^+$-invariant.  By this we mean 
that there is the induced action of the right equivalence group $_{s}\cR$, 
via the action of $\cR$ on $s$ multigerms, and we extend this by the 
group $\R$ acting by diagonal translations on $\R^2$, $(y_1, y_2) \mapsto 
(y_1 + c, y_2 + c)$ for $c \in \R$. \par
For later reference to the transversality theorems in \cite{D5}, we refer 
to the compact--open (or weak) $C^{\infty}$--topology as the {\em regular 
$C^{\infty}$--topology}.  Since the underlying space of $\bgD$ is compact, 
the Whitney and regular $C^{\infty}$ topologies on 
$C^{\infty}(\bgD, \R^{n+1})$ will agree. 
\begin{Thm}
\label{Thm5.0}
For any $i = 1, \dots , m$, let $W$ be a closed Whitney stratified subset of 
${_s}J^{k}(X_i^*,\R)$ that is $ {_s}\cR^+$-invariant.  
\flushpar
{\rm ($a_0$)} Let $Z \subset (X_i^*)^{(s)} \times \R^{n+1}$ be compact. 
Then the set
\begin{align*}
\cW \,\, &= \,\,  \{\Phi \in C^{\infty}(\bgD, \R^{n+1}) : \text{ both } 
{_s}j_1^k \gs_i \text{ and } {_s}j_1^k \gs_i | ((\gS_Q)^{(s)} \times 
\R^{n+1}) \,\,  \\ 
& \hspace{1in} \ol{\pitchfork}\,\, W \subset {_s}J^{k}(X_i^*,\R) \text{ on 
Z}\}  
\end{align*}
is an open dense subset for the regular $C^{\infty}$--topology.  
\flushpar 
{\rm ($b_0$)} The set
\begin{align*}
  \cW \,\, &= \,\, \{\Phi \in C^{\infty}(\bgD, \R^{n+1}) : \text{ both } 
{_s}j_1^k \gs_i \text{ and } {_s}j_1^k \gs_i | ((\gS_Q)^{(s)} \times 
\R^{n+1}) \,\,  \\ 
& \hspace{1in} \ol{\pitchfork}\,\, W \subset {_s}J^{k}(X_i^*,\R)  \text{ on 
} (X_i^*)^{(s)} \times \R^{n+1}\}  
\end{align*}
is a residual subset for the regular $C^{\infty}$--topology.
\end{Thm}
\flushpar
{\bf Note:} \lq\lq $\ol{\pitchfork}$\rq\rq\, denotes transversality of the 
mapping to the strata of the Whitney stratified set.
\par
Next, we give a multi-transversality theorem for the multi-distance and 
height-distance functions.  Although we are principally interested in them 
for the multi-distance function $\rho_0$ to capture the generic linking 
properties, the proof is valid for any region $\gW_i$ and yields a 
corresponding relation between a region $\gW_i$ and its adjoining 
regions.  For any smooth mapping $\Phi : \bgD \to 
\R^{n+1}$ we obtain for any $i$, any integer $m > 0$ with an assignment 
$p \mapsto j_p$, and a partition $\bl = (\ell_1,\dots,\ell_{m})$, the 
partial multijet mappings (\ref{Eqn4.6a}) and (\ref{Eqn4.6b}).  
The transversality theorem concerns these two 
mappings for a class of $ _{\bl}\cR^+$-invariant {\em distinguished 
submanifolds of the partial multijet space} 
$ _{\bl}E^{(k)}(X_{\cJ_i},\R^2)$ (see Definition \ref{Def5.1a}).  The 
$ _{\bl}\cR^+$-action is the induced action of the right equivalence group 
$ _{\bl}\cR$ on the fibers, extended by the diagonal action of $\R$ by 
translations on $\R^2$.
\begin{Thm}
\label{Thm5.1}
Let $\bgD$ be a model for a multi-region configuration.  Then, for any $i$, 
an $m > 0$ with assignment $p \mapsto j_p$, and partition $\bl 
=(\ell_1,\dots,\ell_{q_i})$, let $W$ be a closed Whitney stratified subset 
of $ _{\bl }E^{(k)}(X_{\cJ_i},\R^2)$ whose strata are 
$ _{\bl}\cR^+$--invariant distinguished submanifolds (in the sense of 
Definition \ref{Def5.1a}).  
\flushpar
{\rm ($a_1$)} Let $Z \subset X_{\cJ_i}^{(\bl)} \times (\R^{n+1})^{(m +1)}$ 
be compact. Then, the set
$$ \cW \,\,  = \,\,  \{\Phi \in C^{\infty}(\bgD, \R^{n+1}) : {_{\bl}}j_1^k 
\rho_i \; \ol{\pitchfork} \text{ on Z to } W \subset { _{\bl 
}}E^{(k)}(X_{\cJ_i},\R^2) \}  $$ 
is an open dense subset for the regular $C^{\infty}$--topology.
\flushpar
{\rm ($a_2$)} Let $Z \subset X_{\cJ_0}^{(\bl)} \times (\R^{n+1})^{(m)} 
\times S^n$ be compact.  Then the set
$$ \cW^{\prime} \,\,  = \,\,  \{\Phi \in C^{\infty}(\bgD, \R^{n+1}) : 
{_{\bl}}j_1^k \tau \; \ol{\pitchfork} \text{ on Z to }  W \subset
{ _{\bl }}E^{(k)}(X_{\cJ_0},\R^2)\}  $$ 
is an open dense subset for the regular $C^{\infty}$--topology. \par
\flushpar
{\rm ($b$)} Both of the sets
$$  \cW \,\, = \,\, \{\Phi \in C^{\infty}(\bgD, \R^{n+1}) : { _{\bl }}j_1^k 
\rho_i \; \ol{\pitchfork}\, W \subset { _{\bl }}E^{(k)}(X_{\cJ_i},\R^2) 
\text{ 
on }X_{\cJ_i}^{(\bl)} \times (\R^{n+1})^{(m+1)}\}  $$
and 
$$  \cW \,\, = \,\, \{\Phi \in C^{\infty}(\bgD, \R^{n+1}) : { _{\bl }}j_1^k \tau 
\; \ol{\pitchfork}\, W \subset { _{\bl }}E^{(k)}(X_{\cJ_0},\R^2) \text{ on } 
X_{\cJ_0}^{(\bl)} \times (\R^{n+1})^{(m)} \times S^n\}  $$
are residual subsets for the regular $C^{\infty}$--topology. 
\end{Thm}
\par
Note that in Theorems \ref{Thm5.0} and \ref{Thm5.1}, $(b_0)$ is a 
consequence of $(a_0)$ and $(b)$ is a consequence of $(a_1)$ and $(a_2)$, 
for we may cover $X_i^* \times \R^{n+1}$, $X_{\cJ_i}^{(\bl)} \times 
(\R^{n+1})^{(q_i)}$ and $X_{\cJ_0}^{(\bl)} \times (\R^{n+1})^{(q_0)} \times 
S^n$ with countably many compact sets $C_j$ so that the sets in $(b_0)$ 
and $(b)$ are countable intersections of open dense sets, hence residual.  
We will prove parts $(a_0)$, $(a_1)$ and $(a_2)$ in \S \ref{S:sec9} using a 
variant of the transversality theorems in \cite{D5}.  
\begin{Remark}
\label{Rem5.2}
These results give several types of extensions of the earlier 
transversality theorem due to Looijenga \cite{L} and its extension due to 
Wall \cite{Wa}.
\end{Remark}
\subsection*{Submanifolds Defining Generic Properties of Blum Linking 
Structures}
\label{submflds}
%%% \vspace{3ex}
\par
We next introduce the $ _{\bl}\cR^+$-invariant {\em distinguished 
submanifolds of the partial multijet space} 
%%%  the submanifolds of the partial multijet space 
to which we will apply Theorem \ref{Thm5.1}.  This will yield the 
transversality results implying the generic linking conditions for 
configurations in the Blum case in $\R^{n+1}$ for $n \leq 6$.  
\par
We first introduce the general definition of the distinguished 
submanifolds which involves several classes of $\cR$--invariant 
submanifolds of jet spaces and $_{\ell}\cR^+$--invariant submanifolds of 
multijet spaces.  These will include: both (an explicit list of) orbits of the 
$\cR$ and $_{\ell}\cR^+$ actions; their closures, which form Whitney 
stratified sets, together with the list of $_{\ell}\cR^+$--invariant closed 
Whitney stratified subsets of higher codimension $> n+1$ which will be 
generically avoided.  
\par
\subsubsection*{General Form of the Class of Distinguished  Submanifolds 
in $ _{\bl}E^{(k)}(X_{\cJ_i},\R^2)$} \hfill \par
\label{singsub00}
\par
For a given $i$ and $m > 0$ with an assignment $p \mapsto j_p$ we let 
$\bl=(\ell_1,\ldots,\ell_{m})$ for integers $\ell_p > 0$.  We give the form 
of the distinguished submanifolds in $ _{\bl}E^{(k)}(X_{\cJ_i},\R^2)$. 
By Lemma \ref{Lem4.4}, ${_{\bl}}E^{(k)}(X_{\cJ_i},\R^2)$ is a locally 
trivial fiber bundle.  For $S = (S_1, \dots , S_m)$ in $X_{\cJ_i}^{(\bl)}$ 
with $S_p \subset \cirX_{j_p}$ and $| S_p | = \ell_{p}$, the fiber at $S$ 
equals 
\begin{equation}
\label{Eqn5.1}
\prod_{p=1}^{m} {_{ \ell_p}}J^k(\cirX_{j_p},\R^2)_{S_p} \, \cong \,  
\prod_{p=1}^{m} \, (J^k(n, 2) \times \R^2 )^{\ell_p}.
\end{equation}

\par
We let $W_{0\, p}$ denote $ _{\ell_p}\cR^+$ 
multi--orbits in the fibers $ _{\ell_p}J^k(\cirX_{j_p},\R)_{S_p}$ for $p = 
1, \dots , m$ (so in particular, in such a $ _{\ell_p}\cR^+$ multi--orbit  
the targets lie in $\bgD^{\ell_p}\R$).  As in the multijet case, the 
multi-orbits $W_{0\, p}$ in the fibers over points in 
$\cirX_{j_p}^{(\ell_p)}$ fit together to form subbundles $W_p$ in 
${_{ \ell_p}}J^k(\cirX_{j_p},\R)$ for $p=1,\dots , m$. Then 
$$ W_1\times \cdots \times W_m \,\, \subset \,\, 
\prod_{p=1}^{m}\, {_{ \ell_p}}J^k(\cirX_{j_p},\R)  $$ 
restricts to a subbundle over $X_{\cJ_i}^{(\bl)}$.  
\par
For $p = 1,\dots , m$, let $W^{\prime}_p$ denote a subbundle of 
$J^k(\cirX_{j_p},\R)$ over $\cirX_{j_p}$ with fiber a $\cR^+$--invariant 
submanifold.  Typically it will denote either a $\cR^+$--orbit for a 
$k$-jet of a germ, or a stratum of a Whitney stratification consisting of a 
union of orbits.  We slightly abuse notation and define
\begin{equation}
\label{Eqn5.0a}
W^{\prime} \, \, = \,\, \prod_{p = 1}^{m}( W^{\prime}_p \times 
(J^k(\cirX_{j_p},\R))^{\ell_p-1})\times \gD^{m} \R ,
\end{equation}
where for each $p$, 
$(W^{\prime}_p) \times (J^k(\cirX_{j_p},\R))^{\ell_{p}-1})$ is restricted 
to lie over $\cirX_{j_p}^{(\ell_p)}$ and the entries of 
$\gD^{m} \R$ are in the first factor of each 
$(J^k(\cirX_{j_p},\R))^{\ell_{p}})$.  
\par
Therefore, $W^{\prime} \subset \prod_{p=1}^{m}\, _{ 
\ell_p}J^k(\cirX_{j_p},\R)$ restricts to a fiber bundle over 
$X_{\cJ_i}^{(\bl)}$, and 
\begin{equation}
\label{Eqn5.0b}
  W^{\prime} \times W_1 \times \cdots \times W_m 
\,\, \subset \,\, \prod_{p=1}^{m} \, {_{ \ell_p}}J^k(\cirX_{j_p},\R) 
\times \prod_{p=1}^{m} \, {_{ \ell_p}}J^k(\cirX_{j_p},\R) 
\end{equation}
with the RHS restricting to the bundle ${_{ \bl}}J^k(X_{\cJ_i},\R^2)$ over 
$(X_{\cJ_i})^{(\ell)}$.
\par
Then we define
\begin{Definition}
\label{Def5.1a}
Given $i$, $m > 0$ with an assignment $p \mapsto j_p$, and the ordered 
partition $(\ell_1, \dots , \ell_m)$ with $s = \sum_{p = 1}^{m} \ell_p$.  
Consider a collection of $_{\ell_p}\cR^+$-invariant submanifolds $W_p 
\subset {_{ \ell_p}}J^k(\cirX_{j_p},\R)$ and $\cR^+$-invariant 
submanifolds $W^{\prime}_p \subset J^k(\cirX_{j_p},\R)$ for $p=1,\dots , 
m$, which each have the property that their closures are Whitney 
stratified sets for which they are open strata. Then, we form the 
subbundle of ${_{ \bl}}J^k(X_{\cJ_i},\R^2)$ over $(X_{\cJ_i})^{\ell)}$ by the 
restriction of the LHS of (\ref{Eqn5.0b}).  Any such subbundle will be 
called a {\em distinguished submanifold} in 
$ _{\bl}E^{(k)}(X_{\cJ_i},\R^2)$.
\end{Definition}
\par
Having defined such distinguished submanifolds, we next give the four 
classes of jet and multijet strata formed from $\cR^+$ orbits, 
$_{\ell}\cR^+$ multi-orbits, and strata of $_{\ell}\cR^+$-invariant closed 
Whitney stratifications.  The submanifolds will be one of four types: 
\flushpar
\begin{itemize}
\item[(i)] those formed from simple multigerms; 
\item[(ii)] those formed from the partial multijet orbits from \S 
\ref{S:sec3}; 
\item[(iii)] those which characterize geometric features of boundary 
points for self--linking (these are a special subset of those in (ii)); and 
\item[(iv)] those which arise as strata of closed Whitney stratified sets 
of higher codimension. 
\end{itemize}
\par
Taken together for a given $i$ and $\ell$ we will denote the combined 
collection and the resulting distinguished submanifolds of partial 
multijet space by $\cS(i, \bl)$.  \par

\subsubsection*{Submanifolds for a Single Distance Function} \hfill \par
\label{singsub0}
For (i) we recall the simple multigerms of real-valued functions 
under $\cR^+$-equivalence on $\R^n$ (see \cite{M2}), formed from the $A$, 
$D$, $E$ classification of Arnold  \cite{A}.  As these are finitely 
determined, they are classified by $\cR^+$-multi-orbits in sufficiently 
high multijet space.  The boundary of the region of simple multigerms is 
the closure of the strata of orbits of the simple elliptic singularities 
given by: 
\begin{itemize} 
\item[$\tilde E_6$ :] \,  $x_1^3 + x_2^3 + x_3^3 +a x_1x_2x_3 + Q(x_4, 
\dots , x_n)$,
\item[$\tilde E_7$ :] \, $x_1^4 \pm x_2^4 +a x_1^2x_2^2 + Q(x_3, \dots , 
x_n)$, and 
\item[$\tilde E_8$ :] \, $x_1^3 + x_2^6 +a x_1x_2^4 + Q(x_3, \dots , x_n)$,  
\end{itemize}
where the $Q$ are nonsingular quadratic forms in the indicated variables.  
These germs are respectively $3$-determined, $4$-determined, and 
$6$-determined.  In each case, the set of $\cR^+$-orbits in the appropriate 
jet space (of $3$, $4$ or $6$-jets) forms a semi-algebraic set whose 
closure, denoted $\tilde W_k$, has a canonical Whitney stratification and 
the stratum formed from the $\cR^+$-orbits of the $\tilde E_k$ 
singularities has codimension $k$.  The unions of these orbits in the fibers 
of multijet space form submanifolds whose closures, also denoted 
$\tilde W_k$, have Whitney stratifications in the multijet bundle. \par  
Although $\tilde W_6$ has codimension $6$, its closure consists of jets 
of germs with the same third order terms and hence does not contain 
germs with local minima.  Second, $\tilde W_8$ has codimension $8$; 
hence its closure will consist of strata of codimension $ \geq 8$.  Third,
the codimension of the $\tilde E_7$ stratum $\tilde W_7$ is $7$, and 
where $a \neq 2$ the $\tilde E_7$ germ listed above is $4$--determined 
and if it has positive definite quadratic part and $| a | < 2$ then it has a 
local minimum.  
\par
Thus, for the closures of the $\tilde E_k$ strata, if $n +1 \leq 6$, then 
either the strata have codimension $> n + 1$, or the strata consist of jets 
of germs which are not local minima.  If $n + 1 = 7$, then the same holds 
true with the exception of the $\tilde E_7$-stratum with $a < 2$ which 
has codimension $7$ and the jets are of germs with local minima.  \par
Then, the Whitney stratification for a single function consists of the 
strata of the $\cR^+$--orbits of the simple $A$, $D$, $E$, singularities 
together with the closed Whitney stratified sets obtained as the closures 
of the $\tilde E_k$ strata.  These are $\cR^+$-invariant.  Of these strata, 
the only ones of codimension $\leq 7$ which define germs with local 
minima are the $A_k$, with $k$ odd, and the $\tilde E_7$ with $|a| < 2$, 
both with positive definite quadratic parts.  \par 
Then, the stratification for the multijets is obtained from multi-orbits 
formed by the products of the strata for single germs which are still 
$\cR^+$-invariant. \par 
It is possible to also allow $W^{(0)}$ consisting of the jets of 
non-singular germs with no distinguished image point in $\R$, so that 
$TW^{(0)}_0 = J^{k}(n, 1) \times \R$; however  later we shall find it 
technically simpler when we prove genericity for multi-functions to only 
consider multigerms exhibiting singularities at each point so we shall 
always do so.  \par
\begin{Remark}
\label{Rem5.1}
If $\gS_{n, 1^k}$ is the Thom-Boardman statum $\gS_{n, 1, \dots 1}$, with 
$k$ factors, in the jet space $J^{k+2}(X_i^*, \R)$ or $J^{k+2}(\cirX_i, \R)$, 
then the union of $A_{j+1}$-strata for $j \geq k$ (allowing different 
signs) and the union of $A_{k+1}$-strata is open and dense in $\gS_{n, 
1^k}$.  Then, transversality to the $\gS_{n, 1^k}$ for all $k$ is equivalent 
to the transversality to all of the $A_k$-strata.  
\end{Remark}
\par
Given $\bga = (\ga_1, \dots , \ga_q)$, we let $\bA_{\bga}$ denote the 
multigerm of type $A_{\ga_1}A_{\ga_2}\cdots A_{\ga_q}$, and we let 
$W^{(\bga)}$ denote the $\cR^+$-orbit of multigerms of type 
$\bA_{\bga}$ in the appropriate jet space, either ${_q}J^k(X_i^*, \R)$ or  
${_q}J^k(\cirX_i, \R)$.  These are bundles over $X_i^{*\, (q)}$ or 
$\cirX_i^{(q)}$ with fiber 
\begin{equation}
\label{Eqn5.0}
 W^{(\bga)}_0 \,\, = \,\, W^{(\ga_1)}_0 \times \cdots \times W^{(\ga_q)}_0 
\times \gD^q\R 
\end{equation}
where $W^{(\ga_j)}_0$ denotes the $\cR$-orbit for $A_{\ga_j}$.  \par 
 Then, it will follow from Theorem \ref{Thm5.0} that if $n +1 \leq 7$ then 
generically for each distance function $\gs_i$, the jet map $j_1^k(\gs_i)$ 
will miss the strata of the closures of the $\tilde W_k$ and those of the 
multi-orbits of type $\bA_{\bga}$ of codimension $\geq n +1$.  While it 
will intersect transversally the strata of codimension $\leq n +1$.  Of 
these only the strata of type $\bA_{\bga}$, or if $n+1 = 7$ the $\tilde 
E_7$-stratum, for which the multigerms all involve local minima can be 
multigerms at points of the Blum medial axis. This will yield Mather\rq s 
Classification Theorem \ref{Thm3.1}.  The details will be given in later 
sections.  
\par

\subsubsection*{Distinguished Class of Submanifolds Corresponding to 
Linking Type} \hfill  
\par
Next, for (ii) we define the submanifold of the partial multijet space 
associated to the linking configuration $(\bA_{ \bga} : \bA_{\bgb_1}, 
\dots, \bA_{\bgb_{m}})$ for the multi--distance function (see 
Definition \ref{Def3.4} in  \S \ref{S:sec3}).  Recall it requires that the 
distance function has singularity type $\bA_{ \bga}$ and at each point the 
multigerm for the distance function for the individual $i$-th region 
has type $\bA_{\bgb_i}$.  This involves taking products of the 
multi--orbits in the previous section and Boardman strata.
\par 
In the construction for Definition \ref{Def5.1a}, for $\bga = (\ga_1, \dots 
, \ga_m)$, we let in (\ref{Eqn5.0a}), $W^{\prime}_p = W^{(\ga_p)}_0$, 
which denotes the $\cR$-orbit for $A_{\ga_p}$, so that $W^{\prime} = 
W^{(\bga)}$ given above.  Likewise, for each $\bA_{\bgb_p}$, in 
(\ref{Eqn5.0b}), we let $W_p = W^{(\bgb_p)}$ in ${_{ 
\ell_p}}J^k(\cirX_{j_p},\R)$ for $p=1,\dots , m$.  Together they yield the 
from (\ref{Eqn5.0b}), the distinguished submanifold, denoted $W^{(\ga : 
\bgb )}$, for linking type $(\bA_{ \bga} : \bA_{\bgb_1}, \dots, 
\bA_{\bgb_{m}})$.  
\par
\subsubsection*{Multi-Distance Functions Capturing Geometric Properties 
of the Boundaries} \hfill
\par
For (iii), in using the preceding submanifolds of partial multijet space, 
we will be concerned with the generic interaction of strata for each 
distance function occurring for linking (especially self-linking). These 
depend on the differential geometry of the smooth points of the 
hypersurface $\Phi_j(X_j)=\cB_j$.  We use a Monge representation in 
preparation for analyzing these functions.  This involves the differential 
geometry of the boundary as already studied by Porteous \cite{Po}.  \par 
We let $(x_1, \dots , x_n)$ denote local orthogonal coordinates on 
$T_{y_0} \cB_j$ centered at 
$y_0 = \Phi_j(x_0)$ so that we may locally write the boundary $\cB_i$ in 
Monge form as $(x_1, \dots , x_n, f(x_1, \dots , x_n))$, and use $(x_1, 
\dots , x_n)$ as local coordinates for $X_i$.  Note that at a corner point 
$x_0$, for each smooth stratum containing $x_0$ in the closure, we may 
still obtain a Monge representation for that stratum near $y_0$.  If 
$\gk_1, \dots , \gk_n$ denote the principal curvatures of $\cB_i$ at the 
origin, then we may furthermore choose orthogonal coordinates in the 
principal directions so that 
\begin{equation}
\label{Eqn5.2}
  f(x_1, \dots , x_n) \,\, = \,\, \sum_{i=1}^n \frac{1}{2} \gk_i x_i^2 
+\sum_{|\bga| \geq 3} a_{\bga} x^{\bga}\, .  
\end{equation}
\par
We consider the case where the distance-squared function $\gs_i$ to a 
point $u_0 = (u_{0, 1}, \dots , u_{0, n+1})$ has a critical point so that 
$u_0$ lies on the normal line to the surface at $y_0$.  Thus, $u_{0, i} = 0$ 
for $ i < n+1$.  
Then, $\gs_i$ is given in the local coordinates by
\begin{align}
\label{Eqn5.3}
\gs_i(\cdot, u_0)\,\, &= \,\, \sum_{i=1}^n (x_i)^2 + (\sum_{i=1}^n 
\frac{1}{2} \kappa_i x_i^2 +\sum_{|\bga| \geq 3} a_{\bga} x^{\bga} -u_{0, 
n+1})^2  \\
&= \,\, \sum_{i=1}^n (1-u_{0, n+1} \gk_i ) x_i^2 -2u_{0, n+1}\sum_{|\bga| 
\geq 3} a_{\bga} x^{\bga}+(u_{0, n+1})^2 \notag \\
&\qquad \qquad + \, (\sum_{i=1}^n \frac{1}{2} \gk_i x_i^2 \,  + \, 
\sum_{|\bga| \geq 3} a_{\bga} x^{\bga} )^2. \notag
\end{align} 
\par
For $\gs_i(\cdot, u_0)$ to have a degenerate critical point at $y_0$ 
requires $u_{0, n+1} = \frac{1}{\gk_j}$ for some $j$.  If $\gk_j = 
\gk_{j^{\prime}}$ for some $j \neq j^{\prime}$, then $\gs_j(\cdot, u_0)$ 
has corank $\geq 2$.  The only generic singularities of this form which are 
local minima are the $\tilde E_7$ singularities when $n + 1 = 7$.  
Otherwise we may assume that the coordinates are chosen so that 
$\gk_1 > \gk_2 > \cdots > \gk_n$.  If $1 < j < n$, then $\gs_j(\cdot, 
u_0)$ again does not have a local minimum nor maximum.  Thus, for $y_0$ 
to be a crest point in the positive direction, $u_{0, n+1} = \frac{1}{\gk_1} 
> 0$.  There is a nondegeneracy condition on the coefficient of $x_1^3$ so 
it is an $A_3$ point.  Likewise, for it to be a crest point in the negative 
direction requires $u_{0, n+1} = \frac{1}{\gk_n} < 0$ with a nondegeneracy 
condition on the coefficient of $x_n^3$.  \par
In particular, for self-linking of type $(A_{m_1}:A_{m_2})$ to occur (with 
appropriate positive direction for $u_{0, n+1}$), $\gs_1(\cdot, u_0)$ will 
have an $A_{m_1}$ singularity and $\gs_n(\cdot, u_0^{\prime})$ will have 
an $A_{m_2}$ singularity.  In particular for $(A_{m_1}:A_{m_2})$ to occur 
generically as a transverse intersection requires $n+1 \geq m_1 + m_2 -
3$.  If $n \leq 6$, we obtain for generic self-linking types: $(A_3: A_3)$ 
for $(n \geq 2$): $(A_3: A_5)$ and $(A_5: A_3)$ for ($n \geq 4$); and 
($A_3: A_7$), ($A_7: A_3$), and $(A_5: A_5)$ for ($n = 6$).  \par
In the special case of $n + 1 = 7$, there is also the possibility of linking 
involving an $\tilde E_7$ point of either self-linking of the form $(\tilde 
E_7: A^2_1)$ or simple linking of the form $(A^2_1:\tilde E_7, A_1^2)$.
\par
\subsubsection*{Closed Stratified Sets of Higher Codimension}
\label{singsub3}
\par
For (iv), the final class of submanifolds of $_{\bl}E^{(k)}(X_{\cJ_i},\R^2)$ 
arise as strata of a finite list of closed Whitney stratified sets which 
include all linking configurations that will be shown to be 
non--generic.  Within this list, there are three types of strata: (1) 
submanifolds for simple multigerms; (2) those in the closure of a 
submanifold $W^{(\bga :\bgb)}$; and (3) those representing the degeneracy 
of the geometric conditions on the surface. 
\par
For (1), we have already listed the strata of the closures of the jets of 
simple elliptic singularities $\tilde W_k$ for $k = 6, 7, 8$.  We also 
include $\tilde W_9$, which denotes the closure of the $A_9$ stratum (i.e. 
the $\cR^+$-orbit in $k$-jet space for $k \geq 10$, denoted by $W_8^E$ in 
\cite{M2}).  Also, $\tilde W_6$ contains all $3$-jets with Hessian of 
corank $\geq 3$ and those defining local minima are in the singular strata 
(of codimension $\geq 7$).  These are equivalent to the submanifolds used 
by Mather in \cite{M2}.  
\par
Similarly, for sufficiently large $k$, we 
identify the $k$--multijet orbits of families of multigerms having 
$_{s}\cR_e^+-$codimension $\geq n+2$, which have the property that their 
closures are Whitney stratified sets invariant under $\cR^+$.   We consider 
for each $i$ a multigerm $\bA_{\eta}$, where $\eta = (\eta_1, \dots , 
\eta_s)$, formed from $A_{\eta_j}$, each of codimension $\leq n +1$.  The 
closure of the multi-orbit $W^{\eta}$ in $ _{s}J^k(\cirX_{i},\R)$ contains 
strata of codimension $> n + 1$.  Also, for the case of $n + 1 = 7$,  we also 
must allow the $\tilde E_7$ stratum in place of an $A_{\eta_j}$-orbit. All 
of the closures of these strata of codimension $> n + 1$ are the strata of 
closed Whitney stratified sets of multigerms of higher codimension.  \par
For (2), we consider for distinguished submanifolds of partial multijet 
spaces the closures of $W^{(\bga : \bgb)}$ in 
$\, {_\bl}E^{(k)}(X_{\cJ_i},\R^2)$.  These are represented by 
\begin{equation} 
\label{Eqn5.5}
\ol{W^{(\bga : \bgb)}} \, :=\, \ol{W^{\bga}} \times \ol{W^{\bgb_1}}\times 
\cdots \times \ol{W^{\bgb_q}}
\end{equation}
where $\ol{W^{\bga}}$ and the $\ol{W^{\bgb_p}}$ denote the closures of 
the corresponding multi-orbits.  
where the product on the RHS is restricted to $X_{\cJ_i}$.  This yields a 
finite number of Whitney stratified sets invariant under ${_{\ell}}\cR^+$.  
By the classification of simple multigerms and the closure of their 
complement,  this yields a finite list of closed Whitney stratified sets 
with strata of codimension $> n +1$ in ${_\bl}E^{(k)}(X_{\cJ_i},\R^2)$ that 
will be generically avoided.  
\par
For (3), we consider the submanifolds describing $(A_{r_1}:A_{r_2})$ with 
$n+1 < r_1 + r_2 -3$.  In the case of $n + 1 = 7$, we also consider 
submanifolds involving $\tilde E_7$ points, where either a point is of 
multigerm type $\tilde E_7$ and $A_{r}$ for any $r \geq 1$, or linking 
type $(\tilde E_7:A_{\bgb})$ with $| \bgb| > 2$ or $(A_{\bga}: \tilde E_7, 
A_{\bgb})$ with $| \bga|$ or $| \bgb| > 2$.  Then, we again consider the 
closure of these strata in the partial multijet space and take the strata of 
codimension $> n + 1$.  
 \par

\section{Generic Properties of Blum Linking Structures}
\label{S:sec6} 
\par
In this section, we apply Theorems \ref{Thm5.0} and \ref{Thm5.1} to the 
four classes of submanifolds and closed stratified sets defined in \S 
\ref{S:sec5} to obtain the generic properties of Blum structures for 
generic multi-region configurations.  We recall that for a given an $i$, and 
multi-index $\bl$ with $s =  \sum_{p=1}^{m} \ell_p$, $\cS(i, \bl)$ denotes 
the collection of closed $_{s}\cR^+$-invariant Whitney stratified sets 
constructed in \S \ref{S:sec5} in ${_s}J^k(X_i^*,\R)$ or for an $m > 0$ 
with assignment $p \mapsto j_p$, the distinguished submanifolds of
${_{\bl}}E^{(k)}(X_{\cJ_i},\R^2)$ for various $k$.  We consider the 
consequences of the transversality of the distance functions and both the 
multi-distance functions and the height-distance functions.  Although we 
state the results in terms of the induced stratifications on the union of 
boundaries of $X_i^*$ or $X_{\cJ_i}$, the results then apply to the 
corresponding boundaries of the multi-region configuration $\bgW$ 
defined from a generic model mapping $\Phi : \bgD \to \R^{n+1}$.  \par
\subsection*{Properties of Transversality and Whitney Stratified Sets} 
%%%  \hfill 
\par
We will make use of several simple lemmas regarding transversality in 
various settings, several basic properties of Whitney stratified sets, and 
a push-forward property for Whitney stratified sets.  
\par
\subsubsection*{Simple Properties of Transversality}  
\par
We begin with several lemmas concerning transversality which we state 
together here, and whose proofs we leave to the reader.  \par
The first two lemmas concern mappings involving product spaces.
\begin{Lemma}
\label{Lem6.10}
\begin{enumerate}
\item The smooth mapping $f = (f_1, \dots , f_m) : X \to \prod_{i}^{m} Y_i$ 
between smooth manifolds is transverse to $W = \prod_{i}^{m} W_i$, 
with each $W_i \subset Y_i$ a Whitney stratification, if and only if each 
$f_i$ is transverse to each $W_i$ and the $f_i^{-1}(W_i)$ intersect 
transversally in $X$.  Then, $f^{-1}(W) = 
\cap_{i = 1}^{m} f_i^{-1}(W_i)$.  \par
\item The smooth product mapping $f = \prod_{i} f_i : \prod_{i} X_i \to 
\prod_{i} Y_i$ between smooth manifolds is transverse to $W = \prod_{i} 
W_i$, with each $W_i \subset Y_i$ a Whitney stratification, if and only if 
each $f_i$ is transverse to each $W_i$.  
\end{enumerate}
\end{Lemma}
\hspace{11.5cm} $\Box$
\par
\begin{Lemma}
\label{Lem6.6}
Suppose $f : N \times M \to P$ is a smooth mapping between smooth 
manifolds which is transverse to a smooth submanifold $W \subset P$.  
Then the projection $\pi : N \times M \to N$ restricted to $f^{-1}(W)$is a 
local diffeomorphism onto its image at $(x, u)$, with $f(x, u) = y \in W$, if 
and only if $df(x, u) (T_x M) \cap T_y W = 0$.
\end{Lemma}  
\hspace{11.5cm} $\Box$
\par
Next we consider two situations when we can deduce transversality for 
push-forwards of submanifolds. 
\par

\begin{Lemma}
\label{Lem6.9}  Let $U$ and $Y_i$, $i = 1, 2$, be smooth manifolds.  Let  
$Z_1 \subset Y_1 \times U$ and $Z_2 \subset U \times Y_2$ be smooth 
submanifolds such that $Z_1 \times Y_2$ and $Y_1 \times Z_2$ are 
transverse in $X = Y_1 \times U \times Y_2$.  Let $\pi_1 : U \times Y_2 
\to U$ and $\pi_2 : Y_1 \times U \to U$ denote projection along $Y_1$ 
respectively $Y_2$, onto $U$.  Suppose each projection $\pi_2 | Z_1$ and 
$\pi_1 | Z_2$ is a submersion onto its image $W_1$, respectively $W_2$.  
Then, $W_1$ is transverse to $W_2$ in $U$. 
\end{Lemma}
\hspace{11.5cm} $\Box$
\par
A second lemma is the following.
\begin{Lemma}
\label{Lem6.13}
Suppose $\pi : X \to Y$ is a smooth fibration and that $\pi | Z_1 : Z_1 \to 
W_1$ is the restriction of the fibration to the smooth submanifold $W_1 
\subset Y$.  Let $Z_2 \subset X$ be a smooth submanifold which is 
transverse to $Z_1$.  Let $W_2 = \pi (Z_2)$ and suppose $\pi | Z_2 \to 
W_2$ is a submersion.  Then $W_1$ is transverse to $W_2$. 
\end{Lemma}
\hspace{11.5cm} $\Box$  
\par
\vspace{1ex}
\subsubsection*{Basic Properties of Whitney Stratified Sets}  
\par
 For this we recall several simple properties of Whitney stratifications, 
and refer the reader to Thom \cite{Th2}, Mather \cite{M1}, \cite{M3}, or 
see also \cite{GLDW}.  \par
\vspace{2ex}
\begin{enumerate}
\item {\em Pull-back by a Transverse Mapping:} Suppose $f : X \to Y$ is a 
smooth mapping transverse to a Whitney stratified set $Z \subset Y$, 
which has strata $\{ S_i\}$.  Then, $f^{-1}(Z)$ is a Whitney stratified set 
with strata $\{ f^{-1}(S_i)\}$.  
\item {\em Products of Whitney Stratified Sets:} If $Z_j \subset X_j$ is a 
Whitney stratified set with strata $\{ S_j^{(i)}\}$, for $j = 1, \dots , m$, 
then $\prod_{j = 1}^{m} Z_j \subset \prod_{j = 1}^{m} X_j$ is a Whitney 
stratified set with strata $\{ \prod_{j = 1}^{m} S_j^{(i_j)}\}$ for all 
possible choices $i_j$.
\item {\em Transverse Intersection: } Suppose $Z_j \subset X$ are 
Whitney stratified sets with strata $\{ S_j^{(i)}\}$, for $j = 1, \dots , m$.  
If the $Z_j$ are in general position (i.e. the strata $S_1^{(i_1)}, 
S_2^{(i_2)}, \dots , S_m^{(i_m)}$ are in general position for all possible 
$i_j$), then $\cap_{j = 1}^{m} Z_j$ is a Whitney stratified set with strata 
$\cap_{j = 1}^{m} S_j^{(i_j)}$ for all possible $i_j$.  
\end{enumerate}
For (1) see Thom \cite{Th2}, or for (1) and (2) see \cite{M1}, and then (3) 
follows from (1) and (2) applied 
to the diagonal map being transverse to the product stratification $\prod 
Z_j$.  \par
We remark that by applying the Lemmas to the strata of Whitney 
stratifications, we obtain analogues of the Lemmas for Whitney 
stratifications.  

\subsection*{Consequences of Transversality for (Multi-) Distance 
Functions} 
%%%  \hfill 
\par
We now combine the transversality results with the generic 
transversality properties of multi-distance functions to be followed by 
those resulting for height-distance functions. To keep notation and special 
cases when $n + 1 = 7$ from excessively complicating the discussion, we 
will concentrate on the properties for linking of types $(A_{\bga} : 
A_{\bgb_1}, \dots , A_{\bgb_m})$.  We can then modify the argument 
replacing one of the $A_{\bga}$ or $A_{\bgb_p}$ orbits by the $\tilde 
E_7$-stratum. 
\par 
We consider for $m > 0$ the assignment $p \mapsto j_p$, and a partition 
$\bl = (\ell_1, \dots \ell_{m})$ with $\ell_p$-tuple $x^{(j_p)} = 
(x^{(j_p)}_1, \dots , x^{(j_p)}_{\ell_p}) \in (\cirX_{j_p})^{(\ell_p)}$, for 
each $1 \leq p \leq m$, so that $(x^{(j_1)}_1, \dots , x^{(j_m)}_{1}) \in 
(\cirX_{i})^{(m)}$, and $(u^{j_1}, \dots , u^{j_m}, u^{(i)}) \in (\R^{n+1})^{(m 
+1)}$.  We suppose that at $\{ x^{(j_p)}_1, \dots , x^{(j_p)}_{\ell_p}\}$ the 
distance function $\gs( \cdot , u^{(j_p)})$ has a multigerm of type 
$A_{\bgb_p}$, and at $\{x^{(j_p)}_1, \dots , x^{(j_m)}_{1}\}$ in $\cirX_{i}$ 
the distance function $\gs( \cdot , u^{(i)})$ has a multigerm of type 
$A_{\bga}$.  \par
We begin with the distance function $\gs_i$ and the $\gs_{j_p}$.
First, for  $n \leq 6$, and each $j$ we let $\cP_{j\, \gs} \subset 
\emb(\bgD,\R^{n+1})$ 
consist of all $\Phi$ such that $_{s}j_1^k \gs_j$ is transverse to every 
element of $\cS(j, \bl)$.  The set $\cP_{j\, \gs}$ is a residual subset of 
$\emb(\bgD,\R^{n+1})$ by Theorem \ref{Thm5.0}.  Then, the intersection  
$\cP_{\gs} = \cap_{j} \cP_{j\, \gs}$ is again residual.  We apply Theorem 
\ref{Thm5.0} for $\Phi \in \cP_{\gs}$ to the $W^{(\bga)} \subset 
{_s}J^k(X_j^*, \R)$.  Suppose $S = (x^{(1)}, \dots , x^{(s)}) \times \{u_0\} 
\in \, {_s}j_1^k\gs_j^{-1}(W^{(\bga)})$.  Then, if $k$ is large enough, the 
proof of Mather\rq s Theorem \ref{Thm3.1}, as given in \cite[Thm 9.1, 
Thm 9.2]{M2}, implies that the distance function defines $\cR^+$-versal 
unfoldings 
$$\gs_j : X_j^*  \times \R^{n+1}, S \times \{u_j\} \to \R \, .$$
This will be true for each multigerm $\gs( \cdot , u^{(j_p)})$ for $j = j_p$ 
and $u_j = u^{(j_p)}$ for a partition $\ell = (\ell_1, \dots 
\ell_{m})$ as in \S \ref{S:sec4}, this holds for each $s = \ell_p$ for $p = 
1, \dots , m$.  
Note that Mather does not actually give the details of the proof of these 
theorems and refers to adapting another proof which he does not give.  
However, this result also follows in our situation from a more general 
result given in \cite[\S 5]{D5}, valid for more general geometric subgroups 
of $\cA$ and $\cK$.  \par
Hence, for $u_i \in \intr(\gW_i)$  the Blum medial axis for $\gW_i$, 
which is the Maxwell set for the versal unfolding, exhibits the generic 
local forms given by Mather\rq s Theorem \ref{Thm3.1}.  \par
This establishes that the properties i), and ii) of Theorem \ref{Thm3.6} 
and (1) for Theorem \ref{Thm3.5} hold for a residual set of embeddings of 
the configuration.  
\par
 Second, for $n \leq 6$, we let $\cP_{j\, \rho} \subset 
\emb(\bgD,\R^{n+1})$ consist of all $\Phi$ such that $_{\bl}j_1^k \rho_j$ 
is transverse to every element of $\cS(j, \bl)$.  The set $\cP_{j\, \rho}$ 
is a residual subset of $\emb(\bgD,\R^{n+1})$ by Theorem \ref{Thm5.1}.  
Then, the intersections  $\cP_{\rho} = \cap_{j} \cP_{j\, \rho}$ and 
$\cP_{\rho\, \gs} = \cP_{\rho} \cap \cP_{\gs}$  are again residual sets.  
We deduce the consequences of the transversality conditions.  
\par
For $\Phi \in \cP_{i\, \rho}$, transversality of ${_{\bl}}j_1^k \rho_i$ to 
$W^{(\bga : \bgb)}$ yields transversality statements for certain 
submanifolds in the product space $X_{\cJ_i}^{(\bl)} \times  
(\R^{n+1})^{(m)}$.  We first consider the consequences of  transversality 
to these submanifolds in $X_{\cJ_i}^{(\bl)} \times (\R^{n+1})^{(m)}$ and 
show how this implies transversality results for certain projections of 
the submanifolds, which are needed for the classification of generic 
linking.  
\par
We consider points
\begin{align*}
  (x^{(j_1)}, \dots , x^{(j_{m})}, u^{(i)}, u^{(j_1)}, \dots , u^{(j_{m})}) &\in 
X_{\cJ_i}^{(\bl)} \times (\R^{n+1})^{(m + 1)}, \\  
&\text{with } \quad x^{(j_p)} = (x^{(j_p)}_1, \dots , x^{(j_p)}_{\ell_p}) \in 
\cirX_j^{(\ell_p)} .
\end{align*}  
For each $p$ we have $x^{(j_p)}_1 \in X_{i\, 
j_p}$.  For each $p =1, \dots , m$, we choose disjoint open neighborhoods 
$U_{j_p}^{(q)} \subset \cirX_{j_p}$ of the $x^{(j_p)}_q$ and let $U_{j_p} = 
\prod_{q=1}^{\ell_p}U_{j_p}^{(q)}$.  Likewise, we choose disjoint open 
neighborhoods $V_{j_p} \subset \R^{n+1}$ of $u^{(j_p)}$ and $V_i$ of 
$u^{(i)}$.   Then, we let 
$$  Y \, = \, \prod_{p=1}^{m} U_{j_p} \times V_i \times \prod_{p=1}^{m} 
V_{j_p} $$ 
which is an open subset of $ X_{\cJ_i}^{(\bl)}  \times (\R^{n+1})^{(m + 1)}$.  
We also let $\tilde U = \prod_{p=1}^{m} U_{j_p}^{(1)} \subset 
\cirX_i^{(m)}$; and for each $p$, we let $Y_p = U_{j_p} \times V_{j_p}$ 
so that $Y \simeq (\prod_{p = 1}^{m} Y_p) \times V_i$.  
\par  
Next, we denote the natural projections 
$$  \pi_0 : Y \rightarrow \tilde U \times V_i  $$
and, for $p=1, \dots , m$, 
$$   \pi_p : Y \rightarrow Y_p . $$
We have the jet extension mapping for the distance function $\gs : X 
\times \R^{n+1}$ (which restricts to the $\gs_i$ on $X_i^* \times 
\R^{n+1}$),
\begin{equation}
\label{Eqn6.6}
   {_m}j_1^k \gs(\cdot, u^{(i)}) : \tilde U \times V_i \rightarrow 
 {_{m}}J^k(\cirX_i,\R)   
\end{equation}
and for each $p$, we have 
\begin{equation}
\label{Eqn6.7}
  {_{\ell_p}}j_1^k \gs(\cdot,u^{(j_p)}) : U_{j_p} \times V_{j_p} 
\rightarrow  {_{\ell_p}}J^k(\cirX_{j_p},\R) .  
\end{equation}
We let 
\begin{equation}
\label{Eqn6.10}
  Z^{(\bga)} = ({_m}j_1^k \gs(\cdot , u^{(i)}))^{-1}(W^{(\bga)}) \quad  
\text{and} \quad  Z^{(\bgb_{j_p})} = ({_{\ell_p}}j_1^k \gs(\cdot , 
u^{(j_p)}))^{-1}(W^{(\bgb_{j_p})})  
\end{equation}
 for each $p$.  \par
Then, the consequences for the jet extension mappings are given by the 
following.
\begin{Proposition}
\label{Prop6.3}
For $\Phi \in \cP_{\rho}$, we have the following transversality properties 
for 
${_{\ell_p}}j_1^k \gs(\cdot,u^{(j_p)})$ and ${_{\bl}}j_1^k\rho_i$: \par
\begin{enumerate}
\item ${_m}j_1^k \gs(\cdot,u^{(i)})$ is transverse to $W^{(\bga)}$; 
\item  each ${_{\ell_j}}j_1^k \gs(\cdot , u^{(j_p)})$ is transverse to 
$W^{(\bgb_{j_p})}$; 
\item the $\pi_p^{-1}(Z^{(\bgb_{j_p})})$, for $p  = 1, \dots , m$, are in 
general position; 
\item their intersection is transverse to $\pi_0^{-1}(Z^{(\bga)})$; and 
\item on $Y$
\begin{equation}
\label{Eqn6.11}
 ({_{\bl}}j_1^k\rho_i)^{-1}(W^{(\bga : \bgb)}) \,\, = \,\, \pi_0^{-
1}(Z^{(\bga)}) \cap (\cap_{p = 1}^{m}\pi_p^{-1}(Z^{(\bgb_{j_p})})) . 
\end{equation}
\end{enumerate}
\end{Proposition}
\par
\begin{proof}
To simplify the notation, we let $E_p$ denote the jet space 
${_{\ell_p}}J^k(\cirX_{j_p},\R)$, $h_0$ denote the jet-extension mapping 
${_s}j_1^k \gs(\cdot, u^{(i)})$, $h^{\prime}_0$ denote the jet-extension 
mapping $ {_m}j_1^k \gs(\cdot, u^{(i)})$ in (\ref{Eqn6.6}), and let $h_p: Y_p 
\to E_p$ denote each jet-extension mapping in (\ref{Eqn6.7}).  
Then, the jet extension mapping 
\begin{equation}
\label{Eqn6.8}
{_{\bl}}j_1^k\rho_i: X_{\cJ_i}^{(\bl)}  \longrightarrow \, 
{_{\bl}}E^{(k)}(X_{\cJ_i},\R^2)
\end{equation}
may be written 
\begin{equation}
\label{Eqn6.9} 
h = (h_0, \prod_{p = 1}^{m} h_p \circ \pi_p) : Y \to E^{(1)} 
\times E^{(2)}  
\end{equation}
where each $E^{(i)} = \prod_{p = 1}^{m} E_p$, for $i = 1, 2$. \par
By Theorem \ref{Thm5.1}, the mapping $ _{\bl}j_1^k\rho_i$ in 
(\ref{Eqn6.8}) is transverse to $W^{(\ga : \bgb)}$ in 
$ _{\bl}E^{(k)}(X_{\cJ_i},\R^2)$.  In the form of (\ref{Eqn6.9}), this states 
that $h$ is transverse to $W^{(\bga)} \times W^{(\bgb_{j_1})}\times 
\cdots \times W^{(\bgb_{j_{m}})}$, with $W^{(\bga)} \subset E^{(1)}$ and 
each $W^{(\bgb_{j_p})} \subset E_p$ in $E^{(2)}$.  By Lemma \ref{Lem6.10}, 
this is true if and only if $h_0$ is transverse to $W^{(\bga)}$ in $E^{(1)}$ 
and $h_p \circ \pi_p$ is transverse to $W^{(\bgb_{j_1})}\times 
\cdots \times W^{(\bgb_{j_{m}})}$ in $E^{(2)}$.  
%%% This later statement is equivalent to $h_p$ being transverse to 
%%$W^{(\bgb_{j_1})}\times \cdots \times W^{(\bgb_{j_{m}})}$ in $E^{(2)}$.  
Also, for $h_0$, we note that by (\ref{Eqn5.0a}) and (\ref{Eqn5.0}), 
$W^{(\bga)}$ has fibers over $X_{\cJ_i}^{(\bl)}$of the form 
\begin{equation}
\label{Eqn6.9a}
 W^{(\bga)}_0 \,\, = \,\,  \prod_{p = 1}^{m}( W^{(\ga_p)}_0 \times 
(J^k(\cirX_{j_p},\R))^{\ell_p-1}) \times  \gD^m\R
\end{equation}
Thus, by the form of (\ref{Eqn6.9a}), transversality of ${_s}j_1^k 
\gs(\cdot, u^{(i)})$ to $W^{(\bga)}$ is equivalent to the transversality of 
${_m}j_1^k \gs(\cdot, u^{(i)}) \circ \pi_0$ to $W^{(\bga)} \subset 
{_{m}}J^k(\cirX_i, \R)$ whose fibers are of the form
\begin{equation}
\label{Eqn6.9b}
 W^{(\bga)}_0 \,\, = \,\, W^{(\ga_1)}_0 \times \cdots \times W^{(\ga_q)}_0 
\times \gD^m\R 
\end{equation}
Moreover, 
$${_s}j_1^k \gs(\cdot, u^{(i)})^{-1}(W^{(\bga)}) \,\, = \,\, 
\pi_0^{-1}(Z^{(\bga)}) \,\, = \,\, 
\pi_0^{-1}({_m}j_1^k \gs(\cdot, u^{(i)})^{-1}(W^{(\bga)}))$$
(with the versions of $W^{(\bga)}$ on each side being in their appropriate 
spaces).
\par 
Then, we may apply
Lemma \ref{Lem6.10} and obtain : 
\begin{itemize} 
\item[i)] $h^{\prime}_0 \circ \pi_0$, and hence $h^{\prime}_0$, is 
transverse to 
$W^{(\bga)}$; 
\item[ii)] each $h_p \circ \pi_p$, and hence $h_p$, is transverse to 
$W^{(\bgb_{j_p})}$;  
\item[iii)] the $(h_p \circ \pi_p)^{-1}(W^{(\bgb_{j_p})})$ are in general 
position in $Y$; and 
\item[iv)] $(h^{\prime}_0 \circ \pi_0)^{-1}(W^{(\bga)})$ is transverse to 
$\cap_{p = 1}^{m} (h_p \circ \pi_p)^{-1}(W^{(\bgb_{j_p})})$.  
\end{itemize}
Then, the statements in the proposition are an immediate consequence of 
these properties resulting from the transversality conditions.  
\end{proof}
\par  
\vspace{1ex}
\subsubsection*{Stratification Properties on $\cirX_i$ and $\R^{n+1}$ and 
Generic Linking Properties } \hfill
\par
Now, we begin to apply the preceding properties to obtain the remaining 
conclusions in Theorems \ref{Thm3.5} and \ref{Thm3.6}.  \par
We next turn to the linking properties.  For these we have to establish: i) 
there is a stratification on $\cirX_i$ given by $\bA_{\bga}$-types of 
multigerms for $\rho_i$, two of which are Whitney stratifications; and ii) 
the stratifications on $\cirX_i$ by $\bA_{\bga}$-types and those given by 
the $\bA_{\bgb_{j_p}}$-types intersect transversely.  \par
Because the stratification by the $W^{(\bga)}$ or $W^{(\bgb_{j_p})}$ form 
Whitney stratifications, so does the product stratification by the 
$W^{(\ga : \bgb)}$.  Hence, by the transversality to Whitney 
stratifications, the 
$Z^{(\bga)}$ or $Z^{(\bgb_{j_p})}$ in (\ref{Eqn6.10}) form Whitney 
stratifications, with strata the inverse images of the strata of 
$W^{(\bga)}$ or $W^{(\bgb_{j_p})}$.  Moreover as the $\pi_j$ are 
submersions, by (3) above we deduce that 
$\cap_{p = 1}^{m}\pi_p^{-1}(Z^{(\bgb_{j_p})})$, which is the transverse 
intersection of Whitney 
stratified sets, is again Whitney stratified with strata the intersections 
of the strata of the $\pi_p^{-1}(Z^{(\bgb_{j_p})})$.  Then, by (4) and 
(\ref{Eqn6.11}) so is $(\, _{\bl}j_1^k\rho_i)^{-1}(W^{(\bga : \bgb)})$ 
Whitney stratified on $Y$ with strata the intersections of the strata of 
$\pi_0^{-1}(Z^{(\bga)})$ and 
$\cap_{p = 1}^{m}\pi_p^{-1}(Z^{(\bgb_{j_p})})$.  These stratifications are 
in $X_{\cJ_i}^{(\bl)} \times (\R^{n+1})^{(m)}$ rather than in $\cirX_i$ or 
in $\R^{n+1}$.  Next we relate these stratifications to those in $\cirX_i$ 
and $\R^{n+1}$.  
\par
First, by the openness of versality and the uniqueness of versal 
unfoldings, at 
any point of the Blum medial axes, $u_j \in \intr(\gW_j)$ for some $j$, 
the medial axis is diffeomorphic to a product along the stratum containing 
$u^{(j)}$.  Hence, the projections of any of the $Z^{(\bga)}$ or 
$Z^{(\bgb_{j_p})}$ map them diffeomorphically onto the corresponding 
strata in $\intr(\gW_i)$ or $\intr(\gW_{j_p})$ which form a Whitney 
stratification by multigerm type in the parameter space $\R^{n+1}$ of 
the versal unfolding.  \par 
Second, we have the following corresponding results for projection onto 
the factors $U_j$.
\begin{Proposition}
\label{Prop6.6}
For $\Phi \in \cP_{i\, \rho}$, the strata $Z^{(\bga)}$ and $Z^{(\bgb_{j_p})}$ 
project diffeomorphically onto their images $\gS^{(\bga)}$, resp. 
$\gS^{(\bgb_{j_p})}$, in each $U_{j_p}^{(i)}$.  Hence, they project 
diffeomorphically onto their images in each $\tilde U$ and $U_{j_p}$.  
These strata $\{ \gS^{(\bga)}\}$ or  $\{ \gS^{(\bgb_{j_p})}\}$ form  
stratifications.  Moreover, the subset of strata $\gS^{(\bga)}$ with 
$\ga_{j_p} \geq 3$, or the subset with all $\ga_{j_i} = 1$ form Whitney 
stratifications, and likewise for the subsets of strata 
$\{\gS^{(\bgb_{j_p})}\}$.  
\end{Proposition}  
\begin{proof}[Proof of Proposition \ref{Prop6.6}] 
We apply Lemma \ref{Lem6.6} to the jet extension mappings (\ref{Eqn6.6}) 
and (\ref{Eqn6.7}) with $W$ denoting either $W^{(\bga)}$ or 
$W^{(\bgb_{j_p})}$.  Then, for the first claim it is only necessary to show 
the conditions of Lemma \ref{Lem6.6} are satisfied.  This will follow 
directly from Lemma \ref{Lem10.7} which shows that the jet extension 
map (\ref{Eqn6.6}) viewed as defined on $U_{j_p}^{(i^{\prime})} \times 
(\prod_{p \neq i^{\prime}} U_{j_p}^{(1)} \times V_i)$, with the first factor 
representing $N$, satisfies the conditions of Lemma \ref{Lem6.6}.  
Likewise, as projection onto $\tilde U$ and $U_{j_p}$ composed with the 
further projection onto $U_{j_p}^{(i)}$ is a local diffeomorphism onto its 
image, hence it must be true for the projections onto $\tilde U$ and 
$U_{j_p}$.  
\par
To see that the subsets of strata form a Whitney stratification, we first 
apply Lemma \ref{Lem6.6} to the strata in the Thom-Boardman stratum 
$\gS_{n, 1}$. By Lemma \ref{Lem10.9}, for a germ of a generic distance 
function this stratum is mapped diffeomorphically onto its image in $X$.  
Hence, the strata in the image of $\gS_{n, 1}$ form a Whitney 
stratification.  Second, for the case where all $\ga_{j_p} = 1$, we let 
$x^{(j_p)}_1 \in \cB_{j_p}$ be a point in a stratum of type $A_1^k$.  If 
$x^{(j)}_2 \in \cB_{j_i}$ is another point in the same stratum associated 
to the same point in $M_j$, then they lie in a common $A_1^2$ stratum, 
where we do not consider absolute minimum values for distance.  By 
applying Lemma \ref{Lem6.6} with Lemma \ref{Lem10.7} again, we 
conclude that this stratum is the diffeomorphic image of a smooth 
stratum in the jet space.  In this stratum, the other higher strata $A_1^k$ 
form a Whitney stratification, and hence so does their image under the 
diffeomorphism.  A similar argument works for the case with all 
$\gb_{j_p\, i} = 1$.  
\end{proof}
\par
We now claim that these two stratifications are transverse.  
\begin{Proposition}
\label{Prop6.7}
For $\Phi \in \cP_{i\, \rho}$, the stratifications $\gS^{(\bga)}$ and 
$\gS^{(\bgb_{j_p})}$ for $p = 1, \dots , m$ intersect transversely in 
$\cirX_i$.  Hence,  the intersections of the subset of strata $\gS^{(\bga)}$ 
with $\ga_{j_p} \geq 3$ or the subset with all $\ga_{j_i} = 1$ with one of 
the subsets of strata $\{ \gS^{(\bgb_{j_p})}\}$ with $\gb_{j_p\, 1} \geq 
3$  or the subset with all $\gb_{j_p\, i} = 1$ form Whitney 
stratifications.  
\end{Proposition}
\par
\begin{proof}
We consider transversality to the stratification defined by the 
submanifolds $W^{(\bga, \bgb_{j_p})}$ (recall the definition after 
(\ref{Eqn5.0b})).  Then, 
$$   {_{m}}j_1^k \rho_i^{-1}(W^{(\bga, \bgb_{j_p})}) \,\, = \,\, 
\pi_0^{-1}(Z^{(\bga)}) \cap \pi_p^{-1}(Z^{(\bgb_{j_p})}) $$
is a transverse intersection by (1) of Lemma \ref{Lem6.10}.  Thus, the 
stratifications restricted to $Y$ are transverse.  
They have the following forms:  
\begin{align} 
\label{Eqn6.13}
\pi_0^{-1}(Z^{(\bga)}) \, &= \, Z^{(\bga)} \times \prod_{p=1}^{m} \big( 
\prod_{i=2}^{\ell_p} U_{j_p}^{(i)} \times V_{j_p} \big),\\
\pi_p^{-1}(Z^{(\bgb_{j_p})}) \, &= \, Z^{\gb_j} \times \prod_{p^{\prime} 
\neq p} \big( U_{j_{p^{\prime}}} \times V_{j_{p^{\prime}}}\big) \times 
V_i.  \notag 
\end{align}
\par
Let $\cU = \prod_{p^{\prime} \neq p} (\prod_{r = 
2}^{\ell_p}U_{j_{p^{\prime}}}^{(r)} \times V_{j_{p^{\prime}}})$ 
and 
$\tilde U_{j_p} = \prod_{i = 2}^{\ell_p} U_{j_p}^{(i)}$.  
By (\ref{Eqn6.13}), both $\pi_0^{-1}(Z^{(\bga)})$ and 
$\pi_p^{-1}(Z^{(\bgb_{j_p})})$ form fibrations with fiber $\cU$.  Thus, we 
may project along $\cU$ and obtain by Lemma \ref{Lem6.13} that the 
strata $Z^{(\bga)} \times \tilde U_{j_p} \times V_{j_p}$ and 
$Z^{(\bgb_{j_p})} \times (\prod_{p^{\prime} \neq 
p}U_{j_{p^{\prime}}}^{(1)}) \times V_i$ are transverse in $\tilde U \times 
\tilde U_{j_p} \times V_{j_p} \times V_i$.   \par
If we let $Y_1 = (\prod_{p^{\prime} \neq p}U_{j_{p^{\prime}}}^{(1)}) \times  
V_i$ and $Y_2 = \tilde U_{j_p} \times V_{j_p}$, then, 
$$  Y_1 \times U_{j_p}^{(1)} \times Y_2 \,\, = \,\, \tilde U \times \tilde 
U_{j_p} \times V_{j_p} \times V_i .$$
Hence, we may apply Lemma \ref{Lem6.9} to conclude that the images  
under the projection onto $U_{j_p}^{(1)}$, namely $\gS^{(\bga)}$ and 
$\gS^{(\bgb_{j_p})}$, are transverse in $U_{j_p}^{(1)}$.  \par
The proofs that the subsets of strata form Whitney stratifications follow 
from Proposition \ref{Prop6.6} together with their transverse 
intersections.
\end{proof}
\par
Applying the diffeomorphism $\Phi$ implies that the strata 
$\Sigma^{(\bga)}$ and $\Sigma^{(\bgb_{j_p})}$ intersect transversely in 
$\cB_i$. This holds for all $p=1, \dots , m$, and we conclude that the 
linking configuration $(\bA_{\bga} : \bA_{\bgb})$ is generic.  This proves  
that there is a residual set of mappings such that the generic linking 
property holds.  This yields property (2) for Theorem \ref{Thm3.5} and 
property v) for Theorem \ref{Thm3.6}, except for self-linking.  
\par
We recall that self-linking can occur in two forms: i) multiple points from 
a single region together with other points from another region occur in a 
linking configuration (partial linking); or ii) only points in a single region 
belong to the configuration (self-linking).  In either case, the 
stratifications in partial multijet space are still given for the 
multi-distance functions by the type (ii) submanifolds which we defined 
in \S \ref{S:sec5}.  It is in these cases where we need the assignment 
functions with possible multiple repetitions.  We can again conclude by 
Theorem \ref{Thm5.0} that for both partial linking and self-linking, there 
is a residual set of embeddings $\Phi$ so that the resulting jet maps for 
the distance function exhibit generic properties, and that transversality 
to strata holds for the individual distance functions.  
\par
Lastly, by the Transversality Theorem \ref{Thm5.0}, we may conclude that 
${_s}j_1^k\gs_i^{-1}(W^{(\bga)})$ is transverse to $(\gS_Q)^{(s)} \times 
\R^{n+1}$.  Then, by Proposition \ref{Prop6.6}, ${_s}j_1^k\gs_i^{-
1}(W^{(\ga)})$ projects diffeomorphically onto its image $\gS^{(\bga)}$.   
Hence, by Lemma \ref{Lem6.13} its transverse intersection with 
$(\gS_Q)^{(s)} \times \R^{n+1}$ maps diffeomorphically to the 
intersections of $\gS^{(\bga)}$ with
$\gS^{(\bga)} \subset X_i^*$.  Thus, $\gS_Q$ is transverse to 
$\gS^{(\bga)}$ in $X_i^*$, which establishes property iv) of Theorem 
\ref{Thm3.6} for a residual set of embeddings of the configuration.  

\subsection*{Consequences of Transversality for Height-Distance 
Functions}   %%%  \hfill
\par
We follow the same line of reasoning as we did for the multi-distance 
function case.  Now we deduce from the transversality results, the generic 
properties of height-distance functions. \par 
 For $n \leq 6$, we let $\cP_{i\, \tau} \subset \emb(\bgD,\R^{n+1})$ 
consist of all $\Phi$ such that $_{\bl}j_1^k \tau$ is transverse to every 
element of $\cS(i, \bl)$, $i = 0, \dots , m$.  The set $\cP_{i\, \tau}$ is 
a residual subset of $\emb(\bgD,\R^{n+1})$ by Theorem \ref{Thm5.1}.  
Then, the intersection $\cP_{\tau} = \cap_{i} \cP_{i\, \tau}$ is again a 
residual set.  We deduce the consequences of the transversality 
conditions.  
\par
For $\Phi \in \cP_{i\, \tau}$, transversality of $ _{\bl}j_1^k \tau$ to 
$W^{(\bga : \bgb)}$ yields transversality statements for certain 
submanifolds in the product space $X_{\cJ_0}^{(\bl)} \times 
(\R^{n+1})^{(m)} \times S^n$.  We again first consider the 
consequences of transversality to these submanifolds in 
$X_{\cJ_0}^{(\bl)} \times (\R^{n+1})^{(m)} \times S^n$ and then show how 
this implies transversality results for certain projections of the 
submanifolds. \par
First we deduce from the transversality to the submanifolds in 
$\cS(i, \bl)$ the analogue of Proposition \ref{Prop6.3}.  We use the same 
notation, except now $V_i \subset S^n$ is a neighborhood of $v \in S^n$, 
and let
\begin{equation}
\label{Eqn6.10a}
  \tilde Z^{(\bga)} = ( {_m}j_1^k \nu(\cdot , v))^{-1}(W^{(\bga)}) \quad  
\text{and} \quad Z^{(\bgb_{j_p})} = (\, _{\ell_j}j_1^k \gs(\cdot , 
u^{(j_p)}))^{-1}(W^{(\bgb_{j_p})})  
\end{equation}
 for each $p$.  \par

Then, the consequences for the jet extension mappings are given by the 
following.
\begin{Proposition}
\label{Prop6.3a}
For $\Phi \in \cP_{\tau}$, we have the following transversality 
properties for 
$\, _{\ell_p}j_1^k \nu(\cdot, v)$, with $v \in S^n$, and 
$\, _{\bl}j_1^k\tau$: \par
\begin{enumerate}
\item $\, {_m}j_1^k \nu(\cdot,v)$ is transverse to $W^{(\bga)}$; 
\item  each $\, _{\ell_j}j_1^k \gs(\cdot , u^{(j_p)})$ is transverse to 
$W^{(\bgb_{j_p})}$; 
\item the $\pi_p^{-1}(Z^{(\bgb_{j_p})})$, for $p  = 1, \dots , m$, are in 
general position; 
\item their intersection is transverse to $\pi_0^{-1}(\tilde Z^{(\bga)})$; 
and 
\item on $Y$
\begin{equation}
\label{Eqn6.11a}
 (\, _{\bl}j_1^k\tau)^{-1}(W^{(\bga : \bgb)}) \,\, = \,\, \pi_0^{-1}(\tilde 
Z^{(\bga)}) \cap (\cap_{p = 1}^{m}\pi_p^{-1}(Z^{(\bgb_{j_p})})) . 
\end{equation}
\end{enumerate}
\end{Proposition}
Then we may apply the same reasoning as in the proof of Mather\rq s 
Theorem \ref{Thm3.1}.  Suppose $S = (S_{j_1}, \dots , S_{j_{m}}) \in \, 
_{\bl}j_1^k\tau^{-1}(W^{(\ga : \bgb)})$, with each $S_{j_p} = \{ x^{(j_p)}_1, 
\dots , x^{(j_p)}_{\ell_p}\}$ and each $x^{(j_p)}_1 \in \cirX_i$.  We let 
$S^{\prime} = \{x^{(j_p)}_1, \dots x^{(j_{m})}_1\}$.  \par
As earlier, (2) implies the distance functions $\gs(x, u^{(j_p)}) : 
\cirX_{j_p} \times \R^{n+1}, S_{j_p} \times \{u^{(j_p)}\} \to \R$, $p = 1, 
\dots , m$, define $\cR^+$-versal unfoldings.  If $u^{(j_p)} \in 
\intr(\gW_{j_p})$, the Blum medial axes for $\gW_i$ and the $\gW_{j_p}$, 
which are the Maxwell sets for the versal unfoldings, exhibit the generic 
local forms given by Mather\rq s Theorem \ref{Thm3.1}.  
\par
Then, if $k$ is large enough, (1) implies that the height function $\nu(x, v) 
: X_{\cJ_0} \times S^n, S^{\prime} \times \{v\} \to \R$ defines an 
$\cR^+$-versal unfolding.  The spherical axis $\cZ$ is the Maxwell set for 
$\nu$, and hence has the same local generic structure.  In particular, it 
consists of $v \in S^n$ such that $\nu(\cdot , v)$ has an absolute maximum 
at $S^{\prime}$.  Thus, $S^{\prime} \subset \cB_{\infty}$, and the 
structure of the set of such $v \in S^n$ exhibits the same generic local 
forms as the Blum medial axes, except with one lower dimension.  \par
Second, to determine the structure of $\cB_{\infty}$ and then 
$M_{\infty}$ we need the analogue of Proposition \ref{Prop6.6}. 
\par
\begin{Proposition}
\label{Prop6.6a}
For $\Phi \in \cP_{i\, \tau}$, the strata $\tilde Z^{(\bga)}$ and 
$Z^{(\bgb_{j_p})}$ project diffeomorphically onto their images in each 
$U_{j_p}^{(i)}$.  Hence, they also  project diffeomorphically onto their 
images in $\tilde U$ and $U_{j_p}$.  Moreover, the subset of strata 
$\gS^{(\bga)}$ with $\ga_{j_p} \geq 3$ or the subset with all $\ga_{j_i} = 
1$ form Whitney stratifications.  
\end{Proposition}  
\begin{proof}[Proof of Proposition \ref{Prop6.6a}] 
Proposition \ref{Prop6.6} already applies to the $Z^{(\bgb_{j_p})}$.  
For $\tilde Z^{(\bga)}$, we apply Lemma \ref{Lem10.8} and then Lemma 
\ref{Lem6.6} to the jet extension mapping $\, _sj_1^k \nu(\cdot,v)$.  \par
The argument that the subsets of strata form Whitney stratifications 
follow the same arguments used in the proof of Proposition \ref{Prop6.6}. 
\end{proof}
\par  
We let the projection of the stratum $\tilde Z^{(\bga)}$ onto $U_{j_p}$ be 
denoted by $\gS^{(\bga)}_{\infty}$, and that of $Z^{(\bgb_{j_p})}$ by 
$\gS^{(\bgb_{j_p})}$.   Third, we may apply the same proof as for 
Proposition \ref{Prop6.7} to conclude that the two stratifications are 
transverse.  
\begin{Proposition}
\label{Prop6.7a}
For $\Phi \in \cP_{i, \tau}$, the stratifications $\gS^{(\bga)}_{\infty}$ 
and $\gS^{(\bgb_{j_p})}$ for $p = 1, \dots , m$
intersect transversely in $\cirX_0$.  Hence,  the intersections of the 
subset of strata $\gS^{(\bga)}$ with $\ga_{j_p} \geq 3$ or the subset with 
all $\ga_{j_i} = 1$ with one of the subsets 
of strata $\{ \gS^{(\bgb_{j_p})}\}$ with $\gb_{j_p\, 1} \geq 3$  or 
the subset with all $\gb_{j_p\, i} = 1$ form Whitney stratifications.  
\end{Proposition}
\begin{proof}
The proof follows by an argument analogous to that for Proposition 
\ref{Prop6.7}.   The strata $\gS^{(\bga)}_{\infty}$ transversely intersect 
the strata $\gS^{(\bgb_{j_p})}$, which correspond to the 
strata of the Blum medial axis of $\gW_{j_p}$.  That the intersections of 
subsets of the strata form Whitney stratifications follows from 
Proposition \ref{Prop6.6a} and their transverse intersection.  
\end{proof}
\par 
\subsection*{Generic Properties of $\cB_{\infty}$ and $M_{\infty}$} 
\par
We now use the preceding to establish for the residual set $\cP_{\tau}$ 
the properties of $\cB_{\infty}$ and $M_{\infty}$ given by vi) of Theorem 
\ref{Thm3.6}.  We recall that $\cB_{\infty}$ consists of $x \in \cB$ such 
that a height function has an absolute maximum (or minimum for the 
height function for the opposite direction) and is stratified by the 
$\gS_{\infty}^{(\bga)}$, which correspond to the strata 
$\gS_{\cZ}^{(\bga)}$, and the $\gS_{\cB}^{(\bga)}$.  These strata intersect 
transversally and $M_{i\, \infty}$ consists of the strata in $\tilde M_i$ 
corresponding to the intersections of these strata under the 
correspondence of the strata $\gS_{M}^{(\bga)}$ of $M_i$ with the 
$\gS_{\cB}^{(\bga)}$.   Then, the generic structure of $\cB_{\infty}$ and 
$M_{\infty}$ is given by the following proposition.  
\begin{Proposition}
\label{Prop6.8}
For the residual set $\cP_{\tau}$, the global radial flow $\psi : N_{+} | 
\cB_{\infty}$ is a global diffeomorphism with image the complement of 
the set of points lying in the image of the linking flow.  Moreover, the 
interior points of $\cB_{\infty}$ consist of points where a height function 
has a unique nondegenerate maximum; and the boundary of $\cB_{\infty}$ 
consists of points $x \in S$ such that a height function $h: \cB, S \to \R, 
y$ is a multigerm of type $A_{\bga}$, with either $| \bga | > 1$ or $\ga_1 
\geq 3$ and odd.
\end{Proposition}
\par 
\begin{proof}
First, we prove that $\cB_{\infty}$ lies in the smooth strata of $\cB$, 
which we view as a piecewise smooth manifold.  At any singular point $x$ 
of the boundary of a region $\gW_i$, there is another region $\gW_j$ for 
which  $x$ is also a singular point of its boundary.  Then, $T_x X_{i\, j}$ 
will contain points of $\gW_i$ or $\gW_j$ on each open half-space 
determined by it.  Hence the height function cannot have an absolute 
minimum at $x$.  Thus, $x \notin \cB_{\infty}$.  
\par
Second, by the genericity properties of height functions, if a height 
function has an absolute minimum at a point $x \in \cirX_0$, but it is not 
a nondegenerate minimum, then it is of type $A_k$, $k \geq 3$ and odd.  By 
the normal form for the versal unfolding of such germs, there are 
codimension one strata in the Maxwell set for $A_k$ of type $A_1^2$.  As 
we cross such a stratum tranversely the absolute minimum moves from 
the neighborhood of one $A_1$ point to the other.  Hence, the 
corresponding curve crossing the one $A_1$ stratum in a neighborhood 
will move from $\cB_{\infty}$ to points not in $\cB_{\infty}$.  Thus, 
these strata are in the boundary of $\cB$, and hence so is the $A_k$ point.  
If instead, we have $S \subset \cB$ so a height function $h: \cB, S \to \R, 
y$ is a multigerm of type $A_{\bga}$, with $| \bga | > 1$, then for $x_1 
\in S$ with $\ga_1 >0$, we can again by versality find strata of type 
$A_1^2$ which contain the tuple $(x_1, \dots , x_r)$.  By an analogous 
argument as above, $x_1$ is contained in the boundary of $\cB_{\infty}$.  
\par 
Third, we consider a point $x_0 \in \cB$ where a height function has a 
unique nondegenerate absolute maximum.  We claim that $x_0$ is an 
interior point of $\cB_{\infty}$.  Let $\bn$ be the outward pointing unit 
normal vector field to $\cB$ in a neighborhood of $x$, then the height 
function is given by $h_0(x) = x\cdot \bn(x_0)$.  We claim that there is a 
neighborhood $W$ of $x_0$ in $\cB$ so that for all $x^{\prime} \in W$, the 
height function has a unique nondegenerate absolute maximum at 
$x^{\prime}$.  \par 
To see this, we note that by the $C^{\infty}$-stability of Morse 
singularities, there is a neighborhood $W$ of $x_0$ with compact closure 
such that $h_0 | Cl(W)$ is $C^{\infty}$ stable with only a single singular 
point at $x_0$.  Thus, for $x^{\prime}$ in a sufficiently small 
neighborhood $W^{\prime} \subset W$, for the corresponding height 
function $h_{x^{\prime}}(x) = x\cdot \bn(x^{\prime})$, $h_{x^{\prime}} | 
Cl(W)$ is $C^{\infty}$-equivalent to $h_0$, and thus has a nondegenerate 
maximum at $x^{\prime}$.  We further claim that for $x^{\prime}$ in a 
smaller neighborhood $W^{\prime\prime} \subset W^{\prime}$, 
$h_{x^{\prime}}$ has an absolute maximum at $x^{\prime}$.  If not, then 
for a decreasing sequence of neighborhoods $W_i \subset W^{\prime}$ 
whose intersection is $x_0$, there are $x_i \in W_i$ and $y_i \in \cB$ so 
that $y_i\cdot \bn(x_i) \geq x_i\cdot \bn(x_i)$.  This implies that $y_i 
\notin Cl(W)$.  By compactness we may take a subsequence and assume 
$\lim y_i = y_0 \in \cB \backslash W$.  Then taking limits in the 
inequalities, we obtain $\lim x_i = x_0$ and $\lim \bn(x_i) = \bn(x_0)$ so 
$y_0\cdot \bn(x_0) \geq x\cdot \bn(x_0)$ with $y_0 \neq x_0$, which 
contradicts our assumption about $x_0$.  Thus, $x_0$ is an interior point 
of $\cB_{\infty}$; and by our description of the boundary, the interior of 
$\cB_{\infty}$ consists of such points.  It then follows that 
$\cB_{\infty}$ has a boundary in $\cirX_0$ which has a local form given 
by the $\gS^{(\bga)}_{\infty}$ for the $\cR^+$-versal unfolding of 
multigerms of the height function.  \par
Fourth, we claim that the global radial flow $\psi: N_{+} | \cB_{\infty}$ is 
a diffeomorphism onto its image.  There are two steps: showing that the 
global flow is everywhere nonsingular, and showing that the flow is 
$1-1$.  \par
For nonsingularity, we choose a Monge patch $W$ about $x_0 \in \cB_{i, 
\infty}$, which corresponds to $0$, so that the height function is the 
$x_{n+1}$ coordinate which has an absolute maximum at $x$ and $x_{n+1} 
= f(x_1, \dots x_n)$ locally defines $\cB$ on $W$.  Hence, the Hessian of 
$f$ has nonpositive eigenvalues, which are the principal curvatures of 
$\cB$ at $x_0$.  The global radial flow will be nonsingular out to $\cB_i$, 
and from there, given $t_0 > 1$ we can view it as a \lq\lq new radial 
flow\rq\rq\, defined on $W$ with \lq\lq radial vector field\rq\rq\, $t_0 r 
\bu_i$. If this flow is nonsingular for $t \leq 1$, then the original global 
radial flow is nonsingular for $t \leq t_0+ 1$.  Hence, if this holds for 
arbitrary $t_0$, then it is nonsingular for all $t$.  Lastly, for this vector 
field, the principal radial curvatures are the usual principal curvatures.  
Thus, as all of the principal curvatures are nonpositive, the radial 
curvature condition is vacuously satisfied (see Proposition 
\ref{PropII.2.1}).  \par 
To see that the flow is also $1-1$ on $\cB_{\infty}$, we suppose not. 
Let $x, x^{\prime} \in \cB_{\infty}$ be distinct points such that the 
positive normal half-lines which point out from $\cB$ at these points 
intersect at some point $u$.  Suppose $\| u - x\| \geq \| u - x^{\prime}\|$.  
We choose 
coordinates so $x = 0$ and let $\bv$ be the outward pointing unit vector to 
$\cB$ at $x$ with $L$ the line spanned by $\bv$.  Then, $x^{\prime} \notin 
L$ or our above assumption implies $\bv \cdot x^{\prime} > 0$, a 
contradiction for $x$.  Hence, if $x^{\prime\prime}$ is the orthogonal 
projection of $x^{\prime}$ onto $L$, then by the triangle inequality 
$\|u - x^{\prime\prime}\| < \|u - x^{\prime}\| \leq \| u\|$.  Thus, 
$x^{\prime\prime} = c\bv$ for some $c > 0$, and $x^{\prime\prime}\cdot 
\bv = c > 0$, a contradiction.  Thus, the half-lines do not intersect and the 
global flow on $\cB_{\infty}$ is $1-1$, and hence a global diffeomorphism 
onto its image. \par
There is one final part of property iv) to show, which will follow from the 
next lemma.
\begin{Lemma}
\label{Lem6.9b}
If $x \in \cB \backslash \cB_{\infty}$, then there is a point in $\cB$ 
(which may be $x$ itself through self-linking) to which $x$ is linked. 
Also, any point not in the image of the global radial flow on $ N_{+} | 
\cB_{\infty}$ is in the image of the linking flow. 
\end{Lemma}
\par 
This Lemma establishes property vi) of Theorem \ref{Thm3.6} for the 
residual set $\cP_{\tau}$, completing the proof.
\end{proof}
\par 
It remains to prove the Lemma.  For this proof, given a hyperplane $H$ in 
$\R^{n+1}$ with a distinguished point $x_0 \in H$ and a normal vector 
$\bn$ to $H$, we will use the notion of a {\it family of spheres} $\{ S_a : a 
> 0\}$ of radii $a$, lying on one side of $H$, and tangent to $H$ at $x_0$. 
We can by a change of coordinates assume $x_0 = 0$, $H$ is the coordinate 
hyperplane defined by $x_{n+1} = 0$ and $\bn = \be_{n+1}$, the unit vector 
in the positive $x_{n+1}$-direction.  Then, in this model situation, $S_a$ 
is defined by $\sum_{i = 1}^{n+1} x_i^2 - 2 a x_{n+1} = 0$.  We can easily 
check that if $a \neq a^{\prime} > 0$ then $S_a \cap S_{a^{\prime}} = \{ 
0\}$, and given any $x = (x_1, \dots , x_{n+1}) \in \R^{n+1}$ with $x_{n+1} 
> 0$, there is a unique $S_a$ containing $x$ (and we can write $a$ as a 
smooth function of $x$ on the open half-space).  We make use of such 
families of spheres to prove the Lemma. 
\begin{proof}[Proof of Lemma \ref{Lem6.9b}]
For the first statement, let $x_0 \in \cB \backslash \cB_{\infty}$ and 
$\bn$ denote the unit outward pointing normal vector to $\cB$ at $x_0$.  
Let $\gk_i$ denote the  principal curvatures of $\cB$ at $x_0$, and let
$r_{min} = \min\{ \frac{1}{\gk_i} : \gk_i > 0\}$ or $\infty$ if all $\gk_i 
\leq 0$.  \par
Since  $x_0 \notin \cB_{\infty}$, the height function in the direction 
$\bn$ does not have an absolute maximum at $x_0$; thus, there are points 
of $\cB$ in the half-space defined by $T_{x_0}\cB$ with $\bn$ pointing 
into the half-space.  We consider the family of $S_a$ for this data.  
If $0 < r_0 < r_{min}$, then there is a neighborhood $U$ of $x_0$ such that  
for $0 < a \leq r_0$,  $S_a \cap U = \{x_0\}$.  Then, if $S_a \cap (\cB 
\backslash U) = \emptyset$ for $0 < a \leq r_0$, then  $S_a \cap \cB = 
\{x_0\}$  for $0 < a \leq r_0$.  If this holds for all $0 < r_0 < r_{min}$, 
then $x_0$ is an $A_3$-point of the distance function to the focal point 
$y = x_0 + r_{min}\bn$, which is an edge point of the medial axis.  
\par 
Otherwise, there is a smallest $a$ such that $S_a \cap \cB \backslash 
\{x_0\} \neq \emptyset$.  Hence, $S_a$ is tangent to $\cB$ for any such 
intersection point.  Thus, $x_0$ is linked to such an intersection point.  In 
either case, $x_0$ is linked to a point in $\cB$.  \par
For the second statement, let $y \in \gW_0 \backslash \psi(N_{+} | 
\cB_{\infty})$.  Then, by compactness, there is a point $x_0 \in \cB$ of 
minimum distance to $y$.  If there are more than one such point, then $y 
\in M_0$ and is in the image of the linking flow.  If $x_0$ is the unique 
minimum point, then we consider the family of spheres for $x_0$, 
$T_{x_0}\cB$, and $\bn$ the outward pointing unit normal vector to $\cB$ 
at $x_0$.  If $a = \| y - x_0\|$, then by assumption $S_a \cap \cB = \{ 
x_0\}$, thus $\| y - x_0\| \leq r_{min}$.  \par 
As $y \notin \psi(N_{+} | \cB_{\infty})$, the height function in the 
direction of $\bn$ cannot have an absolute maximum at $x_0$, so there 
are points of $\cB$ in the open half-space defined by $T_{x_0}\cB$ and 
$\bn$.  Hence, by the above argument, there is a smallest $a$ with 
$\| y - x_0\|  \leq  a \leq r_{min}$ satisfying one of the following two 
possibilities.  If $a = r_{min}$ is the minimum positive radius of 
curvature of $\cB$ at $x_0$, then $y^{\prime} = x_0 + r_{min}\bn$ is a 
focal point for $r_{min}$ and thus $y = x_0 + \| y - x_0\|\bn$ lies in the 
image of the linking flow. Otherwise, $a < r_{min}$ and there is another 
$x^{\prime} \in S_a \cap \cB$.  Then, $x_0$ and $x^{\prime}$ are linked via 
$x_0 + a\bn$, and again  $y = x_0 + \| y - x_0\|\bn$ lies in the image of the 
linking flow.  In either case we obtain the second statement of the Lemma.
\end{proof}

\subsection*{Proof of Theorem \ref{Thm3.5} for a Residual Set of 
Embeddings: } \hfill
\par
By the preceding results we have now established properties i), ii) , iv), 
and vi) for Theorem \ref{Thm3.6}.  The remaining properties iii) and v) 
concerning the edge corner points will be established in the next section.  
However, we are now able to conclude the proof of Theorem \ref{Thm3.5} 
for the residual set of mappings in $\cP = \cP_{\rho\, \gs} \cap 
\cP_{\tau}$.  We summarize  the consequences we have obtained by 
applying Theorem \ref{Thm5.1} for $n \leq 6$ to the elements of the 
$\cS(i, \bl)$.   In terms of the configuration $\bgW$ consisting of disjoint 
regions with smooth boundaries, these results yield the following.  
\begin{enumerate}
\item Transversality to the submanifolds representing orbits of simple 
multigerms implies that every region $\Omega_i \subset \R^{n+1}$ for 
$i=1,\dots , r$  has a Blum medial axis exhibiting only the generic local 
normal forms given in Theorem \ref{Thm3.1}.  This also holds for the 
linking medial axis in the complement $\gW_0$.
\item By Propositions \ref{Prop6.6} and \ref{Prop6.7} there are 
stratifications of the smooth regions of the region boundaries $\cB_i$ by 
the multigerm types $\gS^{(\bga)}$ and $\gS^{(\bgb_j)}$, which intersect 
transversally; moreover they are Whitney stratified sets for two of the 
three types described in \S \ref{S:sec3}.  Hence, we obtain the generic 
linking structure (including self-linking) on the smooth points of the 
boundaries $\cB_i$.  
\item  In the case of disjoint regions with smooth boundaries, these give 
properties (1) and (2) of Theorem \ref{Thm3.5}. 
\item  By Propositions \ref{Prop6.6a} and \ref{Prop6.7a}, $\cB_{\infty}$ 
has a stratification formed from height function 
multi--germs of the types described in \S \ref{S:sec3}, which intersects 
transversally the strata $\gS^{(\bgb_j)}$ of the multigerm types for the 
distance functions, and forms a stratified set in $\cB_{j}$.  
\item  Moreover, by Proposition \ref{Prop6.8}, we have also established 
the properties of $\cB_{\infty}$ and $M_{\infty}$.  This gives property (3) 
of Theorem \ref{Thm3.5}.  
\end{enumerate}
\par
\begin{Remark}
\label{Rem6.10}
In the bounded case where the configuration lies in a region $\tilde \gW$ 
which is a manifold with boundaries and corners, we require that the 
radial lines from $M$ which meet $\partial \tilde \gW$ do so 
transversely, by which we mean that at singular points of $\partial \tilde 
\gW$, the radial line is transverse to all limiting tangent planes at that 
point.  By scaling $\tilde \gW$ we obtain a parametrized family, which by 
the parametrized transversality theorem will be transverse to $M_0$ for 
almost all parametrized values.  Alternatively this is equivalent to 
scaling the configuration.  Hence, sufficiently small perturbations will 
make it transverse.  For a convex region $\tilde \gW$, all lines from the 
interior will be transverse to the boundary, so we have the second 
condition for a bounded region.  
\end{Remark}
Thus, in the case of a configuration consisting of disjoint regions with 
smooth boundaries, these yield the properties for a generic Blum linking 
structure for a residual set of embeddings of the configuration.  This 
proves Theorem \ref{Thm3.5} for a residual set of embeddings.  It remains 
to show that it holds for an open set of embeddings.  This will be shown in 
the next section. 
\par
We conclude this section by using the transversality of the stratifications 
to prove Corollary \ref{Cor3.5}.  
\begin{proof}[Proof of Corollary \ref{Cor3.5}] \par
 Let $\Phi \in \cP_{\rho}$ be a generic configuration.  By our earlier 
results, we know that $\dim \gS_{M_i}^{(\ga : \bgb)} = \dim 
\gS_{\cB_i}^{(\ga : \bgb)}$.  Thus it is enough to verify the result for 
$\gS_{M_i}^{(\ga : \bgb)}$.  We use the same notation as earlier $(\bga : 
\bgb_1 , \dots , \bgb_m)$.  Again by the preceding results,  
$\gS_{M_i}^{(\ga : \bgb)}$ is the diffeomorphic image of ${_{\bl}}j_1^k 
\rho_{i}^{-1}(W^{(\ga : \bgb)})$ under projection.  Hence, they have the 
same dimension.  \par
Second, Theorem \ref{Thm5.1} implies that the codimension of $W^{(\bga : 
\bgb)}$ in ${_{\bl}}E^{(k)}(X_{\cJ_i},\R^2)$, which we denote by 
$\codim_{{_{\bl}}E^{(k)}}(W^{(\bga:\bgb)})$, equals the codimension of 
${_{\bl}}j_1^k \rho_{i}^{-1}(W^{(\ga : \bgb)})$ in the space $X_{\cJ_i} 
\times (\R^{n+1})^{(m+1)}$.  Thus, combining the first two statements, 
we conclude
\begin{align}
\label{Eqn6.15}
\dim \gS_{M_i}^{(\ga : \bgb)} \,\, &= \,\, \dim {_{\bl}}j_1^k \rho_{i}^{-
1}(W^{(\ga : \bgb)})  \notag  \\
 &= \,\, \left(\sum_{p=1}^{m} n \ell_{p} + (n+1)(m + 1)  \right) 
\, -  \, \codim_{{_{\bl}}E^{(k)}}(W^{(\bga:\bgb)})\, .
\end{align}
Third, by the local form of ${_{\bl}}E^{(k)}(X_{\cJ_i},\R^2)$ and replacing 
$\R^2$ by $\R$ in (\ref{Eqn4.4b}) we have
\begin{equation}
\label{Eqn6.16}
\codim_{{_{\bl}}E^{(k)}}(W^{(\bga : \bgb)}) \,\, = \,\, 
\codim_{{_{m}}J^k(X_{\cJ_i},\R)}(W^{(\bga)}) \, + \, \sum_{p=1}^{m} 
\codim_{{_{\ell_p}}J^k(X_{j_p},\R)}(W^{(\gb_{j_p})})\, .
\end{equation}
Now, by the proof of Mather\rq s Theorem \ref{Thm3.2}, for a generic 
configuration and a multigerm $\bA_{\eta}$ which is an $m$--tuple,
 \begin{equation}
 \label{Eqn6.17}
 \codim_{{_m}J^k(X,\R)}(W^{\eta}) \,\, = \,\, \cR_e^+\text{-}
\codim(\bA_{\eta}) + nm.
\end{equation}
Now substituting (\ref{Eqn6.17}) into (\ref{Eqn6.16}) for each $\eta = 
\bga$ or $\bgb_{j_p}$ and then substituting  (\ref{Eqn6.16}) into 
(\ref{Eqn6.15}) and simplifying, we obtain 
\begin{align}
\label{Eqn6.18}
 \codim \gS_{M_i}^{(\ga : \bgb)} \, &= \, (n+1) - \dim \gS_{M_i}^{(\ga : 
\bgb)}  \notag \\
 \, &= \, \cR_e^+\text{-}\codim(\bA_{\bga}) +\sum_{p=1}^{m} 
\cR_e^+\text{-}\codim(\bA_{\bgb_{j_p}}) - m .
\end{align}
\end{proof}

\section{Concluding Generic Properties of Blum Linking Structures}
\label{S:sec7}
In this section we complete the proofs of Theorems \ref{Thm3.2}, 
\ref{Thm3.5} and \ref{Thm3.6}.  We first establish Theorem \ref{Thm3.2} 
which for a model configuration provides a normal edge-corner form for 
the closure of the Blum medial structure in the neighborhood of a singular 
point of a boundary.  Second, we deduce for a residual set of embeddings of 
a model configuration that the generic linking properties, structures of 
the stratifications, and generic properties of $\cB_{\infty}$ and 
$M_{\infty}$ hold for general multi-region configurations.  Then, we 
establish the openness of the genericity properties.  We do this in two 
steps.  First, we prove it for the easier case of configurations of disjoint 
regions with smooth boundaries.  Then, we explain how to modify the proof 
for general multi-region configurations.  This then completes the proofs 
of these theorems.  In proving openness, we establish an equivalence 
between the $\cR^{+}$-versality of multigerms for the distance or height 
function and the infinitesimal stability of associated mappings, and then 
adapt Mather\rq s theorem \lq\lq infinitesimal stability implies 
stability\rq\rq to obtain the openness (and note that it implies the 
structural stability of the medial axis and linking structures in the 
generic case).  \par
\subsection*{Blum Medial Structure Near Corner Points} %%%\hfill
\par
In place of the entire configuration, we consider a single region $\gW$ 
whose boundary $\cB$ has corners.  We recall from \S \ref{S:sec1} that 
for a $k$-edge-corner point $x \in \cB$, a local model consists of a 
diffeomorphism $\varphi : U \to \R^{n+1}$ from a neighborhood $U$ of $0$, 
with $\varphi(0) = x$, such that the restriction maps an open subset 
$U^{\prime}$ of $C_k = \R^k_{+} \times \R^{n+1-k}$ diffeomorphically to a 
neighborhood of $x$ in $\gW$.  Given such a local model, we have the 
subspaces $H_j$ defined where the coordinate $x_j = 0$ for $j \leq k$.  We 
then obtain hypersurfaces $S_j = \varphi (H_j \cap U)$.  For $S_j$ we let 
$\bn_j$ be the unit normal vector field on $S_j$ pointing into the region, 
and then we define the corresponding eikonal flow on $S_j$ by 
$\psi_j(x^{\prime}, t) = x^{\prime} + t \bn_j$.  Then there is an $\gevar_j > 
0$ such that $\psi_j | (S_j \times [- \gevar_j, \gevar_j])$ is a 
diffeomorphism onto a neighborhood of $x$.  We let $S_{j\, t} = \psi_j | 
(S_j \times \{ t\})$ denote the level hypersurface of the flow at time $t$. 
\par  
We begin with the following lemma.

\begin{Lemma}
\label{Lem7.1}
Let $x \in \cB$ be a $k$-edge-corner point for a region $\gW \subset 
\R^{n+1}$ with boundary $\cB$.  Then, there is a local model $\varphi : U 
\to \R^{n+1}$ for $\gW$ in a neighborhood of $x$ and an $\gevar > 0$ such 
that each eikonal flow is a diffeomorphism containing a common 
neighborhood $W$ of $x$.  Moreover, the level hypersurfaces $\{ S_{j\, 
t_j}\}$ with $j \leq k$ and all $0 \leq t_j \leq \gevar$ are in general 
position on $W$.  
\end{Lemma}
\begin{proof}
\par
We illustrate the situation in Figure~\ref{fig.7.1}. 
\par
\begin{figure}[ht] 
\includegraphics[width=8cm]{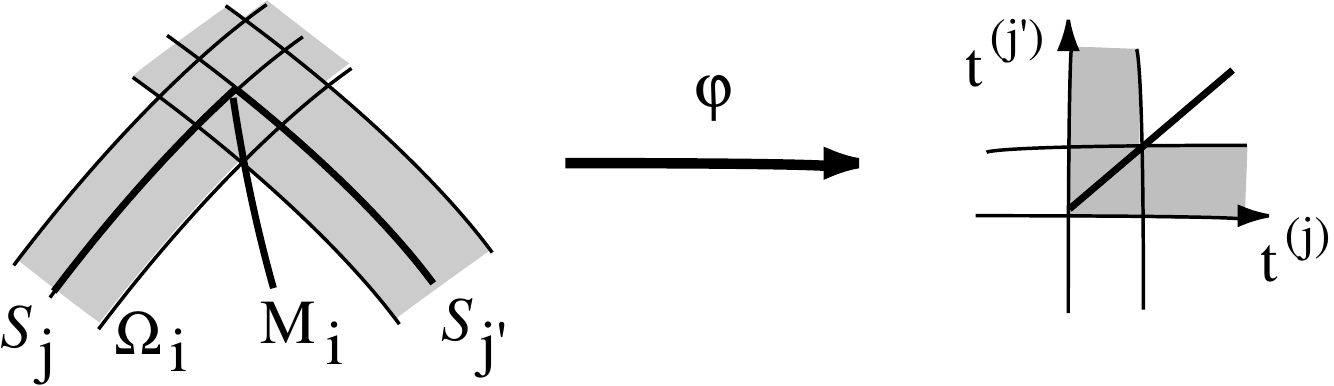}
\caption{\label{fig.7.1} The eikonal flows from the hypersurfaces meeting 
at a corner and the resulting mapping $\varphi$ giving an edge-corner 
normal form.} 
\end{figure} 
\par
By the above remarks, we can find for each $0 < j \leq k$, an $\gevar_j > 
0$ so the eikonal flow $\psi_j | (S_j \times [- \gevar_j, \gevar_j])$ is a 
diffeomorphism onto a neighborhood of $x$.  We let $\gevar = \min \{ 
\gevar_j\}$.  We also let $W = \cap \, \psi_j(S_j \times (- \gevar, 
\gevar))$, which is an open neighborhood of $x$.  Then, for each $j = 1, 
\dots , k$ we can define on $W$ a smooth vector field $\zeta_j$ by 
extending $\bn_j$ along the lines of the eikonal flow in $W$.  \par 
By possibly shrinking $W$ the vector fields $\{ \zeta_j\}$ are linearly 
independent on $W$.  We know that since $\zeta_j(x) = \bn_j$, 
$\{\zeta_1(x^{(1)}), \dots , \zeta_k(x^{(k)})\}$ are linearly independent 
when $x^{(1)} = \dots = x^{(k)} = x$, hence by continuity of the $k$-tuple 
$(\zeta_1, \dots , \zeta_k)$ on the $k$-fold product $W^k = W \times 
\cdots \times W$, there is a neighborhood of $(x, x, \dots , x)$ in $W^k$ on 
which they are linearly independent.  This neighborhood contains 
$W^{\prime \, k}$ for a neighborhood $W^{\prime}$ of $x$.  Then, the $\{ 
\zeta_j\}$ are the normal vectors to the level hypersurfaces $S_{j\, t}$ in 
$W$ and hence on $W^{\prime}$.  Thus, for any $x^{\prime} \in W^{\prime}$ 
with $x^{\prime} \in S_{j\, t^{(j)}}$, for each $j$, we conclude they are in 
general position at $x^{\prime}$.  We let $W^{\prime}$ be our desired 
neighborhood $W$.
\end{proof}
\par
\begin{proof}[Proof of Theorem \ref{Thm3.2}]
We first apply Lemma \ref{Lem7.1}, where we may furthermore choose 
compact neighborhoods $x \in W^{\prime \prime} \subset W^{\prime} 
\subset W$ such that $W^{\prime \prime}$ has diameter $ < a = \frac{1}{2} 
dist(W^{\prime}, \, \gW \backslash W)$, which is positive as $W^{\prime}$ 
is compact.  \par
Then, we consider $x^{\prime} \in W^{\prime \prime}$.  Then, 
$dist(x^{\prime}, S_j) \leq \gevar \leq a$, while $dist(x^{\prime}, \cB 
\backslash W) \geq 2a$.  Thus, a point $x^{\prime} \in W^{\prime \prime}$ 
of the Blum medial axis of type $A_1^{\ell}$ must have radial distance $ 
\leq \gevar$ and has the $A_1$ points in $\cB \cap W^{\prime \prime}$ 
which are points in $S_j$ which map to $x^{\prime}$ under the eikonal 
flows $\psi_j$.  \par
Thus, the Blum medial axis in $W^{\prime \prime}$ is formed from the 
transverse intersections of the eikonal flows.  The inverse of the $j$-th 
eikonal flow on $W^{\prime \prime}$ is smooth and given by the mapping 
$x^{\prime} (= \psi_j(x^{(j)}, t^{(j)})) \mapsto  (x^{(j)}, t^{(j)})$.  Thus, 
$t^{(j)}$, $j = 1, \dots, k$ are smooth functions on $W^{\prime \prime}$.  
These define a smooth mapping $\mu :  W^{\prime \prime} \to \R^k$ by 
$\mu(x^{\prime}) = (t^{(1)}, \dots , t^{(k)})$.  
Then, for each subset $J = \{ j_1 < j_2 \dots < j_{\ell}\}$ with 
$j_{\ell} \leq k$, we let 
$$  R^{\ell}_J \, = \, \{ (t_1, \dots , t_k) \in \R^k : t_{j_r} = \min_{j} \{ 
t_j\}, r = 1, \dots , \ell \} .$$  
By the general position of the $S_{j\, t}$ and $(1)$ of Lemma 
\ref{Lem6.10}, $\mu$ is transverse to all of the $R^{\ell}_J$.  Then, the 
Blum medial axis in $W^{\prime \prime}$ consisting of points of type 
$A_1^{\ell}$ is the disjoint union of the $\mu^{-1}(R^{\ell}_J)$ for subsets 
with $card(J) = \ell$.  \par
Furthermore, $Z_0 = \cap_{j = 1}^{k} S_j$ is a submanifold of dimension 
$n+1-k$, and $T_{x} Z_0 = \cap_{j = 1}^{k} T_x S_j$.  This is the subspace 
orthogonal to the $\bn_j$.  We let $\pi : \R^{n+1} \to T_{x} Z_0$ denote 
orthogonal projection.  Then, $\pi$  projects $Z_0$ submersively onto 
$T_{x} Z_0$ in a neighborhood of $x$, which by shrinking we may assume 
is $W^{\prime \prime}$.  Thus, as $T_{x} Z_0 = \ker(d\mu(x))$, we 
conclude that the smooth map $\tilde \mu  = (\mu, \pi) : W^{\prime 
\prime} \to T_{x} Z_0 \times \R^k \simeq \R^{n+1}$ is a local 
diffeomorphism in a neighborhood of $x$.  It sends $W^{\prime \prime} 
\cap \, \gW$ to a neighborhood of $0$ in $T_{x} Z_0 \times \R_{+}^k 
\simeq \R^{n+1-k} \times \R_{+}^k$, which is the model for a corner 
$C_k$.  Hence, $\tilde \mu^{-1}$ provides the local $k$--corner model as 
in Definition \ref{Def3.3} with $E_k$ mapping to the Blum medial axis in a 
neighborhood of $x$. 
\end{proof}
\par
\subsubsection*{Proof of Theorem \ref{Thm3.6} for a Residual Set of 
Embeddings} 
\hfill 
\par 
We can now complete the proof of Theorem \ref{Thm3.6} for a residual set 
of embeddings.  In the previous section, we have already established the 
properties i), ii) ,iv) and vi) for a residual set of embeddings $\Phi \in 
\cP$.  We have just completed the proof of Theorem \ref{Thm3.2}, which 
establishes the edge-corner normal form for the Blum medial axis at 
corner points, yielding  property iii).  Moreover, Theorem \ref{Thm3.2} 
also shows that there is no linking occurring at corner points, yielding the 
remainder of condition v).  Thus, Theorem \ref{Thm3.6} follows for a 
residual set of embeddings. \par
It remains to establish for both Theorem \ref{Thm3.5} and Theorem 
\ref{Thm3.6} the openness of genericity.  
\subsection*{Openness of the Genericity Conditions} %%%  \hfill 
\par
We first prove the openness of the genericity conditions for the case of a 
multi-region configuration $\bgW$ consisting of disjoint compact regions 
$\gW_i$ with smooth boundaries and without $\tilde E_7$ points on the 
medial axes (e.g. there will generically be no $\tilde E_7$ points for $n + 
1 \leq 6$).  This configuration is modeled by 
$\Phi : \bgD \to \R^{n+1}$ and contained in $\intr(\tilde \gW)$ for a 
compact region $\tilde \gW$.  As earlier $\gW_i = \Phi(\gD_i)$, 
$\cB_i = \Phi(X_i)$.  We let $\gW = \coprod_i \gW_i$, 
$X = \coprod_i X_i$, and $\cB = \coprod_i \cB_i$.  
We suppose it satisfies the genericity conditions 
with Blum medial axes $M_i$ for $\gW_i$ and $M_0$ for $\gW_0 \cap 
\tilde \gW$.  As $\cB$ is a compact (but not connected) smooth 
hypersurface, we let $\bn$ denote the inward pointing unit normal vector 
field on $\cB$.  Then, there is an $\gevar > 0$ such that the eikonal flow 
$x \mapsto \Phi(x) + t \bn(x)$ defines a diffeomorphism $\Psi: X \times [-
\gevar , \gevar ] \to  \R^{n+1}$ onto a tubular neighborhood $T_{\gevar}$ 
of $\cB$.  For $0 < s < \gevar$ we denote the image $T_s = \Psi(X \times 
[-s, s])$.  We may choose $\gevar$ sufficiently small so that $T_{\gevar} 
\subset V \subset \intr(\tilde \gW)$, for an open subset $V$ containing 
$\gW$.

We will use the following Lemma.  
\begin{Lemma}
\label{Lem7.2}
In the above situation, there is an $\gevar^{\prime}$ with $0 < 2 
\gevar^{\prime} < \gevar$ and an open neighborhood $\cU$ of $\Phi$ in 
$\emb(\bgD, \R^{n+1})$ such that if $\Phi^{\prime} \in \cU$, then: \par
\begin{enumerate}
\item $\cB^{\prime} = \Phi^{\prime}(X) \subset 
\intr(T_{\gevar^{\prime}})$;  
\item the eikonal flow for $\cB^{\prime}$, $\Psi^{\prime}: X \times [-
\gevar , \gevar ] \to  \R^{n+1}$ is a diffeomorphism onto its image, which 
is contained in $V \subset \intr(\tilde \gW)$; and 
\item $T_{2 \gevar^{\prime}} \subset \Psi^{\prime}(X \times [-\gevar , 
\gevar ])$. 
\end{enumerate}
\end{Lemma}
\par
We illustrate the Lemma in Figure~\ref{fig.7.2}.
\par
\begin{figure}[ht] 
\includegraphics[width=10cm]{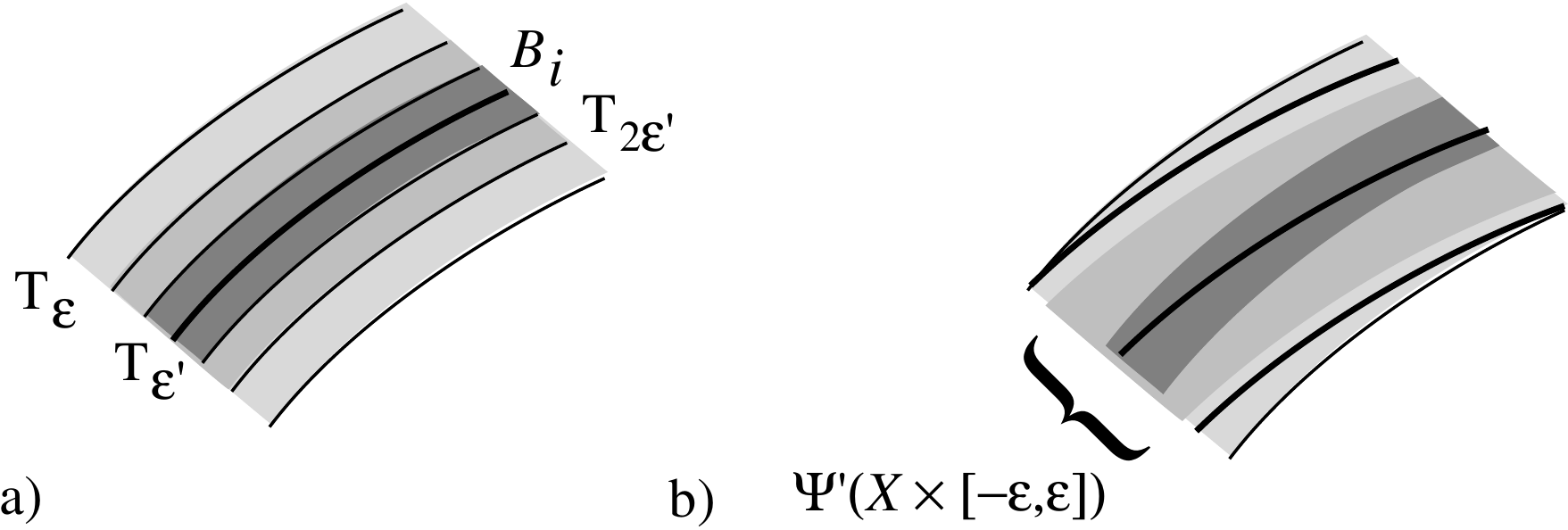}
\caption{\label{fig.7.2} The relation between the tubular neighborhoods of 
$\cB_i$ in a) and their images under the perturbation $\Psi$ in b).} 
\end{figure}
\par
After proving Lemma \ref{Lem7.2}, we will deduce the consequences for 
the openness of the genericity conditions for the case of a multi-region 
configuration consisting of disjoint regions with smooth boundaries.  
\begin{proof}[Proof of Lemma \ref{Lem7.2}]  
\par  
There is a continuous map $\emb(X, \R^{n+1}) \to C^{\infty}(X, S^{n})$ 
sending $\Phi$ to its Gauss map $\bn_{\Phi}$, which is given by the 
inward pointing unit normal vector field.  We may combine the two 
mappings to define a mapping $\Psi^{\prime}(x) = \Phi^{\prime}(x) + t 
\bn_{\Phi^{\prime}}$ on $X \times [-\gevar , \gevar ] \to \R^{n+1}$.  This 
defines a continuous mapping in the Whitney topology 
\begin{align*}
  \tilde G : \emb(\gD, \R^{n+1}) \, &\rightarrow \, C^{\infty}(X \times [-
\gevar , \gevar ], \R^{n+1})  \\
\Phi^{\prime} \, &\mapsto \, \Psi^{\prime} .
\end{align*}

\par
First, as $\intr(T_{\gevar^{\prime}})$ is an open subset of $\R^{n+1}$, the 
set of smooth mappings with images in $\intr(T_{\gevar^{\prime}})$ form 
an open subset $\cU_1$ of $C^{\infty}(X, \R^{n+1})$.  Since the restriction 
map $\emb(\bgD, \R^{n+1}) \to C^{\infty}(X, \R^{n+1})$ is continuous, 
$\cU_1$ is pulled back to an open subset $\cU_1^{\prime} \subset 
\emb(\bgD, \R^{n+1})$.  The configuration embeddings in $\cU_1^{\prime}$ 
then satisfy condition (1).  
\par
Next, for smooth manifolds $N$ and $P$ with $N$ compact, the set of 
$C^{\infty}$ $1-1$ immersions forms an open subset of $C^{\infty}(N, P)$, 
see e.g. \cite[Prop. 5.8]{GG}.  Thus, there is an open subset $\cU$ of 
$C^{\infty}$ $1-1$ immersions in $C^{\infty}(X \times [-\gevar , \gevar ], 
\R^{n+1})$.  As $X \times [-\gevar , \gevar ]$ is compact, such an 
immersion is a diffeomorphism onto its image.  Hence, property (2) holds 
for $\Phi$ in the inverse image of this open set under $\tilde G$. \par
Third, there is an open subset of $C^{\infty}(X \times [-\gevar , \gevar ], 
\R^{n+1})$ containing $\Psi$ which will map $X \times [-\gevar , \gevar ]$ 
into $V$ and each component of $X \times \{-\gevar , \gevar \}$ into 
$(\gW \backslash \intr(T_{2\gevar^{\prime}})) \cap \intr(\tilde \gW)$.  
For a diffeomorphism $\Psi^{\prime}$ contained in a sufficiently small 
open set $\cU^{\prime \prime}$ containing $\Psi$, it maps each $X_i 
\times \{\pm \gevar\}$ to the same complementary components as does 
$\Psi$.  Thus, by a degree argument the image of $\Psi^{\prime}$ must 
contain $T_{2\gevar^{\prime}}$ giving (3).
\end{proof}
We now apply Lemma \ref{Lem7.2} to prove openness.  We note that 
openness cannot be deduced from multi-transversality conditions on all of 
$\cB^{(\ell)}$ (by Theorem \ref{Thm5.1} we only obtain openness on a 
compact subset of $\cB^{(\ell)}$).  Instead we show that we can pass to 
the infinitesimal stability of smooth maps associated to the 
multi-distance and height-distance functions, and apply Mather\rq s 
Theorem  \lq\lq Infinitesimal Stability Implies Stability\rq\rq.  
\par  
\subsection*{Relation with Infinitesimally Stable Mappings } 
We first choose the $\gevar^{\prime}$ satisfying the Lemma and yielding 
an open subset $\cU \subset \emb(\bgD, \R^{n+1})$.  Then, for $i > 0$ we 
let $C_i^{\prime} = \gW_i \backslash \intr(T_{2\gevar^{\prime}})$ and 
$C_i = \gW_i \backslash \intr(T_{\frac{3}{2}\gevar^{\prime}})$.   Next, 
given $\tilde \gW$ we choose compact submanifolds $\tilde \gW_1$ and 
$\tilde \gW^{\prime}$ satisfying
$$\tilde \gW \subset \intr(\tilde \gW_1)  \subset \tilde \gW_1  \subset 
\tilde \intr(\tilde\gW^{\prime}) \subset \tilde\gW^{\prime}\, . $$ 
Then, we replace $\gW_0$ by $\intr(\tilde\gW^{\prime}) \cap \gW_0$ and 
still denote it by $\gW_0$.  Next, we let $C_0^{\prime} = \tilde \gW 
\backslash \intr(T_{2\gevar^{\prime}}))$ and $C_0 = \tilde \gW_1 
\backslash \intr(T_{\frac{3}{2}\gevar^{\prime}}))$.  
Third, for $i > 0$ we let $U_i = \gW_i \backslash 
(T_{\gevar^{\prime}})$ and $U_0 = \intr(\tilde \gW^{\prime})$. 
By Lemma \ref{Lem7.2}, for $\Phi^{\prime} \in \cU$, the corresponding 
Blum medial axes satisfy $M_i^{\prime} \subset C_i$ for $i > 0$ and 
$M_0^{\prime} \cap \tilde \gW  \subset C_0$.  \par 
We then have for all $i \geq 0$, 
$$ C_i^{\prime} \,\, \subset \,\, \intr(C_i)  \,\, \subset \,\, C_i  \,\, 
\subset \,\, U_i  \,\, \subset \,\, \intr(\gW_i)   $$
with $ C_i$ and  $C_i^{\prime}$ compact and $U_i$ open.  
\par 
We will consider various functions related to the distance and height 
functions associated with the embedding $\Phi$.  For $i \geq 0$ we begin 
with
\begin{align}
\label{Eqn7.6a}
  \bar \gs_i: X_i \times \intr(\gW_i)   &\to   \R \times \intr(\gW_i)   \\
(x, u^{(i)}) \quad &\mapsto \quad (\gs(x, u^{(i)}), u^{(i)})\,  , \notag
\end{align}
and 
\begin{align}
\label{Eqn7.7a}
  \bar \nu_i : X_i \times S^{n}  &\to   \R \times S^{n}   \\
(x, v) \quad &\mapsto \quad (\nu(x, v), v)\, .  \notag 
\end{align}
Then, via the following general lemma, we relate the local 
$\cR^{+}$-versality of $\gs_i$ and the local $\cA$-stability of $\bar 
\gs_i$, $i = 0, \dots m$; and similarly for $\nu_i$ and $\bar \nu_i$.
\begin{Lemma}
\label{Lem7.6}
Suppose $\bar F(x, u) : \R^{n+1} \times U, S \times \{u^{(0)}\} \to \R, 0$ is 
an unfolding of $f(x) : \R^{n+1}, S \to \R, 0$, where each germ of $f$ at 
each $x^{(i)} \in S$ is weighted homogeneous.  Then, $\bar F$ is the 
$\cR^{+}$-versal unfolding of $f$ if and only if 
$$F(x, u) = ( \bar{F}(x; u), u): \R^{n+1} \times U, S \times \{u^{(0)}\} \to \R 
\times U, (0, u^{(0)})$$
 is infinitesimally $\cA$-stable.  
\end{Lemma}
\begin{proof}[Proof of Lemma \ref{Lem7.6}]
In either direction it is sufficient to consider a finite set $S \subset 
\R^{n+1}$ with $u^{(0)} \in U$.  We let $S = \{x^{(1)}, \dots , x^{(r)}\}$, and 
choose local coordinates $x^{(j)}_i$ about each $x^{(j)}$.  We note that 
points $x^{(j)}$ at which the multigerm is nonsingular can be ignored 
without affecting the conclusion.  We also let $u = 
(u_1, \dots u_q)$ denote the local coordinates for $U$ about $u^{(0)}$, and 
$y$ a coordinate for $\R$.  For these coordinates, we let $\bar F_j$ denote 
the germ of $\bar F$ at $x^{(j)}$.  Next, we let $\cC_{x^{(j)}}$ denote the 
ring of germs of functions at $x^{(j)}$ with maximal ideal $\itm_{x^{(j)}}$, 
and $\cC_{x^{(j)}, u}$ the ring of germs at $(x^{(j)}, u^{(0)})$, with maximal 
ideal $\itm_{x^{(j)}, u}$, etc.  With $\{\varphi_1, \dots , \varphi_k\}$ 
understood, we abbreviate a module $R\{\varphi_1, \dots , \varphi_k\}$ by 
$R\{\varphi_i\}$.  \par
Then the infinitesimal stability of $F(x, u) : \R^{n+1} \times U, S \times 
\{u^{(0)}\} \to \R \times U, (0, u^{(0)})$ is equivalent to 
\begin{equation}
\label{Eqn7.8}
 \sum_{j = 1}^{r} \cC_{x^{(j)}, u}\{\pd{\bar F_j}{x^{(j)}_i}, \pd{ }{u_i} + 
\pd{\bar F_j}{u_i}\}\, + \cC_{y, u} \{\pd{ }{y}, \pd{ }{u_i}\} \,\, = \,\,  
\oplus_{j = 1}^{r} \cC_{x^{(j)}, u}\{\pd{ }{y}, \pd{ }{u_i}\}  
\end{equation}
where $\cC_{y, u} \{\pd{ }{y}, \pd{ }{u_i}\} \, \subset \, \oplus_{j = 1}^{r} 
\cC_{x^{(j)}, u}\{\pd{ }{y}, \pd{ }{u_i}\}$ is included by the diagonal map 
$F^*$ in each summand.  Since $\{\pd{ }{y}, \pd{ }{u_i} + \pd{\bar F_j}{u_i}, 
i = 1, \dots , q\}$ is also a set of free generators for $\cC_{x^{(j)}, 
u}\{\pd{ }{y}, \pd{ }{u_i}\}$, we may project along each $\cC_{x^{(j)}, 
u}\{\pd{ }{u_i}+ \pd{\bar F_j}{u_i}\}$ onto $\cC_{x^{(j)}, u}\{\pd{ }{y}\}$ 
and obtain that (\ref{Eqn7.8}) is equivalent to 
\begin{equation}
\label{Eqn7.9}
 \sum_{j = 1}^{r} \cC_{x^{(j)}, u}\{\pd{\bar F_j}{x^{(j)}_i}\}\, + \, \cC_{y, 
u} \{\pd{ }{y}, \pd{\bar F}{u_1}, \dots , \pd{\bar F}{u_q}\} \,\, = \,\,  
\oplus_{j = 1}^{r} \cC_{x^{(j)}, u}\{\pd{ }{y}\}\, .  
\end{equation}
Then, the system of rings and ideals $(\cC_{y, u}, \itm_u\cC_{y, u}) 
\overset{F_j^*}{\rightarrow} (\cC_{x^{(j)}, u}, \itm_u\cC_{x^{(j)}, u})$ is 
an adequate system of differentiable algebras in the sense of \cite[\S 
6]{D6}, so by an extension of Mather\rq s algebraic Lemma to the case of 
modules over adequate systems of algebras \cite[Lemma 7.3]{D6},  
(\ref{Eqn7.9}) is equivalent to 
\begin{equation}
\label{Eqn7.10}
 \sum_{j = 1}^{r} \cC_{x^{(j)}}\{\pd{f_j}{x^{(j)}_i}\}\, + \, \cC_{y} \{\pd{ 
}{y}, \pd{\bar F}{u_1}_{| u = u^{(0)}}, \dots , \pd{\bar F}{u_q}_{| u = 
u^{(0)}}\} \,\, = \,\,  \oplus_{j = 1}^{r} \cC_{x^{(j)}}\{\pd{ }{y}\}\, .  
\end{equation}
As each $f_j$ is weighted homogeneous, each
$f_j \in \cC_{x^{(j)}}\{\pd{f_j}{x^{(j)}_i}\}$.  Hence 
$$  \itm_{y} \{\pd{ }{y}, \pd{\bar F}{u_1}_{| u = u^{(0)}}, \dots , \pd{\bar 
F}{u_q}_{| u = u^{(0)}}\} \,\, \subset \,\, \sum_{j = 1}^{r} 
\cC_{x^{(j)}}\{\pd{f_j}{x^{(j)}_i}\}\, .  $$
Thus, (\ref{Eqn7.10}) is equivalent to 
\begin{equation}
\label{Eqn7.11}
 \sum_{j = 1}^{r} \cC_{x^{(j)}}\{\pd{f_j}{x^{(j)}_i}\}\, + \, \langle \pd{ }{y}, 
\pd{\bar F}{u_1}_{| u = u^{(0)}}, \dots , \pd{\bar F}{u_q}_{| u = 
u^{(0)}}\rangle  \,\, = \,\,  \oplus_{j = 1}^{r} \cC_{x^{(j)}}\{\pd{ }{y}\}\, .  
\end{equation}
This is exactly the equation for the $\cR^{+}$-infinitesimal versality of 
the multigerm $\bar F(x, u) : \R^{n+1} \times U, S \times \{u^{(0)}\} \to 
\R, 0$, which by the versal unfolding theorem \cite[Thm 9.3]{D6} is 
equivalent to its $\cR^{+}$-versality.  
\par
As each step is an equivalence, we have established the result.
\end{proof}
\par
We apply Lemma \ref{Lem7.6} as follows.  Let $\Phi \in \cP = (\cP_{\gs\, 
\rho} \cap \cP_{\tau})$, the residual set of mappings of the configuration 
with the generic properties on $\tilde \gW^{\prime}$ from 
\S~\ref{S:sec6}.  Then, the associated distance and height functions are 
$\cR^{+}$-versal unfoldings for any finite set $S \subset X_i$.  By the 
Lemma and the results of Mather, this implies that the associated 
mappings $\bar \gs_i$ and $\bar \nu_i$, which are proper, are 
infinitesimally stable, as global mappings.  This is proven, following 
Mather, by using a partition of unity argument for the local infinitesimal 
stability.  Then, we want to apply Mather\rq s theorem \lq\lq 
infinitesimal stability implies stability\rq\rq\,\cite[Thm 3]{M5}.
\begin{Thm}[Mather]
\label{Thm.Mather}
Let $f: N \to P$ be a smooth mapping between manifolds without 
boundaries and let $M \subset N$ be a closed $0$-codimension submanifold 
which may have boundaries and corners.  Then, if $f | M : M \to P$ is proper 
and infinitesimally stable, then there is a neighborhood $\cW$ of $f | M \in 
C^{\infty}(M, P)$ and continuous mappings $H_1 : \cW \to \Diff(N)$ and 
$H_2 : \cW \to \Diff(P)$, with both $H_1(f|M) = Id_N$ and $H_2(f|M) = 
Id_P$, such that 
$$ g \,\, = \,\, H_2(g)\circ f \circ H_1(g)  \qquad \text{ for all } g \in 
\cW\, .  $$
\end{Thm}
\par
For our situation, we suppose that $M$ is compact with a given compact 
$M^{\prime}$ with $M^{\prime} \subset \intr(M)$.  Then, we may shrink 
$\cW$ to a smaller neighhborhood $\cW^{\prime}$ in the 
$C^{\infty}$-topology so that in addition $M^{\prime} \subset H_1(g)(M)$.  
\par
We apply this version of Mather\rq s Theorem for each of the mappings 
$\bar \gs_i$ and $\bar \nu_i$ associated to the given $\Phi$. 
Each of these mappings is a proper infinitesimally stable mapping.   
Thus, for each mapping there are open neighborhoods: $\cV_{1, i}$ of $\bar 
\gs_i | X_i \times C_i$ in $C^{\infty}(X_i \times C_i, \R \times C_i)$, and 
$\cV_{2, i}$ of $\bar \nu_i$ in $C^{\infty}(X_i \times S^n, \R \times S^n)$ 
satisfying the conclusions of Mather\rq s Theorem and the additional 
condition that $M^{\prime} \subset H_1(g)(M)$ for each case.  However by 
Lemma \ref{Lem7.6}, off of the appropriate compact submanifolds 
$M^{\prime}$ which denote either $X_i \times C_i^{\prime}$ or resp. $X_i 
\times S^n$, the corresponding mappings have at most stable $A_1$ 
singularities.  Also, by Mather\rq s classification theorem for stable 
multigerms, the $\cA$-equivalence type of the stable multigerms, 
which are formed from the simple $A_k$-germs, are determined by the 
$\cR^{+}$-equivalence class $\gS^{(\bga)}$.  For any of 
the mappings $\bar \gs_i$ or $\bar \nu_i$, we denote the set of points of 
singularity type $\bA_{\bga}$ by $\gS^{(\bga)}(\bar \gs_i)$, or 
$\gS^{(\bga)}(\bar \nu_i)$.  \par
The restriction of the distance mapping to $X_i \times C_i$ defines a 
continuous mapping in the Whitney topology $\gs \mapsto \bar \gs_i$, 
$$C^{\infty}(X_i \times \R^{n+1}, \R) \longrightarrow C^{\infty}(X_i 
\times C_i, \R \times C_i)\, . $$  
Hence, there is an open set $\cU^{\prime} \subset \emb(\bgD, \R^{n+1})$ 
containing $\Phi$ which maps into each $\cV_{i, 1}$ and $\cV_{i, 2}$.  
Thus, by Lemma \ref{Lem7.6}, the associated families of distance and 
height maps $\gs_i^{\prime}$ or $\nu_i^{\prime}$ corresponding to 
$\Phi^{\prime} \in \cU^{\prime}$ are $\cR^{+}$-versal unfoldings for any 
finite set $S \subset X_i$.  Thus, we obtain openness of the genericity 
properties of individual stratifications $\gS^{(\bgb_{i})}$, 
$\gS_{\infty}^{(\bga)}$, and $\gS^{(\bga)}$ in $X_i$.  \par 
We next consider the transverse intersection of the strata $\gS^{(\bga)}$ 
for $\gW_0$ and the $\gS^{(\bgb_{i})}$ for the $\gW_i$.  We form the 
associated mappings 
\begin{align}
\label{Eqn7.6c}
  \hat \gs_{i, 0}: X_i \times C_i \times C_0  &\to   \R \times C_i \times 
C_0  \\
(x, u^{(i)}, u^{(0)})) \quad &\mapsto \quad (\gs_i(x, u^{(i)}), u^{(i)}, u^{(0)})\, 
, \notag
\end{align}
\begin{align}
\label{Eqn7.6d}
  \hat \gs_{0, 0}: X_i \times C_i \times C_0   &\to   \R \times C_i \times 
C_0  \\
(x, u^{(i)}, u^{(0)}) \quad &\mapsto \quad (\gs_0(x, u^{(0)}), u^{(i)}, u^{(0)}) 
\, .  \notag 
\end{align}
These are products of $\bar \gs_i$ with identity mappings, so  
$\gS^{(\bgb_{j_p})}(\hat \gs_{i, 0}) = \gS^{(\bgb_{j_p})}(\bar \gs_i) 
\times C_0$ for $i > 0$ or $\gS^{(\bga)}(\hat \gs_{0, 0}) = 
\gS^{(\bga)}(\bar \gs_0) \times C_i$.  These stratifications are the 
images under projection of the stratifications $Z^{(\bgb_{j_p})}$, resp.  
$Z^{(\bga)}$ in Proposition \ref{Prop6.6}.  As they are transverse and 
project diffeomorphically to the strata in the compact manifold (with 
boundaries and corners) $X_i \times C_i \times C_0$, the strata 
$\gS^{(\bgb_{j_p})}(\hat \gs_{i, 0})$ and $\gS^{(\bga)}(\hat \gs_{0, 0})$ 
intersect transversely.  They form a closed Whitney stratification of $X_i 
\times C_i \times C_0$.  
\par
Thus, for $\Phi^{\prime} \in \cU^{\prime}$ with corresponding $\bar 
\gs_i^{\prime}$ and $\bar \nu_i^{\prime}$, we may apply Mather\rq s 
Theorem and obtain for each $i$ the continuous families of 
diffeomorphisms $H^{(i)}_j(\bar \gs_i^{\prime})$, $j = 1, 2$.
\par  
By the  properties of the $M^{\prime}$ for $i > 0$, 
$$H^{(i)}_1(\bar \gs_i^{\prime})(\gS^{(\bgb_{j_p})}(\bar \gs_i)) \cap (X_i 
\times C_i^{\prime}) \,\, = \,\,  \gS^{(\bgb_{j_p})}(\bar \gs_i^{\prime}) 
\cap (X_i \times C_i^{\prime})  \,\, = \,\, \gS^{(\bgb_{j_p})}(\bar 
\gs_i^{\prime}) \, ,$$ 
and similarly 
$$H^{(0)}_1(\bar \gs_0^{\prime})(\gS^{(\bga)}(\bar \gs_0)) \cap (X_i 
\times C_0^{\prime}) \,\, = \,\,  \gS^{(\bga)}(\bar \gs_0^{\prime}) \cap 
(X_i \times C_0^{\prime})  \,\, = \,\, \gS^{(\bga)}(\bar \gs_i^{\prime})\, . 
$$
Thus, from the above, 
$\gS^{(\bgb_{j_p})}(\hat \gs_{i, 0}^{\prime}) \cap \gS^{(\bga)}(\hat 
\gs_{0, 0}^{\prime})$ is the intersection in $X_i \times C_i \times C_0$ 
of the pull-back of the closed Whitney stratification with strata 
$(\gS^{(\bga)}(\bar \gs_i) \times C_0) \times (\gS^{(\bga)}(\bar \gs_0) 
\times C_i)$,by the map 
\begin{equation}
\label{Eqn7.12}
  ((H^{(i)}_1 \times Id_{C_0})^{-1}, (H^{(0)}_1 \times Id_{C_i})^{-1}) : X_i 
\times C_i \times C_0 \longrightarrow (X_i \times C_i \times C_0)^2\, .  
\end{equation}
\par
By the openness of transversality to closed Whitney stratified sets on 
compact sets, and the continuous dependence of the mapping  
(\ref{Eqn7.12}) on $(\bar \gs_i^{\prime}, \bar \gs_0^{\prime})$, and hence 
on $\Phi^{\prime}$, there is a smaller open neighborhood 
$\cU^{\prime\prime}$ of $\Phi$, such that for $\Phi^{\prime} \in 
\cU^{\prime\prime}$, (\ref{Eqn7.12}) is transverse to the product 
stratification.  Hence, $\gS^{(\bgb_{j_p})}(\bar \gs_{i, 0}))$ and 
$\gS^{(\bga)}(\bar \gs_{0, 0}^{\prime})$ intersect transversely, and so do 
their projections in $X_i$ by Lemma \ref{Lem6.9}.  This establishes the 
openness of the condition that the strata $\gS^{(\bga)}$ and 
$\gS^{(\bgb_{j_p})}$ intersect transversely in each $X_i$. \par
Next, we apply an analogous argument, using  
\begin{align}
\label{Eqn7.6e}
  \check \gs_i: X_i \times C_i \times S^{n} &\to   \R \times C_i \times 
S^{n}  \\
(x, u^{(i)}, v)) \quad &\mapsto \quad (\gs_i(x, u^{(i)}), u^{(i)}, v)\,  \notag 
\end{align}
and
\begin{align}
\label{Eqn7.7b}
  \check \nu_i : X_i \times C_i \times S^{n}  &\to   \R \times C_i \times 
S^{n}   \\
(x, u^{(i)}, v) \quad &\mapsto \quad (\nu(x, v),u^{(i)},  v)\, . \notag 
\end{align}
\vspace{2ex}
We then obtain the openness of the genericity condition of transversality 
of the strata $\gS^{(\bgb_{j_p})}$ and $\gS_{\infty}^{(\bga)}$.  \par
Lastly, we apply the argument a third time to the maps 
\begin{align}
\label{Eqn7.6b}
  \hat{\hat{\gs}}_{i, 0}: X_i \times \prod_{i = 0}^{m} C_i   &\to   \R \times 
\prod_{i = 0}^{m} C_i   \\
(x, (u^{(0)}, \dots , u^{(m)})) \quad &\mapsto \quad (\gs_i(x, u^{(i)}), 
(u^{(0)}, \dots , u^{(m)}))\,  \notag
\end{align}
to obtain the openness of the genericity condition that the images of the 
interseections of the strata $\gS^{(\bga)}$ and $\gS^{(\bgb_{j_p})}$ 
intersect in general position in $\gS_{M_0}^{(\bga)}$.
\par  
This completes the proof of openness of the genericity conditions.  \par

\subsection*{Openness of Generic Properties for General Configurations}
\par
For the general case, the argument is slightly complicated by the 
edge-corner points on the common boundaries of regions and the presense 
of $\tilde E_7$ points when $n + 1 = 7$.  We first allow common 
boundaries but still exclude $\tilde E_7$ points.  The total space 
of $\bgD$ is a manifold with boundaries and corners, although the corners 
are concave rather than convex.  We begin with an extension lemma based 
on an extension result of Bierstone \cite{Bi} .
\begin{Lemma}
\label{Lem7.4}
If $\tilde \bgD$ is an open neighborhood of $\bgD$ in $\R^{n+1}$, then 
there is a continuous extension in the Whitney topology
\begin{equation}
\label{Eqn7.5}
\eta : \emb(\bgD, \R^{n+1}) \longrightarrow C^{\infty}(\tilde \bgD, 
\R^{n+1}) \, .
\end{equation}
\end{Lemma}
\begin{proof}
In the case of convex corners, Mather proved an extension result giving a 
continuous extension of smooth functions on an $n$-manifold $M \subset 
\R^n$ with boundaries and corners to an open neighborhood $U$ of $M$ in 
$\R^n$ (see \cite{M4} and the proposition in \cite[\S 5]{M5}).  This method 
uses continuous 
local extensions for all $C^{\infty}$ functions in neighborhoods of the 
corner points, based on the extension method of Seeley \cite{Se}, together 
with a partition of unity.  This method also works in our case provided we 
have a continuous local extension on a neighborhood for the concave corner 
points.  In this case, whether of type $P_k$ or $Q_k$ the complement in 
the edge-corner model is the interior of an edge-corner point.  The region 
itself is locally a closed semi-analytic set which is the closure of its 
interior.  This is exactly the situation where the extension result of 
Bierstone \cite{Bi} applies to give a continuous extension from 
$C^{\infty}$ functions on the region in a neighborhood of the edge-corner 
point to an open neighborhood of the point in $\R^{n+1}$.  With this local 
extension, we can apply the same argument used by Mather to give the 
continuous extension (\ref{Eqn7.5}), where $\tilde \bgD$ denotes an open 
neighborhood of $\bgD$.  
\end{proof}
Then, in $\tilde \bgD$ we can extend each nonempty $Cl(X_{i\, j})$ to an 
open hypersurface $\tilde X_{i\, j}$ and then construct a compact 
submanifold with boundaries and corners $S_{i\, j}$ so that $Cl(X_{i\, j}) 
\subset \intr(S_{i\, j}) \subset \tilde X_{i\, j}$, where $\intr(\cdot)$ is 
taken in the manifold $\tilde X_{i\, j}$.  \par
We generally denote $\eta(\Phi)$ by $\tilde \Phi$.  We choose $\tilde 
\bgD$ small enough so that $\tilde \Phi$ is still an embedding.  Then, we 
first restrict to the open subset $\cU^{\prime}$ which is the pull-back of 
the open subset of $C^{\infty}(\tilde \bgD, \R^{n+1})$ consisting of those 
$\tilde \Phi^{\prime}$ which restricted to each $S_{i\, j}$ are 
embeddings.  \par
We furthermore let $S_i =  \coprod_{j} S_{i\, j}$ and $S = \coprod_i S_i$, 
which is a disjoint union of manifolds with boundaries and corners.  
Then, $\tilde \Phi$ still defines a smooth map denoted $\tilde \Phi : S \to 
\R^{n+1}$, which is an embedding.  Then, we can proceed as earlier to 
define the eikonal flow using the inward pointing unit normal vector 
fields $\bn_i$ for each $\gW_i$.  Again by compactness of each $S_{i\, 
j}$, there is an $\gevar > 0$ so that each $\psi_{i\, j} : S_{i\, j} \times [-
\gevar, \gevar] \to \R^{n+1}$ is a diffeomorphism onto its image $T_{i\, j, 
\gevar}$ for all $i, j$.  Because the singular strata of the boundaries $X_i$ 
form compact sets, we can use the argument in the proof of Theorem 
\ref{Thm3.2} to find a smaller $\gevar > 0$ so that in the union 
$\cup_{i, j} T_{i\, j, \gevar}$, the Blum medial axis consists of the 
edge-corner normal forms for the singular points of $X$.  \par 
Then, we replace the tubular neighborhoods $T_{s}$ by the unions $\tilde 
T_{i, s} = \cup_j T_{i\, j, s}$ and obtain the analogue of Lemma 
\ref{Lem7.2}.
\begin{Lemma}
\label{Lem7.2a}
In the above situation, there is an $\gevar^{\prime}$ with $0 < 2 
\gevar^{\prime} < \gevar$ and an open neighborhood $\cU$ of $\Phi$ in 
$\emb(\bgD, \R^{n+1})$ such that if $\Phi^{\prime} \in \cU$, then: \par
\begin{enumerate}
\item $\cB_i^{\prime} = \Phi^{\prime}(X) \subset 
\intr(\tilde T_{i, \gevar^{\prime}})$;  
\item each eikonal flow for $\cB_i^{\prime}$, $\Psi_{i\, j}^{\prime}: 
S_{i\, j} \times [-\gevar , \gevar ] \to  \R^{n+1}$ is a diffeomorphism onto 
its image; and 
\item $\tilde T_{i, 2 \gevar^{\prime}} \cap \gW_i \subset 
\Psi_i^{\prime}(S_i \times [-\gevar , \gevar ]).$
\end{enumerate}
\end{Lemma}
\par
We again enlarge $\tilde \gW$ via $\tilde \gW_1$ to $\tilde 
\gW^{\prime}$ as earlier and replace $\gW_0$ by $\tilde \gW^{\prime} 
\cap \gW_0$.  With the $\gevar^{\prime}$ from Lemma \ref{Lem7.2a}, we 
now let $C_i^{\prime} = \gW_i \backslash \intr(T_{i, 2\gevar^{\prime}})$ 
and $C_0^{\prime} = \tilde \gW_1 \cap (\gW_0 \backslash \intr(T_{0, 
2\gevar^{\prime}}))$.  Analogously, we let $C_i = \gW_i \backslash 
\intr(T_{i, \frac{3}{2}\gevar^{\prime}})$ and $C_0 = \tilde \gW^{\prime} 
\cap (\gW_0 \backslash \intr(T_{0, \frac{3}{2}\gevar^{\prime}}))$.  Third, 
we let $U_i$ denote $\intr(\gW_i) \backslash T_{i, \gevar}$.
\par
Then, we proceed as earlier.  We have the analogous collection of mappings 
(\ref{Eqn7.6a}) - (\ref{Eqn7.7a}), (\ref{Eqn7.6c}) - (\ref{Eqn7.6d}), 
(\ref{Eqn7.6e}) - (\ref{Eqn7.7b}), and (\ref{Eqn7.6b}), for the 
corresponding families of distance and height based mappings, where we 
replace each $X_i$ by $X_{\cJ_i}$.  First, using Lemma \ref{Lem7.2a},  we 
deduce that the stratifications are contained in the model $M^{\prime}$ 
for each case.  Then, we apply Lemma \ref{Lem7.6} to conclude that these 
mappings are infinitesimally stable.  We then apply the same adapted 
version of Mather\rq s Theorem for the corresponding mappings to obtain 
an open neighborhood of $\Phi^{\prime} \in \cU^{\prime}$, to which we can 
repeat the arguments to conclude that the stratifications are closed 
Whitney stratifications which intersect transversely where appropriate.  
As earlier, we may conclude that they map diffeomorphically onto the 
boundary and intersect transversely.  This completes the proof of 
openness of the genericity conditions for the general multi-configuration 
in the absence of $\tilde E_7$ points. \par
\subsubsection*{Openness of Generic Properties Allowing $\tilde E_7$ 
Points}
\par
Lastly, we include the possibility of $\tilde E_7$ points when $n + 1 = 7$.  
By transversality to the $\tilde E_7$-stratum, these points will occur at 
isolated points. Consider points $x_0 \in X_i$ and $u_0 \in M_i$, such that 
$\gs(\cdot, u_0)$ has a minimum $y_0$ at $x_0$ of type $\tilde E_7$.  By 
transversality, we may suppose $j_1^4\gs(x, u)$ is transverse at $(x_0, 
u_0)$ to the Whitney stratified set defined by the closure of the $\tilde 
E_7$-stratum.  Then, there is an open neighborhood $\cU^{\prime \prime}$ 
of embeddings and open neighborhoods $u_0 \in U$ and $x_0 \in V$ so that 
for $\Phi^{\prime} \in \cU^{\prime \prime}$, the associated distance 
function $\gs^{\prime}$ saisfies:  $j_1^4\gs^{\prime}(x, u)$ transversely 
meets $\tilde E_7$-stratum at a unique point $(x^{\prime}, u^{\prime}) \in 
V \times U$ and the associated $\gs^{\prime}(x, u^{\prime})$ has a 
minimum at $x^{\prime}$.  
Furthermore, by a result of Looijenga \cite{L2}, the transversality to the 
$\tilde E_7$ stratum implies that the resulting unfolding at such a point 
as in  (\ref{Eqn7.6a}) given by 
$$\bar \gs_i: X_i \times \intr(\gW_i)  \to   \R \times \intr(\gW_i)$$
will define a topologically stable mapping near $(x_0, u_0)$.  Looijenga 
\cite{L2} shows that the germ is topologically equivalent to the unfolding 
which is infinitesimally versal except for the modulus term; and the proof 
of topological stability is given in \cite[Thm 4]{D8}.  By this we mean: 
that there is a neighborhood of embeddings, still denoted by $\cU^{\prime 
\prime}$, compact neighborhoods $u_0 \in D_1 \subset \intr(  D_2) 
\subset  D_2  \subset U$,  $x_0 \in V_1 \subset \intr(V_2)  \subset V_2 
\subset V$ and $y_0 \in W_1 \subset \intr(W_2)  \subset W_2  \subset 
\R$ such that if $\Phi^{\prime} \in \cU^{\prime \prime}$, with associated 
mapping $\bar \gs_i^{\prime}$ then: 
\begin{itemize} 
\item[i)] there are homeomorphisms onto their images $\varphi = 
(\varphi_1, \varphi_2) : V_2 \times D_2 \to \cirX_i \times U$, and $\psi : 
W_2 \times D_2 \to \R \times U$; 
\item[ii)] such that $\bar \gs_i^{\prime} = \psi \circ \bar \gs_i \circ 
\varphi^{-1}$ on $\varphi(V_2 \times D_2)$;
\item[iii)]  $\varphi(D_2 \times V_2) \supset D_1 \times V_1$, 
$\psi(W_2 \times V_2) \supset W_1 \times V_1$; and 
\item[iv)] $\varphi$ is a smooth diffeomorphism on the complement of 
$(x_0, u_0)$ and $\psi$ is a smooth diffeomorphism on the complement of 
$(y_0, u_0)$. 
\end{itemize} 
\par
We can repeat this for each $\tilde E_7$ point for $\gs_i$, and ensure that 
the above neighborhoods are distinct.  If they are labeled by an index $j$, 
then we can apply the earlier argument after we remove the 
$\intr(D^{(j)}_1)$ for all $j$. Then, off the union of the $D^{(j)}_1$ the 
generic properites persist in an open neighborhood of $\Phi$.  Also, on 
each $D^{(j)}_1$, the medial axis for $\Phi^{\prime}$ will be the image of 
that of $\Phi$ via the homeomorphism $\varphi$ and will have generic 
properties off $\varphi_2(x_0, u_0)$, which must also be an $\tilde E_7$ 
point, and hence equals the unique point $(x^{\prime}, u^{\prime})$.  Thus, 
the medial axis has generic properties off the $\tilde E_7$ points; and at 
these points its structure is topologically constant.  Because the $\tilde 
E_7$ points are isolated, the codimension conditions and transversality 
only allow the linking as earlier described, so in this sense the generic 
properties for $\tilde E_7$ points hold for an open set of embeddings.  
\par

\section{Reductions of the Proofs of the Transversality Theorems}
\label{S:sec9}
\par
We begin the proofs of Theorems \ref{Thm5.0} and \ref{Thm5.1} by first 
outlining the three main steps of the proofs: \flushpar 
\begin{enumerate}
\item reduce Theorem \ref{Thm5.0} to a relative transversality theorem 
and Theorem \ref{Thm5.1} to a \lq\lq hybrid transversality theorem\rq\rq 
which combines the relative and absolute transversality theorems in 
\cite{D5};  
\item introduce the families of perturbations needed to prove the 
transversality for the space of mappings and compute the necessary 
infinitesimal deformations; and  
\item  verify the transversality conditions for the families of 
perturbations.  
\end{enumerate}
We explain each of these steps in more detail.  \par 
\vspace{2ex}

The \lq\lq hybrid transversality theorem\rq\rq is an extension of Thom\rq 
s transversality theorem which applies to a continuous mapping $\Psi: \cH 
\to C^{\infty}(M,N)$, where $\cH$ is a Baire space, and $M$ and $N$ are 
smooth manifolds (where we allow $M$ to have boundaries and corners).  
We assume there is given a subbundle $\cH^{k}(M, N)$ of the jet space 
$J^{k}(M, N)$ which consists of $k$-jets $j^k(\Psi(h))(x)$ for all $h \in 
\cH$ and all $x \in M$.  Then, for any $h \in \cH$ the associated jet 
mapping  $j^k(\Psi(h))(M) \subset \cH^{k}(M, N)$.  We consider either 
closed Whitney stratified sets  $W \subset \cH^{k}(M, N)$ or submanifolds 
$W$ whose closures form Whitney stratified sets with $W$ a stratum.  We 
refer to the latter $W$ as being \lq\lq relatively Whitney 
stratifiable\rq\rq.  Then, by \cite[Thm 1.3]{D5} and \cite[Thm 1.5]{D5}, 
provided (in an appropriate sense) $\Psi$ is \lq\lq transverse to 
$W$\rq\rq\, relative to $\cH^{k}(M, N)$, then on any compact subset $C 
\subset M$  there is an open dense subset of $h \in \cH$ for which 
$j^k(\Psi(h))$ is transverse on $C$ to the closed Whitney stratified set 
determined by $W$.  We may alternatively consider the situation where for 
a closed Whitney stratified set $Y \subset M\backslash E$, with $E 
\subset M$ a closed subset, $\cH^{k}(M\backslash E, N)$ is the restriction 
of the bundle to $M\backslash E$.  If $\Psi$ is \lq\lq transverse on $Y$ to 
$W$\rq\rq\, relative to $\cH^{k}(M\backslash E, N)$, then, there is a 
corresponding transversality result that on any compact subset $C \subset 
M\backslash E$, there is an open dense subset of $h \in \cH$ for which 
$j^k(\Psi(h))$ is transverse on $Y \cap C$ to $W$.  \par
We apply either of these results to the geometric mappings that associate 
to an embedding $\Phi : \bgD \to \R^{n+1}$ either the multi-distance 
functions $\rho_i$ or the height-distance function $\tau$.  In each case 
we must verify that \lq\lq the mapping $\Psi$ is transverse to the 
appropriate closed Whitney stratified sets\rq\rq (we recall the definition 
below).  The transversality conditions in Theorem \ref{Thm5.1} are shown 
to be equivalent to conditions in the setting of the hybrid transversality 
theorem for the appropriate Whitney stratified sets.  \par
It then remains to show that the corresponding  mappings $\Psi$ satisfy 
the transversality conditions.  For this we consider specific families of 
perturbations of $\Phi$ within $\emb (\bgD, \R^{n+1})$ and verify the 
appropriate local transversality conditions for the families.  

\subsection*{Hybrid Transversality Theorem}

We give a modified form of \cite[Thms 1.3 and 1.5]{D5} for a continuous 
mapping $\Psi: \cH \to C^{\infty}(M,N)$ which is transverse off a closed 
subset $E \subset M$  to $W \subset \cH^k(M,N)$, which is either a closed 
Whitney stratified set or its closure $Cl(W)$ is a Whitney stratified set 
with $W$ a stratum.  While in \cite{D1} $M$ was a smooth manifold, here 
we allow $M \subset \R^{\ell}$ to be a closed stratified set such that 
$sing(M) 
\subset E$, a closed subset of $M$.  We also let $Y \subset M\backslash E$ 
be a closed Whitney stratified set.  Then, the smoothness of a function $f : 
M \backslash E \to N$ is still defined.  

\begin{Definition}
\label{Def9.1}
The map $\Psi$ is said to be {\em transverse} or {\em completely 
transverse} to $W \subset  \cH^k(M,N)$ off $E$ if, given an open subset 
$\cU \subset \cH$, $x \in M$ and $y \in N$, there are open sets $x \in U 
\subset M$, and $y \in V \subset N$ such that for every map $h \in \cU$, 
there exists a finite-dimensional smooth manifold $ \cT \subset \cU$ 
with $h \in \cT$ such that for the family
\begin{align*}
\gG: M \times \cT &\rightarrow N,\\
(x,f) &\mapsto \Psi(f)(x)
\end{align*}
 the  $k$-jet extension
\begin{align*}
j^k\gG: (U\backslash E) \times \cT &\rightarrow \cH^k(U\backslash E,N), 
\\  
(x, f) &\mapsto  j^k(\Psi(f))(x) 
\end{align*}
is transverse to $W$, respectively transverse to $Cl(W)$, at all points 
$\{(x ,h)\}$ with $x \in U \cap \Psi(h)^{-1}(V)$. \par
If instead $j^k\gG | (Y \times \cT)$ is transverse (on each stratum $Y_i 
\times \cT$) to $W$, respectively transverse to $Cl(W)$, at all points 
$\{(x ,h)\}$ with $x \in Y \cap U \cap \Psi(h)^{-1}(V)$, then we say $\Psi$ 
restricted to $Y$ is {\em transverse} or {\em completely transverse} to 
$W$.  
\end{Definition}
Then, for a continuous $\Psi : \cH \to C^{\infty}(M,N)$ such that for each 
$h \in \cH$, $j^k(\Psi(h) | M \backslash E)$ maps into the subbundle 
$\cH^k(M \backslash E, N) = \cH^{k}(M, N) \cap J^k(M \backslash E, N)$, the 
hybrid transversality theorem takes the following form.
\begin{Thm}[Hybrid Transversality Theorem] 
\label{Thm9.1}
Let $E \subset M$ be closed with $Y \subset M \backslash E$ a closed 
Whitney stratified set.  Suppose $W \subset \cH^{k}(M \backslash E, N)$ is 
a closed Whitney stratified subset.  If both $\Psi$ and $\Psi$ restricted 
to $Y$  are transverse to $W$, then for a subset $X \subset M \backslash 
E$:
\begin{itemize}
\item[i)]
$$\cW \, = \, \{h \in \cH: j^k \Psi(h) \text{ and } j^k \Psi(h) | Y \text{ are 
transverse on $X$ to } W \text{ in } \cH^{k}(M \backslash E, N)\} $$
 is open and dense if $X$ is compact and in general is a residual set in the 
regular 
$C^{\infty}$--topology;
\item[ii)]  in particular, if $W$ is relatively Whitney stratifiable and $X$ 
is compact then 
\begin{align*}
\tilde \cW \, &= \, \{h \in \cH: j^k \Psi(h) \text{ and } j^k \Psi(h) | Y \text{ 
are completely transverse on $X$ to }  \\ 
& \hspace{.5in} W \text{ in } \cH^{k}(M \backslash E, N)\} 
\end{align*}
is open and dense in the regular $C^{\infty}$--topology.  
\end{itemize}
\end{Thm}
\par
\begin{Remark}
\label{Rem9.1}
The hybrid transversality theorem combines versions of a \lq\lq relative  
transversality theorem\rq\rq, where $j^k \Psi(\cH) = \cH^{k}(M, N)$ but 
for a Whitney stratified set $X \subset M \backslash E$, which requires 
transversality of ${_s}j_1^k(\Psi(h)) | X$ to $W$, and an \lq\lq absolute 
transversality theorem\rq\rq, where $\cH^{k}(M, N) = J^{k}(M, N)$, but only 
requires transversality for ${_s}j_1^k(\Psi(h))$, possibly on a compact 
subset (see \cite{D5}).  These are formulated so that $\Psi$ can be an 
operation involving differential or integral operators (as well as 
geometrically defined mappings) to obtain genericity results for solutions 
to PDE\rq s.  
\end{Remark}
\par
The proof of this version of the transversality theorem follows the proof 
of the relative transversality theorem \cite[Thm 1.3]{D5} modifying it to 
only require that $\Psi$ be transverse to $W$ off $E$ in $\cH^{k}(M 
\backslash E, N)$ as in the absolute transversality theorem \cite[Thm 
1.5]{D5}.  The failure of $Y$ to be smooth on $M$ does not affect the proof.  
\par 
If we compose $\Psi$ with the continuous mapping $C^{\infty}(M, N) \to 
C^{\infty}(M^s, N^s)$ defined by $f \mapsto f \times \cdots \times f$  and 
let $E = \gD^{(s)} M$ we obtain a multi-transversality version of the 
theorem (extending Corollaries 1.9 and 1.11 given in \cite{D5}).  Here 
recall from \S \ref{S:sec4} the standard notation that $X^s =  X \times X 
\times \cdots \times X$ with $s$ factors while $\gD^{(s)} X$  denotes the 
generalized diagonal of $X^s$.  However, to deduce Theorem \ref{Thm5.1} 
we need a version valid for families of mappings.  We apply Theorem 
\ref{Thm9.1} to deduce a multijet version of the hybrid transversality 
theorem for families.  \par 
Given a smooth mapping $f \in C^{\infty}(M \times Z, N)$, we let 
${_s}j_1^k(f)(x_1, \dots , x_s, z) = {_s}j^k(f(\cdot , z)(x_1, \dots , x_s)$, 
with $f(\cdot , z)$ denoting for the fixed value $z \in Z$ the function on 
$M$.  We consider a continuous mapping $\Psi : \cH \to C^{\infty}(M \times 
Z, N)$, and suppose $j_1^k(\Psi(h)) \in \cH^k(M, N)$.  We may define 
${_s}\cH^k(M, N)$ as the restriction of $(\cH^k(M, N))^s$ to $M^{(s)} \times 
N^s$, yielding 
${_s}\cH^k(M, N) \subset {_s}J^k(M, N)$.  This implies that for $h \in \cH$ 
the map ${_s}j_1^k(\Psi(h)) : M^{(s)} \times Z \to {_s}\cH^k(M, N)$.  We may 
extend Definition \ref{Def9.1} to the case of $\Psi : \cH \to C^{\infty}(M 
\times Z, N)$.  We let $E \subset M^{(s)}$ be a closed subset such that  
$M^{(s)} \backslash E$ is smooth, and let $Y \subset M^{(s)} \backslash E$ 
be a closed Whitney stratified subset.  
\begin{Definition}
\label{Def9.2}
The map $\Psi : \cH \to C^{\infty}(M \times Z, N)$ is said to be {\em 
transverse} or {\em completely transverse} to $W \subset {_s}\cH^k(M ,N)$ 
off $E$ if, given an open subset $\cU \subset \cH$, $x \in M^{(s)} 
\backslash E$, $z \in Z$, and $y \in N$, there are open sets $x \in U 
\subset M^{(s)} \backslash E$,  
$z \in U^{\prime} \subset Z$, and $y \in V \subset N$ such that for every 
map $h \in \cU$, there exists a finite-dimensional smooth manifold $ \cT 
\subset \cU$ with $h \in \cT$ such that 
for the family
\begin{align*}
\gG: M \times Z \times \cT &\rightarrow N,\\
(x, z, h) &\mapsto \Psi(h)(x, z)
\end{align*}
 the $k$-jet extension
\begin{align*}
{_s}j_1^k(\gG): (U \times U^{\prime}) \times \cT &\rightarrow 
{_s}\cH^k(M,N), \\ 
((x_1, \dots , x_s), z, h) &\mapsto {_s}j_1^k(\Psi(h))((x_1, z) \dots , (x_s, 
z)) 
\end{align*}
is transverse to $W$, respectively transverse to $Cl(W)$, relative to 
${_s}\cH^k(M \backslash E, N)$, at all points $\{((x_1, \dots , x_s), z ,h)\}$ 
with $((x_1, \dots , x_s), z) \in (U \times U^{\prime}) \cap 
\Psi(h)^{-1}(V)$.  \par 
If instead ${_s}j_1^k(\gG) | (Y \times U^{\prime}) \times \cT$ is 
transverse (on each stratum $Y_i \times \cT$) to $W$, respectively 
transverse to $Cl(W)$, at all points $\{(x ,h)\}$ with $x \in ((Y \cap U) 
\times U^{\prime}) \cap \Psi(h)^{-1}(V)$, then we say $\Psi$ restricted to 
$Y$ is {\em transverse} or {\em completely transverse} to $W$. 
\end{Definition}
Then there is the following multijet version of the Hybrid Transversality 
Theorem.  
\begin{Thm}[Hybrid Multi-Transversality Theorem]  
\label{Thm9.2}
 Let $W \subset {_s}\cH^{(k)}(M, N)$ and $Y \subset M^{(s)}\backslash E$ be 
closed Whitney stratified sets and let $X \subset M^{(s)}\backslash E$.  
Suppose that both $\Psi$ and $\Psi$ restricted to $Y$ are transverse to 
$W$ off $E$. Then, the set
\begin{align*}
  \cW &= \{h \in \cH: {_s}j_1^k \Psi(h) \text{ and } {_s}j_1^k \Psi(h)| (Y 
\times Z) \text{ are transverse on $X$ to }  \\  
& \hspace{3.1in} W \text{ in }{_s}\cH^{(k)}(M, N)\}  
\end{align*}
is a residual set, and if $X$ is compact, $\cW$ is open and dense in the 
regular $C^{\infty}$--topology.  \par 
If $W$ is relatively Whitney stratifiable and $X$ is compact, then 
\begin{align*}
\tilde \cW \, &= \, \{h \in \cH: {_s}j^k \Psi(h) \text{ and } {_s}j_1^k 
\Psi(h)| (Y \times Z) \text{ are completely transverse on $X$ to }   \\ 
& \hspace{3.5in} W \text{ in } {_s}\cH^{k}(M, N)\} 
\end{align*}
is open and dense in the regular $C^{\infty}$--topology.  
\end{Thm}
\par
The proof follows from Theorem \ref{Thm9.1} by composing $\Psi$ with 
the continuous mapping $C^{\infty}(M \times Z, N) \to C^{\infty}(M^s 
\times Z, N^s)$ defined by $f \mapsto \tilde f$  where $\tilde f(x_1, \dots 
, x_s, z) = (f(x_1, z), f(x_2, z), \dots , f(x_s, z))$ and letting $Y \times Z$ 
be used for $Y$ and  $E \cup (\gD^{(s)} M \times Z)$ be used for $E$ in 
Theorem \ref{Thm9.1}.   This yields Theorem \ref{Thm9.2}.  \par
\subsubsection*{Deducing the Transversality Theorems from the Hybrid 
Transversality Theorems} 
\par  
We will deduce Theorems \ref{Thm5.0} and \ref{Thm5.1} from the 
preceding hybrid transversality theorems.  To do so we must show that 
Theorems \ref{Thm5.0} and \ref{Thm5.1} can be placed in the frameworks 
of these transversality theorems.  
\par  
We first consider Theorem \ref{Thm5.0} and apply Theorem \ref{Thm9.2}.  
If we decompose $X_i^*$ into connected components $X_{i\, \ell}^*$, 
where $\ell$ is just an index for the components. Then the closure of each 
$X_{i\, \ell}^*$ is a compact manifold with boundaries and corners.  
We form the disjoint union $\bar X_i^* = \coprod Cl(X_{i\, j}^*)$, over the 
connected components of $X_i^*$.  Then, we let 
$M = (\bar X_i^*)^{(s)} \times \R^{n+1}$, $N = (\R)^s$, and 
$Y = \gS_Q \times \R^{n+1}$, and 
$E = \partial (\bar X_i^*)^{(s)} \times \R^{n+1}$.  Also, we let 
$\cH^k(M, N) = {_s}J^k(\bar X_i^*, \R)$.  We define 
$$  \Psi_{\gs_i} : \emb(\bgD, \R^{n+1})  \rightarrow C^{\infty}((\bar 
X_i^*)^{(s)} \times \R^{n+1}, (\R)^s) $$ 
such that $\Psi_{\gs_i}(\Phi)$ is defined by $((x_1, \dots , x_s), u) 
\mapsto (\gs_i(x_1, u), \dots \gs_i(x_s, u))$ with $\gs_i$ the distance 
function associated to $\Phi$ for $X_i$.  
Then, Theorem \ref{Thm9.2} implies the conclusion of Theorem 
\ref{Thm5.0}, once we have shown that $\Psi_{\gs_i}$ is continuous (see 
Lemma \ref{Lem9.3}) and satisfies the perturbation transversality 
condition in Definition \ref{Def9.1}.  This will be verifed in what follows.  
\par
Second, we shall also use Theorem \ref{Thm9.2} to prove Theorem 
\ref{Thm5.1}.  To do so we must define the setting which applies to the 
multi-distance and height-distance functions in each case.  Given $m > 0$ 
with an assignment $p \mapsto j_p$ and a partition $\bl = (\ell_1, \dots , 
\ell_m)$, we embed $X_{\cJ_i}^{(\bl)}$ in a closed compact manifold with 
boundaries and corners.  First, for any $j$ we let $\bar X_j = 
\coprod_{j^{\prime} \in \cJ_j} Cl(X_{j\, j^{\prime}})$, which is a 
manifold with boundaries and corners such that $\cirX_j \subset \bar 
X_j$ is the smooth interior submanifold.  Then, we let $\bar X_{\cJ_i} = 
\coprod_{j \in \cJ_i} \bar X_j$.  Then $\bar X_{\cJ_i}$ is a manifold with 
boundary and corners with interior $X_{\cJ_i}$. \par
We next extend $X_{\cJ_i}^{(\bl)}$.  For the assignment $p \mapsto j_p$, 
and partition $\bl$, we define 
\begin{equation}
\label{Eqn9.30}
\bar X_{\cJ_i}^{(\bl)} \,\, = \,\,  \{ (x^{(j_1)}, \dots , x^{(j_{m})}) \in \bar 
X_{j_1}^{(\ell_1)}\times \cdots \times \bar X_{j_{m}}^{(\ell_{m})} :  
x_1^{(j_p)} \in X_{i\, j_p} \text{ for all $p$} \} 
\end{equation}

%%%\notag \\ 
%%%& \qquad \qquad \text{ if } j_p = j_{p^{\prime}}, \text{then } 
%%% x_q^{(j_p)} \neq x_{q^{\prime}}^{(j_{p^{\prime}})} \text{for any } q, 
%%% q^{\prime}\} 
%%% $$\bar X_{\cJ_i}^{(\bl)} \,\, = \,\, \left( \prod_{p = 1}^{m} 
%%% \bar X_{j_p}^{\ell_p}\right) $$.  \par 
Then, we let $M = \bar X_{\cJ_i}^{(\bl)}$, which is a manifold with 
boundaries and corners which contains as a dense open smooth 
submanifold $X_{\cJ_i}^{(\bl)}$.  We may represent $X_{\cJ_i}^{(\bl)} = 
\bar X_{\cJ_i}^{(\bl)} \backslash E$ where $E = \partial M \cup E^{(\bl)}$
with 
\begin{align}
\label{Eqn9.30a}
E^{(\bl)} \,\, &= \,\,  \{ (x^{(j_1)}, \dots , x^{(j_{m})}) \in \bar 
X_{\cJ_i}^{(\bl)} : \text{ there are } p \neq p^{\prime}, q, 
q^{\prime} \text{ so that}  \\
& \qquad \qquad \qquad j_p = j_{p^{\prime}}, \text{and } 
x_q^{(j_p)} = x_{q^{\prime}}^{(j_{p^{\prime}})} \}.  \notag 
\end{align}
\par
Here, $E$ is a closed (stratified) set.  \par 
First, for the $i$--th multi-distance function we let $Z = 
(\R^{n+1})^{(m+1)}$, $N = \R^2$, and define 
$$\Psi_{\rho_i} : \emb(\bgD, \R^{n+1})  \rightarrow C^{\infty}(\bar 
X_{\cJ_i} \times (\R^{n+1})^{(m+1)}, \R^2) $$ 
such that for $j \in \cJ_i$
$\Psi_{\rho_i}(\Phi) | \bar X_{j_p} \times (\R^{n+1})^{(m+1)}$ applied to 
$(x, (u^{(j_1)}, \dots , u^{(j_m)}), u^{(i)})$ equals $(\gs(x, u^{(i)}), \gs(x, 
u^{(j_p)}))$.  This is the smooth extension of the multi-distance function 
$\rho_i$ associated to $\Phi$.
%%% and $\Psi_{\rho_i}(\Phi)$
%%%  to the boundary $\partial \bar X_{\cJ_i} = \coprod_{p = 1}^{m} 
%%% \partial \bar X_{j_p}$.  
\par
Likewise for the height-distance function we let $Z = (\R^{n+1})^{(m)} 
\times S^n$, $N = \R^2$, and define 
$$\Psi_{\tau} : \emb(\bgD, \R^{n+1})  \rightarrow C^{\infty}(\bar 
X_{\cJ_0} \times (\R^{n+1})^{(m)} \times S^n, \R^2) $$ 
such that $\Psi_{\tau}(\Phi) | X_{\cJ_0} \times (\R^{n+1})^{(m)} \times 
S^n) = \tau$ and $\Psi_{\tau}(\Phi)$ smoothly extends $\tau$ to the 
closure $\bar X_{\cJ_0}$.  
\begin{Lemma}
\label{Lem9.3} 
All of the $\Psi_{\gs_i}$, $\Psi_{\rho_i}$, and $\Psi_{\tau}$ are 
continuous in the Whitney topology (i.e. regular $C^{\infty}$--topology as 
source spaces are compact).  
\end{Lemma}
\begin{proof}[Proof of Lemma \ref{Lem9.3}] 
\par
We give the proof for $\Psi_{\rho_i}$ and that for $\Psi_{\gs_i}$  and 
$\Psi_{\tau}$ is similar.  \par  
As each restriction $\Phi \mapsto \Phi | \bar X_{j_p k}$ defines a 
continuous map $\emb(\bgD, \R^{n+1}) \to C^{\infty}(\bar X_{j_p k}, 
\R^{n+1})$, the disjoint union over $k$ of the maps gives a continuous map 
\begin{equation}
\label{Eqn9.4a}
\emb(\bgD, \R^{n+1}) \rightarrow C^{\infty}(\bar X_{j_p}, \R^{n+1})\, .
\end{equation} 
Second, we may form the product with the identity map on 
$(\R^{n+1})^{(m+1)}$ to yield a continuous map in the Whitney topology 
\begin{equation}
\label{Eqn9.4b}
  C^{\infty}(\bar X_{j_p}, \R^{n+1}) \rightarrow C^{\infty}(\bar X_{j_p} 
\times (\R^{n+1})^{(m+1)}, \R^{n+1} \times (\R^{n+1})^{(m+1)}) .
\end{equation}  
This follows by an argument similar to that in Proposition 3.10 
\cite[Chap. 2]{GG}.  Third, we may compose maps in $C^{\infty}(\bar 
X_{j_p} \times (\R^{n+1})^{(m+1)}, \R^{n+1} \times (\R^{n+1})^{(m+1)})$ 
with the smooth multi-distance map $\gs_j :\R^{n+1} \times 
(\R^{n+1})^{(m+1)} \to \R^2$ given by $\gs_j = (\gs \circ \pi_{m+1}, \gs 
\circ \pi_j)$.  Here $\pi_k$ denotes projection onto the $k$--th factor of 
$(\R^{n+1})^{(m+1)}$.  As $\gs_j$ is smooth, the composition with 
$\gs_j$ defines for $j = j_p$ a continuous map 
\begin{equation}
\label{Eqn9.4c}
 C^{\infty}(\bar X_{j_p} \times (\R^{n+1})^{(m+1)}, \R^{n+1} \times 
(\R^{n+1})^{(m+1)}) \rightarrow  C^{\infty}(\bar X_{j_p} \times 
(\R^{n+1})^{(m+1)}, \R^2) 
\end{equation} 
(see e.g. Proposition 3.9 of \cite[Chap. 2]{GG}). 
\par  
Then, the composition of (\ref{Eqn9.4a}), (\ref{Eqn9.4b}), and 
(\ref{Eqn9.4c}) yields a continuous map 
\begin{equation}
\label{Eqn9.5}
\emb(\bgD, \R^{n+1}) \rightarrow C^{\infty}(\bar X_{j_p} \times 
(\R^{n+1})^{(m+1)}, \R^2) . 
\end{equation}
Taking the disjoint unions of these maps over the disjoint $\bar X_{j_p} 
\times (\R^{n+1})^{(m+1)})$ for $p = 1, \dots m$ defines a continuous 
map  
\begin{equation}
\label{Eqn9.6}
\emb(\bgD, \R^{n+1}) \rightarrow C^{\infty}(\bar X_{\cJ_i} \times 
(\R^{n+1})^{(m+1)}, \R^2) . 
\end{equation}
This is $\Psi_{\rho_i}$. 
\par
\end{proof}
\par 
Then, for $\Phi \in \emb(\bgD, \R^{n+1})$, $\Psi_{\rho_i}(\Phi)$ extends 
$\rho_i$ associated to $\Phi$ and its multijet restricted to $\bar 
X_{\cJ_i}^{(\bl)} \backslash E = X_{\cJ_i}^{(\bl)}$ is contained in 
${_s}\cH^k(M\backslash E, N) = \, {_{\bl}}E^{(k)}(X_{\cJ_i},\R^2)$.  
Then, to apply Theorem \ref{Thm9.2} it remains to verify the local 
transversality condition in Definition \ref{Def9.2}.  \par
For the height-distance function, the argument will be similar except we 
let $Z = (\R^{n+1})^{(m)} \times S_n$, and ${_s}\cH^k(M \backslash E, N) = 
\, {_{\bl}}E^{(k)}(X_{\cJ_0},\R^2)$ over $X_{\cJ_0}^{(\bl)}$.  \par

\subsubsection*{Perturbation Transversality Conditions for the 
Multi-Distance and Height-Distance Functions} \hfill  
\par
We will first consider the case of the multi-distance functions and 
indicate the modifications that must be applied for the height-distance 
function.  
We consider $m > 0$ with assignmment $p \mapsto j_p$ and partition $\bl 
= (\ell_1, \dots , \ell_m)$.  Suppose 
$W \subset \, _{\bl}E^{(k)}(X_{\cJ_i},\R^2)$ is a distinguished submanifold 
(which is invariant under the action of ${_{\bl }}\cR^{+}$). Let $\Phi \in 
\emb(\bgD, \R^{n+1})$, and let $(x, u) \in X_{\cJ_i} \times 
(\R^{n+1})^{(m+1)})$ with $x = (x^{(j_1)}, \dots , x^{(j_{m})})$, $x^{(j_p)} 
=(x^{(j_p)}_1, \dots ,  x^{(j_p)}_{\ell_{j_p}})$ and $u = 
(u^{(j_1)},\dots,u^{(j_{m})}, u^{(i)})$.  Suppose that the associated 
multi-distance map $\rho_i$ satisfies $\, _{\bl}j_1^k(\rho_i)(x, u) \in W$.  
In the notation of Definition \ref{Def9.2}, we let $N = \R^2$ and $V=\R^2$.  
Then, we must give a finite dimensional submanifold $\cT \subset 
\emb(\bgD, \R^{n+1})$, with $\Phi \in \cT$, yielding
%%%  the following holds.  For an appropriate smooth manifold 
%%% $\cT\subset \emb(\bgD,\R^{n+1})$
 the finite-dimensional family 
\begin{align}
\label{Eqn9.7}
\gG: X_{\cJ_i} \times (\R^{n+1})^{(m+1)} \times \cT &\rightarrow \R^{2}  
\\
( x, u ,\Phi^{\prime}) &\mapsto \Psi_{\rho_i}(\Phi^{\prime})(x, u) = 
\rho_i^{\prime}(x, u)  \notag
\end{align}
where $\rho_i^{\prime}$ is the multi-distance function associated to 
$\Phi^{\prime} \in \cT$.  
It will have the property that the corresponding partial multi $k$-jet 
extension
\begin{align}
\label{Eqn9.8}
{_{\bl }}j_1^k\gG: X_{\cJ_i}^{(\bl)} \times (\R^{n+1})^{(m+1)} \times \cT 
&\rightarrow {_s}J^k(X_{\cJ_i},\R^{2}),\\
\;(x , u ,\Phi^{\prime}) &\mapsto {_{\bl }}j_1^k \rho_i^{\prime}(x , u) 
\notag
\end{align}
is transverse on $U^{(\bl )} \times (\R^{n+1})^{(m+1)} \times \cT$ to $W$ 
relative to ${_{\bl }}E^{(k)}(X_{\cJ_i},\R^2)$.  \par
We proceed in an analogous fashion for the height function constructing 
the submanifold $\cT$ and corresponding family
\begin{align}
\label{Eqn9.9}
\gG: X_{\cJ_0} \times (\R^{n+1})^{(m)}  \times S^n \times \cT 
&\rightarrow \R^{2}  \\
( x, u ,\Phi^{\prime}) &\mapsto \Psi_{\tau}(\Phi^{\prime})(x, u) = 
\tau^{\prime}(x, u)  \notag
\end{align}
where $\tau^{\prime}$ is the multi-distance function associated to 
$\Phi^{\prime}$.  
\par
In the next section we will construct the submanifolds $\cT$ of 
$\emb(\bgD, \R^{n+1})$ using families of perturbations, and verify the 
required conditions. 

\section{Families of Perturbations and their Infinitesimal Properties} 
\label{S:sec10} 
\par
We begin by giving a general scheme for constructing the submanifolds 
$\cT$ for families of perturbations, and then specialize to those 
perturbations arising from polynomial mappings. 
\subsection*{Construction of the Families of Perturbations} %%%  \hfill
\par
To find an appropriate finite-dimensional manifold $\Phi \in \cT \subset 
\emb(\bgD, \R^{n+1})$ which will define a finite family of perturbations, 
we modify the method used by Looijenga \cite{L} and Wall \cite{Wa} for 
constructing a family of perturbations in the case of a single distance 
function.  Then, extra care must now be taken in the multi-function case 
to verify the transversality conditions which will not result from 
submersions.  We shall carry out the local derivative computations for 
specific families of perturbations using local Monge patches and the 
algebraic representation of fibers in jet space.   \par  
As $\bgD$ is compact, the composition with a smooth embedding $\Phi : 
\bgD \to \R^{n+1}$ defines a continuous map in the Whitney topology (see 
e.g. \cite[Prop. 3.9, Chap. 2]{GG}) 
$$ \Phi^* : C^{\infty}(\R^{n+1}, \R^{n+1}) \longrightarrow  C^{\infty}(\bgD, 
\R^{n+1})\, . $$ 
If $T$ is a finite dimensional vector space of smooth mappings in 
$C^{\infty}(\R^{n+1}, \R^{n+1})$, then $\Phi^*$ restricts to give a 
continuous linear map $\Phi^*| T : T \to C^{\infty}(\bgD, \R^{n+1})$.  If 
$\Phi^*| T$ is not injective, we can always restrict to a complement to 
$\ker(\Phi^*| T)$.  Hence, we suppose $\Phi^*| T$ is injective.  For 
example, as $\Phi(\bgD) = \bgW$ has non-empty interior, if $T$ is a vector 
space of polynomial mappings, then $\Phi^*| T$ is injective.  This gives a 
submanifold $T^{\prime} = \Phi^*(T) \subset C^{\infty}(\bgD, \R^{n+1})$, as 
is the translate $T^{\prime\prime} = \Phi + T^{\prime}$.  As $\Phi \in 
\emb(\bgD, \R^{n+1})$, which is open in $C^{\infty}(\bgD, \R^{n+1})$, there 
is an open subset $\cT \subset T^{\prime\prime}$ containing $\Phi$ and 
lying in $\emb(\bgD, \R^{n+1})$.  \par 
We will use this method to construct the submanifolds $\cT \subset 
\emb(\bgD, \R^{n+1})$ which will satisfy the transversality condition for 
$\rho_i$ and the distinguished submanifolds 
$W \subset { _{\bl }}E^{(k)}(X_{\cJ_i},\R^2)$.  \par
\vspace{1ex}
\subsubsection*{Perturbation Family from Polynomial Mappings}
\par
We recall that for a point in $X_{\cJ_i}^{\bl }$ the $x^{(j_p)}$ for each $p$ 
satisfy $x^{(j_p)}_{1} \in X_{i\, j_p}$.  Then, for each $p =1, \dots , m$, we 
choose disjoint open neighborhoods 
$U_{j_p}^{(q)} \subset \R^{n+1}$ of the $x^{(j_p)}_{q}$ and let $U_{j_p} = 
\prod_{q=1}^{\ell_p}U_{j_p}^{(q)}$.  Likewise, we choose disjoint open 
neighborhoods $V_{j_p} \subset \R^{n+1}$ of $u^{(j_p)}$ and $V_i$ of 
$u^{(i)}$.  We note that if there are $j_p, j_{p^{\prime}} \in \cT_i$ such 
that 
$X_{j_p\, j_{p^{\prime}}} \neq \emptyset$, then it is possible that 
$x^{(j_p)}_{q} = x^{(j_{p^{\prime}})}_{q^{\prime}}$ for some $q$ and 
$q^{\prime}$.  
In this case we choose the neighborhoods 
$U_{j_p}^{(q)} = U_{j_{p^{\prime}}}^{(q^{\prime})} \subset X_{j_p\, 
j_{p^{\prime}}}$.  However, there will be different $u^{(j_p)}$, 
$u^{(j_{p^{\prime}})}$ and disjoint neighborhoods $V_{j_p}$ 
and $V_{j_{p^{\prime}}}$.  
\par
In each $U_{j_p}^{(q)}$ we choose a $C^{\infty}$ bump function 
$\chi_{j_p\, q}$ with support in $U_{j_p}^{(q)}$ and $\equiv 1$ in a 
neighborhood of $x^{(j_p)}_{q}$.  
For $k > 0$ let $\cP^{(k)}(\R^{n+1})$ denote the vector space of polynomial 
mappings $g : \R^{n+1} \to \R^{n+1}$ of degree $\leq k$, of the form 
$g(y) =  (g_1(y), \dots , g_{n+1}(y))$ with $y = (y_1,\dots, y_{n+1})$. 
Then we define vector spaces 
$$T^{(k)}_{j_p\, q} \, = \, \{\chi_{j_p\, q}\cdot g : g \in 
\cP^{(k)}(\R^{n+1})\}  \quad \text{ and } \quad  T \, = \, 
\times_{x^{(j_p)}_q} T^{(k)}_{j_p\, q}\, .  $$
where the product for $T$ is over distinct points $x^{(j_p)}_q$.  So if 
$x^{(j_p)}_q = x^{(j_{p^{\prime}})}_{q^{\prime}}$, then there is a single 
perturbation space $T^{(k)}_{j_p\, q} = T^{(k)}_{j_{p^{\prime}}\, 
q^{\prime}}$ which will provide perturbations for each of them in the 
single neighborhood $U_{j_p}^{(q)} = U_{j_{p^{\prime}}}^{(q^{\prime})}$.  
We shall see in \S \ref{S:sec11} that nonetheless in 
this case the appropriate perturbation transversality conditions are still 
satisfied using the perturbation family we are about to define.  
%%%T \, = \, \times_{p = 1}^{m} (\times_{q = 1}^{\ell_p} 
%%% T^{(k)}_{j_p\, q})\, . 
\par
As we have described, $\cT$ will be an open subset of $T^{\prime\prime}$ 
containing $\Phi$, giving the corresponding families (\ref{Eqn9.7}) or 
(\ref{Eqn9.9}).  Our desired $U$ will be obtained by shrinking, if necessary, 
$U = \prod_{p=1}^{m}U_{j_p}$.  
Hence, to determine the effect of elements of $T$, we observe that 
elements of $T^{(k)}_{j_p\, q}$ only affect the open set $U_{j_p}^{(q)}$ 
containing $x^{(j_p)}_{q}$.  Thus these perturbations act independently and 
we can consider their effects individually.  \par
For a single $x^{(j_p)}_{q}$, we simplify notation in the derivative 
calculations that follow and denote it by $x^{(0)}$ 
and let $y^{(0)} = \Phi(x^{(0)})  \in \cB$.  We locally represent a 
neighborhood of $\cB$ about $y^{(0)}$ using a Monge patch 
\begin{equation}
\label{Eqn10.1}
  (y_1, \dots , y_{n+1}) \,\, = \,\, (x_1,  \dots , x_n, f(x_1,  \dots , x_n))  
\end{equation}
so that $y = \Phi(x)$ locally is given by (\ref{Eqn10.1}), with $y = (y_1, 
\dots , y_{n+1})$ denoting coordinates in $\R^{n + 1}$ and $x = (x_1,  \dots 
, x_n)$ denoting local coordinates for $X$ near $x^{(0)}$.  We may also 
suppose in the case of $x^{(0)} = x^{(j_p)}_1$
that the region $\gW_i$ is below the Monge patch graph and $\gW_{j_p}$ 
above the graph.  We may obtain such a form by an orthogonal 
transformation and a translation; hence, such transformations still send 
$\cP^{(k)}(\R^{n+1})$ to itself.  
We use the following notation to describe the generators for 
$\cP^{(k)}(\R^{n+1})$.   
 Let $\bga =(\ga_1,\dots,\ga_{n+1})$ be a multi-index, $y^{\bga} = 
y^{\ga_1} \cdots y^{\ga_{n+1}}$, and use the standard notation $ |\bga | 
=\sum_{j=1}^{n+1} \ga_j$.  For the standard basis $\{e_i\}$ for $\R^{n+1}$  
we let $w_{\bga, \ell}(y) = y^{\bga} e_{\ell}$ (so the $\ell$-th coordinate 
is $y^{\bga}$, and the others are $0$). Then a basis for 
$\cP^{(k)}(\R^{n+1})$ is given by $\{ w_{\bga, \ell} : 0 \leq | \bga | \leq k 
\text{ and } 0 \leq \ell  \leq n+1\}$.  Thus, any $g(y)  \in 
\cP^{(k)}(\R^{n+1})$ may be written as $g(y) = \sum_{\ell 
=1}^{n+1}\sum_{|\bga | \leq k} \, t_{\bga, \ell}\,w_{\bga,\ell}$.  We will 
use $\bt = (t_{{\bga, \ell}})$ as coordinates for $\cP^{(k)}(\R^{n+1})$.  
\par
By restricting to a sufficiently small neighborhood of $x^{(0)}$ we may 
assume  $\chi \equiv 1$, so sufficiently near $x^{(0)}$, the elements of 
$T^{\prime}$ have the form  
\begin{equation}
\label{Eqn10.2}
\tilde \Phi \,\, = \,\, \Phi + \sum_{\ell =1}^{n+1}\sum_{|\bga | \leq k} 
t_{\bga, \ell}\,w_{\bga,\ell}\circ \Phi \, .
\end{equation}
\par

\subsection*{Computation of Derivatives for Families of Perturbations}  
%%% \hfill 
\par
In this section, we use (\ref{Eqn10.2}) to define the perturbed functions 
associated to $\tilde \Phi$ from the basic distance and height functions 
$\gs^{\prime}(y, u) = \| y - u\|^2$ and $\nu^{\prime}(y, v) = y\cdot v$ on 
$\R^{n + 1}$.  The {\em perturbed distance function} is defined by $\tilde 
\gs = \gs^{\prime} \circ (\tilde \Phi \times id)$ so that $\tilde \gs(x, u, 
\bt) = \| \tilde \Phi(x, \bt) - u\|^2$.  For a tuple $(u^{(j_1)}, \dots , 
u^{(j_m)}, u^{(i)})$, we let $\tilde \gs_{j_p} (x, \bt) = \tilde \gs (x, 
u^{(j_p)}, \bt) = \| \tilde \Phi(x, \bt) - u^{(j_p)}\|^2$ for each $p$, and also 
for $i$ in place of $j_p$.  Second, the {\em perturbed height function} 
$\tilde \nu = \nu^{\prime} \circ (\tilde \Phi \times id)$ so $\tilde \nu (x, 
v, \bt) = \tilde \Phi(x, \bt)\cdot v$.  Then, from these we define the {\em 
perturbed 
multi-distance function} $\tilde \rho_i$ by replacing in (\ref{Eqn4.2}) 
$\gs (x, u^{(j_p)})$ by $\tilde \gs (x, u^{(j_p)}, \bt)$ for all $p$, and 
similarly for $i$ in place of $j_p$.  Likewise, the {\em perturbed 
height-distance function}  $\tilde \tau$ is defined from (\ref{Eqn4.3}) by 
replacing $\nu (x, v)$ by $\tilde \nu (x, v, \bt)$, and $\gs (x, u^{(j_p)})$ by 
$\tilde \gs (x, u^{(j_p)}, \bt)$ for all $p$.  \par
For these we determine the resulting infinitesimal deformations resulting 
from the perturbation $\tilde \Phi$.  We progressively proceed from germs 
and multigerms of the distance and height functions to the 
multi-distance and height-distance functions.  For the derivatives with 
respect to $u$ we 
use orthonormal coordinates $(u_1, \dots, u_{n+1})$ in a neighborhood of 
$u^{(0)}$.  For $v^{(0)} \in S^n$, we choose for a neighborhood of $v^{(0)}$ 
an orthonormal basis $\{v_j\}$ for $T_{v^{(0)}} S^n$ (so orthogonal to 
$v^{(0)}$), with coordinates $\bs = (s_1, \dots , s_n)$ defined by $(s_1, 
\dots , s_n) \mapsto \sum_{j = 1}^{n} s_i v_i + s_{n+1} v^{(0)}$, where 
$s_{n+1} = \sqrt{1 - \sum_{j = 1}^{n} s_i^2}$.   \par 
Then, we compute
\begin{align}
\label{Eqn10.3}
 \pd{\gs^{\prime}}{y_i} =  2(y - u)\cdot e_i \qquad &\text{ and } \qquad 
\pd{\nu^{\prime}}{y_i} =  e_i \cdot v^{(0)} ;  \notag \\
\pd{\gs^{\prime}}{u_i} = -2(y - u)\cdot e_i \qquad &\text{ and } \qquad 
\pd{\nu^{\prime}}{s_i} =  y \cdot (v_i - \frac{s_i}{s_{n+1}} v^{(0)}) \, ,
\end{align}
where here and below \lq\lq $\cdot$\rq\rq denotes dot product on 
$\R^{n + 1}$. \par

Then, we use (\ref{Eqn10.3}) and the chain rule to compute the derivatives 
of the perturbations of the distance function $\tilde \gs(\cdot , u^{(0)})$ 
at $(x^{(0)}, u^{(0)})$ and the height function $\tilde \nu(\cdot , v^{(0)})$ at 
$(x^{(0)}, v^{(0)})$.  To compute the derivatives at $x = 0$, $\bs = 0$, $\bt = 
0$, and $u = u^{(0)}$, we let $z = (x, \bt, \bs, u)$ and $z_0 = (0, 0, 0, u_0)$ 
and obtain  \flushpar
\vspace{1ex}
\begin{itemize}
\item[i)] {\it derivatives of the perturbed distance function :}
\begin{align}
\label{Eqn10.4a}
\pd{\tilde \gs}{x_i}_{| z = z_0}  \,\, &= \,\,  
2 (\Phi - u^{(0)})\cdot \pd{\Phi}{x_i} \, ; \notag  \\
 \pd{\tilde \gs}{u_j}_{|  z = z_0} \,\, &= \,\,  
-2 (\Phi - u^{(0)})\cdot e_j  \, ; \notag  \\
 \pd{\tilde \gs}{t_{{\bga, \ell}}}_{| z = z_0} \,\, &= \,\,  
2 (\Phi - u^{(0)})\cdot (y^{\bga}\circ \Phi) e_{\ell}   \, ;  \notag  \\
 \pd{\tilde \gs}{s_j}_{|  z = z_0} \,\, &= \,\,  0 \, ;
\end{align}
\flushpar
and
\vspace{1ex}
\item[ii)] {\it derivatives of the perturbed height function :}
\begin{align}
\label{Eqn10.4b}
\pd{\tilde \nu}{x_i}_{|  z = z_0} \,\, &= \,\,  
\pd{\Phi}{x_i}\cdot v^{(0)}   \, ; \notag  \\
\pd{\tilde \nu}{u_i}_{|  z = z_0} \,\, &= \,\,  0  \, ; \notag  \\
\pd{\tilde \nu}{t_{{\bga, \ell}}}_{|  z = z_0} \,\, &= \,\,
(y^{\bga}\circ \Phi) e_{\ell} \cdot v^{(0)}  \, ;  \notag  \\
\pd{\tilde \nu}{s_j}_{|  z = z_0} \,\, &= \,\,  \Phi \cdot v_j \, .
\end{align}
\end{itemize}
\par
To evaluate these derivatives, we distinguish whether or not the distance 
function $\tilde \gs(\cdot , u^{(0)})$ and the height function $\tilde 
\nu(\cdot , v^{(0)})$ have critical points at $x^{(0)}$.  
For the distance function, the condition for a critical point is that 
$u^{(0)}$ is lying in the normal line to the surface, which for the Monge 
patch is the $y_{n+1}$-axis so $u^{(0)}_j = 0$ for $j \leq n$.  Then, the 
distance function will be of the form (\ref{Eqn5.3}).  Thus, it will be an 
$A_1$ point unless $u^{(0)}_{n+1} = \frac{1}{\gk_j}$ for some $j \leq n$ 
and then it is an $A_k$ point with $k \geq 2$.  We use the same notation 
as in (\ref{Eqn5.3}), so for a local maximum we must have $u^{(0)}_{n+1} = 
\frac{1}{\gk_1}$.  
\par  
For the height function, the condition is that $v^{(0)}$ is normal to the 
surface so that $v^{(0)}_j = 0$ for $j \leq n$, and moreover that $x^{(0)}$ 
is a maximum for the height function for $v^{(0)}$ requires that $v^{(0)}$ 
is in the positive $y_{n+1}$-direction.  
\par
\subsubsection*{Evaluating Derivatives of Jet Extension Mappings}
\par 
Next, we use (\ref{Eqn10.4a}) and  (\ref{Eqn10.4b}) to evaluate the 
derivatives of the jet mappings ${_s}j_1^k(\tilde \gs)$ and 
${_s}j_1^k(\tilde \nu)$.  To do so we use the local representation
of the partial multijet spaces as products in (\ref{Eqn4.4b}) of Definition 
\ref{Def4.6}. 
\begin{equation}
\label{Eqn10.5a}
 {_{\bl}}E^{(k)}(X_{\cJ_i},\R^2)  \,\, = \,\, \prod_{p=1}^{m} \, 
{_{\ell_p}}J^k(\cirX_{j_p},\R^2)\, | \, (X_{\cJ_i})^{(\bl)} \, .
\end{equation}
Furthermore, the multijet space ${_{\ell_p}}J^k(\cirX_{j_p},\R^2)$ is 
itself locally a product of jet spaces $J^k(X_{i\, j_p},\R^2)$ and the jet 
mapping acts as a product when we fix a value in $\R^{n+1}_i$ or $S^{n}$:  
\begin{align}
\label{Eqn10.5b}
 j_1^k(\tilde \gs) : U_{j_p}^{(q)} \times V_{j_p} \times T_{j_p}^{(q)} 
&\longrightarrow J^k(U_{j_p}^{(q)}, \R) \,\, \simeq \,\, U_{j_p}^{(q)} 
\times \R \times J^k(n, 1)\, , \notag   \\
j_1^k(\tilde \nu) : U_{j_p}^{(q)} \times S^n \times T_{j_p}^{(q)} 
&\longrightarrow J^k(U_{j_p}^{(q)}, \R) \,\, \simeq \,\, U_{j_p}^{(q)} 
\times \R \times J^k(n, 1) \, .
\end{align}
As usual $J^k(n, 1)$ denotes the $k$-jets of germs $\R^n, 0 \to \R, 0$, and 
$J^k(n, 1) \times \R \simeq \cC_x\backslash \itm^{k+1}$, where $\cC_x$ 
denotes the ring of germs on $(\R^{n}, 0)$ with maximal ideal $\itm$.  A 
$\cR^+$-invariant submanifold $W$ of (multi-)jet space is a bundle over 
the base with fiber $W_0$.  For a local representation of the jet bundle  as 
a product, let $\pi_f$ denote projection of the jet bundle onto a fiber.  
Then, transversality of a multijet map ${_s}j_1^k(\psi)$ to $W$ is 
equivalent to transversality of $\pi_f \circ {_s}j_1^k(\psi)$ to the fiber 
$W_0$.  We refer to $\pi_f \circ {_s}j_1^k(\psi)$ as the {\it fiber jet 
extension map}, which we denote by ${_s}j_f^k(\psi)$.  \par 
If $w$ denotes any of the local coordinates and $\psi$ denotes either the 
distance or height function, then as partial derivatives commute, the fiber 
jet extension maps can be computed by
\begin{equation}
\label{Eqn10.6}
  dj_f^k(\psi)_{| z = z_0}(\pd{ }{w}) \,\, = \,\, \pd{\psi}{w}_{|  z = z_0} 
\quad \mod \itm^{k+1}  .
\end{equation}
This reduces the verification of the transversality conditions to $W$ in 
Definitions \ref{Def9.1} and \ref{Def9.2} to verifying the transversality of 
the appropriate fiber jet extension map to the fiber $W_0$.  \par
We first use (\ref{Eqn10.4a}) and (\ref{Eqn10.4b}) in (\ref{Eqn10.6}) to 
compute the derivatives for the Monge patch in the case that $\tilde 
\gs(\cdot, u^{(0)})$ and $\tilde \nu(\cdot, v^{(0)})$ have critical points at 
$0$.  As $\Phi$ is of the form (\ref{Eqn5.2}), $\pd{\Phi}{x_i}  = e_i + 
\pd{f}{x_i}e_{n+1}$.  Evaluating (\ref{Eqn10.4a}) and (\ref{Eqn10.4b}), 
using the conditions on $u^{(0)}$ and $v^{(0)}$ for critical points, we 
obtain the following:  
\flushpar
\vspace{1ex}
\begin{itemize}
\item[i)]
{\it derivatives of the distance function at a critical point:}
\begin{align}
\label{Eqn10.6a}
\pd{\tilde \gs}{x_i}_{| z = z_0}  \,\, &= \,\,  
2 (x_i + (f - u^{(0)}_{n+1})\pd{f}{x_i}) \, ; \notag  \\
 \pd{\tilde \gs}{u_j}_{| z = z_0} \,\, &= \,\, 
\left\{  \begin{array}{lr}   
 -2 x_j &  \displaystyle 0 \leq j \leq n \\ 
-2(f - u^{(0)}_{n+1})  & \displaystyle j = n+1      
\end{array}  \right. \, ;  \notag  \\
 \pd{\tilde \gs}{t_{{\bga, \ell}}}_{| z = z_0} \,\, &= \,\,  
\left\{  \begin{array}{lr}       
 2 x_j(y^{\bga}\circ \Phi)  &  \displaystyle 0 \leq \ell \leq n \\ 
2(f - u^{(0)}_{n+1})(y^{\bga}\circ \Phi)  &  \displaystyle \ell = n+1      
\end{array}  \right. \, ;  \notag  \\
 \pd{\tilde \gs}{v_j}_{| z = z_0} \,\, &= \,\,  0 \, ;
\end{align}
 \flushpar
and
\vspace{1ex}
\item[ii)]{\it derivatives of the height function at a critical point:}
\begin{align}
\label{Eqn10.6b}
\pd{\tilde \nu}{x_i}_{| z = z_0} \,\, &= \,\,  
\pd{f}{x_i} \, ;  \notag  \\
\pd{\tilde \nu}{u_i}_{| z = z_0} \,\, &= \,\,  0  \, ;  \notag  \\
\pd{\tilde \nu}{t_{{\bga, \ell}}}_{| z = z_0} \,\, &= \,\,
\left\{  \begin{array}{lr}       
 0 &  \displaystyle 0 \leq \ell \leq n \\ 
(y^{\bga}\circ \Phi)  & \displaystyle \ell  = n+1      
\end{array}  \right.  \, ;   \notag  \\
\pd{\tilde \nu}{v_j}_{| z = z_0} \,\, &= \,\,  x_j \, .
\end{align}
\end{itemize}
\par

\subsection*{Multijet Properties Implying Stratification Properties on 
the Boundaries} %%%  \hfill
\par
Before completing the proofs of the transversality theorems, we first 
take a brief detour to use these calculations to give the proofs of several 
Lemmas referred to in \S \ref{S:sec6}.  These Lemmas were used there to 
prove certain stratification properties of the strata in the boundaries of 
the regions corresponding to the Blum types.  \par 
Specifically we consider the infinitesimal deformations of the 
multigerms of either multi-distance or height-distance functions arising 
from the variation of either $u$ or $v$ and all but one of the multigerm 
points on the boundary.  In these cases, $\bt = 0$ so we are considering 
the derivative calculations as they apply to $\gs$ and $\nu$.   \par 
We then prove that on the inverse images of certain classes of 
submanifolds $W^{(\bga)}$ under the (multi)jet maps, these functions are 
not singular.  
\begin{Lemma}
\label{Lem10.7}
Suppose the distance function $\gs$ defines the $\cR^+$--versal 
unfolding at $x = \{x^{(1)}, \dots, x^{(s)}\}$ and $u = u^{(0)}$ of the 
multigerm $g(x) = \gs(x, u^{(0)})$ of type $A_{\bga}$.  If either $s > 
1$, or $s = 1$ and $\bga = (k)$ with $k \geq 2$, then 
$$  d{_s}j_1^k(\gs)(x, u^{(0)}) (\oplus_{j = 2}^{s} T_{x^{(j)}} X \oplus 
T_{u^{(0)}} R^{n+1}) \, \cap \, T_{{_s}j_1^k(g)(x)} W^{(\bga)} \,\, = \,\, 0 \, . 
$$
\end{Lemma}
\par
There is an analogous result for the height function.  
\begin{Lemma}
\label{Lem10.8}
Suppose the height function $\nu$ defines the $\cR^+$--versal 
unfolding at $x = \{x^{(1)}, \dots, x^{(s)}\}$ and $v = v^{(0)}$ of the 
multigerm $h(x) =\nu(x, v^{(0)})$ of type $A_{\bga}$.  If either $s > 
1$, or $s = 1$ and $\bga = (k)$ with $k \geq 2$, then 
$$  d{_s}j_1^k(\nu)(x, v^{(0)}) (\oplus_{j = 2}^{s} T_{x^{(j)}} X \oplus 
T_{v^{(0)}} S^n) \, \cap \, T_{{_s}j_1^k(h)(x)} W^{(\bga)} \,\, = \,\, 0 \, . $$
\end{Lemma}
Note that these lemmas are equally valid replacing the first factor by any 
other.  \par
A third property involves the Thom-Boardman strata.
\begin{Lemma}
\label{Lem10.9}
Suppose the distance function $\gs$, resp. the height function 
$\nu$, defines the $\cR^+$--versal unfolding at $x = \{x^{(1)}, \dots, 
x^{(s)}\}$ for $u = u^{(0)}$ of the germ $g(x) =\nu(x, u^{(0)})$, resp. 
for $v = v^{(0)}$ of the germ $h(x) =\nu(x, v^{(0)})$.  If these germs 
are of Thom-Boardman type $\gS_{n, 1}$ then 
\begin{align}
\label{Eqn10.12}
dj_1^k(\gs)(x, u^{(0)}) (T_{u^{(0)}} R^{n+1}) \, \cap \, T_{j_1^k(g)(x)} 
\gS_{n, 1} \,\,&= \,\, 0 \, , \notag  \\
dj_1^k(\nu)(x, v^{(0)}) (T_{v^{(0)}} S^n) \, \cap \, T_{j_1^k(h)(x)} 
\gS_{n, 1} \,\, &= \,\, 0 \, .
\end{align}
\end{Lemma}

\begin{proof}[Proof of Lemma \ref{Lem10.7}] \par
We first consider the case when $s = 1$ and use the Monge patch given by  
(\ref{Eqn10.1}).  By (\ref{Eqn10.6a}), 
\begin{equation}
\label{Eqn10.8a} 
dj_1^k(\gs)(0, u^{(0)}) (T_{u^{(0)}} R^{n+1})  \,\, = \,\, \langle -2x_1, 
\dots -2x_n, -2(f - u^{(0)}_{n+1}) \rangle  \, . 
\end{equation}
For this Monge patch, the distance function $g(x) =\gs(x, u^{(0)})$ 
has the form (\ref{Eqn5.3}).  
\begin{align}
\label{Eqn10.9a}
g(x)\,\, &= \,\, \sum_{i=1}^n (1-u^{(0)}_{n+1} \gk_i ) x_i^2 -2u_{0, 
n+1}\sum_{|\bga| \geq 3} a_{\bga} x^{\bga}+(u^{(0)}_{n+1})^2 \\
&\qquad \qquad + \, (\sum_{i=1}^n \frac{1}{2} \gk_i x_i^2 \,  + \, 
\sum_{|\bga| \geq 3} a_{\bga} x^{\bga} )^2. \notag
\end{align}
\par
Hence, as $g$ is an $A_k$ with $k \geq 2$, $u^{(0)}_{n+1} = 
\frac{1}{\gk_i}$ for some $i$.  For simplicity we assume $i = 1$; then 
\begin{equation}
\label{Eqn10.10a} 
T \cR^+ \cdot g \,\, \subset \,\,  \langle 1 \rangle  \, + \, \itm_x (x_2, 
\dots , x_n) + \itm_x^3  \, .
\end{equation}
Then, since $\gk_1 \neq 0$, $f$ has a nonzero term $x_1^2$, and so the 
RHS of (\ref{Eqn10.8a}) is a complement of dimension $n+1$ to the 
subspace of $\cC_x$ given by the RHS of (\ref{Eqn10.10a}).  Hence, its 
intersection with the subspace $T \cR^+ \cdot g$ is $0$.  
\par
Next, suppose that $s > 1$.  We let $x_0 = (x^{(1)}_0, \dots , x^{(s)}_0) 
\subset \cirX^{(s)}$ and $g(x) =\gs(x, u^{(0)})$.  Let $g_j$ denote 
the germ of $g$ at $x^{(j)}_0$ expressed in terms of a local Monge patch in 
a neighborhood $U_j$ about $x^{(j)}_0$ with local coordinates $(x^{(j)}_1, 
\dots , x^{(j)}_n)$.  We suppose that the multigerm of $g$ at $x_0$ is of 
type $\bga = (\ga_1, \dots \ga_s)$.  To simplify notation, we let $z_j = 
j^k(g_j)(x^{(j)})$, $z = (z_1, \dots , z_s)$, $z^{\prime} = (z_2, \dots , 
z_s)$, $x^{\prime}_0 = (x^{(2)}_0, \dots , x^{(s)}_0)$ and $\bga^{\prime} = 
(\ga_2, \dots \ga_s)$.  We note that $z = {_s}j_1^k(\gs)(x_0, 
u^{(0)})$.  \par
To prove the Lemma in this case it is sufficient to show that for a set of 
generators $\{w_i\}$ of $\oplus_{j = 2}^{s} T_{x^{(j)}} X \oplus T_{u^{(0)}} 
R^{n+1}$, that $\{d{_s}j_1^k(\gs)(x_0, u^{(0)}) (w_i)\}$ are 
linearly independent in $T_{z}\, {_s}J^k(\cirX, \R)/ T_{z} W^{(\bga)}$.  
We use the set of generators 
$$\{ \pd{ }{x^{(j)}_i}, j = 2, \dots , s, i = 1, \dots , n; \pd{ }{u_i}, i = 1, 
\dots , n+1\}  .$$ \par

We first note that for any $g_j$ of type $A_{\ga_j}$, that 
$\{\pd{g_j}{x^{(j)}_1}, \dots , \pd{g_j}{x^{(j)}_n}\}$ form a regular 
sequence in $\cC_{x^{(j)}}$; and hence are linearly independent $\mod 
\itm_{x^{(j)}} (\pd{g_j}{x^{(j)}_1}, \dots , \pd{g_j}{x^{(j)}_n})$.  
Thus, in $J^k(U_j, \R) \simeq U_j \times \R \times 
(\itm_{x^{(j)}}/\itm_{x^{(j)}}^{k+1})$, 
$$ dj^k(g_j)(T_{x^{(j)}_0} X) \quad \text{is spanned by } 
\{\pd{ }{x^{(j)}_i} + \pd{g_j}{x^{(j)}_i}, i = 1, \dots , n\} ,$$
 which are linearly independent in $T_{z_j}J^k(U_j, \R)/ T_{z_j} 
W^{(\ga_j)}$.  
Thus, 
$\{ \pd{g_j}{x^{(j)}_i}, j = 2, \dots , s,\, i = 1, \dots , n\}$ are linearly 
independent in $T_{z^{\prime}}\, {_{(s-1)}}J^k(\cirX, \R)/ T_{z^{\prime}} 
W^{(\bga^{\prime})}$.  
Hence, 
$$\{ \pd{g_j}{x^{(j)}_i}, j = 2, \dots , s,\, i = 1, \dots , n;\, \pd{(g_1, \dots 
, g_s)}{u_i}, i = 1, \dots , n+1\}$$
 are linearly dependent in $T_{z}\, {_s}J^k(\cirX, \R)/ T_{z} W^{(\bga)}$ if 
and only if 
$\{\pd{(g_1, \dots , g_s)}{u_i}, i = 1, \dots , n+1\}$ together with $(1, 
\dots , 1)$ are linearly dependent in 
$$
 \cC_{x^{(1)}}/ (T\cR^+ g_1) \oplus_{j = 2}^{s} 
\cC_{x^{(j)}}/(\pd{g_j}{x^{(j)}_1}, \dots , \pd{g_j}{x^{(j)}_n}) \, . $$
\par
As $(\pd{g_j}{x^{(j)}_1}, \dots , \pd{g_j}{x^{(j)}_n}) \subset 
\itm_{x^{(j)}}$,
it is sufficient to show that $\{\pd{(g_1, \dots , g_s)}{u_i}, i = 1, \dots , 
n+1\}$ together with $(1, \dots , 1)$ are linearly independent in 
\begin{equation}
\label{Eqn10.11b}
 L \,\, = \,\,  \cC_{x^{(1)}}/ (T\cR^+ g_1) \oplus_{j = 2}^{s} 
\cC_{x^{(j)}}/\itm_{x^{(j)}} . \end{equation} 
Now
$$ \pd{g_j}{u_i} \,\, = \,\, 
\left\{  \begin{array}{lr}
-2x^{(j)}_i &  \displaystyle i \leq n, \\
-2(g_j - u^{(j)}_{0\, n+1})  & \displaystyle i = n+1      
\end{array}  \right. 
$$
where $u^{(j)}_0$ is the point $u^{(0)}$ in the coordinates $x^{(j)}$.  
Then, 
$$  \pd{(g_1, \dots , g_s)}{u_i} \, \equiv \,  (-2x_i, 0, \dots , 0)  \quad  
\mod L  \qquad \text{ for } i = 1, \dots , n$$
and 
$$  \pd{(g_1, \dots , g_s)}{u_{n+1}} \, \equiv \,  (-2(g_j - u^{(0)}_{n+1}), 2 
u^{(2)}_{0\, n+1}, \dots , 2 u^{(s)}_{0\, n+1})  \quad  \mod L \, . $$
\par  
If $\ga_1 \geq 2$, then we can apply the same argument for the case $s = 
1$ to conclude that these together with $(1, \dots , 1)$ are linearly 
independent in $L$.  \par
If instead $\ga_1 = 1$, we examine the relations for the first two 
coordinates in $L$.  Again there are two cases based on whether 
$u^{(0)} - x^{(1)}$ and $u^{(0)} - x^{(2)}$ are or are not collinear.
First, if $u^{(0)} - x^{(1)}$ and $u^{(0)} - x^{(2)}$ are not collinear, then 
they span a plane $P$ which intersects each tangent plane in a line.  We 
choose an orthonormal basis $\{e_i, i = 1, \dots , n-1\}$ for the orthogonal 
complement to the line.  For the tangent plane at $x^{(1)}$ we let $e_n$ be 
a unit vector in the line, and $e_{n+1}$ be along $u^{(0)} - x^{(1)}$.  For the 
second tangent plane we let $e_i^{\prime} = e_i, i = 1, \dots , n-1$, 
$e_n^{\prime}$ be along the line of intersection of $P$ with the tangent 
plane and $e_{n+1}^{\prime}$ be along $u^{(0)} - x^{(1)}$ so that $\{ 
e_{n}^{\prime}, e_{n+1}^{\prime}\}$ have the same orientation in $P$ as 
does $\{ e_{n}, e_{n+1}\}$.  Hence, there is a rotation by an angle $\theta$ 
sending 
$\{ e_{n}^{\prime}, e_{n+1}^{\prime}\}$ to $\{ e_{n}, e_{n+1}\}$.  By 
(\ref{Eqn10.8a}) applied to each germ $g_j$, and $i = 1, \dots , n-1$,
\begin{equation}
\label{Eqn10.11a}
d{_2}j_1^k(\gs)((x^{(1)}, x^{(2)}), u^{(0)})(e_i) \, \equiv \,\,  (-
2x^{(1)}_i, 0) \quad \mod \itm_{x^{(1)}}^2 \oplus \itm_{x^{(2)}} \, . 
\end{equation}
\par
It suffices to show that $d{_2}j_1^k(\gs)((x^{(1)}, x^{(2)}), 
u^{(0)})(e_i)$ for $i = 1, \dots , n+1$ and $(1, 1)$ are linearly independent 
in $\cC_{x^{(1)}}/ \itm_{x^{(1)}}^2 \oplus \cC_{x^{(2)}}/ \itm_{x^{(2)}}$.  
Already (\ref{Eqn10.11a}) shows this is true for $i < n$.  
For $\{ e_{n}, e_{n+1}\}$, we use (\ref{Eqn10.8a}) for $g_0$, and 
(\ref{Eqn10.8a}), but for $g_1$, together with 
\begin{align}
\label{Eqn10.12a}
 e_{n} \,\, &= \,\, \cos (\theta) e_{n}^{\prime} \, + \, \sin (\theta)  
e_{n+1}^{\prime} \, , \notag  \\
e_{n+1} \,\, &= \,\, -\sin (\theta) e_{n}^{\prime} \, + \, \cos (\theta)  
e_{n+1}^{\prime} \, .
\end{align} 
We obtain $\mod \itm_{x^{(1)}}^2 \oplus \itm_{x^{(2)}}$
\begin{align}
\label{Eqn10.13a} 
d{_2}j_1^k(\gs)((x^{(1)}, x^{(2)}), u^{(0)})(e_n) &\equiv  -2( x_n^{(1)} 
, - \sin (\theta) u^{(1)}_{n+1} ) , \notag \\ 
d{_2}j_1^k(\gs)((x^{(1)}, x^{(2)}), u^{(0)})(e_{n+1}) &\equiv  -2( - 
u^{(0)}_{n+1}, - \cos (\theta) u^{(1)}_{n+1} ) \, .
\end{align}
In the subspace spanned by $\{ (x_n^{(1)}, 0), (1, 0), (0, 1)\}$, as 
$\cos(\theta) \neq 1$, the two terms in (\ref{Eqn10.13a}) together with 
the term $(1, 1)$ are checked to be linearly independent.  These together 
with those in (\ref{Eqn10.11a}) are linearly independent in $\cC_{x^{(1)}}/ 
\itm_{x^{(1)}}^2 \oplus \cC_{x^{(2)}}/ \itm_{x^{(2)}}$.  This gives the 
result in this case.  The case when $u^{(0)} - x^{(1)}$ and $u^{(0)} - x^{(2)}$ 
are collinear is easier as we may let $e_{n}^{\prime} = e_n$ and 
$e_{n+1}^{\prime} = - e_{n+1}$ so that (\ref{Eqn10.11a}) holds for $i \leq 
n$ and (\ref{Eqn10.13a}) becomes 
$$  d{_2}j_1^k(\gs)((x^{(1)}, x^{(2)}), u^{(0)})(e_{n+1}) \equiv  -2( - 
u^{(0)}_{n+1}, u^{(1)}_{n+1}) . $$
As $u^{(0)}_{n+1}$ and $u^{(1)}_{n+1}$ have the same sign, this term and 
$(1, 1)$ are linearly independent so again the result is true.  
\par
\end{proof}
\begin{proof}[Proof of Lemma \ref{Lem10.8}] \par
The proof here is similar but much easier.  Again for $s = 1$, we use the 
Monge patch in (\ref{Eqn10.1}) so $v^{(0)} = e_{n+1}$, and $T_{v^{(0)}} S^n$ 
is spanned by the orthonormal basis $\{e_i, i = 1, \dots , n\}$.  By 
(\ref{Eqn10.6b}),
$$  \pd{\nu}{v_j}_{| (x, \bt) = (0, 0)} \,\, = \,\,  x_j \quad \text{ for 
} j = 1, \dots , n  \, .$$
Hence, its intersection with the subspace $T \cR^+ \cdot h$ is $0$.
If $s > 1$, the case $\ga_1 \geq 2$ reduces to an argument analogous to 
the case $s = 1$.  If $\ga_1 = 1$ we again consider the first two factors.  
Now we may use the same orthogonal basis for the multigerm $h$ at  
$\{x^{(1)}, x^{(2)}\}$ (using the same notation and Monge patches as in the 
preceding proof).  We obtain in place of (\ref{Eqn10.11a}) for $i = 1, \dots 
n$, 
\begin{equation}
\label{Eqn10.14a}
d{_2}j_1^k{\nu}((x^{(1)}, x^{(2)}), v^{(0)})(e_i) \,\, \equiv \,\,  
(x^{(1)}_n, 0) \quad \mod \itm_{x^{(1)}}^2 \oplus \itm_{x^{(2)}} \, . 
\end{equation}
Hence, its intersection with the subspace $T {_2}\cR^+ \cdot h = \langle 
(1, 1) \rangle \, + \, \itm_{x^{(1)}}^2 \oplus \itm_{x^{(2)}}$ is $0$.
\end{proof}
\begin{proof}[Proof of Lemma \ref{Lem10.9}] \par
The proof is similar to the first part of each of the two preceding 
Lemmas.  We give the argument for the distance function with the height 
function being analogous.  We use the Monge patch given by  (\ref{Eqn10.1}) 
with the same notation as in the previous Lemmas.  This time we use 
Mather\rq s formula (\cite[\S 5] {M6}) for the tangent space to $\gS_{n, 
1}$ for the distance function $g(x) =\gs(x, u^{(0)})$ of the form 
(\ref{Eqn10.9a}), where $u^{(0)}_{n+1} = \frac{1}{\gk_1}$.  In the fiber 
$J^2(n, 1) \times \R$,
$$ T_{j^2g}\gS_{n, 1} \oplus \R \,\, \equiv \,\, \langle 1 \rangle + \itm_x 
(\pd{g}{x_1}, \dots , \pd{g}{x_n}) + \gb(g) \quad \mod \itm_x^3 $$
where the operator $\gb$ applied to the ideal $(g)$ satisfies $$\gb(g)\,\,  
\equiv \,\, (g) + (\pd{g}{x_1}, \dots , \pd{g}{x_n})^2 \quad \mod \itm_x^3 
\, .$$  Hence, 
\begin{equation}
\label{Eqn10.15a}
 T_{j^2g}\gS_{n, 1}  \,\, \equiv \,\, \itm_x (\pd{g}{x_1}, \dots , 
\pd{g}{x_n}) + (g) \quad \mod \itm_x^3\, .
\end{equation}
Now, we follow the proof of the first part of Lemma \ref{Lem10.7} and 
apply (\ref{Eqn10.8a}) using that $f$ has a nonzero term $x_1^2$, and so 
the RHS of (\ref{Eqn10.8a}) is a complement of dimension $n+1$ to the 
subspace of $\cC_x$ given by the RHS of (\ref{Eqn10.15a}).  Hence, its 
intersection with the subspace $T_{j^2g}\gS_{n, 1}$ is $0$.  
\end{proof}

\section{Completing the Proofs of the Transversality Theorems} 
\label{S:sec11}
\par
We are ready to return to the derivative calculations of the previous 
section to complete the proofs of Theorems \ref{Thm5.0} and 
\ref{Thm5.1}.  
\subsection*{ Perturbation Transversality Conditions via Fiber Jet 
Extension Maps} 
\par
We consider a given $i$, an $m > 0$ with an assignment $p \mapsto j_p$, a 
partition $\bl = (\ell_1, \dots , \ell_p)$, and a distinguished submanifold 
$W \subset {_\bl}E^k(X_{\cJ_i}, \R^2)$.  \par
In order to verify the conditions in Definitions \ref{Def9.1} 
and \ref{Def9.2}, we consider the family of perturbations $\tilde \Phi$ 
constructed in the last section (\ref{Eqn10.2}) for a point $x = (x^{(j_1)}, 
\dots , x^{(j_{m})}) \in X_{\cJ_i}$.  We have seen that because the 
subspaces of perturbations act locally on one factor, it is sufficient to 
consider for each $x^{(j_p)} = (x^{(j_p)}_1, \dots , x^{(j_p)}_{\ell_p})$, the 
local representations in (\ref{Eqn10.5a}) and (\ref{Eqn10.5b}) in the 
neighborhoods of a point $x^{(j_p)}_q$.  \par
We establish transversality progressively by beginning first with 
$\cR^+$-invariant submanifolds of jet space, then second for 
${_s}\cR^+$-invariant submanifolds of multijet space; and then we use 
(\ref{Eqn10.6}), to establish transversality for the multi-distance and 
height-distance functions to distinguished submanifolds of partial 
multijet space.  \par

\subsubsection*{ Fiber Jet Extension Maps to Jet Spaces} \hfill 
\par
For a $\cR^+$-invariant submanifold $W$ in jet 
space, it is sufficient to verify transversality of the fiber jet extension 
map to the fiber $W_0$ in the fiber $J^k(n, 1) \times \R \simeq 
\cC_x/\itm_x^{k+1}$.  
\begin{Proposition}
\label{Prop11.1}
The fiber jet extension maps for the distance and height functions are 
submersions in all cases, except for the distance function when $u^{(0)} = 
x^{(0)}$.  In that case, it is transverse to the $\cR^+$-orbit for $A_1$-type 
singularities.  
\end{Proposition}
\begin{proof}
\par
First we use  (\ref{Eqn10.4a}) and (\ref{Eqn10.4b}) together with 
(\ref{Eqn10.6}) to compute the image of the derivative of the fiber jet 
extension maps for the basic distance and height functions.  
For $\bga$ of the form $\bga = (\ga_1, \dots \ga_n, 0)$, $y^{\bga}\circ 
\Phi = x^{\bga}$.  By (\ref{Eqn10.6a}),
\begin{equation}
\label{Eqn11.1a}
\pd{\tilde \gs_i}{t_{{\bga, \ell}}}_{| z = z_0} \,\, = \,\, 
\left\{  \begin{array}{lr}
2(x_{\ell}- u^{(0)}_{\ell})x^{\bga} &  \displaystyle \ell \leq n, \\
2(f - u^{(0)}_{n+1})x^{\bga}  & \displaystyle \ell = n+1      
\end{array}  \right. \, . 
\end{equation}
We consider two cases depending on whether $u^{(0)} =$ or $\neq x^{(0)}$. 
\par
 In the case $u^{(0)} \neq x^{(0)}$, there is some $u^{(0)}_{\ell} \neq 0$.  If 
$\ell \leq n$, then $x_{\ell}- u^{(0)}_{\ell}$ is a unit in $\cC_x$, while if  
$\ell = n+1$, then $f - u^{(0)}_{n+1}$ is a unit in $\cC_x$.  Hence, in either 
case for appropriate $\ell$, the set of elements $\{\pd{\tilde 
\gs_i}{t_{{\bga, \ell}}}_{| z = z_0}\}$, as we vary over $\bga$ with 
$\ga_{n+1} = 0$ and $| \bga | \leq k$, generate $\cC_x/\itm_x^{k+1}$ as a 
vector space.  Thus, the fiber jet extension map $j_f(\tilde \gs_i)$ is a 
submersion in a neighborhood of $(x^{(0)}, u^{(0)})$.  \par  
If instead $u^{(0)} = x^{(0)}$, then $u^{(0)}_j = 0$ for all $j$.  The set of 
elements $\{\pd{\tilde \gs_i}{t_{{\bga, \ell}}}_{| z = z_0} = 
2x_{\ell}\,x^{\bga}\}$ for $\ell \leq n$ generate as a vector space 
$\itm_x/\itm_x^{k+1}$.  Now, the distance function has an $A_1$ 
singularity at $x^{(0)}$, so $W_0 = W^{(1)} \times \R$.  As $T W_0 = 
\langle 1\rangle + \itm_x^2$, $j_f(\tilde \gs_i)$ is transverse to $W_0$ 
in a neighborhood of $(x^{(0)}, u^{(0)})$.  
Thus, in each case the condition in Definition \ref{Def9.1} is satisfied in a 
neighborhood of $z$.  \par
For the height function we perform a similar analysis using  
\begin{equation}
\label{Eqn11.1b}
\pd{\tilde \nu}{t_{{\bga, \ell}}}_{| z = z_0} \,\, = \,\, v^{(0)}_{\ell}x^{\bga} 
\, .
\end{equation}
If $x^{(0)}$ is a critical point for the height function $\tilde \nu$ for 
$v^{(0)}$, then $v^{(0)}_{n+1} = \pm 1$.  Then, by (\ref{Eqn11.1b}) the set 
of elements $\{\pd{\tilde \nu}{t_{{\bga, n+1}}}_{| z = z_0} = \pm 
x^{\bga}\}$, as we vary over $\bga$ with $\ga_{n+1} = 0$ and $| \bga | \leq 
k$, again generate as a vector space $\cC_x/\itm_x^{k+1}$.  If instead 
$x^{(0)}$ is not a critical point, then there is a $j \leq n$ such that 
$v^{(0)}_j \neq 0$.  Now, we use instead $\{\pd{\tilde \nu}{t_{{\bga, j}}}_{| 
z = z_0} = v^{(0)}_j\, x^{\bga}\}$ and obtain $\cC_x/\itm_x^{k+1}$.  In 
either case, the fiber jet extension map $j_f(\tilde \nu)$ is a submersion 
in a neighborhood of $(x^{(0)}, v^{(0)})$.  \par
Thus, the fiber jet extension maps are locally submersions in all cases, 
except for the distance map when $u^{(0)} = x^{(0)}$; and then it is 
transverse to the $\cR^+$-orbit for $A_1$-type singularities.  
\end{proof}
\par
We will refer to the  transversality condition in Definition \ref{Def9.2} as 
the {\it perturbation transversality condition}.  We see that it is satisfied 
for both the jet maps for distance and height functions.  
\subsubsection*{ Fiber Jet Extension Maps to Multijet Spaces} \hfill 
\par
Second, we extend the preceding results to fiber multijet maps.  
\begin{Proposition}
\label{Prop11.2}
The perturbation transversality condition is satisfied for multijet maps 
for both distance and height functions.
\end{Proposition}
\begin{proof}
Let $W = \prod_{j = 1}^{q} W_j \times \gD^q \R$.  First we consider the 
multigerm of the distance function $\tilde \gs(\cdot, u^{(0)})$ at $x = 
(x^{(1)}, \dots , x^{(q)}) \in X^{(q)}$.  First, suppose $u^{(0)} = x^{(j)}$ for 
some $j$.  Since $\gs(x^{(\ell)}, u^{(0)}) = \gs(x^{(j)}, u^{(0)}) = 0$ for all 
other $\ell$, and the $x^{(\ell)}$ are distinct, we can only have $q = 1$.  
Thus, we are back to the case of a single germ.  \par
Otherwise, $u^{(0)} \neq x^{(j)}$ for any $j$, and we construct a family of 
perturbations as already described.  We restrict to a product neighborhood 
$U = \prod_{j=1}^{q} U_j$ and a neighborhood $V$ of $u^{(0)}$.  We denote 
by $T_j$ the space of localized polynomial perturbations on $U_j$ and let 
$T = \times_{j = 1}^{q} T_j$ with the resulting perturbation of $\Phi$.  
First, if $u^{(0)} \neq x^{(j)}$ for any $j$, then $dj_f(\tilde \gs(\cdot, 
u^{(0)}))(x)_{| T_j}$ is a local submersion onto $J^k(n, 1) \times \R$ for 
each $j$.  Hence, the fiber multijet map $d{_q}j_f(\tilde \gs(\cdot, 
u^{(0)}))(x)_{| T}$ is a local submersion in a neighborhood of $(x, 
u^{(0)}, 0)$.  This argument applies to all points in the neighborhood.  \par
The argument for the height-distance function is similar as the fiber jet 
map on each perturbation space $T_j$ will again be a local submersion by 
the case of single germs; thus so is the fiber multijet map.  
Thus, the perturbation transversality condition is satisfied in either 
multijet case.  
\end{proof}
\begin{Remark}
Because of the surjectivity of the perturbation map when we replace 
$X_i$ by $X_i^*$, the perturbation transversality condition remains valid 
for any submanifold which is invariant under ${_q}\cR^+$.  Then, the 
general transversality theorem for multijets implies Theorem 
\ref{Thm5.0} for multijets ${_s}j_1^k( \gs_i)$ (and jets for $s = 1$).  This 
yields Mather\rq s theorem. 
\end{Remark}
\par
Next, we shall use these, along with the derivative computations for the 
multi-function cases, to verify the perturbation transversality condition 
for the partial multijet mappings. \par

\subsubsection*{ Fiber Jet Extension Maps to Partial Multijet Spaces} 
\hfill 
\par
We now consider the multi-distance and height-distance functions.  Given 
$i$, and $m > 0$ with assignment $p \mapsto j_p$,and partition $\bl = 
(\ell_1, \dots , \ell_m)$, we let 
$x = (x^{(j_1)}, \dots , x^{(j_{m})}) \in X_{\cJ_i}^{(\bl)}$ with $x^{(j_p)} = 
(x^{(j_p)}_1, \dots , x^{(j_p)}_{\ell_p})$ and where by our earlier 
definition of $ X_{\cJ_i}^{(\bl)}$, $x^{(j_p)}_1 \in X_{i \, j_p}$ for all $p$.  
For the multi-distance function we 
also consider $u = (u^{(j_1)}, \dots , u^{(j_{m})}, u^{(i)}) \in (\R^{n+1})^{(m 
+1)}$.  We use the local representation of $ X_{\cJ_i}^{(\bl)}$ and 
${_\bl}E^{(k)}(X_{\cJ_i}, \R^2)$ given in (\ref{Eqn10.5a}) and in \S 
\ref{S:sec6}.  We suppose ${_\bl}j^k(\tilde \gs_i)(x, u, \bt) \in W 
\subset {_\bl}E^{(k)}(X_{\cJ_i}, \R^2)$.  Here $W$ is a distinguished 
submanifold and has the following form by (\ref{Eqn5.0b}) contained in 
(\ref{Eqn11.0a}), where we restrict the products on the RHS to 
$X_{\cJ_i}^{(\bl)}$:
\begin{align}
\label{Eqn11.0a}
{_{\bl}}E^{(k)}(X_{\cJ_i},\R^2)  \,\, &= \,\, \prod_{p=1}^{m} \,
{_{\ell_p}}J^k(\cirX_{j_p},\R^2)    \\
&= \,\, \prod_{p=1}^{m} \left( \prod_{q=1}^{\ell_p}
(J^k(\cirX_{j_p},\R) \times J^k(\cirX_{j_p},\R)) \right) \, .   \notag 
\end{align}
Then, for fixed $u$ and $\bt$, the partial multijet of the multi-distance 
function is the product of jet maps $\prod_{p=1}^{m} 
\prod_{q=1}^{\ell_p} j_1^k(\tilde \gs_i, \tilde\gs_{j_p})$, with each 
$j_1^k(\tilde \gs_i, \tilde\gs_{j_p}) : X_{i\, j_p} \to J^k(X_{i\, j_p}, 
\R^2)$.  \par
Because the perturbations act independently in a neighborhood of each 
$x^{(j_p)}_q$ it is enough to consider two cases.  One is for $x^{(j_p)} = 
(x^{(j_p)}_1, \dots , x^{(j_1)}_{\ell_p})$ with $x^{(j_p)}_1 \in X_{i\, j_p}$.  
We then will determine the derivatives of the perturbations of each factor 
for each local multi-distance map $(\tilde \gs_i, \tilde\gs_{j_p})$.  The 
second case is when there are $j_p, j_{p^{\prime}} \in \cJ_i$ with $j_p 
\neq j_{p^{\prime}}$, and $q, q^{\prime}$, such that 
$x^{(j_p)}_q = x^{(j_{p^{\prime}})}_{q^{\prime}}$ as points in 
$X_{j_p \, j_{p^{\prime}}}$.  Then, we have an analogous situation except 
we have to simultaneously consider the perturbations applied at the 
common point for both distance functions $\tilde \gs_{j_p}$ and  
$\tilde\gs_{j_{p^{\prime}}}$.  We consider case 1, and then indicate the 
modifications for case 2. \par   
Unlike the two preceding special cases for jets and multijets, we must 
now examine the contributions to the fiber partial multijet map of the 
perturbed multi-distance map.  \par 
To simplify notation (and agree with earlier notation involving derivatives 
of distance and height functions), we consider $p$ and denote 
$x^{(j_p)}_1$ as $x^{(0)}$, 
%%%  $x^{(j_p)}_1$ as $x^{(p-1)}$ for $p >1$, 
%%%  $x^{(j_p)}_q$ as $x^{(q-1) \, \prime}$ for $q >1$, 
$u^{(i)}$ as $u^{(0)}$, and 
$u^{(j_p)}$ as $u^{(1)}$.  
\par
\subsection*{Derivatives of Multi-Distance and Height-Distance 
Functions} \hfill 
\par
We then first obtain from (\ref{Eqn10.4a}) the form of the derivatives of 
the multi-distance functions at $(x^{(0)}, u^{(0)}, u^{(1)})$ with $x^{(0)} \in 
X_{i, j}$ and $u^{(0)} \neq u^{(1)}$.  Locally we are reduced to the 
multi-distance function $\tilde \rho_i = (\tilde \gs_i,\tilde \gs_j)$, 
with $\tilde \gs_i = \tilde \gs(\cdot , u^{(0)})$, and  $\tilde \gs_j = \tilde 
\gs(\cdot , u^{(1)})$.  The derivative computations are with respect to two 
sets of orthonormal coordinates $(u_j)$ and $(u_j^{\prime})$ about the 
points $u^{(0)}$ and $u^{(1)}$, where $u^{(0)}$ and $u^{(1)}$ vary 
independently.  Here $u^{(0)} = (u^{(0)}_1, \dots , u^{(0)}_{n+1})$ and 
$u^{(1)} = (u^{(1)}_1, \dots , u^{(1)}_{n+1})$ at a point $(x^{(0)}, u^{(0)}, 
u^{(1)})$.  We compute the derivatives with respect to coordinates used 
earlier, $z = (x, s, \bt, u, u^{\prime})$ and $z_0 = (0, 0, 0, u^{(0)}, u^{(1)})$, 
where $u^{\prime}$ denotes orthogonal coordinates about $x^{(0)}$ for 
$u^{(1)}$.  \flushpar
{\it Derivatives of the multi-distance function:}
\begin{align}
\label{Eqn11.1}
 \pd{\tilde \rho_i}{x_i}_{| z = z_0} \,\, &= \,\,  \left(2 (\Phi - 
u^{(0)})\cdot\pd{\Phi}{x_i}, 2 (\Phi - u^{(1)})\cdot\pd{\Phi}{x_i}\right)  \, 
;  \notag \\  
\pd{\tilde \rho_i}{u_j}_{| z = z_0} \,\, &= \,\,  \left(-2 (\Phi - 
u^{(0)})\cdot e_j, 0 \right) \, ; \notag \\
\pd{\tilde \rho_i}{u_j^{\prime}}_{| z = z_0} \,\, &= \,\,  \left(0, -2 (\Phi - 
u^{(1)})\cdot e_j \right) \, ; \notag \\
\pd{\tilde \rho_i}{t_{{\bga, \ell}}}_{| z = z_0} \,\, &= \,\,  \left(2 (\Phi - 
u^{(0)})\cdot (y^{\bga}\circ \Phi) e_{\ell}, \, 2 (\Phi - u^{(1)})\cdot 
(y^{\bga}\circ \Phi) e_{\ell}\right) \, , 
\end{align}
where as earlier \lq\lq $\cdot$\rq\rq denotes the dot product in $\R^{n + 
1}$. 
\par
Second, we use (\ref{Eqn10.4a}) and (\ref{Eqn10.4b}) to obtain the 
derivatives of the height-distance function at a point $(x^{(0)}, u^{(0)}, 
v^{(0)})$, with $x^{(0)}$ again denoting $x^{(j_p)}_1 \in \cirX_i$, $u^{(0)} 
\in \R^{n+1}$ and $ v^{(0)} \in S^n$.  Locally we are reduced to the 
height-distance function $\tilde \tau = (\tilde \nu, \tilde \gs_i)$, with  
$\tilde \nu = \tilde \nu(\cdot , v^{(0)})$ and  $\tilde \gs_i = \tilde 
\gs(\cdot , u^{(0)})$.  The derivative computations are with respect to the 
same coordinates used earlier about $x^{(0)}$, $u^{(0)}$, and $v^{(0)}$, with 
$z = (x, s, \bt, u)$ and $z_0 = (0, 0, 0, u^{(0)})$.  \flushpar
{\it Derivatives of the height-distance function:}
\begin{align}
\label{Eqn11.2}
\pd{\tilde \tau}{x_i}_{| z = z_0} \,\, &= \,\,  \left(\pd{\Phi}{x_i}\cdot 
v^{(0)}, 2 (\Phi - u^{(0)})\cdot\pd{\Phi}{x_i}\right) \, ; \notag  \\
\pd{\tilde \tau}{u_j}_{| z = z_0} \,\, &= \,\,  \left(0, 2 (\Phi - 
u^{(0)})\cdot e_j \right) \, ;  \notag  \\ 
\pd{\tilde \tau}{t_{\bga, \ell}}_{| z = z_0} \,\, &= \,\,  \left((y^{\bga}\circ 
\Phi) e_{\ell} \cdot v^{(0)}, 2 (\Phi - u^{(0)})\cdot (y^{\bga}\circ \Phi) 
e_{\ell} \right)  \, ;  \notag  \\
\pd{\tilde \tau}{s_j}_{| z = z_0} \,\, &= \,\,  (\Phi \cdot v_j, 0) \, .
\end{align}
\par
Then we compute the contribution of these terms to the derivative of the 
fiber partial multijet map.  \par
For $\bga$ of the form $\bga = (\ga_1, \dots \ga_n, 0)$, $y^{\bga}\circ 
\Phi = x^{\bga}$.  Then, by (\ref{Eqn10.6a}) and (\ref{Eqn11.1})
$$ \pd{\tilde \rho_i}{t_{{\bga, \ell}}}_{| z = z_0} \,\, = \,\, 
\left\{  \begin{array}{lr}
2((x_{\ell}- u^{(0)}_{\ell})x^{\bga}, (x_{\ell}- u^{(1)}_{\ell})x^{\bga}) &  
\displaystyle \ell \leq n \, , \\
2((f - u^{(0)}_{n+1})x^{\bga}, (f - u^{(1)}_{n+1})x^{\bga})  & \displaystyle 
\ell = n+1  \, .     
\end{array}  \right. $$
Thus, the set of elements $\{\pd{\tilde \rho_i}{t_{{\bga, \ell}}}_{| z = 
z_0}\}$ as we vary over $\bga$ with $\ga_{n+1} = 0$ and $| \bga | \leq k$ 
generate as a vector space the submodule 
\begin{align}
\label{Eqn11.3}
K_1 \,\, &= \,\,  \cC_x\cdot \{(x_{\ell} - u^{(0)}_{\ell}, x_{\ell} - 
u^{(1)}_{\ell}), \ell = 1, \dots , n, (f - u^{(0)}_{n+1}, f -u^{(1)}_{n+1})\}  
\notag   \\  
& \hspace{2in} \quad \mod (\itm_x^{k+1} \oplus \itm_x^{k+1}) \, . 
\end{align}
Also, the vector spaces spanned by 
$$\{\pd{\tilde \rho_i}{u_{\ell}}_{| z = z_0}:  \ell = 1, \dots , n+1\},\,\,  
\text{ resp., } \,\,  \{\pd{\tilde \rho_i}{u_{\ell}^{\prime}}_{| z = z_0}:  \ell 
= 1, \dots , n+1\} $$ 
are resp. 
\begin{align}
\label{Eqn11.4}
K_2 \,\, &= \,\, \langle (x_{\ell}- u^{(0)}_{\ell}), 0), \ell = 1, \dots n, (f - 
u^{(0)}_{n+1}, 0) \rangle \, ,\notag \\
K_3 \,\, &= \,\, \langle (0, x_{\ell}- u^{(1)}_{\ell})), \ell = 1, \dots n, (0, f 
- u^{(1)}_{n+1}) \rangle \, .
\end{align}
Third, by (\ref{Eqn11.1}) we have for i =  1, \dots , n,  
$$  \pd{\tilde \rho_i}{x_i}_{| z = z_0} \,\, = \,\, 
2((x_{i}- u^{(0)}_{i} + (f - u^{(0)}_{n+1})\pd{f}{x_{i}}), (x_{i}- u^{(1)}_{i}) + 
(f - u^{(1)}_{n+1})\pd{f}{x_{i}})) .$$
Thus, $\pd{\tilde \rho_i}{x_i}_{| z = z_0} \in K_1 + K_2 + K_3$. \par 
For the height-distance function, we have a corresponding set of 
conclusions using (\ref{Eqn11.2}).  
The set of elements $\{\pd{\tilde \tau}{t_{{\bga, \ell}}}_{| z = z_0}\}$ as 
we vary over $\bga$ with $\ga_{n+1} = 0$ and $| \bga | \leq k$ generate as 
a vector space the submodule 
\begin{align}
\label{Eqn11.3a}
L_1 \,\, &= \,\,  \cC_x\cdot \{(v^{(0)}_{\ell}, x_{\ell} - u^{(0)}_{\ell}), \ell 
= 1, \dots , n, (v^{(0)}_{n+1}, f - u^{(0)}_{n+1})\}  \\  
& \hspace{2in} \quad \mod (\itm_x^{k+1} \oplus \itm_x^{k+1}) \, .  \notag  
\end{align}
Also, the vector spaces spanned by 
$$\{\pd{\tilde \tau}{u_{\ell}}_{| z = z_0}:  \ell = 1, \dots , n+1\},\,\,  
\text{ resp., } \,\,  \{\pd{\tilde \tau}{v_{\ell}}_{| z = z_0}:  \ell = 1, \dots , 
n\} $$ 
are resp. 
\begin{align}
\label{Eqn11.4a}
L_2 \,\, &= \,\, \langle (0, x_{\ell}- u^{(0)}_{\ell})), \ell = 1, \dots n, (0, f 
- u^{(0)}_{n+1}) \rangle  \, ,\\
L_3 \,\, &= \,\, \langle (x_{\ell}, 0), \ell = 1, \dots n \rangle \, . \notag 
\end{align}
We will not have need of the $\pd{\tilde \tau}{x_i}_{| z = z_0}$.

\subsection*{Verifying the Perturbation Transversality Conditions} 
%%%  \hfill
\par
With the given $i$, the $m > 0$ with assignment $p \mapsto j_p$, and the 
partition $\bl = (\ell_1, \dots , \ell_p)$ already chosen, we consider a 
compact subset $Z \subset X_{\cJ_i}^{(\bl)} \times (\R^{n +1})^{(m+1)}$.  
We want to establish the perturbation multi-transversality condition for 
a distinguished submanifold $W \subset {_\bl}E^{(k)}(X_{\cJ_i}, \R^2)$ and 
the partial multijet map 
$$ {_\bl}j_1^k(\rho_i) : X_{\cJ_i}^{(\bl)} \times (\R^{n +1})^{(m+1)} \to 
{_\bl}E^{(k)}(X_{\cJ_i}, \R^2) \,   $$
restricted to $Z$.  We may cover $Z$ by a finite number of compact sets of 
the form $(\prod_{p = 1}^{m} Z_{j_p}) \times (\prod_{j = 1}^{m +1} 
Z^{\prime}_{j})$ with $Z_{j_p} =  (\prod_{q = 1}^{\ell_p} Z^{\prime}_{q}) 
\subset X_{j_p}^{(\ell_p)}$ for each $p$ and $\prod_{j = 1}^{m +1} 
Z^{\prime\prime}_{j} \subset (\R^{n +1})^{(m+1)}$.  Thus, we may assume 
$Z$ has this form.
We consider $(x, u) \in X_{\cJ_i}^{(\bl)}\times (\R^{n +1})^{(m+1)}$ with as 
before $x = (x^{(j_1)}, \dots , x^{(j_m)})$, $x^{(j_p)}  = (x^{(j_p)}_1, \dots , 
x^{(j_p)}_{\ell_p})$, $u = (u^{(j_1)}, \dots, u^{(j_m)}, u^{(i)})$.   \par 
We suppose $(x, u) \in Z$, so each $x^{(j_p)} \in Z_{j_p}$ and each 
$u^{(j_p)} \in Z^{\prime\prime}_{j_p}$, and in additiion 
${_\bl}j_1^k(\rho_i)(x, u) \in W$.  By Theorem \ref{Thm5.0} for jets and 
multijets, we have that there is an open dense subset $\cU \subset 
\emb(\bgD, \R^{n + 1})$ such that for $\Phi \in \cU$, the conclusions of 
Mather\rq s Theorem hold for multigerms: $\gs_{j_p}$ at $S_{j_p} \subset 
X_{j_p}$ with $S_{j_p} = \{x^{(j_p)}_1, \dots , x^{(j_p)}_{\ell_p}\} \subset 
Z_{j_p}$ and $u^{(j_p)} \in Z^{\prime}_{j_p}$ for each $p$, and $\gs_{i}$ at 
$S_{i} \subset X_{i}$ with $S_{i} = \{x^{(j_1)}_1, \dots , x^{(j_m)}_1\} 
\subset (\prod_{p = 1}^{m} Z^{\prime}_{1})$ and $u^{(i)} \in 
Z^{\prime\prime}_{i}$.  \par 
We will verify that the perturbation transversality conditions to the 
distinguished submanifolds hold on $Z$ for the partial multijets 
corresponding to $\Phi \in \cU$.  It will then follow by the Hybrid 
Multi-Transversality Theorem \ref{Thm9.2}, that there is an open dense 
subset $\cU^{\prime} \subset \cU$, such that for $\Phi \in \cU^{\prime}$,  
the partial multijets ${_\bl}j_1^k(\rho_i)$ are transverse to 
distinguished submanifolds on $Z$.  Then, $\cU^{\prime}$ is then also open 
and dense in $\emb(\bgD, \R^{n + 1})$.  As this is true for each of the 
finite number of $i$, $m$, assignments $p \mapsto j_p$, partitions $\bl$, 
and distinguished submanifolds $W$, the intersection of these is still 
open and dense in $\emb(\bgD, \R^{n + 1})$.  Thus, 
($a_1$) of the Transversality Theorem \ref{Thm5.1} holds, and as 
explained earlier it follows that (b) holds for multi-distance functions.  
There will be an analogous argument for the height-distance functions, 
completing the proof of Theorem  \ref{Thm5.1}. 
\par
Then, we consider $(x, u) \in Z$ with 
${_\bl}j_1^k(\rho_i)(x, u) \in W \subset {_{\bl}}E^{(k)}(X_{\cJ_i},\R^2)$.  
Unlike the cases of simple germs and multigerms, we will not always 
have a submersion onto the jet spaces.  We first use (\ref{Eqn11.3}) and 
(\ref{Eqn11.4}) to compute the image of the subspaces of infinitesimal 
perturbations and prove transversality to $W$.  \par 
As $W$ is a distinguished submanifold, we may represent $W$ in the 
following form
$$  W = \left( \prod_{p=1}^{m} \big((\prod_{q=1}^{\ell_p} W^{(j_p)}_q) 
\times \gD^{\ell_p}\R \big) \right) \times \gD^{m}\R\, ,  $$
%%% where $\gD^{m}\R$ has nonzero entries for each $x^{(j_p)}_1$, and 
%%% each $\gD^{\ell_p}\R$ has nonzero entries for fixed $p$ and each 
%%% $x^{(j_p)}_q$.  
Also, each $W^{(j_p)}_q$ has the form
$$ W^{(j_p)}_q  \,\, = \,\, 
\left\{  \begin{array}{lr}
W^{(j_p)}_{1\, 0} \times W^{(j_p) \prime}_{1\, 0} &  \displaystyle  
x^{(j_p)}_q \in X_{i\, j_p} \text{ a singular point of } \gs_i, \\
J^k(n, 1)  \times W^{(j_p) \prime}_{q\, 0}  & \displaystyle 
\text{otherwise}      
\end{array}  \right. \, , $$
where the $W^{(j_p)}_{q\, 0}$ and $W^{(j_p) \prime}_{q\, 0}$ are the 
$\cR$-invariant submanifolds of jets of singular germs, introduced in \S 
\ref{S:sec5}.  \par
Let $S_i = \{x^{(j_1)}_1, \dots , x^{(j_m)}_1\}$ and $S_{j_p} = 
\{x^{(j_p)}_1, \dots , x^{(j_p)}_{\ell_p}\}$ for all $p$.  Then, by assumption 
$(\gs_i(x^{(j_1)}_1), \dots , \gs_i(x^{(j_m)}_1)) \in \gD^{m}\R$, and 
$(\gs_{j_p}(x^{(j_p)}_1), \dots , \gs_{j_p}(x^{(j_p)}_{\ell_p})) \in 
\gD^{\ell_p}\R$ for all $p$; thus there are common values 
$\gs_i(x^{j_p}_1) = y^{(i)}$ for all $p$, and $\gs_{j_p}(x^{(j_p)}_q) = 
y^{(j_p)}$ for all $q$.   Since $\Phi$ belongs to the open dense set $\cU$, 
the multigerms $\gs_i : X_i \times \R^{n + 1}, S_i \times \{u^{(i)}\} \to 
\R, y^{(i)}$ and $\gs_{j_p} : X_{j_p} \times \R^{n + 1}, S_{j_p} \times 
\{u^{(j_p)}\} \to \R, y^{(j_p)}$ are $\cR^+$-versal unfoldings for all cases 
except $n + 1 = 7$ when one of the points is an $\tilde E_7$ point.  In that 
case, that is the only point in the $S_i$, resp. $S_{j_p}$, and we consider 
that case  separately later.  
\par
Then for the perturbation, we are considering the partial multijet of the 
multi-distance function 
$\tilde \rho_i$ about the points $(x^{(j_1)}_1, \dots , x^{(j_m)}_1)$, 
where $\tilde \rho_i = (\tilde \gs_i, \tilde \gs_{j_p})$ about 
$x^{(j_1)}_p$.  
 Here we view $\tilde \gs_{j_p}$ as a multigerm about $(x^{(j_p)}_1, 
\dots , x^{(j_p)}_{\ell_p})$ with each $x^{(j_p)}_1 \in X_{i\, j_p}$.  \par 
There are two cases involving $x^{(0)} (= x^{(j_p)}_1)$, $u^{(0)}$, and 
$u^{(1)}$.  \flushpar
\begin{itemize}
\item[(a)] $u^{(0)} = x^{(0)}$, (and then $u^{(1)} \neq x^{(0)}$);
\item[(b)]  $u^{(0)} \neq x^{(0)}$ . 
\end{itemize}
\par
We consider each of these cases. \par
\vspace{1ex}
\flushpar
{\it Case $(a)$: } Suppose $u^{(0)} = x^{(0)}$.  As $u^{(0)} \neq u^{(1)}$, 
$x^{(0)} \neq u^{(1)}$.
Then, $\gs(x^{(0)}, u^{(0)}) = 0$ so as in an earlier case, there are no other 
$x^{(j_p)}$ so the multigerm is just a germ at $x^{(j_1)}$, with $m = 
\ell_1$.  Second, as $u^{(1)} \neq x^{(0)}$, then also $u^{(1)} \neq 
x^{(j_1)}_{q}$, for any $q$.  Hence, 
by the preceding case for multigerms, each factor of $dj_1^k(\tilde \gs( 
\cdot, u^{(1)})) _{| T_{j_1, q}}$ at $(x^{(j_1)}_q, u^{(1)}, 0)$ is a submersion 
onto $J^k(n, 1) \times \R$.  Hence, $j_f(\tilde \gs_i, \tilde \gs_{j_p})_{| 
T_{j_1, q}}$ is transverse to the factor $(J^k(n, 1) \times \R) \times W$ 
in $J^k(n, 2) \times \R^2$.  \par
For the point $x^{(0)}$, $u^{(0)}_j = 0$ for all $j$, while as $u^{(1)} \neq 
x^{(0)}$, $u^{(1)}_j \neq 0$ for some $j$.  Thus, if $j \leq n$ then $x_j - 
u^{(1)}_j$ is a unit, or if $j = n+1$, then $f -u^{(1)}_{n+1}$ is a unit.  As 
$\tilde \gs(\cdot, u^{(0)})$ has an $A_1$ singularity at $x^{(0)}$, then in 
the first factor
$$ K_2 + \langle 1 \rangle + \itm_x^2  \,\, = \,\,  \cC_x   \, . $$
Then, we apply (\ref{Eqn11.3}) and obtain 
\begin{equation}
\label{Eqn11.5}
 K_1  \,\, \equiv \,\, \langle (x_{\ell} , x_{\ell} - u^{(1)}_{\ell}), \ell = 1, 
\dots , n, (0, -u^{(1)}_{n+1})  \rangle \quad \mod (\itm_x^{k+1} \oplus 
\itm_x^{k+1})  \, .
\end{equation}
Hence,  
$$ K_1 + (\cC_x \oplus 0) \,\, \equiv \,\,  \cC_x^2 \quad \mod 
(\itm_x^{k+1} \oplus \itm_x^{k+1}) .$$
Thus, $j_f(\tilde \gs_i, \tilde \gs_{j_1})_{| T_{j_1\, 1} \times 
\R^{n+1}_i}$ is transverse in a neighborhood of $x^{(0)}$ to $W^{(j_p)}_{1\, 
0} \times W^{(j_p) \prime}_{1\, 0}$.  This applies for each stratum in 
$\bar W$.  Thus, by openness of transversality to closed Whitney 
stratified sets, $j_f(\tilde \gs_i, \tilde \gs_{j_1})_{| T \times 
\R^{n+1}_i}$ is transverse to $W$ in a neighborhood of $z_0$ for a 
neighborhood of $0 \in T$.  Thus, the perturbation transversality condition 
is satisfied. \par
Note that if we interchange the roles of $u^{(0)}$ and $u^{(1)}$, then we 
obtain the same conclusion for transversality of $j_f(\tilde \rho_i)$ at 
$x^{(j_1)}_1$ (by using $K_3$ in place of $K_2$). \par
\vspace{1ex}
\flushpar
{\it Case $(b)$:}  $x^{(0)} \neq u^{(0)}$.  \par
First, we consider when also $x^{(0)} \neq u^{(1)}$.  As the partial multijet 
map has image at $(x, u)$ in a distinguished submanifold $W$, both $\tilde 
\gs( \cdot, u^{(0)})$ and $\tilde \gs( \cdot, u^{(1)})$ have critical points at 
$x^{(0)}$.   Then, both $u^{(i)}$ lie on the normal line to the boundary at 
$x^{(0)}$.  \par  
We separately consider the multigerms of $\gs_i$ and $\gs_{j_p}$ for 
$x^{(0)} = x^{(j_p)}$.  
First, for $\gs_{j_p}$, we note that because the partial multijet has 
image of $(x, u)$ in a distinguished submanifold,  any $x^{(j_p)}_{q} \in 
X_{i\, j_p}$ is a critical point for $\gs_{j_p}$ (and satisfies 
$\gs_{j_p}(x^{(j_p)}_{q}) = \gs_{j_p}(x^{(j_p)}_{1})$).  Also,
$x^{(j_1)}_{q}$ is also be a critical point of $\gs_i$.  Thus, both $u^{(0)}$ 
and $u^{(1)} = u^{(j_p)}$ would lie in the normal line to $X_i$ at 
$x^{(j_p)}_{q}$.  However, they already lie in the 
normal line to $x^{(0)} = x^{(j_p)}_{1}$.  Hence, these normal lines are the 
same line.  Then, there is a unique point on this line that satisfies $\gs(x, 
u^{(0)}) = y^{(i)} = \gs(x^{(0)}, u^{(0)})$ and $\gs(x, u^{(1)}) = y^{(p)} = 
\gs(x^{(0)}, u^{(1)})$.  Thus, $x^{(j_p)}_{q} = x^{(j_p)}_{1}$, a contradiction. 
Hence, each $x^{(j_p)}_{q} \notin X_{i\, j_p}$ for $1 < q \leq \ell_p$.
As $\gs(x^{(0)}, u^{(0)}), \gs(x^{(0)}, u^{(1)}) \neq 0$, then $\gs(x^{(j_p)}_q, 
u^{(1)}) \neq 0$ for $q > 1$.  This is true for each $p$. 
 \par
Using the Monge patch representation at $x^{(0)}$, by the normality 
condition, both $u^{(0)}_j, u^{(1)}_j = 0$ for $j \leq n$ and $u^{(0)}_{n+1}, 
u^{(1)}_{n+1} \neq 0$.  Thus, we obtain from (\ref{Eqn11.3})
\begin{equation}
\label{Eqn11.5b}
K_1 \,\, = \,\, \itm_x\cdot \{(1, 1)\} \,\, + \,\, \cC_x\cdot \{(f - 
u^{(0)}_{n+1}, f -u^{(1)}_{n+1})\} \quad \mod (\itm_x^{k+1} \oplus 
\itm_x^{k+1}) \, .
\end{equation} 
Since $u^{(0)}_{n+1} \neq u^{(1)}_{n+1}$, then $(1, 1)$ and $(- u^{(0)}_{n+1}, 
-u^{(1)}_{n+1})$ are independent in $\R^2$; hence, by Nakayama\rq s 
Lemma,  $$  \cC_x^2  \,\, = \,\, \cC_x\{(1, 1), (f - u^{(0)}_{n+1}, f -
u^{(1)}_{n+1})\} \, .$$
Thus, by (\ref{Eqn11.5b}) $\cC_x^2$ is spanned by $K_1$ and a single 
constant term $(1, 0)$.  In addition, we see that $\{\pd{\tilde 
\rho_i}{u_{\ell}}_{| z = z_0}:  \ell = 1, \dots , n+1\}$, which
by (\ref{Eqn11.4}) spans $K_2$, includes the constant term $(-u_{n+1}, 
0)\, \mod (\itm_x \oplus \itm_x)$ (for the Monge patch).  This is true for 
each $j_p$.  Moreover, as $\gs_i : X_i \times \R^{n + 1}, S_i \times 
\{u^{(i)}\} \to \R, y^{(i)}$ is an $\cR^+$-versal unfolding, and each 
$x^{(j_p)}_1$ is a singular point of $\gs_i$,  it follows that $\{\pd{\tilde 
\gs_i}{u_{\ell}}_{| z = z_0}:  \ell = 1, \dots , n+1\}$ together with $(1, 
\dots, 1)$ must span $\oplus_{p = 1}^{m} \cC_x/\itm_x$, the constant 
terms of the summands for each $x^{(j_p)}_1$.  \par 
Thus, the factor maps $j_f(\tilde \gs_i, \tilde \gs_{j_p})_{| T_{j_p, 1}}$, 
$p = 1, \dots , m$ together with the infinitesimal variations of the 
$u^{(i)}$ and the term $(1, \dots , 1)$ (spanning $T\bgD^m\R$) span the 
tangent spaces of the fibers $(J^k(n, 1) \times \R)^2$.  By the remark at 
the end of case a), this continues to apply for any $p$ for which $u^{(j_p)} 
= x^{(j_p)}_1$. \par
Thus, 
$j_f(\tilde \gs_i, \tilde \gs_{j_1})_{| (\times_{p} T_{j_p, 1}) \times 
\R^{n+1}_{j_p}}$ is transverse to $(\times_{p} W^{(j_p)}_1) \times 
\bgD^m\R$.  
\par
We can apply a similar argument for a fixed $p$ and $\tilde \rho_i = 
(\tilde \gs_i, \tilde \gs_{j_p})$ on $S_{j_p} = \{ x^{(j_p)}_1, \dots , 
x^{(j_p)}_{\ell_p}\}$.  First, by the preceding argument, at $x^{(j_p)}_1$ 
the local perturbations from $T_{j_p, 1}$ together with $\{\pd{\tilde 
\rho_i}{u_{\ell}}_{| z = z_0}:  \ell = 1, \dots , n+1\}$ and $(1, \dots , 1)$ 
span $(J^k(n, 2) \times \R)^2$.  Next for $x^{(j_p)}_q$ with $1 < q \leq 
\ell_p$, $W^{(j_p)}_q = J^k(n, 1) \times W^{(j_p) \prime}_{0 \, q}$.  If 
$x^{(j_p)}_q \neq x^{(j_{p^{\prime}})}_{q^{\prime}}$ for $p^{\prime} \neq p$ 
and any $q^{\prime}$, then by the argument for multitransversality 
$dj_f(\tilde \gs_{j_p})(x^{(j_p)}_q)_{| ( T_{j_p, q})}$ is surjective onto 
$J^k(n, 2) \times \R$.  For each $x^{(j_p)}_q$ with $x^{(j_p)}_q = 
x^{(j_{p^{\prime}})}_{q^{\prime}}$ for some unique $p^{\prime} \neq p$ and 
$q^{\prime}$, then 
\begin{equation}
\label{Eqn11.6}
W^{(j_p)}_q \times W^{(j_{p^{\prime}})}_{q^{\prime}} \,\, = \,\,  (J^k(n, 1) 
\times W^{(j_p)\, \prime}_{0 \, q}) \times (J^k(n, 1) \times 
W^{(j_{p^{\prime}})\, \prime}_{0 \, q^{\prime}}) \, .
\end{equation} 
By the argument given for $(\tilde \gs_i, \tilde \gs_{j_p})$ applied to 
$(\tilde \gs_{j_p}, \tilde \gs_{j_p^{\prime}})$, we see that the 
corresponding $K_1$ for $T_{j_p, q}$ will span a codimension one 
subspace of $(J^k(n, 1) \times \R) \times (J^k(n, 1)\times \R)$ 
corresponding to the second and fourth factors, with $(0, 1, 0, 0)$ 
spanning the complement.  Then, again as $\gs_{j_p} : X_{j_p} \times \R^{n 
+ 1}, S_{j_p} \times \{u^{(j_p)}\} \to \R, y^{(j_p)}$ is an $\cR^+$-versal 
unfolding, the infinitesimal variations from $u^{(j_p)}$, giving 
$\{\pd{\tilde \gs_{j_p}}{u^{(j_p)}_{\ell}}_{| z = z_0}:  \ell = 1, \dots , 
n+1\}$, together with $(1, 1, \dots , 1)$ spans the sum $\oplus_{q = 
1}^{\ell_p} \cC_x/\itm_x$, the constant terms of the summands for each 
$x^{(j_p)}_q$.  
Thus, $j_f(\tilde \gs_i, \tilde \gs_{j_p})_{| 
\times_{q = 1}^{\ell_p} T_{j_p, q} \times \R^{n+1}_{j_p}}$ is transverse to 
$(\times_{q = 1}^{\ell_p} W^{(j_p)}_q \times \bgD^{\ell_p}\R) \times 
(\times_{x^{(j_{p^{\prime}})}_{q^{\prime}}} 
W^{(j_{p^{\prime}})}_{q^{\prime}})$.
where the last product is over those $x^{(j_{p^{\prime}})}_{q^{\prime}} = 
x^{(j_p)}_q$ for some $q > 1$.
Hence, taken together, the surjectivity for each $j_p$ implies that 
${_\bl}j^k_f( \tilde\rho_i)$ is transverse to $W$ at $(x, u)$, and hence 
also on a neighborhood of $(x, u)$ by the argument above.  Thus, the 
perturbation transversality condition for the partial multijet map for 
$\tilde \rho_i$ is satisfied when there are no $\tilde E_7$ points.  \par
Lastly, if say $x^{(j_1)}_1$ is an $\tilde E_7$ point for $\gs_i$, then by 
the multitransversality condition being satisfied for $\gs_i$, the 
codimension of the $\tilde E_7$ stratum implies there are no other 
singular points in $S_i$, so $m = 1$.  Then, at $x^{(j_1)}_1$ we first use 
the above arguments for $\gs_{i}$ at $x^{(j_1)}_1$ and then for 
$\gs_{j_1}$ at $S_{j_1}$ to verify the perturbation transversality 
condition.  If instead there is a $j_p$ such that $\gs_{j_p}$ has an $\tilde 
E_7$ point at $x^{(j_p)}_1$, then again by codimension conditions for the 
multigerm $\gs_{j_p}$ at $S_{j_p}$, there is only one point $x^{(j_p)}_1$, 
and just the first step of the above argument guarantees transversality to 
$W^{(j_p)}_1$, so the perturbation transversality condition is satisfied. 
\par
\begin{Remark}
\label{Rem11.1}
This last case implies by the multitransversality results for partial 
multijets that the only linking that occurs involving $\tilde E_7$ points 
either is self-linking of the form $(\tilde E_7: A^2_1)$ or simple linking 
of the form $(A^2_1:\tilde E_7, A_1^2)$.
\end{Remark}
\subsubsection*{Perturbation Transversality Conditions for 
Height-Distance Functions} \hfill
\par
We apply similar reasoning to the height-distance function.  However, we 
note that for $n + 1 = 7$, by dimension reasons, there will generically not 
be $\tilde E_7$ points for the height function $\nu$, but there can be for 
the distance functions $\gs_{j_p}$.  We consider the case where there are 
no $\tilde E_7$ points for either when $n + 1 \leq 7$, and the special 
argument for such points for $n + 1 = 7$ follows as above.  \par
Now $X_i$ is replaced by $X_0$, but we still consider an $m > 0$ with 
assignment $p \mapsto j_p \in \cJ_0$ and partition $\bl = (\ell_1, \dots , 
\ell_m)$.  We use the same notation for $(x, u)$ as above with $x = 
(x^{(j_1)}, \dots , x^{(j_m)})$, $u = (u^{(j_1)}, \dots , u^{(j_m)})$.  In 
addition we consider $v \in S^n$ with $(x, u, v) \in Z$ for a compact 
$Z \subset X_{\cJ_0} \times (\R^{n+1})^{m} \times S^n$.  As above we may 
reduce to the case where $Z$ is a product of compact subspaces of each 
factor. \par 
By the multitransversality theorem applied to both $\nu$ and $\gs_j$,  
there is an open dense subset $\cU \subset \emb(\bgD, \R^{n+1})$ such 
that for $\Phi \in \cU$ and $(x, u, v) \in Z$ : i) $\nu(\cdot, v)$ defines a 
multigerm with critical points at $S_0 = \{x^{(j_1)}_1, \dots , 
x^{(j_m)}_1\}$ where $x^{(j_p)}_1 \in  X_{0\, j_p}$ for each $p$; and $\nu : 
\cirX_0 \times\, S^n , S_0 \times \{ v\} \to \R, t_0$ is an 
$\cR^+$-versal unfolding; ii) likewise, each $\gs_{j_p}(\cdot, u^{(j_p)})$ 
defines a multigerm of critical points at $S_{j_p} = \{x^{(j_p)}_1, \dots , 
x^{(j_p)}_{\ell_p}\} \subset \cirX_{j_p}$ and $\gs_{j_p} : \cirX_0 \times 
\R^{n + 1}, S_{j_p} \times \{u^{(j_p)}\} \to \R, y^{(p)}$ is an $\cR^+$-versal 
unfolding.  \par 
First, for a given $p$, let $x^{(j_p)}_{1}$ be denoted by $x^{(0)}$ and 
$u^{(j_p)}$ by $u^{(0)}$, and the specific $v$ by $v^{(0)}$.  As the distance 
function $\gs(\cdot , u^{(0)})$ and height function 
$\nu(\cdot , v^{(0)})$ have critical points at $x^{(0)}$, it follows 
that $u^{(0)}$ is lying in the normal line to the surface, which for the 
Monge patch is the $y_{n+1}$-axis, so $u^{(0)}_j = 0$ for $j \leq n$.  
For the height function, the condition is that 
$v^{(0)}$ is normal to the surface so that $v^{(0)}_j = 0$ for $j \leq n$.  
\par  
We determine the derivative of the fiber jet map for 
the multigerm of $\tilde \tau(\cdot, u^{(0)}, v^{(0)})$ at $\{x^{(j_1)}_1, 
\dots , x^{(j_m)}_1\}$.  For the Monge patch at $x^{(0)}$, $v^{(0)}$ has 
coordinates $v^{(0)}_j = 0$, for $j \leq n$ and 
$v^{(0)}_{n+1} = 1$.  We apply (\ref{Eqn11.3a}) for $x^{(0)}$ and obtain for 
$dj_f^k(\tilde \tau)(x)_{| T^{(k)}_{0\, j_1}}$,  
$$ L_1 \,\, \equiv \,\, \cC_x \{(v^{(0)}_{n+1}, - u^{(0)}_{n+1})\} \oplus 
\itm_x \{(0, 1)\} \quad  
\mod (\itm_x^{k+1} \oplus \itm_x^{k+1}) \, . $$
This is true for each $p$.  As $v^{(0)}_{n+1} = 1$, the complement is 
spanned by the constant term $(0, 1)$, which by (\ref{Eqn11.4a}) is 
obtained from $L_2$ provided $u^{(0)} \neq x^{(0)}$ so $u^{(0)}_{n+1} \neq 
0$.  Since the terms in $L_2$  are generated by the infinitesimal 
variations $\{\pd{\tilde \tau}{u^{(j_p)}_{\ell}}_{| z = z_0}:  \ell = 1, \dots , 
n+1\}$, we also have to determine their contribution to the perturbations 
at the $x^{(j_p)}_{q}$ for all $q$.  
\par 
Again there are two cases depending on whether $u^{(0)} = $ or 
$ \neq x^{(0)}$ for each $j_p$.  \par 
If $u^{(0)} = x^{(0)}$, then we argue as in Proposition 
\ref{Prop11.1} so there is only one $x^{(j_p)}_1 = x^{(0)}$ in $S_{j_p}$.  
Also, $\gs_{j_p} = \gs(\cdot, u^{(0)})$ has an $A_1$ point at $x^{(0)}$; and 
$j_f(\tilde \gs_{j_p})$ is transverse in a neighborhood of $(x^{(0)}, 
u^{(0)})$ to $W_1$, the $\cR^+$-orbit for $A_1$-type singularities.  This 
gives the term $(0, 1)$ to $W_{j_1}$ and establishes the perturbation 
transversality condition for $j_f(\tilde \tau)$ at $x^{(j_p)}_1$.  
\par
In the case $u^{(0)} \neq x^{(0)}$, then $u^{(0)}$ is in the normal line to the 
boundary at $ x^{(0)}$.  By an analogous argument to that given earlier, 
$x^{(j_p)}_q \notin X_{0\, j_p}$ for $q > 1$.  Then, we compute 
$dj_f(\tilde \nu, \tilde \gs_{j_p})_{| \times_{q} T_{j_p\, q}}$ for 
$S_{j_p} =  \{ x^{(j_p)}_1, \dots , x^{(j_p)}_{\ell_p}\}$ with 
$\gs_{j_p}(x^{(j_p)}_q) = \gs_{j_p}(x^{(j_p)}_1)$ for all $q$.  We consider 
the other points $x^{(j_p)}_q$, with $q > 1$.  \par 
If $x^{(j_p)}_q \neq x^{(j_{p^{\prime}})}_{q^{\prime}}$ for any $p^{\prime} 
\neq p$, then by Proposition \ref{Prop11.1} for  $\tilde \gs_{j_p}$, the 
fiber jet map ${_{\ell_q}}j_f^k(\tilde \gs_{j_p})(x)_{| T_{j_p\, q}}$ is 
locally a submersion in a neighborhood of $(x^{(0)}, u^{(0)}, v^{(0)}, 0)$ onto 
$J^k(n, 1) \times \R$, providing the constant term $(0, 1)$.  Thus, 
$j_f^k(\tilde \nu(\cdot, v^{(0)}), \tilde \gs_{j_p})(\cdot, u^{(0)})_{| 
T_{j_p\, q}}$ is transverse at $x^{(0)}$ to $(J^k(n, 1) \times \R) \times 
W_{j_p}$.  As $W^{(j_1)}_1 = W^{(j_1)}_{1\, 0} \times W^{(j_1) 
\prime}_{1\, 0}$ is $\cR^+$-invariant, $(1, 1) \in  TW^{(j_1)}_1$.  
\par  
In the case there are $x^{(j_p)}_q = x^{(j_{p^{\prime}})}_{q^{\prime}}$ for 
some $p^{\prime} \neq p$, we repeat the argument given above.  Since  
$W^{(j_p)}_q \times W^{(j_{p^{\prime}})}_{q^{\prime}}$ for $\tau$ at 
$x^{(j_p)}_q$ has the same form given in (\ref{Eqn11.6}), we reduce to the 
same argument for $dj_f(\tilde \gs_i, \tilde \gs_{j_p})$ using the 
$\cR^+$-versality of $\gs_{j_p}$ at $S_{j_p}$.  
Then, $j_f(\tilde \tau, \tilde \gs_{j_p})_{| 
\times_{q = 1}^{\ell_p} T_{j_p, q} \times \R^{n+1}_{j_p}}$ is transverse to 
$(\times_{q = 1}^{\ell_p} W^{(j_p)}_q \times \bgD^{\ell_p}\R) \times 
(\times_{x^{(j_{p^{\prime}})}_{q^{\prime}}} 
W^{(j_{p^{\prime}})}_{q^{\prime}})$,
where the last product is over those $x^{(j_{p^{\prime}})}_{q^{\prime}} = 
x^{(j_p)}_q$ for some $q > 1$.
\par
Hence, by taking the product of the factor maps, we deduce that 
${_\bl}j_f^k(\tilde \tau)(x)_{| T}$ is transverse to $W$.  Again this holds 
for the strata in $\bar W$ so it remains true in a neighborhood of $z_0$ 
for $\bt$ in a neighborhood of $0 \in T$.  We conclude that the 
perturbation transversality 
condition is satisfied.  \par

\subsection*{Concluding the Proofs} 
%%%  \hfill
\par
It remains to deduce Theorems \ref{Thm5.0} and 
\ref{Thm5.1}.   By Lemma \ref{Lem9.3}, $\Psi_{\gs}$, 
$\Psi_{\rho_i}$, and $\Psi_{\tau}$ are continuous, and perturbation 
transversality conditions are satisfied for each of them for compact 
subsets $Z$ and distinguished submanifolds $W$ in the cases : $Z \subset 
X_{\cJ_i} \times \R^{n + 1}$ and $W \subset {_s}J^k(X^*_i, \R)$; $Z \subset 
X_{\cJ_i}^{(\bl)} \times (\R^{n + 1})^{(m+1)}$ and $W \subset 
{_\bl}E^{(k)}(X_{\cJ_i}, \R^2)$; and $Z \subset X_{\cJ_0}^{(\bl)} \times 
(\R^{n + 1})^{(m)} \times S^n$ and $W \subset {_\bl}E^{(k)}(X_{\cJ_0}, 
\R^2)$.  Then, we directly apply the Hybrid Multi-Transversality Theorem 
\ref{Thm9.2} for $Y = \emptyset$ to conclude that the subsets 
$\cW$ and $\tilde \cW$ in Theorems \ref{Thm5.0} and \ref{Thm5.1} are 
open and dense.  We have already explained how to obtain from these 
results the remaining results in the theorems with the exception of 
establishing the transversality of ${_s}j_1^k(\Phi(h)) | \gS_{Q}$ in 
Theorem \ref{Thm5.0}.
\par
To complete the proof for this claim, we enlarge 
the family of perturbations.  In addition to the space $T$ of localized 
polynomial perturbations, we also introduce a finite dimensional manifold 
of diffeomorphisms of $X_i$ to account for the stratified set $\gS_Q$ in 
Theorem \ref{Thm5.0}.  We also choose an open neighborhood $\gS_Q 
\subset U \subset X_i \backslash sing(X_i)$.  Then, as $\gS_Q$ is 
compact, we can use the isotopy theorem to find a finite dimensional 
manifold $T_0$ and a finite parameterized family of diffeomorphisms of 
$X_i$, $\{ \phi(\cdot, t) : X_i \times T_0 \to X_i\}$, such that $\phi(\cdot, 
t_0) = id$, each $\phi(\cdot, t) \equiv id$ for $x \notin U$, and for each 
$x_0 \in \gS_Q$, $\{ \pd{\phi(x_0, t)}{v} : v \in T_{t_0}T_0\} = 
T_{x_0}X_i$.   Then together with $T$ we form $T^{\prime \prime} = T 
\times T_0$, and map $(\bt, t_1) \mapsto \tilde \Phi \circ (\phi(\cdot, 
t_1) \times id)$.  Then, if ${_s}j_1^k(\tilde \gs(\Phi^{\prime}, u_0)$ is 
transverse to $W \subset {_s}J^k(X_i, \R)$ for $\Phi^{\prime} \in 
T^{\prime}$, then by the parametrized transversality theorem, for almost 
all $t_1 \in T_0$, ${_s}j_1^k(\tilde \gs(\Phi^{\prime}\circ \phi(\cdot, 
t_1), u_0)$ is transverse to $W$.  Thus, we have established the 
perturbation transversality conditions for both $\Psi$ and $\Psi$ 
restricted to $\gS_{Q} \times \R^{n+1}$.  Then, Theorem \ref{Thm5.0} 
follows from Theorem \ref{Thm9.2}. \par 
\newpage
\section{ \bf Appendix:  List of Frequently Used Notation} 
\label{notationTable} 
\par 
Here we give a list of frequently used notation and the section in which it 
is introduced.
\par

\begin{longtable}{lp{9cm}lp{9cm}}
   Symbol &   Meaning     \\  
\hline\\

\endhead
\multicolumn{2}{c}{\textbf{\S 2}}\\\\

$\Omega_i$ & compact connected region in $\R^{n+1}$ that is a smooth 
manifold with 
boundaries and corners\\
$\cB_i$ & boundary of $\Omega_i$\\
$\R^k_+$ & $\{(x_1,\ldots,x_k) \in \R^{k}:x_i \geq 0\}$ \\
$C_k$ & $\R^k_+ \times \R^{n+1-k}$ \\
$L_k$& $\{y \in \R^{k+1}:\sum_{i=1}^{k+1}y_i=0\}$ \\
$Y_k$ & $\{y \in L_k:\text{for some } i \neq j, y_i=y_j \leq y_l, l \neq 
i,j\}$\\
$P_k$ & $Y_k \times \R^{n+1-k}$\\
$Z_{n+1}$& hyperplane $x_{n+1}=0$ \\
$H_{n+1}$& half-space $x_{n+1}\geq 0$ \\
$Q_k$ & $Z_{n+1} \cup (H_{n+1}\cap P_k)$  \\
$\Sigma_{Q_i}$& compact Whitney stratified set in $\Omega_i$ 
consisting of 
smooth $Q_k$ points\\
$\Sigma_Q$ & $\cup_i \Sigma_{Q_i}$\\
$\boldsymbol{\Omega}$& multi-region configuration consisting of regions 
$\{\Omega_i\}_{i=1}^m$ \\
$\Omega_0$ & closure of the complement $\R^{n+1}\setminus 
\cup_{i=1}^m \Omega_i$\\
$\boldsymbol \Delta$& model configuration for a multi-region 
configuration \\
$\emb(\boldsymbol \Delta,\R^{n+1})$& space of smooth embeddings giving 
configurations of type $\boldsymbol \Delta$ \\
$\Phi$ & smooth embedding $\boldsymbol \Delta \rightarrow \R^{n+1}$ \\
$\Delta_i$& model for $\Omega_i$\\
$X_i$& boundary of $\Delta_i$ \\
$X$ & $\cup_i X_i$\\
int & interior\\
Cl & closure\\
$\tilde{\Omega}$ & bounding region containing the configuration $\bgW$ 
\\\\

\multicolumn{2}{c}{\textbf{\S 3}}\\\\

$M_i$ & skeletal set that is Whitney stratified \\
$U_i$ & multi-valued radial vector field\\
$\bu_i$ & unit radial vector field\\
$(M_i,U_i)$& skeletal structure \\
$M_{i\,sing}$& singular strata of $M_i$ \\
$M_{i\,reg}$& smooth strata of $M_i$\\
$\bdyM_i$ & singular strata where $M_i$ locally a manifold with 
boundary\\
$\barbdyM_i$ & closure of $\bdyM_i$\\
$\tilde{M_i}$ & double of $M_i$ \\
$N$ & normal line bundle on a skeletal set\\
$N_+$ & half-line bundle, $\{c\cdot U_i:c \geq 0\}$\\
$r_i$ & $||U_i||$, radial function \\
$\ell_i$ & linking function\\
$L_i$ & linking vector field\\
$\{(M_i,U_i,\ell_i)\}$ & skeletal (or medial) linking structure \\
$\cS_i$ & labeled refinement of stratification of $\tilde{M}_i$ \\
$S_{ik}$ & stratum of $\cS_i$ \\
$M_0$ & linking axis in the complement\\
$W_{ij}$ & strata of linking axis \\
$M$ & union of all $M_i$ \\
$\tilde M$ & union of all $\tilde M_i$ \\
$\cB$ & union of all $\cB_i$\\
$\pi_i$ & canonical projection $\tilde{M_i}\rightarrow M_i$\\
$M_{i\infty}$ & unlinked portion of $\tilde{M}_i$ \\
$M_{\infty}$ & union of all $M_{i \infty}$ \\
$\cB_{i \infty}$ & region on $\cB_i$ corresponding to $M_{i \infty}$\\
$\cB_{\infty}$ & union of regions on $\cB$ corresponding to $M_{\infty}$ 
\\
$\Omega_{i \infty}$ & region of $\Omega_i$ corresponding to $M_{i 
\infty}$\\
$\Omega_{\infty}$ & $\cup_i \Omega_{i \infty}$\\
$\lambda_i$ & linking flow from $M_i$ \\
$\lambda$ & total linking flow on $M$ \\
$\lambda_t$ & linking flow for fixed $t$, $\lambda(\cdot,t)$ \\
$\cB_{i t}$ & level set of the linking flow at time $t$\\\\

\multicolumn{2}{c}{\textbf{\S 4}}\\\\

$\gs$ & distance-squared function on $\R^{n+1} \times \R^{n+1}$ \\
$\rho$ & $\sigma|(\cB \times \text{int}(\Omega))$ \\
$A_k$ & type of singularity (single germ) \\
$\bA_{\ga}$ & multigerm singularity type \\
$(\bA_{\ga}:\bA_{\gb_1}, \cdots , \bA_{\gb_k})$ & linking configuration 
\\
$(\bA_{\ga}:\bA_{\bgb})$ & linking configuration where $\bA_{\bgb} = 
(\bA_{\gb_1}, \cdots , \bA_{\gb_k})$\\
$\Sigma^{(\underline{\alpha})}_{M_i}$ & stratum consisting of 
$A_{\alpha}$-type points in $M_i$\\
$\Sigma^{(\underline{\alpha})}_{\cB_i}$ &  stratum in $\cB_i$ 
corresponding to $\Sigma^{(\underline{\alpha})}_{M_i}$  \\
$S^n$ & unit sphere in $\R^{n+1}$ \\
$\nu$ & height function on $\R^{n+1}\times S^n$ \\
$\tau$ & $\nu|(\cB \times S^n)$ \\
$\cZ$ & spherical axis \\
$\Sigma^{(\underline{\alpha})}_{\cZ}$ & stratum consisting of 
$A_{\alpha}$-type points in $\cZ$ \\
$h$ & height function associated to $\cZ$\\
$V$ & multi-valued vector field on $\cZ$\\
$(\cZ,h,V)$ & spherical structure \\
$\Sigma^{(\alpha:\beta)}_{M_0}$ & set of points in $\Omega_0$ exhibiting 
generic Blum linking properties \\
$\Sigma^{(\alpha:\beta)}_{M_i}$ & corresponding set in $M_i$ \\
$\Sigma^{(\alpha:\beta)}_{\cB_{j}}$ & corresponding set in $\cB_{j}$ \\\\
%%%   &    \\\\
\multicolumn{2}{c}{\textbf{\S 7}}\\\\

$S_{rad}$& radial shape operator \\
$S_{\boldsymbol v}$ & matrix representation of $S_{rad}$\\
$S_E$ & edge shape operator \\
$S_{E\boldsymbol v}$ & matrix representation of $S_{E}$\\
$I_{n-1,1}$ & $n\times n$ diagonal matrix with 1's in first $n-1$ diagonal 
positions, 0 in last\\
$T_{x_0}M_i$ & tangent space to $M_i$ at $x_0 \notin \bdyM_i$ \\
$T_{x_0}\bdyM_i $ & tangent space to $M_i$ at $x_0 \in \barbdyM_i$ \\
$\eta_U$ & compatibility $1$-form \\
$\kappa_{ri}$ & principal radial curvature on $M_i$ \\
$\kappa_{Ei}$ & principal edge curvature on $M_i$ \\
$\psi$ & radial flow \\
$\psi_t$ & radial flow at time $t$\\
$U_0$ & radial vector field on $M_0$ \\
$\bu_0$ & unit radial vector field on $M_0$ \\\\

\multicolumn{2}{c}{\textbf{\S 9}}\\\\

$M_{i\rightarrow j}$ & strata of $\tilde{M}_i$ linked to $\tilde{M}_j$ \\
$\Omega_{i \rightarrow j}$ & region of $\Omega_i$ linked to 
$\Omega_j$\\
$\cN_{i \rightarrow j}$ & linking neighborhood of $\Omega_i$ linked to 
$\Omega_j$ \\
$\cN_i$ & total linking neighborhood of $\Omega_i$\\
$\cN_{i \infty}$ & region in $\Omega_0$ corresponding to $M_{i \infty}$\\ 
$\cB_{i \rightarrow j}$ & boundary region of $\cB_i$ linked to $\cB_j$ \\
$\cR_{i \rightarrow j}$ & total region for $\Omega_i$ linked to 
$\Omega_j$ \\
$\cB_{i0}$ & portion of $\cB_i$ not shared with other regions\\\\

\multicolumn{2}{c}{\textbf{\S 10}}\\\\

$\pi$ & canonical projection $\tilde{M}\rightarrow M$\\
$g$ & multi-valued function on $M$ \\
$\tilde{g}$ & $g \circ \pi$\\
$dM_i$ & skeletal (or medial) Borel measure on $\tilde{M}_i$ \\
$\psi_{i1}$ & radial map $\tilde{M}_i\rightarrow \cB_i$\\
$R_i$ & Borel measurable region of $\cB_i$ \\
$\tilde{R}_i$ & Borel measurable region of $\tilde{M}_i$ \\
$\cB_{sing}$ & singular points of $\cB$\\\\

\multicolumn{2}{c}{\textbf{\S 11}}\\\\

$\ell_i'$ & truncated linking function on $M_i$\\
$L_i'$ & truncated linking vector field on $M_i$ \\
$c_{i\rightarrow j}$ & closeness measure from $\gW_i$ linked to 
$\gW_j$\\
$c_{ij}$ & multiplicative closeness measure between $\gW_i$ and 
$\gW_j$\\
$c_{ij}^a$ & additive closeness measure between $\gW_i$ and $\gW_j$\\
$s_i$ & positional significance measure\\
$\Gamma$ &  graph \\
$\Gamma_b$ & subgraph \\
$\Lambda$ & tiered linking graph\\\\

\multicolumn{2}{c}{\textbf{\S 12}}\\\\

$X_{ij}$& union of smooth strata of $X_i \cap X_j$\\
$X_{i0}$& $X\setminus \cup_{j>0}\text{Cl}(X_{ij})$\\
$\cJ_i$& index set, $\{j \neq i:X_{ij}\neq \emptyset\}$ \\
$\cJ_0$& index set,  $\{j > 0:X_{0j}\neq \emptyset\}$  \\
$q_i$& cardinality of $\cJ_i$ \\
$\cirX_i$ & set of points in smooth strata of $X_i$\\
$p \mapsto j_p$ & assignment function for $1 \leq p \leq m$ and $j_p \in 
\cJ_i$ \\
$X_{\cJ_i}$& union of $\cirX_j$ sharing strata with $X_i$\\
\quad & or for assignment function, the disjoint union of $X_{j_p}$ \\
$X_i^*$ & set of smooth points of $X_i$ obtained by removing $P_k$ and 
singular $Q_k$ points\\
$Y^r$ & $Y \times \ldots \times Y$ ($r$ times) \\
$\Delta^rY$ & diagonal in $Y^r$\\
$\Delta^{(r)}Y$ & generalized diagonal in $Y^r$\\
$Y^{(r)}$ & $Y^r \setminus \Delta^{(r)}Y$ \\
$\R^{n+1}_j$ & copy of $\R^{n+1}$ indexed by $j$\\
$(\R^{n+1})^{(q)}$ & complement of $\Delta^{(q)}\R^{n+1}$ \\
$\gs$ & distance-squared function on $X \times \R^{n+1}$ \\
$\gs_i$ & distance-squared function on $X^*_i \times \R^{n+1}$ \\
$\rho_i$ & multi-distance function on $X_{\cJ_i} \times 
(\R^{n+1})^{(m+1)}$ \\
$\tau$ & height-distance function on $X_{\cJ_0} \times 
(\R^{n+1})^{(m+1)}\times S^n$\\
${_s}J^k(X,\R^2)$ & $k-$multijet space \\
${_{\bl}}E^{(k)}(X_{\cJ_i},\R^2)$ & partial $\ell$-multi $k-$jet subspace 
\\
${_{\ell}}j^kf$ & partial multijet map of parametrized family\\
${_{\ell}}j_1^kf$ & ${_{\ell}}j^kf$ for fixed parameter values\\\\

\multicolumn{2}{c}{\textbf{\S 13}}\\\\

$\overline{\pitchfork}$ & transverse to submanifold or Whitney stratified 
set\\
${_{\ell}}S$ & submanifolds/stratified sets in jet spaces defining generic 
properties of Blum linking structures\\
$W^{(\alpha)}$ & $\cR^+-$orbit of multigerms of type $\boldsymbol 
A_{\alpha}$\\
$W_0^{(\alpha)}$ & fiber of orbit\\
$W^{(\alpha:\beta)}$ & submanifold of ${_{\bl}}E^{(k)}(X_{\cJ_i},\R^2)$ 
corresponding to linking configuration $(\boldsymbol 
A_{\alpha}:\boldsymbol A_{\beta})$\\\\

\multicolumn{2}{c}{\textbf{\S 14}}\\\\

$S(i,\ell)$& closed stratified sets in ${_{\ell}}S$  \\
$\cP_{i \sigma}$& set of all embeddings such that ${_s}j_1^k\sigma_i$ is 
transverse to every element of $S(i,\ell)$  \\
$\cP_{\sigma}$ & $\cap_i \cP_{i \sigma}$\\
$\cP_{i \rho}$ & set of all embeddings such that ${_s}j_1^k\rho_i$ is 
transverse to every element of $S(i,\ell)$  \\
$\cP_{\rho}$ & $\cap_i \cP_{i \rho}$ \\
$\cP_{\rho\sigma}$ & $\cP_{\rho}\cap \cP_{\sigma}$ \\
$\cP_{i \tau}$ & set of all embeddings such that ${_s}j_1^k\tau$ is 
transverse to every element of $S(i,\ell)$  \\
$\cP_{\tau}$ & $\cap_i \cP_{i \tau}$ \\
$\cP$ & $\cP_{\rho \sigma} \cap \cP_{\tau}$\\\\

\multicolumn{2}{c}{\textbf{\S 15}}\\\\

$\cC_{x^{(j)}}$ & ring of germs of functions at $x^{(j)}$ with maximal 
ideal $m_{x^{(j)}}$\\
$\cC_{x^{(j)},u}$ & ring of germs of functions at $(x^{(j)},u^{(0)})$ with 
maximal ideal $m_{x^{(j)},u}$\\

\multicolumn{2}{c}{\textbf{\S 17}}\\\\
$\tilde \gs$ & perturbed distance-squared function for $\gs$  \\
$\tilde \gs_i$ & perturbed distance-squared function for $\gs_i$ \\
$\tilde \rho_i$ & perturbed multi-distance function for $\rho_i$ \\
$\tilde \nu$ & perturbed distance-squared function for $\nu$  \\
$\tilde \tau$ & perturbed height-distance function for $\tau$\\

\end{longtable}

\end{document}